\documentclass[a4paper,oneside,leqno,12pt]{report}

\usepackage{sectsty}
\usepackage{setspace}
\usepackage{indentfirst}
\usepackage{amsmath}
\usepackage{amsthm}
\usepackage{amssymb}
\usepackage{hyperref}
\usepackage{caption}
\usepackage{subcaption}
\usepackage{graphicx}
\usepackage{lineno}
\usepackage{cleveref}

\chapternumberfont{\huge} 
\chaptertitlefont{\Large}
\sectionfont{\raggedright}

\AtBeginDocument{%
	\abovedisplayskip=9pt plus 3pt minus 6pt
	\abovedisplayshortskip=0pt plus 3pt
	\belowdisplayskip=9pt plus 3pt minus 6pt
	\belowdisplayshortskip=7pt plus 3pt minus 4pt
}

\newtheorem{theorem}{Theorem}
\newtheorem{lemma}[theorem]{Lemma}
\newtheorem{proposition}[theorem]{Proposition}
\newtheorem{hypothesis}[theorem]{Hypothesis}
\newtheorem{conjecture}[theorem]{Conjecture}
\newtheorem{corollary}[theorem]{Corollary}
\newtheorem{definition}[theorem]{Definition}
\newtheorem{problem}[theorem]{Problem}

\newtheorem{procedure}[theorem]{Procedure}
\newtheorem{example}[theorem]{Example}
\newtheorem{remark}[theorem]{Remark}

\newtheorem{namedalgorithmtheorem}{Algorithm}
\newenvironment{namedalgorithm}[2][]
{\renewcommand\thenamedalgorithmtheorem{#2}\begin{namedalgorithmtheorem}[#1]}
	{\end{namedalgorithmtheorem}}

\newcommand{\abs}[1]{|#1|}
\newcommand{\set}[1]{\left\{#1\right\}}

\newcommand{\arrowsv}[0]{\overset{v}{\rightarrow}}
\newcommand{\arrowse}[0]{\overset{e}{\rightarrow}}

\newcommand{\mH}[0]{\mathcal{H}}
\newcommand{\mL}[0]{\mathcal{L}}
\newcommand{\mA}[0]{\mathcal{A}}
\newcommand{\mB}[0]{\mathcal{B}}
\newcommand{\mM}[0]{\mathcal{M}}
\newcommand{\mR}[0]{\mathcal{R}}

\newcommand{\uni}[2]{{#1}\big\vert_{#2}}

\newcommand{\wHv}[3]{\widetilde{\mH}_v(\uni{#1}{#2}; {#3})}
\newcommand{\wHvn}[4]{\widetilde{\mH}_v(\uni{#1}{#2}; {#3}; {#4})}
\newcommand{\wFv}[3]{\widetilde{F}_v(\uni{#1}{#2}; {#3})}

\newcommand{\rp}[0]{{r'_0}}
\newcommand{\rpp}[0]{{r''_0}}

\DeclareMathOperator{\G}{G}
\DeclareMathOperator{\V}{V}
\DeclareMathOperator{\E}{E}
\DeclareMathOperator{\N}{N}

\begin{document}

\begin{center}

{\setstretch{2.0}
	
\noindent\\
\vspace{0em}

\mbox{\textsc{\large Sofia University "St. Kliment Ohridski"}}

\mbox{\textsc{\large Faculty of Mathematics and Informatics}}

\vspace{4em}

{\Large Aleksandar Sotirov Bikov}

\vspace{4em}

\textbf{\MakeUppercase{\large Computation and Bounding of\break Folkman Numbers}}

\vspace{4em}

{\MakeUppercase{\large Thesis}}

\vspace{2.5em}

\vspace{2.5em}

\vspace{2.5em}

\vspace{4em}

{\large PhD supervisor}

{\Large Prof. DSc Nedyalko Nenov}

\vspace{2em}

{\large Sofia, 2018}

}

\end{center}

\thispagestyle{empty}

\clearpage

\pagenumbering{roman}

\setcounter{page}{1}

\tableofcontents

\clearpage

\pagenumbering{arabic}

\setcounter{page}{1}

\addcontentsline{toc}{chapter}{Introduction}
\chapter*{\LARGE Introduction}
\vspace{-0.5em}
Only simple graphs are considered in this thesis, i.e. finite, non-oriented graphs without loops and multiple edges. The vertex set and the edge set of a graph $G$ are denoted by $\V(G)$ and $\E(G)$ respectively. The complete graph on $n$ vertices is denoted by $K_n$.

One fundamental problem in graph theory is the following:
Let $\mH$ be a class of graphs. What is the minimum number of vertices of the graphs in $\mH$?
$$\min\set{\abs{\V(G)} : G \in \mH} = ?$$

For many important classes $\mH$ this problem is still unsolved. In this thesis we consider such problems. In some cases we will compute $\min\set{\abs{\V(G)}}$ exactly, and in other cases we will obtain new bounds on $\min\set{\abs{\V(G)}}$.

It is well known that in every coloring of the edges of the graph $K_6$ in two colors there is a monochromatic triangle. We will denote this property by $K_6 \arrowse (3, 3)$. It is clear that if $G$ contains $K_6$ as a subgraph, then $G \arrowse (3, 3)$. In 1967 Erd\"os and Hajnal posed the following problem:
\begin{center}
	\emph{Does there exist a graph $G \arrowse (3, 3)$ which does not contain $K_6$ ?}
\end{center}

Denote:

$\mH_e(3, 3; q) = \set{G : G \arrowse (3, 3) \mbox{ and } G \not\supseteq K_q}$.\\

\vspace{-1em}
The edge Folkman number $F_e(3, 3; q)$ is defined with:

$F_e(3, 3; q) = \min{\set{\abs{\V(G)} : G \in \mH_e(3, 3; q)}}$\\

\vspace{-0.5em}
From $K_6 \arrowse (3, 3)$ and $K_5 \not\arrowse (3, 3)$ it follows that $F_v(3, 3; q) = 6, \ q \geq 7$. The first example of a graph $G$, such that $G \not\supseteq K_6$ and $G \arrowse (3, 3)$, was given by van Lint. Later, Graham showed that $K_3 + C_5 \arrowse (3, 3)$ and proved $F_e(3, 3; 6) = 8$.

The computation of the numbers $F_e(3, 3; 5)$ and $F_e(3, 3; 4)$ is very hard. The number $F_e(3, 3; 5)$ was finally computed in 1998 after 30 years of history. The upper bound $F_e(3, 3; 5) \leq 15$ was obtained by Nenov in \cite{Nen81a}, who constructed the first 15-vertex graph in $\mH_e(3, 3; 5)$. The lower bound $F_e(3, 3; 5) \geq 15$ was obtained much later with the help of a computer by Piwakowski, Radziszowski and Urbanski in \cite{PRU99}. Without a computer, it is impossible to prove $F_e(3, 3; 5) \geq 15$. A more detailed view on the results related to the number $F_e(3, 3; 5)$ is given in the paper \cite{PRU99} and the book \cite{Soi08}.

\vspace{1em}

The number $F_e(3, 3; 4)$ is not computed. It is sometimes referred to as the most wanted Folkman number. In 1970 Folkman \cite{Fol70} proved that $\mH_e(3, 3; 4) \neq \emptyset$. The graph obtained by the construction of Folkman has a very large number of vertices. Because of this, in 1975 Erd\"os \cite{Erd75} posed the problem to prove the inequality $F_e(3, 3; 4) < 10^{10}$. In 1986 Frankl and R\"odl \cite{FR86} almost solved this problem by showing that $F_e(3, 3; 4) < 7.02 \times 10^{11}$. In 1988 Spencer \cite{Spe88} proved the inequality $F_e(3, 3; 4) < 3 \times 10^9$ by using probabilistic methods. In 2008 Lu \cite{Lu08} constructed a 9697-vertex graph in $\mH_e(3, 3; 4)$, thus considerably improving the upper bound on $F_e(3, 3; 4)$. Soon after that, Lu`s result was improved by Dudek and R\"odl \cite{DR08b}, who proved $F_e(3, 3; 4) \leq 941$. The best known upper bound on this number is $F_e(3, 3; 4) \leq 786$, obtained in 2012 by Lange, Radziszowski and Xu \cite{LRX14}.

In 1972 Lin \cite{Lin72} proved that $F_e(3, 3; 4) \geq 11$. The lower bound was improved by Nenov \cite{Nen83}, who showed in 1981 that $F_e(3, 3; 4) \geq 13$. In 1984 Nenov \cite{Nen84} proved that every 5-chromatic $K_4$-free graph has at least 11 vertices, from which it is easy to derive that $F_e(3, 3; 4) \geq 14$. From $F_e(3, 3; 5) = 15$ \cite{Nen81a}\cite{PRU99} it follows easily, that $F_e(3, 3; 4) \geq 16$. The best lower bound known on $F_e(3, 3; 4)$ was obtained in 2007 by Radziszowski and Xu \cite{RX07}, who proved with the help of a computer that $F_e(3, 3; 4) \geq 19$. According to Radziszowski and Xu \cite{RX07}, any method to improve the bound $F_e(3, 3; 4) \geq 19$ would likely be of significant interest.\\

A summary of the history of $F_e(3, 3; 4)$ is given in Table \ref{table: history of F_e(3, 3; 4)}.\\

\begin{table}[t]
	\centering
	\begin{tabular}{  l  l  l  r  }
		\hline
		\hline
		Year	& Reference										& Lower				& Upper			\\
		\hline
		1970	& Folkman \cite{Fol70}							& 					& exists		\\
		1972	& Lin \cite{Lin72}								& 11				&				\\
		1981	& Nenov \cite{Nen83}							& 13				&				\\
		1984	& Nenov \cite{Nen84}							& 14				&				\\
		1986	& Frankl, R\"odl \cite{FR86}					&			& $7.02 \times 10^{11}$	\\
		1988	& Spencer \cite{Spe88}							&			& $3 \times 10^9$		\\
		1999	& Piwakowski, Radziszowski, Urbanski \cite{PRU99}& 16				&				\\
		2007	& Radziszowski, Xu \cite{RX07}					& 19				& 				\\
		2008	& Lu \cite{Lu08}								& 					& 9697			\\
		2008	& Dudek, R\"odl \cite{DR08b}					& 					& 941			\\
		2012	& Lange, Radziszowski, Xu \cite{LRX14}			& 					& 786			\\
		\hline
	\end{tabular}
	\caption{History of the Folkman number $F_e(3, 3; 4)$}
	\label{table: history of F_e(3, 3; 4)}
\end{table}

In the last Chapter 9 of this thesis we improve the lower bound on $F_e(3, 3; 4)$ by proving $F_e(3, 3; 4) \geq 20$. This is one of the main results in the thesis.

\vspace{1em}

The computation of $F_e(3, 3; q)$ is a special case of the following more general problem:

\textit{
For given positive integers $a_1, ..., a_s$, $q$, $a_i \geq 2, i = 1, ..., s$, determine the minimum number of vertices of the graphs which do not contain the complete graph on $q$ vertices $K_q$ and have the following property: in every coloring of the edges in $s$ colors there exist $i \in \set{1, ..., s}$ such that there is a monochromatic $a_i$-clique of color $i$.
}

This minimum is denoted by $F_e(a_1, ..., a_s; q)$ and is called edge Folkman number. It is known that
\begin{equation}
\label{(introduction)F_e(a_1, ..., a_s; q) exists}
F_e(a_1, ..., a_s; q) \mbox{ exists } \Leftrightarrow q > \max\set{a_1, ..., a_s}.
\end{equation}
In the case $s = 2$, (\ref{(introduction)F_e(a_1, ..., a_s; q) exists}) is proved by Folkman in \cite{Fol70}, and in the general case (\ref{(introduction)F_e(a_1, ..., a_s; q) exists}) is proved by Nesetril and R{\"o}dl in \cite{NR76}.

The numbers $F_e(a_1, ..., a_s; q)$ are a generalization of the classic Ramsey numbers $R(a_1, ..., a_s)$. Furthermore, $F_e(a_1, ..., a_s; q) = R(a_1, ..., a_s)$, if $q > R(a_1, ..., a_s)$.

\vspace{1em}

The vertex Folkman numbers $F_v(a_1, ..., a_s; q)$ are defined in the same way as the edge Folkman numbers $F_e(a_1, ..., a_s; q)$, but instead of coloring the edges, the vertices of the graphs are colored. Very often results for vertex Folkman numbers $F_v(a_1, ..., a_s; q)$ are used in the computation and bounding of the edge Folkman numbers. Let us also note, that the numbers of the form $F_v(2, ..., 2; q)$ determine the minimum number of vertices of the graphs with given chromatic number and given clique number (see \cite{Nen84}, \cite{JR95} and \cite{Goe17}).

This thesis consists of two parts. The first part is dedicated to the vertex Folkman numbers, and the second - to the edge Folkman number $F_e(3, 3; 4)$, which we discussed at the beginning of the introduction. The relation between the two parts is not obvious. The complicated computations for the proof of the inequality $F_e(3, 3; 4) \geq 20$ are done with the help of Algorithm A8. However, this algorithm is a modification of the algorithms from the first part. Therefore, to better understand the connection between the two parts, the algorithms A1, ..., A8 must be carefully followed.

The results in this thesis are obtained using both theoretical and computer methods. Some of the main theoretical results that we obtain are the results in Theorem \ref{theorem: rp}, Theorem \ref{theorem: rpp}, and Theorem \ref{theorem: m_0}. According to these theorems, the computation of some infinite sequences of Folkman numbers is reduced to computing only the first several members. Other main theoretical result is the introduction of the modified vertex Folkman numbers. With the help of these numbers we obtain upper bounds on the vertex Folkman numbers (Theorem \ref{theorem: F_v(2_(m - p), p; q) leq F_v(a_1, ..., a_s; q) leq wFv(m)(p)(q)}).

We develop eight new computer algorithms for computing and bounding Folkman numbers, which we denote by A1, ..., A8. These algorithms are optimized for high performance in terms of computational time. Using the new algorithms, we obtain results which are considered beyond the reach of existing algorithms, even if very powerful computer hardware is used. For example, the computation of the lower bound $F_e(3, 3; 4) \geq 19$ in \cite{RX07} was completed in a few hours, but it is practically impossible to use the same method to further improve this bound. In comparison, using Algorithm A8 on a similarly capable computer, we obtained the result $F_e(3, 3; 4) \geq 19$ in less than a second, and we needed just several hours of computational time to prove the new bound $F_e(3, 3; 4) \geq 20$. At first glance, some of the presented algorithms seem similar, but they have important distinctions. A particular algorithm can produce good results in the problems where it is used, and not be effective in others. Therefore, we had to develop specific algorithms for the different problems considered in this thesis.\\

A more precise description of the main results in each chapter follows:\\

{\large\textbf{Chapter 1}}\\

In this chapter, the necessary graph theory definitions and definitions related to vertex Folkman numbers are given.
 
Let $a_1, ..., a_s$ be positive integers. The expression $G \arrowsv (a_1, ..., a_s)$ means that in every coloring of $\V(G)$ in $s$ colors ($s$-coloring) there exists $i \in \set{1, ..., s}$ such that there is a monochromatic $a_i$-clique of color $i$.

Define:
$$\mH_v(a_1, ..., a_s; q) = \set{ G : G \arrowsv (a_1, ..., a_s) \mbox{ and } G \not\supseteq K_q }.$$

The vertex Folkman numbers $F_v(a_1, ..., a_s; q)$ are defined by the equality:

$$F_v(a_1, ..., a_s; q) = \min\set{\abs{\V(G)} : G \in \mH_v(a_1, ..., a_s; q)}.$$

Folkman proved in \cite{Fol70} that
\begin{equation}
\label{(introduction)equation: F_v(a_1, ..., a_s; q) exists}
F_v(a_1, ..., a_s; q) \mbox{ exists } \Leftrightarrow q > \max\set{a_1, ..., a_s}.
\end{equation}

For arbitrary positive integers $a_1, ..., a_s$ the following terms are defined
\begin{equation}
\label{(introduction)equation: m and p}
m(a_1, ..., a_s) = m = \sum\limits_{i=1}^s (a_i - 1) + 1 \quad \mbox{ and } \quad p = \max\set{a_1, ..., a_s}.
\end{equation}

It is easy to see that $K_m \arrowsv (a_1, ..., a_s)$ and $K_{m - 1} \not\arrowsv (a_1, ..., a_s)$. Therefore
$$F_v(a_1, ..., a_s; q) = m, \ q \geq m + 1.$$

In \cite{LU96} it was proved that
$$F_v(a_1, ..., a_s; m) = m + p.$$

Not much is known about the vertex Folkman numbers $F_v(a_1, ..., a_s; q)$ when $q < m$. Thanks to the work of different authors, the exact values of all numbers of the form $F_v(a_1, ..., a_s; m - 1)$ where $\max\set{a_1, ..., a_s} \leq 4$ were obtained. The only other known number of this form is $F_v(3, 5; 6) = 16$, \cite{SLPX12}.\\

The expression $G \arrowsv \uni{m}{p}$ means that for every choice of positive integers $a_1, ..., a_s$ ($s$ is not fixed), such that $m = \sum\limits_{i=1}^s (a_i - 1) + 1$ and $\max\set{a_1, ..., a_s} \leq p$, we have $G \arrowsv (a_1, ..., a_s)$.

\vspace{1em}
Define:
$$\wHv{m}{p}{q} = \set{G : G \arrowsv \uni{m}{p} \mbox{ and } G \not\supseteq K_q}.$$

The modified vertex Folkman numbers are defined by the equality:
\begin{equation*}
\wFv{m}{p}{q} = \min\set{\abs{\V(G)} : G \in \wHv{m}{p}{q}}.
\end{equation*}

At the end of the chapter, we prove the following main result. For convenience, instead of $F_v(\underbrace{2, ..., 2}_{m - p}, p; q)$ we write $F_v(2_{m - p}, p; q)$.

\vspace{1em}
\textbf{Theorem \ref{theorem: F_v(2_(m - p), p; q) leq F_v(a_1, ..., a_s; q) leq wFv(m)(p)(q)}.}\textit{
	Let $a_1, ..., a_s$ be positive integers, let $m$ and $p$ be defined by (\ref{(introduction)equation: m and p}), and $q > p$. Then,
	$$F_v(2_{m - p}, p; q) \leq F_v(a_1, ..., a_s; q) \leq \wFv{m}{p}{q}.$$
}

Further, we will compute and bound the numbers $F_v(a_1, ..., a_s; q)$ by computing and obtaining bounds on the border numbers in Theorem \ref{theorem: F_v(2_(m - p), p; q) leq F_v(a_1, ..., a_s; q) leq wFv(m)(p)(q)}, $F_v(2_{m - p}, p; q)$ and $\wFv{m}{p}{q}$.

\vspace{4em}

{\large\textbf{Chapter 2}}\\

The vertex Folkman numbers of the form $F_v(a_1, ..., a_s; m - 1)$ where $\max\set{a_1, ..., a_s} = 5$ are considered. By (\ref{(introduction)equation: F_v(a_1, ..., a_s; q) exists}), these numbers exist when $m \geq 7$. In the border case $m = 7$, the only numbers of this form are $F_v(2, 2, 5; 6)$ and $F_v(3, 5; 6)$. It is known that $F_v(3, 5; 6) = 16$ \cite{SLPX12}. We prove

\vspace{1em}
\textbf{Theorem \ref{theorem: F_v(2, 2, 5; 6) = 16}.}\textit{
$F_v(2, 2, 5; 6) = 16$.
}
\vspace{1em}

With the help of the number $F_v(2, 2, 5; 6)$ we compute all other numbers in the infinite sequence $F_v(2_{m - 5}, 5; m - 1), \ m \geq 7,$ by proving

\vspace{1em}
\textbf{Theorem \ref{theorem: rp(5) = 2}.}\textit{
$F_v(2_{m - 5}, 5; m - 1) = m + 9, \ m \geq 7$.
}
\vspace{1em}

\vspace{1em}
We obtain the exact values of all modified vertex Folkman numbers in the form $\wFv{m}{5}{m - 1}$:

\vspace{1em}
\textbf{Theorem \ref{theorem: wFv(m)(5)(m - 1) = ...}.}\textit{
	The following equalities are true:
	$$\wFv{m}{5}{m - 1} =
	\begin{cases}
	17, & \emph{if $m = 7$}\\
	m + 9, & \emph{if $m \geq 8$}.
	\end{cases}
	$$
}
\vspace{1em}

At the end of this chapter, using Theorem \ref{theorem: rp(5) = 2} and Theorem \ref{theorem: wFv(m)(5)(m - 1) = ...}, we complete the computation of the numbers $F_v(a_1, ..., a_s; m - 1)$ where $\max\set{a_1, ..., a_s} = 5$ by obtaining the following main result:

\vspace{2em}
\textbf{Theorem \ref{theorem: F_v(a_1, ..., a_s; m - 1) = m + 9, max set(a_1, ..., a_s) = 5}.}\textit{
	Let $a_1, ..., a_s$ be positive integers, $m = \sum\limits_{i=1}^s (a_i - 1) + 1$, $\max\set{a_1, ..., a_s} = 5$ and $m \geq 7$. Then,
	\begin{equation*}
	F_v(a_1, ..., a_s; m - 1) = m + 9.
	\end{equation*}
}
\vspace{4em}

{\large\textbf{Chapter 3}}\\

According to (\ref{(introduction)equation: F_v(a_1, ..., a_s; q) exists}), the vertex Folkman numbers of the form $F_v(a_1, ..., a_s; m - 1)$ where $\max\set{a_1, ..., a_s} = 6$ exist when $m \geq 8$. In the border case $m = 8$, the only numbers of this form are $F_v(2, 2, 6; 7)$ and $F_v(3, 6; 7)$. We compute the exact values of these numbers by showing that

\vspace{1em}
\textbf{Theorem \ref{theorem: F_v(2, 2, 6; 7) = 17 and abs(mH_v(2, 2, 6; 7; 17)) = 3}.}\textit{
$F_v(2, 2, 6; 7) = 17$.
}
\vspace{1em}

\textbf{Theorem \ref{theorem: F_v(3, 6; 7) = 18}.}\textit{
$F_v(3, 6; 7) = 18$.
}
\vspace{1em}

In \cite{KN06c} Nenov and Kolev pose the following question:

\emph{Does there exist a positive integer $p$ for which $F_v(2, 2, p; p + 1) \neq F_v(3, p; p + 1)$?}

Theorem \ref{theorem: F_v(2, 2, 6; 7) = 17 and abs(mH_v(2, 2, 6; 7; 17)) = 3} and Theorem \ref{theorem: F_v(3, 6; 7) = 18} give a positive answer to this question, and 6 is the smallest possible value for $p$ for which $F_v(2, 2, p; p + 1) \neq F_v(3, p; p + 1)$.\\

With the help of the number $F_v(2, 2, 6; 7)$ we compute all other numbers in the infinite sequence $F_v(2_{m - 6}, 6; m - 1), \ m \geq 8$. We also use the number $F_v(3, 6; 7)$ to compute all other numbers in the infinite sequence $F_v(2_{m - 8}, 3, 6; m - 1)$:

\vspace{1em}
\textbf{Theorem \ref{theorem: rp(6) = 2}.}\textit{
$F_v(2_{m - 6}, 6; m - 1) = m + 9, \ m \geq 8$.
}
\vspace{1em}

\textbf{Theorem \ref{theorem: rpp(6) = 0}.}\textit{
$F_v(2_{m - 8}, 3, 6; m - 1) = m + 10, \ m \geq 8$.
}
\vspace{1em}

We obtain the exact values of all modified vertex Folkman numbers in the form $\wFv{m}{6}{m - 1}$:

\vspace{1em}
\textbf{Theorem \ref{theorem: wFv(m)(6)(m - 1) = m + 10}.}\textit{
	$\wFv{m}{6}{m - 1} = m + 10, \ m \geq 8.$
}
\vspace{1em}

\vspace{1em}

Using Theorem \ref{theorem: rp(6) = 2}, Theorem \ref{theorem: rpp(6) = 0}, and Theorem \ref{theorem: wFv(m)(6)(m - 1) = m + 10}, we complete the computation of the numbers $F_v(a_1, ..., a_s; m - 1)$ where $\max\set{a_1, ..., a_s} = 6$ by obtaining the following main result:

\vspace{1em}
\textbf{Theorem \ref{theorem: F_v(a_1, ..., a_s; m - 1) = ..., max set(a_1, ..., a_s) = 6}.}\textit{
Let $a_1, ..., a_s$ be positive integers such that	$2 \leq a_1 \leq ... \leq a_s = 6$ and $m = \sum\limits_{i=1}^s (a_i - 1) + 1 \geq 8$. Then:
\vspace{1em}\\
(a)	$F_v(a_1, ..., a_s; m - 1) = m + 9$, if $a_1 = ... = a_{s - 1} = 2$.
\vspace{1em}\\
(b)	$F_v(a_1, ..., a_s; m - 1) = m + 10$, if $a_{s - 1} \geq 3$.
}
\vspace{1em}

Since $F_v(a_1, ..., a_s; q)$ is a symmetric function of $a_1, ..., a_s$, in Theorem \ref{theorem: F_v(a_1, ..., a_s; m - 1) = ..., max set(a_1, ..., a_s) = 6} we actually compute all numbers of the form $F_v(a_1, ..., a_s; m - 1)$, where $\max\set{a_1, ..., a_s} = 6$. 

\vspace{4em}

{\large\textbf{Chapter 4}}\\

In this chapter we consider the vertex Folkman numbers of the form $F_v(a_1, ..., a_s; m - 1)$ where $\max\set{a_1, ..., a_s} = 7$. By (\ref{(introduction)equation: F_v(a_1, ..., a_s; q) exists}), these numbers exist when $m \geq 9$. We compute the number

\vspace{1em}
\textbf{Theorem \ref{theorem: F_v(2, 2, 7; 8) = 20}.}\textit{
$F_v(2, 2, 7; 8) = 20$.
}
\vspace{1em}

With the help of the number $F_v(2, 2, 6; 7)$ we compute all other numbers in the infinite sequence $F_v(2_{m - 7}, 7; m - 1), \ m \geq 9$:

\vspace{1em}
\textbf{Theorem \ref{theorem: rp(7) = 2}.}\textit{
$F_v(2_{m - 7}, 7; m - 1) = m + 11, \ m \geq 9$.
}
\vspace{1em}

We obtain the following bounds:

\vspace{1em}
\textbf{Theorem \ref{theorem: F_v(a_1, ..., a_s; m - 1) leq m + 12, max set(a_1, ..., a_s) = 7}.}\textit{
	Let $a_1, ..., a_s$ be positive integers such that $\max\set{a_1, ..., a_s} = 7$ and $m = \sum\limits_{i=1}^s (a_i - 1) + 1 \geq 9$. Then,
	$$m + 11 \leq F_v(a_1, ..., a_s; m - 1) \leq m + 12.$$ 
}

\vspace{4em}

{\large\textbf{Chapter 5}}\\

Very little is known about the vertex Folkman numbers of the form $F_v(a_1, ..., a_s; m - 2)$. The exact values of these numbers are not computed when $\max\set{a_1, ..., a_s} \geq 3$. We obtain new bounds on the two smallest unknown numbers of this form, namely $F_v(2, 2, 2, 3; 4)$ and $F_v(2, 3, 3; 4)$:

\vspace{1em}
\textbf{Theorem \ref{theorem: 20 leq F_v(2, 2, 2, 3; 4) leq 22}.}\textit{
	$20 \leq F_v(2, 2, 2, 3; 4) \leq 22$.
}

\vspace{1em}

\vspace{1em}
\textbf{Theorem \ref{theorem: 20 leq F_v(2, 3, 3; 4) leq 24}.}\textit{							 
	$20 \leq F_v(2, 3, 3; 4) \leq 24.$
}
\vspace{1em}

\vspace{4em}

{\large\textbf{Chapter 6}}\\

Let us remind that:
$$F_v(a_1, ..., a_s; q) \mbox{ exist } \Leftrightarrow q > \max\set{a_1, ..., a_s}.$$
The computation of the numbers of the form $F_v(a_1, ..., a_s; q)$ in the border case $q = \max\set{a_1, ..., a_s} + 1$ is very hard. The numbers $F_v(2, 2, p; p + 1), \ p \leq 4$, and $F_v(3, p; p + 1), \ p \leq 5$, are already computed. We prove that $F_v(2, 2, 5; 6) = 16$ (Theorem \ref{theorem: F_v(2, 2, 5; 6) = 16}), $F_v(2, 2, 6; 7) = 17$ (Theorem \ref{theorem: F_v(2, 2, 6; 7) = 17 and abs(mH_v(2, 2, 6; 7; 17)) = 3}), $F_v(3, 6; 7) = 18$ (Theorem \ref{theorem: F_v(3, 6; 7) = 18}), and $F_v(2, 2, 7; 8) = 20$ (Theorem \ref{theorem: F_v(2, 2, 7; 8) = 20}). The only other computed number of this form is $F_v(2, 2, 2, 2; 3) = 22$, \cite{JR95}.

\vspace{1em}

The numbers $F_v(p, p; p + 1)$ are of significant interest. The only known numbers of this form are $F_v(2, 2; 3) = 5$ and $F_v(3, 3; 4) = 14$, \cite{Nen81a} \cite{PRU99}. The following general lower bounds on these numbers are known:
\begin{equation}
\label{(introduction)equation: F_v(p, p; p + 1) geq 4p - 1}
F_v(p, p; p + 1) \geq 4p - 1. \cite{XS10}
\end{equation}

\vspace{1em}

We obtain the following bounds:
\begin{equation}
\label{(introduction)equation: F_v(p, p; p + 1) geq F_v(2_(p - 1), p; p + 1) geq F_v(2, 2, p; p + 1) + 2p - 6, p geq 3}
F_v(p, p; p + 1) \geq F_v(2, 2, p; p + 1) + 2p - 6, \ p \geq 3.
\end{equation}
In \cite{Nen00} it is proved that $F_v(2, 2, p; p + 1) \geq 2p + 4$. If $F_v(2, 2, p; p + 1) = 2p + 4$, then the inequality (\ref{(introduction)equation: F_v(p, p; p + 1) geq 4p - 1}) gives a better bound for $F_v(p, p; p + 1)$ than the inequality (\ref{(introduction)equation: F_v(p, p; p + 1) geq F_v(2_(p - 1), p; p + 1) geq F_v(2, 2, p; p + 1) + 2p - 6, p geq 3}). It is interesting to note that it is not known whether the equality $F_v(2, 2, p; p + 1) = 2p + 4$ holds for any $p$. If $F_v(2, 2, p; p + 1) = 2p + 5$, then the bounds for $F_v(p, p; p + 1)$ from (\ref{(introduction)equation: F_v(p, p; p + 1) geq 4p - 1}) and (\ref{(introduction)equation: F_v(p, p; p + 1) geq F_v(2_(p - 1), p; p + 1) geq F_v(2, 2, p; p + 1) + 2p - 6, p geq 3}) coincide, and if $F_v(2, 2, p; p + 1) > 2p + 5$, then the inequality (\ref{(introduction)equation: F_v(p, p; p + 1) geq F_v(2_(p - 1), p; p + 1) geq F_v(2, 2, p; p + 1) + 2p - 6, p geq 3}) gives a better bound for $F_v(p, p; p + 1)$.\\

We improve the bounds on the numbers of the form $F_v(p, p; p + 1)$ which follow from (\ref{(introduction)equation: F_v(p, p; p + 1) geq 4p - 1}) and (\ref{(introduction)equation: F_v(p, p; p + 1) geq F_v(2_(p - 1), p; p + 1) geq F_v(2, 2, p; p + 1) + 2p - 6, p geq 3}) in the cases $p = 4, 5, 6, 7$:

\vspace{1em}
\textbf{Theorem \ref{theorem: F_v(4, 4; 5) geq F_v(2, 3, 4; 5) geq F_v(2, 2, 2, 4; 5) geq 19}.}\textit{
	$F_v(4, 4; 5) \geq F_v(2, 3, 4; 5) \geq F_v(2, 2, 2, 4; 5) \geq 19$.
}
\vspace{1em}

\vspace{1em}
\textbf{Theorem \ref{theorem: F_v(5, 5; 6) geq F_v(2, 2, 2, 2, 5; 6) geq 23}.}\textit{
	$F_v(5, 5; 6) \geq F_v(2, 2, 2, 2, 5; 6) \geq 23$.
}
\vspace{1em}

\vspace{1em}
\textbf{Theorem \ref{theorem: F_v(a_1, ..., a_s; 7) leq F_v(6, 6; 7) leq 60}.}\textit{
	Let $a_1, ..., a_s$ be positive integers such that $\max\set{a_1, ..., a_s} = 6$ and $m = \sum\limits_{i=1}^s (a_i - 1) + 1$. Then:
\vspace{1em}\\
	(a) $22 \leq F_v(a_1, ..., a_s; 7) \leq F_v(4, 6; 7) \leq 35$ if $m = 9$.
\vspace{1em}\\
	(b) $28 \leq F_v(a_1, ..., a_s; 7) \leq F_v(6, 6; 7) \leq 70$ if $m = 11$.
}
\vspace{1em}
\vspace{1em}

\textbf{Theorem \ref{theorem: F_v(a_1, ..., a_s; 8) geq 3m - 10}.}\textit{
	If $m \geq 13$ and $\max\set{a_1, ..., a_s} = 7$, then
	$$F_v(a_1, ..., a_s; 8) \geq 3m - 10.$$
	In particular, $F_v(7, 7; 8) \geq 29$.
}

\vspace{4em}

{\large\textbf{Chapter 7}}\\

In this chapter the necessary definitions related to edge Folkman numbers are given. 

Let $a_1, ..., a_s$ be positive integers. The expression $G \arrowse (a_1, ..., a_s)$ means that in every coloring of $\E(G)$ in $s$ colors ($s$-coloring) there exists $i \in \set{1, ..., s}$ such that there is a monochromatic $a_i$-clique of color $i$.

Define:

$\mH_e(a_1, ..., a_s; q) = \set{ G : G \arrowse (a_1, ..., a_s) \mbox{ and } G \not\supseteq K_q }.$

The edge Folkman numbers $F_e(a_1, ..., a_s; q)$ are defined by the equality:
\begin{equation*}
F_e(a_1, ..., a_s; q) = \min\set{\abs{\V(G)} : G \in \mH_e(a_1, ..., a_s; q)}.
\end{equation*}

The numbers $F_e(3, 3; q)$ are of significant interest.

\vspace{4em}

{\large\textbf{Chapter 8}}\\

The graph $G$ is a minimal graph in $\mH_e(3, 3)$ if $G \arrowse (3, 3)$, and its every proper subgraph $H \not\arrowse (3, 3)$. All minimal graphs in $\mH_e(3, 3)$ with up to 9 vertices are known.

We find all minimal graphs in $\mH_e(3, 3)$ with up to 13 vertices.\hspace{2em} (\textbf{Theorem \ref{theorem: 10-vertex minimal graphs in mH_e(3, 3)}}, \textbf{Theorem \ref{theorem: 11-vertex minimal graphs in mH_e(3, 3)}}, \textbf{Theorem \ref{theorem: 12-vertex minimal graphs in mH_e(3, 3)}}, and \textbf{Theorem \ref{theorem: 13-vertex minimal graphs in mH_e(3, 3)}})\\

We also obtain all minimal graphs $G \in \mH_e(3, 3)$ with $\alpha(G) \geq \abs{\V(G)} - 8$, and all minimal graphs in $G \in \mH_e(3, 3; 5)$ with $\alpha(G) \geq \abs{\V(G)} - 9$. Using these results, we derive the following upper bounds on the independence number of the graphs in $\mH_e(3, 3)$:

\vspace{1em}
\textbf{Corollary \ref{corollary: alpha(G) leq abs(V(G)) - 9, G in mH_e(3, 3; 6)}.}\textit{
	Let $G$ be a minimal graph in $\mH_e(3, 3)$ and $\abs{\V(G)} \geq 27$. Then $\alpha(G) \leq \abs{\V(G)} - 9$.
}
\vspace{1em}

\textbf{Corollary \ref{corollary: alpha(G) leq abs(V(G)) - 10, G in mH_e(3, 3; 5)}.}\textit{
	Let $G$ be a minimal graph in $\mH_e(3, 3)$ such that $G \not\supseteq K_5$ and $\abs{\V(G)} \geq 30$. Then $\alpha(G) \leq \abs{\V(G)} - 10$.
}
\vspace{1em}

At the end of the chapter, we obtain the following lower bounds on the minimum degree of the minimal graphs in $\mH_e(3, 3)$.

\vspace{1em}
\textbf{Theorem \ref{theorem: delta(G) geq 5, G in mH_e(3, 3; 5)}.}\textit{
	Let $G$ be a minimal graph in $\mH_e(3, 3; 5)$. Then $\delta(G) \geq 5$. If $v \in \V(G)$ and $d(v) = 5$, then $G(v) = N_{5.i}$ for some $i \in \set{1, 2, 3}$ (see Figure \ref{figure: N_5_1 N_5_2 N_5_3}).
}
\vspace{1em}

\textbf{Theorem \ref{theorem: delta(G) geq 8, G in mH_e(3, 3; 4)}.}\textit{
	Let $G$ be a minimal graph in $\mH_e(3, 3; 4)$. Then $\delta(G) \geq 8$. If $v \in \V(G)$ and $d(v) = 8$, then $G(v) = N_{8.i}$ for some $i \in \set{1, ..., 7}$ (see Figure \ref{figure: N_8_1 N_8_2 N_8_3 N_8_4 N_8_5 N_8_6 N_8_7}).
}

\vspace{4em}

{\large\textbf{Chapter 9}}\\

At the begining of the introduction we presented in more details the history of the edge Folkman number $F_e(3, 3; 4)$. In this chapter we obtain а new lower bound on the number $F_e(3, 3; 4)$:

\vspace{1em}
\textbf{Theorem \ref{theorem: F_e(3, 3; 4) geq 20}.}\textit{
	$F_e(3, 3; 4) \geq 20.$
}
\vspace{1em}

\vspace{4em}
I was introduced to this field by my supervisor Professor N. Nenov. I would like to thank Professor Nenov for his help and support throughout the research and preparation of this thesis.

\vspace{4em}

\clearpage

{\Large\textbf{Publications}}\\

All results of the thesis are published in \cite{Bik16}, \cite{BN15a}, \cite{BN15b}, \cite{BN16}, \cite{Bik17}, \cite{BN17a}, and \cite{BN17b}. The paper $\cite{BN17a}$ is accepted for publication. All other papers are published. Preprints of the papers are published in ResearchGate and arXiv.
Five of these papers are published jointly with my supervisor Professor N. Nenov.

\vspace{2em}

{\Large\textbf{Citations}}\\

The paper \cite{BN16} is cited in \cite{XLR17a}, \cite{XLR17b}, \cite{Rad17}, \cite{RXL17a}, and \cite{RXL17b}.

The paper \cite{BN15a} is cited in \cite{XLR18}.

The paper \cite{BN17a} is cited in \cite{XLR18}.

The paper \cite{BN17b} is cited in \cite{XLR18}.

In \cite{XLR18}, unpublished joint results with N. Nenov are cited as personal communications.

\vspace{2em}

{\Large\textbf{Approbation of the results}}\\

The results of the thesis are presented in the following conferences:

1. Spring Scientific Session of the Faculty of Mathematics and Informatics at Sofia University. Sofia, Bulgaria, March 2016.

2. 45th Spring Conference of the Union of Bulgarian Mathematicians. Pleven, Bulgaria, April 2016.
 
3. VI Congress of Mathematicians of Macedonia. Ohrid, Macedonia, June 2016.

4. 46th Spring Conference of the Union of Bulgarian Mathematicians. Borovets, Bulgaria, April 2017. 

5. Conference dedicated to the 100th Birthday Anniversary of Professor Yaroslav Tagamlitzki. Sofia, Bulgaria, September 2017.

6. National Seminar on Coding Theory "Professor Stefan Dodunekov", Chiflika, Bulgaria, December 2017.

\vspace{2em}

{\Large\textbf{Author's reference}}\\

In the author's opinion, the main contributions of the thesis are:\\

1. A new method for bounding the Folkman numbers $F_v(a_1, ..., a_s; q)$ by $F_v(2_{m - p}, p; q)$ and the modified Folkman numbers $\wFv{m}{p}{q}$ is presented.\\

2. New algorithms for computation and bounding of Folkman numbers are developed (algorithms A1, ..., A8).\\

3. The following Folkman numbers are computed:

$F_v(2, 2, 5; 6) = 16$ (Theorem \ref{theorem: F_v(2, 2, 5; 6) = 16}),

$F_v(2, 2, 6; 7) = 17$ (Theorem \ref{theorem: F_v(2, 2, 6; 7) = 17 and abs(mH_v(2, 2, 6; 7; 17)) = 3}),

$F_v(3, 6; 7) = 18$ (Theorem \ref{theorem: F_v(3, 6; 7) = 18}),

$F_v(2, 2, 7; 8) = 20$ (Theorem \ref{theorem: F_v(2, 2, 7; 8) = 20}).\\

4. The following infinite series of Folkman numbers are computed:

$F_v(a_1, ..., a_s; m - 1)$, where $\max\set{a_1, ..., a_s} = 5$ (Theorem \ref{theorem: F_v(a_1, ..., a_s; m - 1) = m + 9, max set(a_1, ..., a_s) = 5}),

$F_v(a_1, ..., a_s; m - 1)$, where $\max\set{a_1, ..., a_s} = 6$ (Theorem \ref{theorem: F_v(a_1, ..., a_s; m - 1) = ..., max set(a_1, ..., a_s) = 6}),

$F_v(\underbrace{2, ..., 2}_{m - 7}, 7; m - 1)$, where $m \geq 9$ (Theorem \ref{theorem: rp(7) = 2}).\\

5. New bounds on the following Folkman numbers are obtained:

$20 \leq F_v(2, 2, 2, 3; 4) \leq 22$ (Theorem \ref{theorem: 20 leq F_v(2, 2, 2, 3; 4) leq 22}),

$20 \leq F_v(2, 3, 3; 4) \leq 24$ (Theorem \ref{theorem: 20 leq F_v(2, 3, 3; 4) leq 24}).\\

6. New lower bounds on the numbers of the form $F_v(p, p; p + 1)$ are obtained in the cases $p = 4, 5, 6, 7$:

$F_v(4, 4; 5) \geq 19$ (Theorem \ref{theorem: F_v(4, 4; 5) geq F_v(2, 3, 4; 5) geq F_v(2, 2, 2, 4; 5) geq 19}),

$F_v(5, 5; 6) \geq 23$ (Theorem \ref{theorem: F_v(5, 5; 6) geq F_v(2, 2, 2, 2, 5; 6) geq 23}),

$F_v(6, 6; 7) \geq 28$ (Theorem \ref{theorem: F_v(a_1, ..., a_s; 7) leq F_v(6, 6; 7) leq 60}),

$F_v(7, 7; 8) \geq 29$ (Theorem \ref{theorem: F_v(a_1, ..., a_s; 8) geq 3m - 10}).\\

7. All minimal graphs in $\mH_e(3, 3)$ with up to 13 vertices are obtained (Theorem \ref{theorem: 10-vertex minimal graphs in mH_e(3, 3)}, Theorem \ref{theorem: 11-vertex minimal graphs in mH_e(3, 3)}, Theorem \ref{theorem: 12-vertex minimal graphs in mH_e(3, 3)}, and Theorem \ref{theorem: 13-vertex minimal graphs in mH_e(3, 3)}).\\

8. New upper bounds on the independence number of the minimal graphs in $\mH_e(3, 3)$ are obtained.\\

9. New lower bounds on the minimum degree of the minimal graphs in $\mH_e(3, 3)$ are obtained.\\

10. The new lower bound $F_e(3, 3; 4) \geq 20$ is proved.\\

\numberwithin{equation}{chapter}
\numberwithin{theorem}{chapter}

\part{Vertex Folkman Numbers}

\chapter{Definition and properties of the vertex Folkman numbers}

\section{Graph theory notations}

Only finite, non-oriented graphs without loops and multiple edges are considered in this thesis. 

The following notations are used:

$\V(G)$ - the vertex set of $G$;

$\E(G)$ - the edge set of $G$;

$\overline{G}$ - the complement of $G$;

$\omega(G)$ - the clique number of $G$;

$\alpha(G)$ - the independence number of $G$;

$\chi(G)$ - the chromatic number of $G$;

$\N(v), \N_G(v), v \in \V(G)$ - the set of all vertices of G adjacent to $v$;

$d(v), v \in \V(G)$ - the degree of the vertex $v$, i.e. $d(v) = \abs{\N(v)}$;

$\Delta(G)$ - the maximum degree of $G$;

$\delta(G)$ - the minimum degree of $G$;

$G-v, v \in \V(G)$ - subgraph of $G$ obtained from $G$ by deleting the vertex $v$ and all edges incident to $v$;

$G-e, e \in \E(G)$ - subgraph of $G$ obtained from $G$ by deleting the edge $e$;

$G+e, e \in \E(\overline{G})$ - supergraph of G obtained by adding the edge $e$ to $\E(G)$.

$G[W]$ - subgraph of $G$ induced by $W \subseteq \V(G)$;

$G(v)$ - subgraph of $G$ induced by $N_G(v)$;

$K_n$ - complete graph on $n$ vertices;

$C_n$ - simple cycle on $n$ vertices;

$G_1 + G_2$ - graph $G$ for which $\V(G) = \V(G_1) \cup \V(G_2)$ and $\E(G) = \E(G_1) \cup \E(G_2) \cup E'$, where $E' = \set{[x, y] : x \in \V(G_1), y \in \V(G_2)}$, i.e. $G$ is obtained by connecting every vertex of $G_1$ to every vertex of $G_2$.

$Aut(G)$ - the automorphism group of $G$;

$R(p, q)$ - Ramsey number;

$\mR(p, q) = \set{G : \alpha(G) < p \mbox{ and } \omega(G) < q}$.

$\mR(p, q; n) = \set{G : G \in \mR(p, q) \mbox{ and } \abs{\V(G)} = n}$.

All undefined terms can be found in \cite{Wes01}.\\

$\mH_v(a_1, ..., a_s; q)$ and $\mH_v(a_1, ..., a_s; q; n)$ - Section 1.2;

$F_v(a_1, ..., a_s)$ - Section 1.2;

$\mH_{extr}(a_1, ..., a_s; q)$ - Section 1.3;

$\mH_{max}(a_1, ..., a_s; q)$ and $\mH_{max}(a_1, ..., a_s; q; n)$ - Section 1.3;

$\mH_{+K_p}(a_1, ..., a_s; q)$ and $\mH_{+K_p}(a_1, ..., a_s; q; n)$ - Section 1.3;

$\mH_{max}^t(a_1, ..., a_s; q; n)$ and $\mH_{+K_p}^t(a_1, ..., a_s; q; n)$ - Section 1.3;

$\wHv{m}{p}{q}$ and $\wHvn{m}{p}{q}{n}$ - Section 1.4;

$\wFv{m}{p}{q}$ - Section 1.4;

$\mH_e(a_1, ..., a_s; q)$ and $\mH_e(a_1, ..., a_s; q; n)$ - Chapter 7;

$F_e(a_1, ..., a_s)$ - Chapter 7;

$\rp(p) = \rp$ - see Theorem \ref{theorem: rp};

$\rpp(p) = \rpp$ - see Theorem \ref{theorem: rpp};

$m_0(p) = m_0$ - see Theorem \ref{theorem: m_0};

\section{Definition of vertex Folkman numbers and some known results}

Let $a_1, ..., a_s$ be positive integers. The expression $G \arrowsv (a_1, ..., a_s)$ means that in every coloring of $\V(G)$ in $s$ colors ($s$-coloring) there exists $i \in \set{1, ..., s}$ such that there is a monochromatic $a_i$-clique of color $i$.

In particular, $G \arrowsv (a_1)$ means that $\omega(G) \geq a_1$. Further, for convenience, instead of $G \arrowsv (\underbrace{2, ..., 2}_r)$ we write $G \arrowsv (2_r)$ and instead of $G \arrowsv (\underbrace{2, ..., 2}_r, a_1, ..., a_s)$ we write $G \arrowsv (2_r, a_1, ..., a_s)$. It is easy to see that
\begin{equation}
\label{equation: G arrowsv (2_r) Leftrightarrow chi(G) geq r + 1}
G \arrowsv (2_r) \Leftrightarrow \chi(G) \geq r + 1.
\end{equation}

Define:

$\mH_v(a_1, ..., a_s; q) = \set{ G : G \arrowsv (a_1, ..., a_s) \mbox{ and } \omega(G) < q }.$

$\mH_v(a_1, ..., a_s; q; n) = \set{ G : G \in \mH_v(a_1, ..., a_s; q) \mbox{ and } \abs{\V(G)} = n }.$

\begin{remark}
	\label{remark: mH_v(a_1; q; n) = ...}
	In the special case $s = 1$, $G \arrowsv (a_1)$ means that $\omega(G) \geq a_1$, and therefore
	
	$\mH_v(a_1; q; n) = \set{ G : a_1 \leq \omega(G) < q \mbox{ and } \abs{\V(G)} = n }$.
\end{remark}

The vertex Folkman numbers $F_v(a_1, ..., a_s; q)$ are defined by the equality:
\begin{equation*}
F_v(a_1, ..., a_s; q) = \min\set{\abs{\V(G)} : G \in \mH_v(a_1, ..., a_s; q)}.
\end{equation*}

Folkman proves in \cite{Fol70} that:
\begin{equation}
	\label{equation: F_v(a_1, ..., a_s; q) exists}
	F_v(a_1, ..., a_s; q) \mbox{ exists } \Leftrightarrow q > \max\set{a_1, ..., a_s}.
\end{equation}
Other proofs of (\ref{equation: F_v(a_1, ..., a_s; q) exists}) are given in \cite{DR08a} and \cite{LRU01}. In the special case $s = 2$, a very simple proof of this result is given in \cite{Nen85} with the help of corona product of graphs.\\

Obviously, if $a_i = 1$, then
\begin{equation}
	\label{equation: G arrowsv (a_1, ..., a_(i-1), a_(i+1), ..., a_s) Rightarrow G arrowsv (a_1, ..., a_s)}
	G \arrowsv (a_1, ..., a_{i-1}, a_{i+1}, ..., a_s) \Rightarrow G \arrowsv (a_1, ..., a_s),
\end{equation}
and therefore
\begin{equation*}
	F_v(a_1, ..., a_s; q) = F_v(a_1, ..., a_{i-1}, a_{i+1}, ..., a_s; q).
\end{equation*}
It is also clear that
\begin{equation}
\label{equation: G arrowsv (a_1, ..., a_s) Rightarrow G arrowsv (a_(sigma(1)), ..., a_(sigma(s)))}
G \arrowsv (a_1, ..., a_s) \Rightarrow G \arrowsv (a_{\sigma(1)}, ..., a_{\sigma(s)})
\end{equation}
for every permutation $\sigma$ of the numbers $1, ..., s$. Therefore, $F_v(a_1, ..., a_s; q)$ is a symmetric function of $a_1, ..., a_s$, and it is enough to consider only such Folkman numbers $F_v(a_1, ..., a_s; q)$ for which
\begin{equation}
	\label{equation: 2 leq a_1 leq ... leq a_s}
	2 \leq a_1 \leq ... \leq a_s.
\end{equation}
The numbers $F_v(a_1, ..., a_s; q)$ for which the inequalities (\ref{equation: 2 leq a_1 leq ... leq a_s}) hold are called canonical vertex Folkman numbers.

In \cite{LU96} for arbitrary positive integers $a_1, ..., a_s$ the following terms are defined
\begin{equation}
	\label{equation: m and p}
	m(a_1, ..., a_s) = m = \sum\limits_{i=1}^s (a_i - 1) + 1 \quad \mbox{ and } \quad p = \max\set{a_1, ..., a_s}.
\end{equation}
It is easy to see that $K_m \arrowsv (a_1, ..., a_s)$ and $K_{m - 1} \not\arrowsv (a_1, ..., a_s)$. Therefore
\begin{equation*}
	F_v(a_1, ..., a_s; q) = m, \ q \geq m + 1.
\end{equation*}

The following theorem for the numbers $F_v(a_1, ..., a_s; m)$ is true:
\begin{theorem}
	\label{theorem: F_v(a_1, ..., a_s; m) = m + p}
	Let $a_1, ..., a_s$ be positive integers and let $m$ and $p$ be defined by the equalities (\ref{equation: m and p}). If $m \geq p + 1$, then:
	\vspace{1em}\\
	(a)	$F_v(a_1, ..., a_s; m) = m + p$, \cite{LU96},\cite{LRU01}. 
	\vspace{1em}\\
	(b)	If $G \in \mH_v(a_1, ..., a_s; m)$ and $\abs{\V(G)} = m + p$, then $G = K_{m+p} - C_{2p + 1} = K_{m - p - 1} + \overline{C}_{2p + 1}$, \cite{LRU01}.
\end{theorem}
The condition $m \geq p + 1$ is necessary according to (\ref{equation: F_v(a_1, ..., a_s; q) exists}). Other proofs of Theorem \ref{theorem: F_v(a_1, ..., a_s; m) = m + p} are given in \cite{Nen00} and \cite{Nen01a}.\\

Little is known about the numbers $F_v(a_1, ..., a_s; m - 1)$. According to (\ref{equation: F_v(a_1, ..., a_s; q) exists}), we have
\begin{equation}
	\label{equation: F_v(a_1, ..., a_s; m - 1) exists}
	F_v(a_1, ..., a_s; m - 1) \mbox{ exists } \Leftrightarrow m \geq p + 2.
\end{equation}

The following general bounds are known:
\begin{equation}
	\label{equation: m + p + 2 leq F_v(a_1, ..., a_s; m - 1) leq m + 3p}
	m + p + 2 \leq F_v(a_1, ..., a_s; m - 1) \leq m + 3p,
\end{equation}
where the lower bound is true if $p \geq 2$, and the upper bound is true if $p \geq 3$. The lower bound is obtained in \cite{Nen00} and the upper bound is obtained in \cite{KN06a}. In the border case $m = p + 2$ the upper bounds in (\ref{equation: m + p + 2 leq F_v(a_1, ..., a_s; m - 1) leq m + 3p}) are significantly improved in \cite{SXP09}.

Let $m$ and $p$ be defined by (\ref{equation: m and p}). Then,
\begin{equation}
	\label{equation: F_v(a_1, ..., a_s, m - 1) = ...}
	F_v(a_1, ..., a_s, m - 1) = \begin{cases}
		m + 4, & \emph{if $p = 2$ and $m \geq 6$, \cite{Nen83}}\\
		m + 6, & \emph{if $p = 3$ and $m \geq 6$, \cite{Nen02b}}\\
		m + 7, & \emph{if $p = 4$ and $m \geq 6$, \cite{Nen02b}}.
	\end{cases}
\end{equation}
The remaining canonical numbers $F_v(a_1, ..., a_s; m - 1)$, $p \leq 4$ are: $F_v(2, 2, 2; 3) = 11$, \cite{Myc55} and \cite{Chv79}, $F_v(2, 2, 2, 2; 4) = 11$, \cite{Nen84} (see also \cite{Nen98}), $F_v(2, 2, 3; 4) = 14$, \cite{Nen81a} and \cite{CR06}, $F_v(3, 3; 4) = 14$, \cite{Nen81a} and \cite{PRU99}. From these facts it becomes clear that we know all Folkman numbers of the form $F_v(a_1, ..., a_s; m - 1)$ when $\max\set{a_1, ..., a_s} \leq 4$. It is also known that $F_v(3, 5; 6) = 16$, \cite{SLPX12}. \\

It is easy to see that in the border case $m = p + 2$ when $p \geq 3$ there are only two canonical numbers of the form $F_v(a_1, ..., a_s; m - 1)$, namely $F_v(2, 2, p; p + 1)$ and $F_v(3, p; p+1)$. Some graphs, with which upper bounds for $F_v(3, p; p + 1)$ are obtained, can be used for obtaining general upper bounds for $F_v(a_1, ..., a_s; m - 1)$. For example, the graph $\Gamma_p$ from \cite{Nen00}, which witnesses the bound $F_v(3, p; p + 1) \leq 4p + 2, p \geq 3$, helps to obtain the upper bound in (\ref{equation: m + p + 2 leq F_v(a_1, ..., a_s; m - 1) leq m + 3p}). With the help of the numbers $F_v(2, 2, p; p + 1)$, the general lower bound in (\ref{equation: m + p + 2 leq F_v(a_1, ..., a_s; m - 1) leq m + 3p}) can be improved (see Theorem \ref{theorem: F_v(2, 2, p; p + 1) leq 2p + 5, then ...}). In this thesis we will clarify more thoroughly the role of the numbers $F_v(2, 2, p; p + 1)$ in obtaining lower bounds on the numbers $F_v(a_1, ..., a_s; m - 1)$. Thus, obtaining bounds for the numbers $F_v(a_1, ..., a_s; m - 1)$ and computing some of them is related with computation and obtaining bounds for the numbers $F_v(2, 2, p; p + 1)$ and $F_v(3, p; p + 1)$.
It is easy to see that $G \arrowsv (3, p) \Rightarrow G \arrowsv (2, 2, p)$. Therefore, the following inequality holds:
\begin{equation}
\label{equation: F_v(2, 2, p; p + 1) leq F_v(3, p; p + 1)}
F_v(2, 2, p; p + 1) \leq F_v(3, p; p + 1), \ p \geq 3.
\end{equation}

\vspace{1em}

In \cite{KN06c} Nenov and Kolev pose the following problem:
\begin{problem}
	\label{problem: F_v(2, 2, p; p + 1) < F_v(3, p; p + 1)}
	\cite{KN06c}
	Does there exist a positive integer $p$ for which the inequality (\ref{equation: F_v(2, 2, p; p + 1) leq F_v(3, p; p + 1)}) is strict?
\end{problem}
As we noted after (\ref{equation: F_v(a_1, ..., a_s, m - 1) = ...}), if $p = 3$ or $4$, $F_v(2, 2, p; p + 1) = F_v(3, p; p + 1)$. Further, we will prove that $F_v(2, 2, 5; 6) = F_v(3, 5; 6) = 16$ (Theorem \ref{theorem: F_v(2, 2, 5; 6) = 16} and Corollary \ref{corollary: F_v(3, 5; 6) = 16}), but $F_v(2, 2, 6; 7) = 17$ while $F_v(3, 6; 7) = 18$ (Theorem \ref{theorem: F_v(2, 2, 6; 7) = 17 and abs(mH_v(2, 2, 6; 7; 17)) = 3} and Theorem \ref{theorem: F_v(3, 6; 7) = 18}). Therefore, $p = 6$ is the smallest positive integer for which the inequality (\ref{equation: F_v(2, 2, p; p + 1) leq F_v(3, p; p + 1)}) is strict.\\

In the following Chapters 2 and 3 we compute the numbers $F_v(a_1, ..., a_s; m - 1)$ where $\max\set{a_1, ..., a_s} =$ $5$ or $6$ (Theorem \ref{theorem: F_v(a_1, ..., a_s; m - 1) = m + 9, max set(a_1, ..., a_s) = 5} and Theorem \ref{theorem: F_v(a_1, ..., a_s; m - 1) = ..., max set(a_1, ..., a_s) = 6}). In Chapter 4 we compute an infinite sequence of numbers of the form $F_v(a_1, ..., a_s; m - 1)$ where $\max\set{a_1, ..., a_s} = 7$.

Not much is known about the numbers $F_v(a_1, ..., a_s; q)$ when $q \leq m - 2$. In Chapter 5 we obtain new bounds on some numbers of the form $F_v(a_1, ..., a_s; m - 2)$.

New bounds on some numbers of the form $F_v(a_1, ..., a_s; q)$ where $q = \max\set{a_1, ..., a_s} + 1$ are obtained in Chapter 6.

\section{Auxiliary definitions and propositions}

Let $a_1, ..., a_s$ be positive integers and $m = \sum\limits_{i=1}^s (a_i - 1) + 1$.

Obviously, if $a_i \geq t \geq 2$, then
\begin{equation}
\label{equation: G arrowsv (a_1, ..., a_s) Rightarrow G arrowsv (a_1, ..., a_(i - 1), t, a_i - t + 1, a_(i + 1), ..., a_s)}
G \arrowsv (a_1, ..., a_s) \Rightarrow G \arrowsv (a_1, ..., a_{i - 1}, t, a_i - t + 1, a_{i + 1}, ..., a_s).
\end{equation}
In the special case $t = 2$ we have
\begin{equation}
\label{equation: G arrowsv (a_1, ..., a_s) Rightarrow G arrowsv (a_1, ..., a_(i - 1), 2, a_i - 1, a_(i + 1), ..., a_s)}
G \arrowsv (a_1, ..., a_s) \Rightarrow G \arrowsv (a_1, ..., a_{i - 1}, 2, a_i - 1, a_{i + 1}, ..., a_s).
\end{equation}
If some of the numbers $a_1, ..., a_{i - 1}, 2, a_i - 1, a_{i + 1}, ..., a_s$ are greater than 2, we apply (\ref{equation: G arrowsv (a_1, ..., a_s) Rightarrow G arrowsv (a_1, ..., a_(i - 1), 2, a_i - 1, a_(i + 1), ..., a_s)}) again. Thus, by repeatedly applying (\ref{equation: G arrowsv (a_1, ..., a_s) Rightarrow G arrowsv (a_1, ..., a_(i - 1), t, a_i - t + 1, a_(i + 1), ..., a_s)}) we obtain
\begin{equation}
\label{equation: G arrowsv (a_1, ..., a_s) Rightarrow G arrowsv (2_(m - 1))}
G \arrowsv (a_1, ..., a_s) \Rightarrow G \arrowsv (2_{m - 1}).
\end{equation}
From (\ref{equation: G arrowsv (a_1, ..., a_s) Rightarrow G arrowsv (2_(m - 1))}) and (\ref{equation: G arrowsv (2_r) Leftrightarrow chi(G) geq r + 1}) it follows
\begin{equation}
\label{equation: G arrowsv (a_1, ..., a_s) Rightarrow chi(G) geq m}
G \arrowsv (a_1, ..., a_s) \Rightarrow \chi(G) \geq m, \cite{Nen01a}.
\end{equation}
Simple examples of graphs for which equality in (\ref{equation: G arrowsv (a_1, ..., a_s) Rightarrow chi(G) geq m}) is reached are obtained in \cite{Nen85} and \cite{XLR18}.

\vspace{1em}

If $a_s = p$, by applying the method for obtaining (\ref{equation: G arrowsv (a_1, ..., a_s) Rightarrow G arrowsv (2_(m - 1))}) only for the numbers $a_1, ..., a_{s - 1}$ we obtain
\begin{equation}
\label{equation: G arrowsv (a_1, ..., a_s) Rightarrow G arrowsv (2_(m - p), p)}
G \arrowsv (a_1, ..., a_s) \Rightarrow G \arrowsv (2_{m - p}, p),
\end{equation}
If additionally $a_{s - 1} \geq 3$, using the same method we obtain
\begin{equation}
\label{equation: G arrowsv (a_1, ..., a_s) Rightarrow G arrowsv (2_(m - p - 2), 3, p)}
G \arrowsv (a_1, ..., a_s) \Rightarrow G \arrowsv (2_{m - p - 2}, 3, p).
\end{equation}

\vspace{1em}

It is easy to prove the following
\begin{proposition}
	\label{proposition: G - A arrowsv (a_1, ..., a_(i - 1), a_i - 1, a_(i + 1_, ..., a_s)}
	Let $G$ be a graph, $G \arrowsv (a_1, ..., a_s)$, and $A \subseteq \V(G)$ be an independent set. Then, if $a_i \geq 2$
	\begin{equation*}
	G - A \arrowsv (a_1, ..., a_{i - 1}, a_i - 1, a_{i + 1}, ..., a_s).
	\end{equation*}
\end{proposition}

\vspace{1em}

We will need the following theorem which gives an upper bound on the independence number of the graphs in $\mH_v(a_1, ..., a_s; m - 1)$.
\begin{theorem} \cite{Nen02a}
\label{theorem: alpha(G) leq V(G) + m + p - 1}
Let $a_1, ..., a_s$ be positive integers and $m$ and $p$ defined by (\ref{equation: m and p}). Let $G \in \mH_v(a_1, ..., a_s; m - 1)$. Then,
\vspace{1em}\\
(a) $\alpha(G) \leq \abs{\V(G)} - m - p + 1$.
\vspace{1em}\\
(b) If $\abs{\V(G)} < m + 3p$, then $\alpha(G) \leq \abs{\V(G)} - m - p$.
\end{theorem}

\vspace{1em}

The graph $G$ is called an extremal graph in $\mH_v(a_1, ..., a_s; q)$ if $G \in \mH_v(a_1, ..., a_s; q)$ and $\abs{\V(G)} = F_v(a_1, ..., a_s; q)$. We denote by $\mH_{extr}(a_1, ..., a_s; q)$ the set of all extremal graphs in $\mH_v(a_1, ..., a_s; q)$.

\begin{conjecture}
	\label{conjecture: chi(G) leq m + 1}
	If $G \in \mH_{extr}(a_1, ..., a_s; m - 1)$, then $\chi(G) \leq m + 1$.
\end{conjecture}
For all known examples of extremal graphs, including the extremal graphs obtained in this thesis, this inequality holds.\\

The graph $G$ is called a $K_q$-free graph if $\omega(G) < q$. $G$ is a maximal $K_q$-free graph if $\omega(G) < q$ and $\omega(G + e) = q, \forall e \in \E(\overline{G})$.

Let $W \subseteq \V(G)$. $W$ is a maximal $K_q$-free subset of $\V(G)$ if $W$ does not contain a $q$-clique and $W \cup \set{v}$ contains a $q$-clique for all $v \in \V(G) \setminus W$.

We say that $G$ is a maximal graph in $\mH_v(a_1, ..., a_s; q)$ if $G \in \mH_v(a_1, ..., a_s; q)$ but $G + e \not\in \mH_v(a_1, ..., a_s; q), \forall e \in \E(\overline{G})$, i.e. $\omega(G + e) = q, \forall e \in \E(\overline{G})$. $G$ is a minimal graph in $\mH_v(a_1, ..., a_s; q)$ if $G \in \mH_v(a_1, ..., a_s; q)$ but $G - e \not\in \mH_v(a_1, ..., a_s; q), \forall e \in \E(G)$, i.e. $G - e \not\arrowsv (a_1, ..., a_s), \forall e \in \E(G)$.

If $G$ is both maximal and minimal graph in $\mH_v(a_1, ..., a_s; q)$, then we say that $G$ is a bicritical graph in $\mH_v(a_1, ..., a_s; q)$. 

\vspace{1em}

For convenience, we also define the following term:
\begin{definition}
	\label{definition: (+K_p)}
	The graph $G$ is called a $(+K_p)$-graph if $G + e$ contains a new $p$-clique for all $e \in \E(\overline{G})$.
\end{definition}
Obviously, $G \in \mH_v(a_1, ..., a_s; q)$ is a maximal graph in $\mH_v(a_1, ..., a_s; q)$ if and only if $G$ is a $(+K_q)$-graph. We denote by $\mH_{max}(a_1, ..., a_s; q)$ the set of all maximal graphs in $\mH_v(a_1, ..., a_s; q)$ and by $\mH_{+K_p}(a_1, ..., a_s; q)$ the set of all $(+K_p)$-graphs in $\mH_v(a_1, ..., a_s; q)$. We will also use the following notations:

\vspace{0.5em}

$\mH_{+K_p}(a_1, ..., a_s; q; n) = \set{G \in \mH_{+K_p}(a_1, ..., a_s; q) : \abs{\V(G)} = n}$

\vspace{0.5em}

$\mH_{max}(a_1, ..., a_s; q; n) = \set{G \in \mH_{max}(a_1, ..., a_s; q) : \abs{\V(G)} = n}$

\vspace{0.5em}

$\mH_{+K_p}^t(a_1, ..., a_s; q; n) = \set{G \in \mH_{+K_p}(a_1, ..., a_s; q; n) : \alpha(G) \leq t}$

\vspace{0.5em}

$\mH_{max}^t(a_1, ..., a_s; q; n) = \set{G \in \mH_{max}(a_1, ..., a_s; q; n) : \alpha(G) \leq t}$

\begin{remark}
	\label{remark: mH_(max)(a_1; q; n)}
	If $a_1 \leq n \leq q - 1$ then $\mH_{max}(a_1; q; n) = \set{K_n}$.
	
	If $a_1 \leq q - 1 \leq n$ and $G \in \mH_{max}(a_1; q; n)$, then $\omega(G) = q - 1$, and therefore $\mH_{max}(a_1; q; n) = \mH_{max}(q - 1; q; n).$
\end{remark}

Further, we will need the following 
\begin{proposition}
	\label{proposition: H is a (+K_(q - 1))-graph}
	Let $G$ be a maximal $K_q$-free graph, $A$ be an independent set of vertices of $G$, and $H = G - A$. Then $H$ is a $(+K_{q - 1})$-graph.
\end{proposition}

\begin{proof}
Suppose that for some $e \in \overline{\E}(H)$, $e + H$ does not contain a new $(q - 1)$-clique. Then $e + G$ does not contain a $q$-clique.
\end{proof}

For convenience, we will formulate the following proposition which follows directly from Proposition \ref{proposition: G - A arrowsv (a_1, ..., a_(i - 1), a_i - 1, a_(i + 1_, ..., a_s)} and Proposition \ref{proposition: H is a (+K_(q - 1))-graph}:
\begin{proposition}
	\label{proposition: G - A in mH_(+K_(q - 1))(a_1 - 1, a_2, ..., a_s; q; n - abs(A))}
	Let $G \in \mH_{max}(a_1, ..., a_s; q; n)$, $a_1 \geq 2$, and $A$ be an independent set of vertices of $G$. Then,
	$$G - A \in \mH_{+K_{q - 1}}(a_1 - 1, a_2, ..., a_s; q; n - \abs{A}).$$
\end{proposition}
We will use Proposition \ref{proposition: G - A in mH_(+K_(q - 1))(a_1 - 1, a_2, ..., a_s; q; n - abs(A))} in many of the algorithms.\\

We will also need the following improvement of the lower bound in (\ref{equation: m + p + 2 leq F_v(a_1, ..., a_s; m - 1) leq m + 3p})
\begin{theorem}
	\label{theorem: F_v(a_1, ..., a_s; m - 1) geq m + p + 3}
	\cite{Nen02a}
	Let $a_1, ..., a_s$ be positive integers, let $m$ and $p$ be defined by the equalities (\ref{equation: m and p}), $p \geq 3$, and $m \geq p + 2$. If $F_v(2, 2, p; p + 1) \geq 2p + 5$, then
	\begin{equation*}
	F_v(a_1, ..., a_s; m - 1) \geq m + p + 3.
	\end{equation*}
\end{theorem}

\section{Modified vertex Folkman numbers}

In this section we will define the modified vertex Folkman numbers, which help us to obtain upper bounds on the vertex Folkman numbers.

\begin{definition}
	\label{definition: G arrowsv uni(m)(p)}
	Let $G$ be a graph and let $m$ and $p$ be positive integers. The expression
	\begin{equation*}
	G \arrowsv \uni{m}{p}
	\end{equation*}
	means that for every choice of positive integers $a_1, ..., a_s$ ($s$ is not fixed), such that $m = \sum\limits_{i=1}^s (a_i - 1) + 1$ and $\max\set{a_1, ..., a_s} \leq p$, we have
	\begin{equation*}
	G \arrowsv (a_1, ..., a_s).
	\end{equation*}
\end{definition}

\begin{example}
	\label{example: K_m arrowsv uni(m)(p)}
	Obviously, $K_m \arrowsv \uni{m}{p}, \ \forall p$.
\end{example}

\begin{example}
	\label{example: overline(C)_(2p + 1) arrowsv uni(p + 1)(p)}
	\cite{LRU01}
	Let us notice that $\overline{C}_{2p + 1} \arrowsv \uni{(p + 1)}{p}$. Indeed, let $b_1, ..., b_s$ be positive integers such that $\sum\limits_{i = 1}^{s}(b_i - 1) + 1 = p + 1$ and $\max\set{b_1, ..., b_s} \leq p$. Assume that there exists $s$-coloring $\V(G) = V_1 \cup ... \cup V_s$ such that $V_i$ does not contain a $b_i$-clique. Then $\abs{V_i} \leq 2b_i - 2$ and $\abs{\V(G)} = \sum\limits_{i = 1}^{s}\abs{V_i} \leq 2\sum\limits_{i = 1}^{s}(b_i - 1) = 2p$, which is a contradiction.
\end{example}

Further, we will need the following
\begin{lemma}
	\label{lemma: K_(m - m_0) + G arrowsv uni(m)(p)}
	\cite{Nen02b}
	Let $m_0$ and $p$ be positive integers and $G \arrowsv \uni{m_0}{p}$. Then for every positive integer $m \geq m_0$ it is true that $K_{m - m_0} + G \arrowsv \uni{m}{p}$.
\end{lemma}
This lemma is formulated in an obviously equivalent way and is proved by induction with respect to $m \geq m_0$ in \cite{Nen02b} as Lemma 3.\\

Let $G$ be a graph and $m$ and $p$ be positive integers. If $m < p$, then from Example \ref{example: K_m arrowsv uni(m)(p)} it follows easily that $G \arrowsv \uni{m}{p}$ if and only if $\omega(G) \geq m$. With the help of the following proposition we can significantly reduce the number of checks required to determine if a graph $G$ has the property $G \arrowsv \uni{m}{p}$ (see the proofs of Theorem \ref{theorem: wFv(m)(5)(m - 1) = ...} and Theorem \ref{theorem: wFv(m)(6)(m - 1) = m + 10}).

\begin{proposition}
	\label{proposition: a_1 + a_2 - 1 > p}
	Let $G$ be a graph and let $m > p$ be positive integers. If $G \arrowsv (a_1, ..., a_s)$ for every choice of positive integers $a_1, ..., a_s, \ s \geq 2$ ($s$ is not fixed), such that $m = \sum\limits_{i=1}^s (a_i - 1) + 1$, $2 \leq a_1 \leq ... \leq a_s \leq p$, and $a_1 + a_2 - 1 > p$, then $G \arrowsv \uni{m}{p}$.
\end{proposition}

\begin{proof}
	Let the graph $G$ satisfy the condition of the proposition. We will prove via induction by $s$ that $G \arrowsv (a_1, ..., a_s)$ for every choice of $s$ positive integers $a_1, ..., a_s$ such that $m = \sum\limits_{i=1}^s (a_i - 1) + 1$, $\max\set{a_1, ..., a_s} \leq p$, and $\min\set{a_1, ..., a_s} \geq 2$. Then from (\ref{equation: G arrowsv (a_1, ..., a_(i-1), a_(i+1), ..., a_s) Rightarrow G arrowsv (a_1, ..., a_s)}) it follows that $G \arrowsv \uni{m}{p}$. According to (\ref{equation: G arrowsv (a_1, ..., a_s) Rightarrow G arrowsv (a_(sigma(1)), ..., a_(sigma(s)))}), it is enough to prove that $G \arrowsv (a_1, ..., a_s)$ when $2 \leq a_1 \leq ... \leq a_s \leq p$.
		
	Since $\sum\limits_{i=1}^s (a_i - 1) + 1 = m > p$ and $a_1 \leq p$, it follows that $s \geq 2$. Therefore, the basis of the induction is the case $s = 2$.
	
	If $s = 2$, then $(a_1 - 1) + (a_2 - 1) + 1 = a_1 + a_2 - 1 = m > p$, and by the condition of the proposition, $G \arrowsv (a_1, a_2)$.
	
	Let $s > 3$. If $a_1 + a_2 - 1 > p$, then by the condition of the proposition, $G \arrowsv (a_1, ..., a_s)$. Now, consider the case $a_1 + a_2 - 1 \leq p$. Let $\set{b_1, ..., b_{s - 1}} = \set{a_1 + a_2 - 1, a_3, ..., a_s}$. Then $\sum\limits_{i=1}^{s - 1} (b_i - 1) + 1 = \sum\limits_{i=1}^{s} (a_i - 1) + 1 = m$. By the induction hypothesis, $G \arrowsv (b_1, ..., b_{s - 1})$. Obviously, if $V_1 \cup V_2 \cup ... \cup V_s$ is an $(a_1, a_2, ..., a_s)$-free coloring of $\V(G)$, i.e. $V_i$ does not contain an $a_i$-clique, then $(V_1 \cup V_2) \cup ... \cup V_s$ is an $(a_1 + a_2 - 1, a_3, ..., a_s)$-free coloring of $\V(G)$. Therefore, from $G \arrowsv (b_1, ..., b_{s - 1})$ it follows that $G \arrowsv (a_1, ..., a_s)$. Thus, the proposition is proved.
\end{proof}

Define:

$\wHv{m}{p}{q} = \set{G : G \arrowsv \uni{m}{p} \mbox{ and } \omega(G) < q}$.

$\wHvn{m}{p}{q}{n} = \set{G : G \in \wHv{m}{p}{q} \mbox{ and } \abs{\V(G)} = n}$.\\

The modified vertex Folkman numbers are defined by the equality:
\begin{equation*}
\wFv{m}{p}{q} = \min\set{\abs{\V(G)} : G \in \wHv{m}{p}{q}}.
\end{equation*}

\begin{proposition}
	\label{proposition: wFv(m)(p)(q) exists}
	$\wHv{m}{p}{q} \neq \emptyset$, i.e. $\wFv{m}{p}{q}$ exists $\Leftrightarrow$ $q > \min\set{m, p}$.
\end{proposition}

\begin{proof}
	Let $\wHv{m}{p}{q} \neq \emptyset$ and $G \in \wHv{m}{p}{q}$. If $m \leq p$, then $G \arrowsv (m)$, and it follows $\omega(G) \geq m$. Since $\omega(G) < q$, we obtain $q > m$. Let $m > p$. Then there exist positive integers $a_1, ..., a_s$, such that $m = \sum_{i = 1}^s (a_i - 1) + 1$ and $p = \max\set{a_1, ..., a_s}$, for example $a_1 = ... = a_{m - p} = 2$ and $a_{m - p + 1} = p$. Since $G \arrowsv (a_1, ..., a_s)$, it follows that $\omega(G) \geq p$ and $q > p$. Therefore, if $\wHv{m}{p}{q} \neq \emptyset$, then $q > \min\set{m, p}$.
	
	Let $q > \min\set{m, p}$. If $m \geq p$, then $q > p$. According to (\ref{equation: F_v(a_1, ..., a_s; q) exists}), for every choice of positive integers $a_1, ..., a_s$, such that $m = \sum_{i = 1}^s (a_i - 1) + 1$ and $\max\set{a_1, ..., a_s} \leq p$, there exists a graph $G(a_1, ..., a_s) \in \mH_v(a_1, ..., a_s; q)$. Let $G$ be the union of all graphs $G(a_1, ..., a_s)$. It is clear that $G \in \wHv{m}{p}{q}$. If $m \leq p$, then $m < q$, and therefore $K_m \in \wHv{m}{p}{q}$.
\end{proof}

The following theorem for the modified Folkman numbers is true:
\begin{theorem}
	\label{theorem: wFv(m)(p)(m - m_0 + q) leq wFv(m_0)(p)(q) + m - m_0}
	Let $m$, $m_0$, $p$ and $q$ be positive integers, $m \geq m_0$, and $q > \min\set{m_0, p}$. Then,
	\begin{equation*}
	\wFv{m}{p}{m - m_0 + q} \leq \wFv{m_0}{p}{q} + m - m_0.
	\end{equation*}
\end{theorem}

\begin{proof}
Let $G_0 \in \wHv{m_0}{p}{q}$, $\abs{V(G_0)} = \wFv{m_0}{p}{q}$ and $G = K_{m - m_0} + G_0$. According to Lemma \ref{lemma: K_(m - m_0) + G arrowsv uni(m)(p)}, $G \overset{v}{\rightarrow} \uni{m}{p}$. Since $\omega(G) = m - m_0 + \omega(G_0) < m - m_0 + q$, it follows that $G \in \wHv{m}{p}{m - m_0 + q}$. Therefore, $\wFv{m}{p}{m - m_0 + q} \leq \abs{\V(G)} = \wFv{m_0}{p}{q} + m - m_0$.
\end{proof}

We see that if we know the value of one number $\wFv{m_0}{p}{q}$ we can obtain an upper bound for $\wFv{m}{p}{m - m_0 + q}$ where $m \geq m_0$.

From the definition of the modified Folkman numbers it becomes clear that if $a_1, ..., a_s$ are positive integers and $m$ and $p$ are defined by (\ref{equation: m and p}), then 
\begin{equation}
\label{equation: F_v(a_1, ..., a_s; q) leq wFv(m)(p)(q)}
F_v(a_1, ..., a_s; q) \leq \wFv{m}{p}{q}.
\end{equation}
Defining and computing the modified Folkman numbers is appropriate because of the following reasons:

1) On the left side of (\ref{equation: F_v(a_1, ..., a_s; q) leq wFv(m)(p)(q)}) there is actually a whole class of numbers, which are bound by only one number $\wFv{m}{p}{q}$.

2) The upper bound for $\wFv{m}{p}{q}$ is easier to compute than the numbers $F_v(a_1, ..., a_s)$, according to Theorem \ref{theorem: wFv(m)(p)(m - m_0 + q) leq wFv(m_0)(p)(q) + m - m_0}.

3) As we will see further in Theorem \ref{theorem: m_0}, the computation of the numbers $\wFv{m}{p}{m - 1}$ is reduced to finding the exact values of the first several of these numbers (bounds for the number of exact values needed are given in Theorem \ref{theorem: m_0} (c)).\\

The following theorem gives bounds for the numbers $F_v(a_1, ..., a_s; q)$. Let us remind that $F_v(2_{m - p}, p; q) = F_v(\underbrace{2, ..., 2}_{m - p}, p, q)$.

\begin{theorem}
	\label{theorem: F_v(2_(m - p), p; q) leq F_v(a_1, ..., a_s; q) leq wFv(m)(p)(q)}
	Let $a_1, ..., a_s$ be positive integers, let $m$ and $p$ be defined by (\ref{equation: m and p}), and $q > p$. Then,
	\begin{equation*}
	F_v(2_{m - p}, p; q) \leq F_v(a_1, ..., a_s; q) \leq \wFv{m}{p}{q}.
	\end{equation*}
\end{theorem}

\begin{proof}
	The right inequality follows from the inclusion
	\begin{equation*}
	\wHv{m}{p}{q} \subseteq \mH_v(a_1, ..., a_s; q).
	\end{equation*}
	From (\ref{equation: G arrowsv (a_1, ..., a_s) Rightarrow G arrowsv (2_(m - p), p)}) it follows
	\begin{equation*}
	F_v(a_1, ..., a_s; q) \geq F_v(2_{m - p}, p; q).
	\qedhere
	\end{equation*}
\end{proof}

\begin{remark}
	It is easy to see that if $q > m$, then $F_v(a_1, ..., a_s; q) = \wFv{m}{p}{q} = m$. From Theorem \ref{theorem: F_v(a_1, ..., a_s; m) = m + p} it follows $F_v(a_1, ..., a_s; m) = \wFv{m}{p}{m} = m + p$. If $q = m - 1$ and $p \leq 4$, according to (\ref{equation: F_v(a_1, ..., a_s, m - 1) = ...}), we also have $F_v(a_1, ..., a_s; q) = \wFv{m}{p}{q}$. The first case in which the upper bound in Theorem \ref{theorem: F_v(2_(m - p), p; q) leq F_v(a_1, ..., a_s; q) leq wFv(m)(p)(q)} is not reached is $m = 7, p = 5, q = 6$, since $\wFv{7}{5}{6} = 17$ (see Theorem \ref{theorem: wFv(m)(5)(m - 1) = ...}) and the corresponding numbers $F_v(a_1, ..., a_s; q)$ are not greater than 16.
\end{remark}

Other modifications of the Folkman numbers are defined and studied in \cite{XLR17a}, \cite{XLR17b}, \cite{RXL17b}, \cite{EG18}, \cite{LL15}, \cite{LL17}, and \cite{KWR17}. These modifications of the Folkman numbers are called generalized Folkman numbers. In this thesis we do not consider generalized Folkman numbers.

\vspace{1em}
Theorem \ref{theorem: wFv(m)(p)(m - m_0 + q) leq wFv(m_0)(p)(q) + m - m_0} and Theorem \ref{theorem: F_v(2_(m - p), p; q) leq F_v(a_1, ..., a_s; q) leq wFv(m)(p)(q)} are published in \cite{BN15a}.

\chapter{Computation of the numbers $F_v(a_1, ..., a_s; m - 1)$ where $\max\{a_1, ..., a_s\} = 5$}

In this chapter we will prove the following main result:

\begin{theorem}
	\label{theorem: F_v(a_1, ..., a_s; m - 1) = m + 9, max set(a_1, ..., a_s) = 5}
	Let $a_1, ..., a_s$ be positive integers, $m = \sum\limits_{i=1}^s (a_i - 1) + 1$, $\max\set{a_1, ..., a_s} = 5$ and $m \geq 7$. Then,
	\begin{equation*}
	F_v(a_1, ..., a_s; m - 1) = m + 9.
	\end{equation*}
\end{theorem}
According to (\ref{equation: F_v(a_1, ..., a_s; m - 1) exists}), the condition $m \geq 7$ in Theorem \ref{theorem: F_v(a_1, ..., a_s; m - 1) = m + 9, max set(a_1, ..., a_s) = 5} is necessary.

\vspace{1em}

\section{Algorithm A1}

In the border case $m = 7$ of Theorem $\ref{theorem: F_v(a_1, ..., a_s; m - 1) = m + 9, max set(a_1, ..., a_s) = 5}$ there are only two canonical numbers of the form $F_v(a_1, ..., a_s; m - 1)$, namely $F_v(2, 2, 5; 6)$ and $F_v(3, 5; 6)$. It is known that $F_v(3, 5; 6) = 16$ \cite{SLPX12}, and it follows that $F_v(2, 2, 5; 6) \leq 16$. We will prove that $F_v(2, 2, 5; 6) = 16$ and we will also obtain all extremal graphs in $\mH_v(2, 2, 5; 6)$.

The naive approach for finding all graphs in $\mH_v(2, 2, 5; 6; 16)$ suggests to check all graphs of order 16 for inclusion in $\mH_v(2, 2, 5; 6)$. However, this is practically impossible because the number of non-isomorphic graphs of order 16 is too large (see \cite{McK_c}). The method that is described uses an algorithm for effective generation of all maximal graphs in $\mH_v(2, 2, 5; 6; 16)$. The remaining graphs in $\mH_v(2, 2, 5; 6; 16)$ are obtained by removing edges from the maximal graphs.

The idea for Algorithm \ref{algorithm: A1} comes from \cite{PRU99}. Similar algorithms are used in \cite{CR06}, \cite{XLS10}, \cite{LR11}, and \cite{SLPX12}. Also, with the help of computer, results for Folkman numbers are obtained in \cite{JR95}, \cite{SXP09}, \cite{SXL09}, and \cite{DLSX13}.

Let us remind that the notations $\mH_{max}$ and $\mH_{+K_{q-1}}$ are defined in Section 1.3.

\begin{namedalgorithm}{A1}
	\label{algorithm: A1}
	The input of the algorithm is the set $\mA = \mH_{max}(2_{r - 1}, p; q; n - k)$, where $r, p, q, n, k$ are fixed positive integers.
	
	The output of the algorithm is the set $\mB$ of all graphs $G \in \mH_{max}(2_r, p; q; n)$ with $\alpha(G) \geq k$.
	
	\emph{1.} By removing edges from the graphs in $\mA$ obtain the set
	
	$\mA' = \mH_{+K_{q - 1}}(2_{r - 1}, p; q; n - k)$.
	
	\emph{2.} For each graph $H \in \mA'$:
	
	\emph{2.1.} Find the family $\mM(H) = \set{M_1, ..., M_l}$ of all maximal $K_{q - 1}$-free subsets of $\V(H)$.
	
	\emph{2.2.} For each $k$-element multiset $N = \set{M_{i_1}, ..., M_{i_k}}$ of elements of $\mM(H)$ construct the graph $G = G(N)$ by adding new independent vertices $v_1, ..., v_k$ to $\V(H)$ such that $N_G(v_j) = M_{i_j}, j = 1, ..., k$. If $\omega(G + e) = q, \forall e \in \E(\overline{G})$, then add $G$ to $\mB$.
	
	\emph{3.} Remove the isomorphic copies of graphs from $\mB$.
	
	\emph{4.} Remove from $\mB$ all graphs $G$ for which $G \not\arrowsv (2_r, p)$.
	
\end{namedalgorithm}

\begin{theorem}
	\label{theorem: algorithm A1}
	After the execution of Algorithm \ref{algorithm: A1}, the obtained set $\mB$ coincides with the set of all graphs $G \in \mH_{max}(2_r, p; q; n)$ with $\alpha(G) \geq k$.
\end{theorem}

\begin{proof}
Suppose that after the execution of Algorithm \ref{algorithm: A1} the graph $G \in \mB$. Then $G = G(N)$ and $G - \set{v_1, ..., v_k} = H \in \mA'$, where $N$, $v_1, ..., v_k$, and $H$ are the same as in step 2.2. Since $H \in \mA'$, we have $\omega(H) < q$. Since $N_G(v_i)$ are $K_{q - 1}$-free sets for all $i \in \set{1, ..., k}$, it follows that $\omega(G) < q$. The check at the end of step 2.2 guarantees that $G$ is a maximal $K_q$-free graph, and the check in step 4 guarantees that $G \arrowsv (2_r, p)$. Therefore, $G \in \mH_{max}(2_r, p; q; n)$. Since the vertices $v_1, ..., v_k$ are independent, it follows that $\alpha(G) \geq k$.

Let $G \in \mH_{max}(2_r, p; q; n)$ and $\alpha(G) \geq k$. We will prove that, after the execution of Algorithm \ref{algorithm: A1}, $G \in \mB$. Let $v_1, ... v_k$ be independent vertices in $G$ and $H = G - \set{v_1, ..., v_k}$. Using Proposition \ref{proposition: G - A in mH_(+K_(q - 1))(a_1 - 1, a_2, ..., a_s; q; n - abs(A))} we derive that $H \in \mA'$. Since $G$ is a maximal $K_q$-free graph, $N_G(v_i)$ are maximal $K_{q - 1}$-free subsets of $V(H)$ for all $i \in \set{1, ..., k}$. Therefore, $N_G(v_i) \in \mM(H)$ (see step 2.1). Thus, $G = G(N)$ where $N = \set{N_G(v_1), ..., N_G(v_k)}$, and in step 2.2 $G$ is added to $\mB$. Clearly, after step 4 the graph $G$ remains in $\mB$.
\end{proof}

\begin{remark}
	\label{remark: algorithm A1 k = 2}
	Note that if $G \in \mH_{max}(2_r, p; q; n)$ and $n \geq q$, then $G$ is not a complete graph and $\alpha(G) \geq 2$. Therefore, if $n \geq q$ and $k = 2$, Algorithm \ref{algorithm: A1} finds all graphs in $\mH_{max}(2_r, p; q; n)$.
\end{remark}

\vspace{1em}
Theorem \ref{theorem: algorithm A1} is published in \cite{BN15a}. Algorithm \ref{algorithm: A1} is a slightly modified version of Algorithm 2.2 in \cite{BN15a}. In the special case $p = 5$ and $k = 2$, Algorithm \ref{algorithm: A1} coincides with Algorithm 6.2 in \cite{BN15a}.

\section{Computation of the number $F_v(2, 2, 5; 16)$}

Intermediate problems, that are solved, are finding all graphs in $\mH_v(5; 6; 10)$ and $\mH_v(2, 5; 6; 13)$.

\begin{theorem}
\label{theorem: abs(mH_v(5; 6; 10)) = 1 724 440}
$\abs{\mH_v(5; 6; 10)} = 1\ 724\ 440.$
\end{theorem}

\begin{proof}
It is clear that $\mH_v(5; 6; 10)$ is the set of 10 vertex graphs with clique number 5. Using \emph{nauty} \cite{MP13} it is easy to generate all 12 005 168 non-isomorphic graphs of order 10. Among these graphs 1 724 440 have clique number 5.
\end{proof}

\begin{theorem}
	\label{theorem: abs(mH_v(2, 5; 6; 13)) = 20 013 726}
	$\abs{\mH_v(2, 5; 6; 13)} = 20\ 013\ 726.$
\end{theorem}

\begin{proof}
	
The graphs in $\mH_{max}(2, 5; 6; 13)$ with independence number 2 are a subset of $\mR(3, 6; 13)$. All 275 086 graphs in $\mR(3, 6; 13)$ are known and are available on \cite{McK_r}. Among these graphs we find all 61 graphs in $\mH_{max}(2, 5; 6; 13)$ with independence number 2.
	
Next, we find all graphs in $\mH_{max}(2, 5; 6; 13)$ with independence number greater than 2. All 1 724 440 graphs in $\mH_v(5; 6; 10)$ were obtained in Theorem \ref{theorem: abs(mH_v(5; 6; 10)) = 1 724 440}. Among them there are 18 maximal graphs in $\mH_{max}(5; 6; 10)$. We execute Algorithm \ref{algorithm: A1}($n = 13;\ r = 1;\ p = 5;\ q = 6;\ k = 3$) with input the set $\mA = \mH_{max}(5; 6; 10)$ to obtain the set $\mB$ of all 326 graphs in $\mH_{max}(2, 5; 6; 13)$ with independence number greater than 2.
	
Thus, we obtained all 387 graphs in $\mH_{max}(2, 5; 6; 13)$. By removing edges from them, we obtain all 20 013 726 graphs in $\mH_v(2, 5; 6; 13)$. 
\end{proof}

\begin{theorem}
	\label{theorem: abs(mH_v(2, 2, 5; 6; 16)) = 147}
	$\abs{\mH_v(2, 2, 5; 6; 16)} = 147$.
\end{theorem}

\begin{proof}
The graphs in $\mH_{max}(2, 2, 5; 6; 16)$ with independence number 2 are a subset of $\mR(3, 6; 16)$. All 2 576 graphs in $\mR(3, 6; 16)$ are known and are available on \cite{McK_r}. Among these graphs we find all 5 graphs in $\mH_{max}(2, 2, 5; 6; 16)$ with independence number 2.

Next, we find all graphs in $\mH_{max}(2, 2, 5; 6; 16)$ with independence number greater than 2. In the proof of Theorem \ref{theorem: abs(mH_v(2, 5; 6; 13)) = 20 013 726} we obtained all 387 graphs in $\mH_{max}(2, 5; 6; 13)$. We execute Algorithm \ref{algorithm: A1}($n = 16;\ r = 2;\ p = 5;\ q = 6;\ k = 3$) with input the set $\mA = \mH_{max}(2, 5; 6; 13)$ to obtain the set $\mB$ of all 32 graphs in $\mH_{max}(2, 2, 5; 6; 16)$ with independence number greater than 2.

Thus, we obtained all 37 graphs in $\mH_{max}(2, 2, 5; 6; 16)$. By removing edges from them, we obtain all 147 graphs in $\mH_v(2, 2, 5; 6; 16)$. 
\end{proof}

\begin{figure}
	\centering
	\begin{subfigure}[b]{0.5\textwidth}
		\centering
		\includegraphics[height=256px,width=128px]{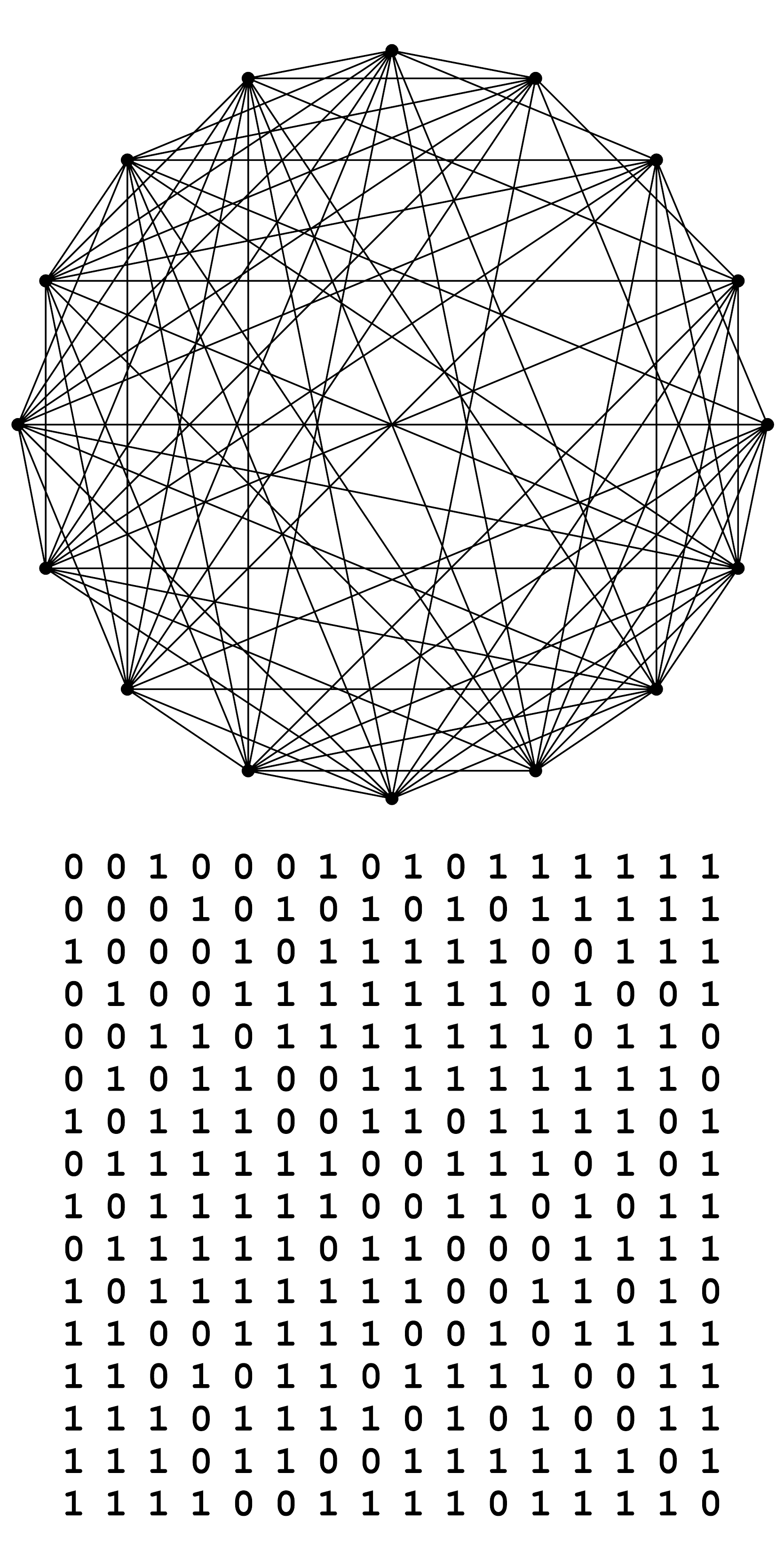}
		\caption*{\emph{$G_{74}$}}
		\label{figure: G_74}
	\end{subfigure}%
	\begin{subfigure}[b]{0.5\textwidth}
		\centering
		\includegraphics[height=256px,width=128px]{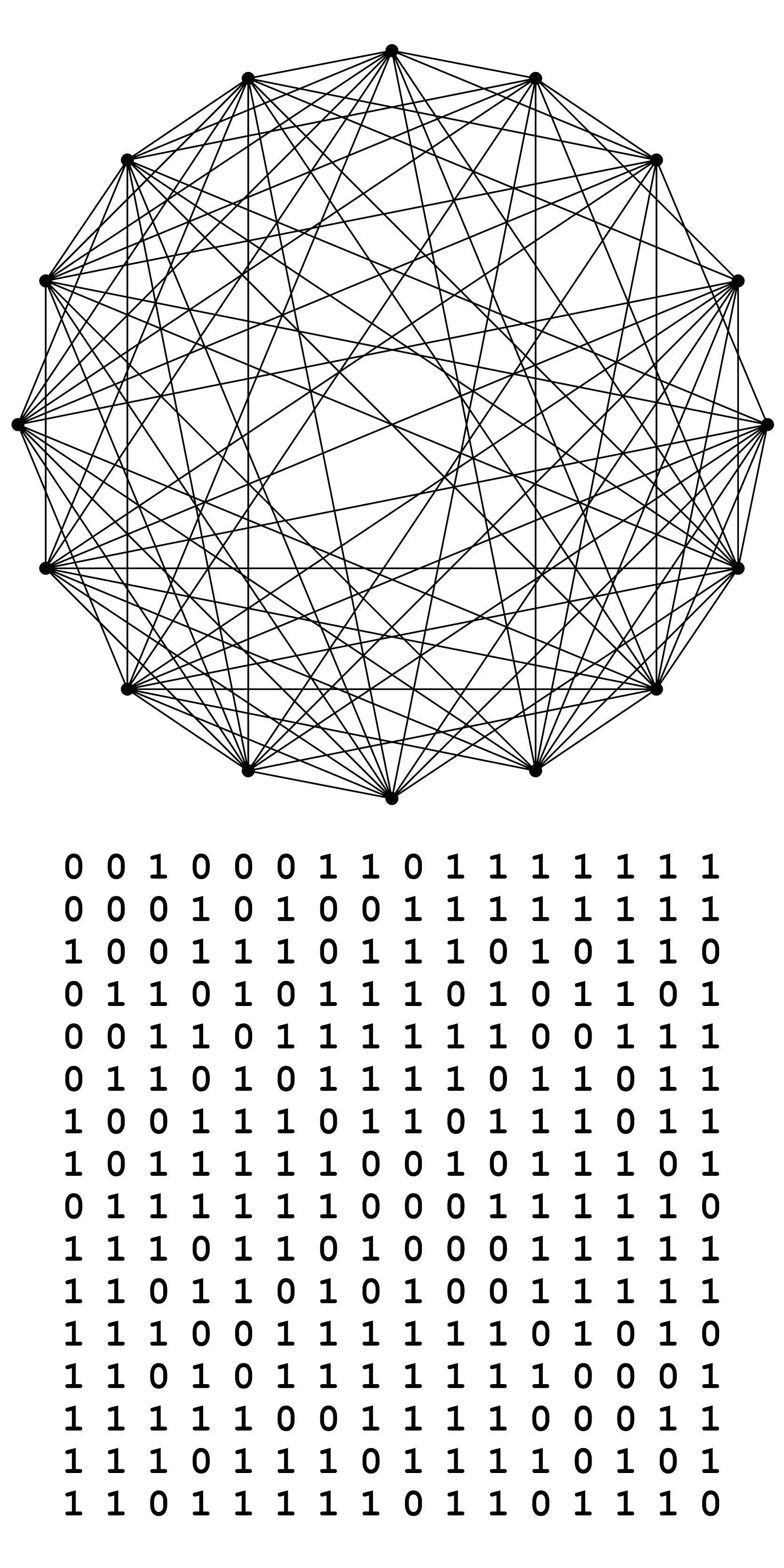}
		\caption*{\emph{$G_{78}$}}
		\label{figure: G_78}
	\end{subfigure}
	
	\begin{subfigure}[b]{0.5\textwidth}
		\centering
		\includegraphics[height=256px,width=128px]{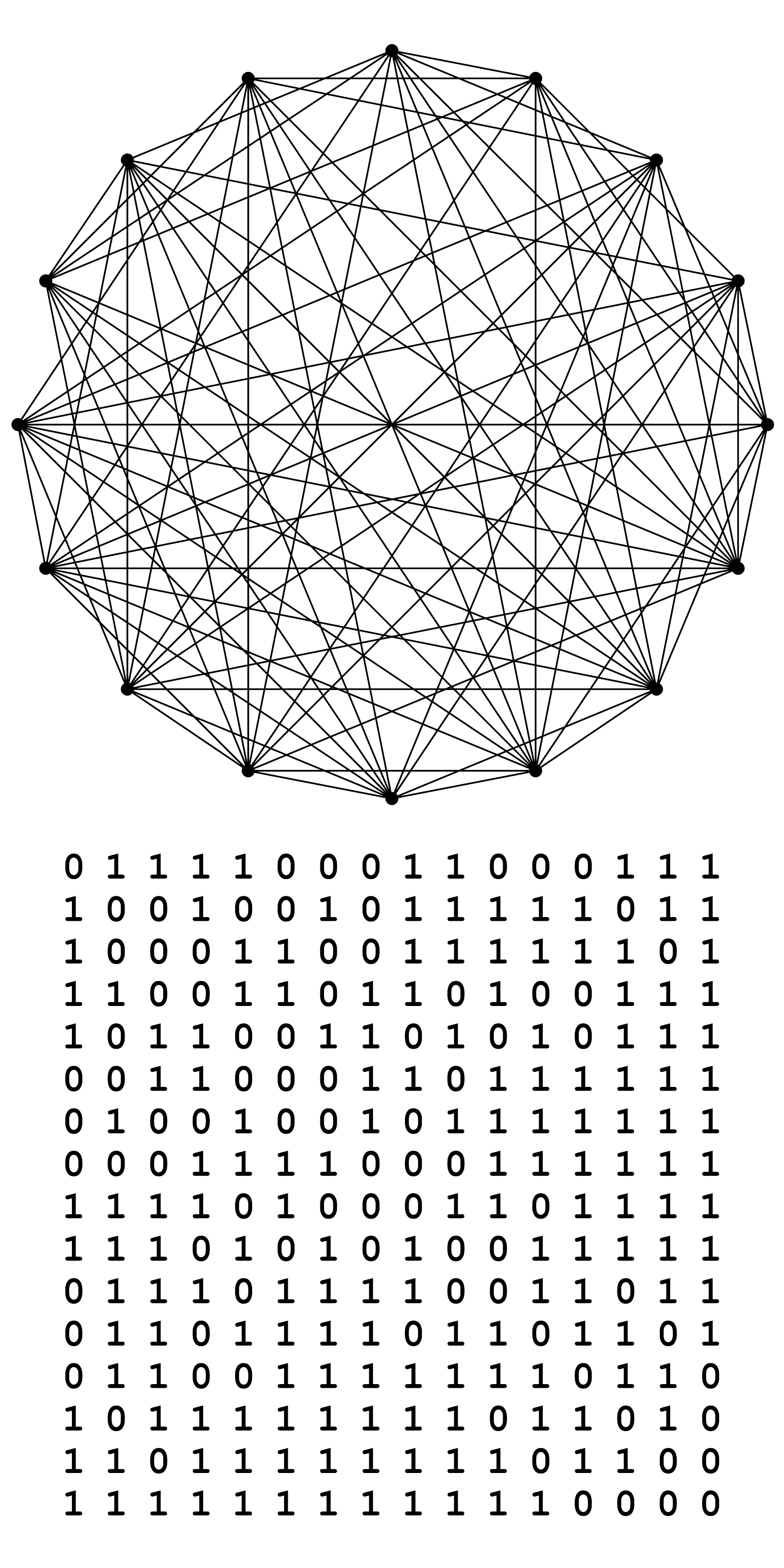}
		\caption*{\emph{$G_{134}$}}
		\label{figure: G_134}
	\end{subfigure}%
	\begin{subfigure}[b]{0.5\textwidth}
		\centering
		\includegraphics[height=256px,width=128px]{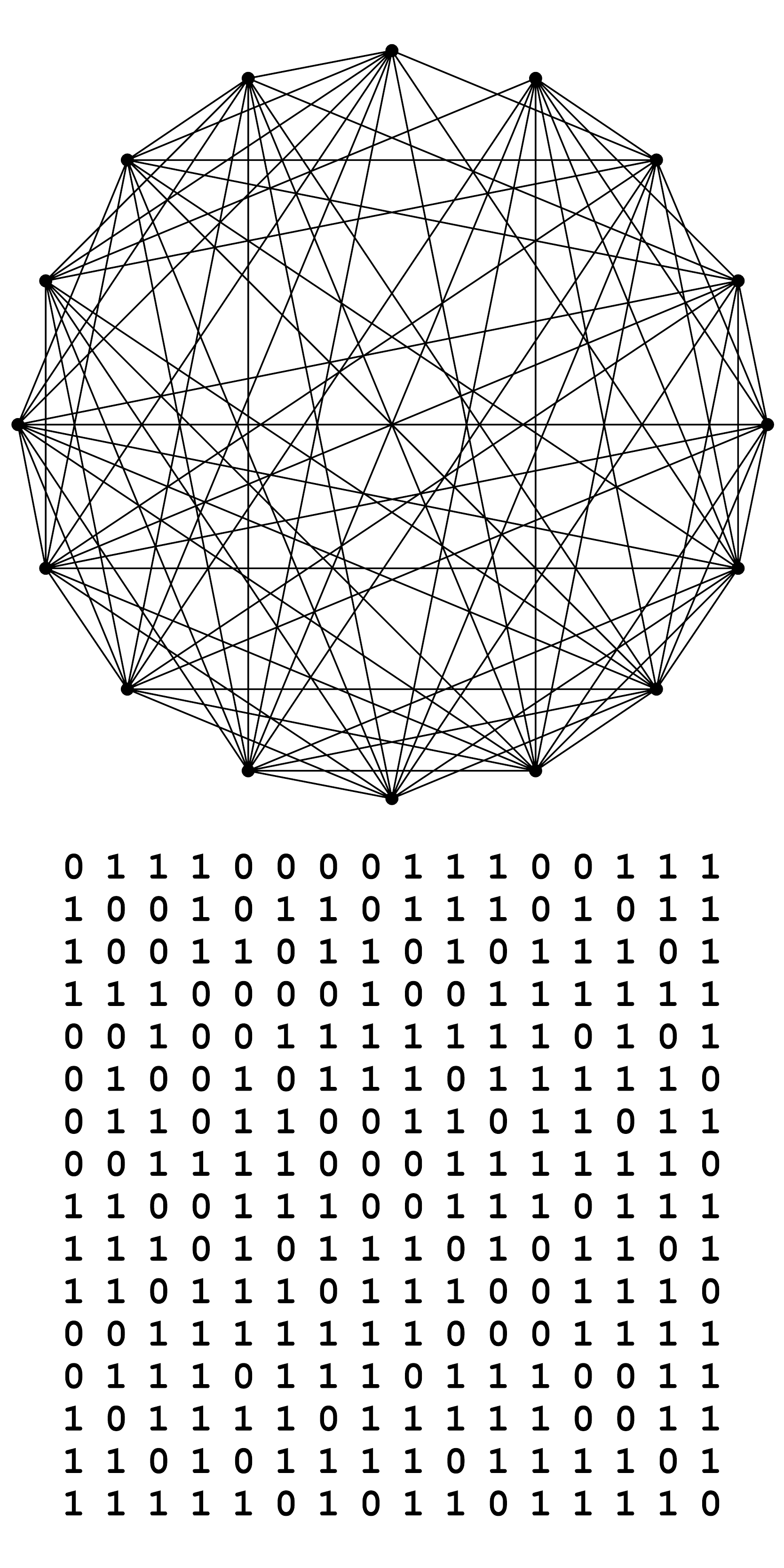}
		\caption*{\emph{$G_{135}$}}
		\label{figure: G_135}
	\end{subfigure}
	
	\vspace{1em}
	\caption{All 4 bicritical graphs in $\mH_v(2, 2, 5; 6; 16)$}
	\label{figure: H_v(2, 2, 5; 6; 16) bicritical}
\end{figure}

We denote by $G_1, ..., G_{147}$ the graphs in $\mH_v(2, 2, 5; 6; 16)$. The indexes correspond to the order defined in the \emph{nauty} programs. In Table \ref{table: H_v(2, 2, 5; 6; 16) properties} are listed some properties of the graphs in $\mH_v(2, 2, 5; 6; 16)$. It is interesting to note that there are no graphs with independence number greater than 3 in $\mH_v(2, 2, 5; 6; 16)$. In our opinion this is a non-trivial fact which cannot be obtained without a computer. Among the graphs in $\mH_v(2, 2, 5; 6; 16)$ there are 37 maximal, 41 minimal, and 4 bicritical graphs, i.e. graphs that are both maximal and minimal (see Figure \ref{figure: H_v(2, 2, 5; 6; 16) bicritical}). Some properties of the bicritical graphs are listed in Table \ref{table: H_v(2, 2, 5; 6; 16) bicritical properties}.

All computations were done on a personal computer. The slowest part was the step of finding all graphs in $\mH_{max}(2, 2, 5; 6; 16)$, which took several days to complete.

\begin{table}
	\centering
	\resizebox{\textwidth}{!}{
		\begin{tabular}{ | l r | l r | l r | l r | l r | l r | }
			\hline
			\multicolumn{2}{|c|}{\parbox{5em}{$\abs{\E(G)}$ \hfill $\#$}}&
			\multicolumn{2}{|c|}{\parbox{5em}{$\delta(G)$ \hfill $\#$}}&
			\multicolumn{2}{|c|}{\parbox{5em}{$\Delta(G)$ \hfill $\#$}}&
			\multicolumn{2}{|c|}{\parbox{5em}{$\alpha(G)$ \hfill $\#$}}&
			\multicolumn{2}{|c|}{\parbox{5em}{$\chi(G)$ \hfill $\#$}}&
			\multicolumn{2}{|c|}{\parbox{5em}{$\abs{Aut(G)}$ \hfill $\#$}}\\
			\hline
			83			&  7		& 7			& 2			& 11		& 24		& 2			& 21		& 7			& 65		& 1			& 84		\\
			84			&  25		& 8			& 36		& 12		& 123		& 3			& 126		& 8			& 82		& 2			& 44		\\
			85			&  42		& 9			& 61		& 			& 			& 			& 			& 			& 			& 3			& 1			\\
			86			&  37		& 10		& 47		& 			& 			& 			& 			& 			& 			& 4			& 8			\\
			87			&  29		& 11		& 1			& 			& 			& 			& 			& 			& 			& 6			& 8			\\
			88			&  6		& 			& 			& 			& 			& 			& 			& 			& 			& 8			& 1			\\
			89			&  1		& 			& 			& 			& 			& 			& 			& 			& 			& 96		& 1			\\
			\hline
		\end{tabular}
	}
	\vspace{-0.5em}
	\caption{Some properties of the graphs in $\mH_v(2, 2, 5; 6; 16)$}
	\label{table: H_v(2, 2, 5; 6; 16) properties}
	
	\vspace{0.5em}
	
	\centering
	\resizebox{\textwidth}{!}{
		\begin{tabular}{ | l | r | r | r | r | r | r | }
			\hline
			{\parbox{4em}{Graph}}&
			{\parbox{4em}{\hfill$\abs{\E(G)}$}}&
			{\parbox{4em}{\hfill$\delta(G)$}}&
			{\parbox{4em}{\hfill$\Delta(G)$}}&
			{\parbox{4em}{\hfill$\alpha(G)$}}&
			{\parbox{4em}{\hfill$\chi(G)$}}&
			{\parbox{4em}{\hfill$\abs{Aut(G)}$}}\\
			\hline
			$G_{74}$		&  86		& 9			& 12		& 3			& 7			& 1			\\
			$G_{78}$		&  87		& 10		& 12		& 3			& 7			& 2			\\
			$G_{134}$		&  85		& 9			& 12		& 3			& 7			& 2			\\
			$G_{135}$		&  85		& 9			& 12		& 3			& 7			& 1			\\
			\hline
		\end{tabular}
	}
	\vspace{-0.5em}
	\caption{Some properties of the bicritical graphs in $\mH_v(2, 2, 5; 6; 16)$}
	\label{table: H_v(2, 2, 5; 6; 16) bicritical properties}
\end{table}

\vspace{0.5em}

\begin{theorem}
	\label{theorem: F_v(2, 2, 5; 6) = 16}
	$F_v(2, 2, 5; 6) = 16$ and the graphs from Theorem \ref{theorem: abs(mH_v(2, 2, 5; 6; 16)) = 147} are all the graphs in $\mH_{extr}(2, 2, 5; 6)$.
\end{theorem}

\begin{proof}
From $\mH_v(2, 2, 5; 6; 16) \neq \emptyset$, it follows that $F_v(2, 2, 5; 6) \leq 16$. Since none of the graphs in $\mH_v(2, 2, 5; 6; 16)$ have an isolated vertex, we have $\mH_v(2, 2, 5; 6; 15) = \emptyset$, and therefore $F_v(2, 2, 5; 6) \geq 16$. Thus, the theorem is proved.
\end{proof}

To check the correctness of our computer programs, we used a simple algorithm to remove one vertex from each graph in $\mH_v(2, 2, 5; 6; 16)$ and confirmed that none of the obtained graphs belong to $\mH_v(2, 2, 5; 6)$. Thus, we derived in another way that $\mH_v(2, 2, 5; 6; 15) = \emptyset$.

\begin{remark}
The lower bound $F_v(2, 2, 5; 6) \geq 16$ can be proved simpler in terms of computational time. With the help of Algorithm \ref{algorithm: A1}, we proved independently that $\mH_v(2, 2, 5; 6; 15) = \emptyset$, which is a lot faster.
\end{remark}

\begin{figure}
	\centering
	\begin{subfigure}[b]{0.5\textwidth}
		\centering
		\includegraphics[height=256px,width=128px]{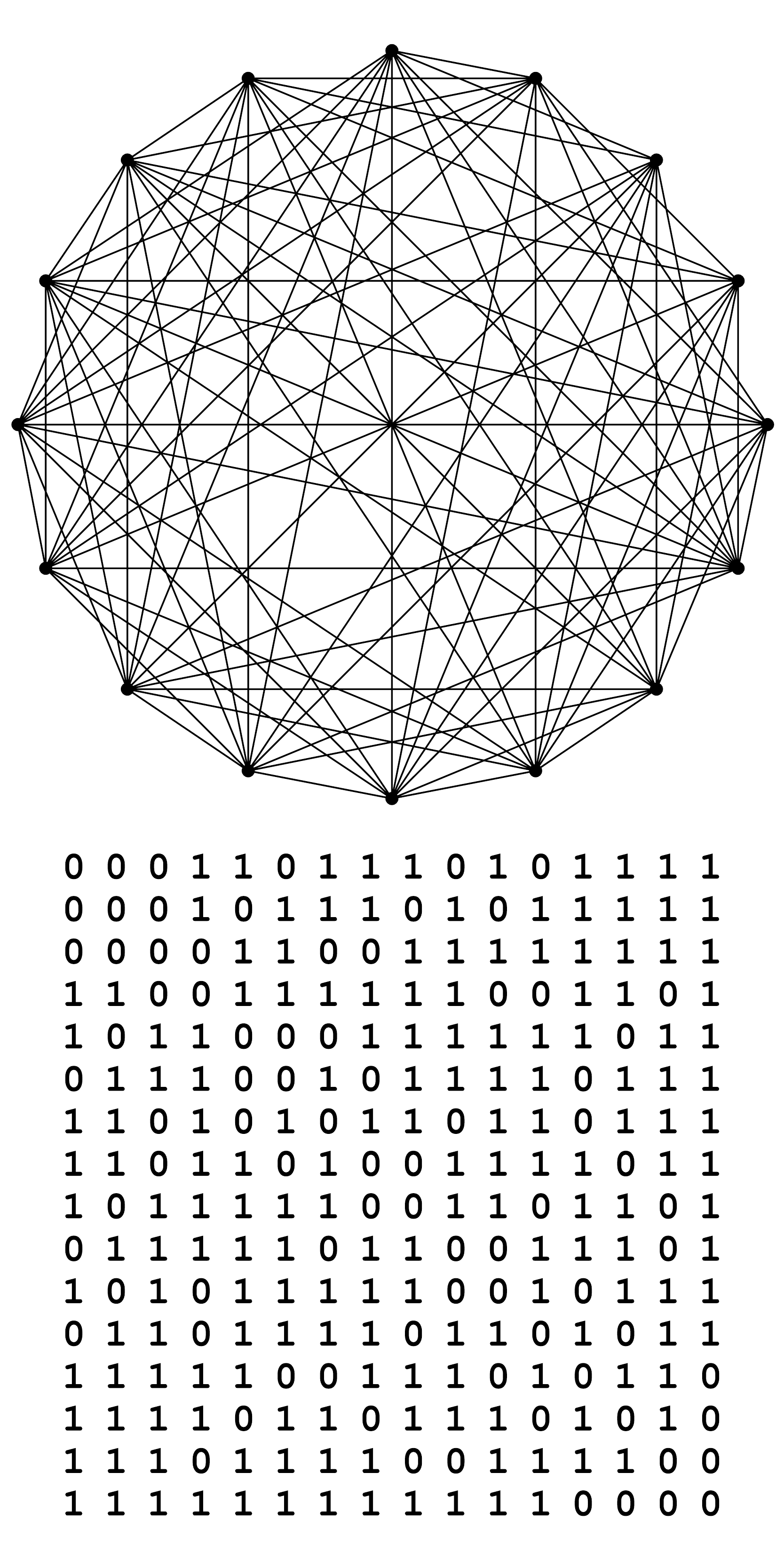}
		\caption*{\emph{$G_{50}$}}
		\label{figure: G_50}
	\end{subfigure}%
	\begin{subfigure}[b]{0.5\textwidth}
		\centering
		\includegraphics[height=256px,width=128px]{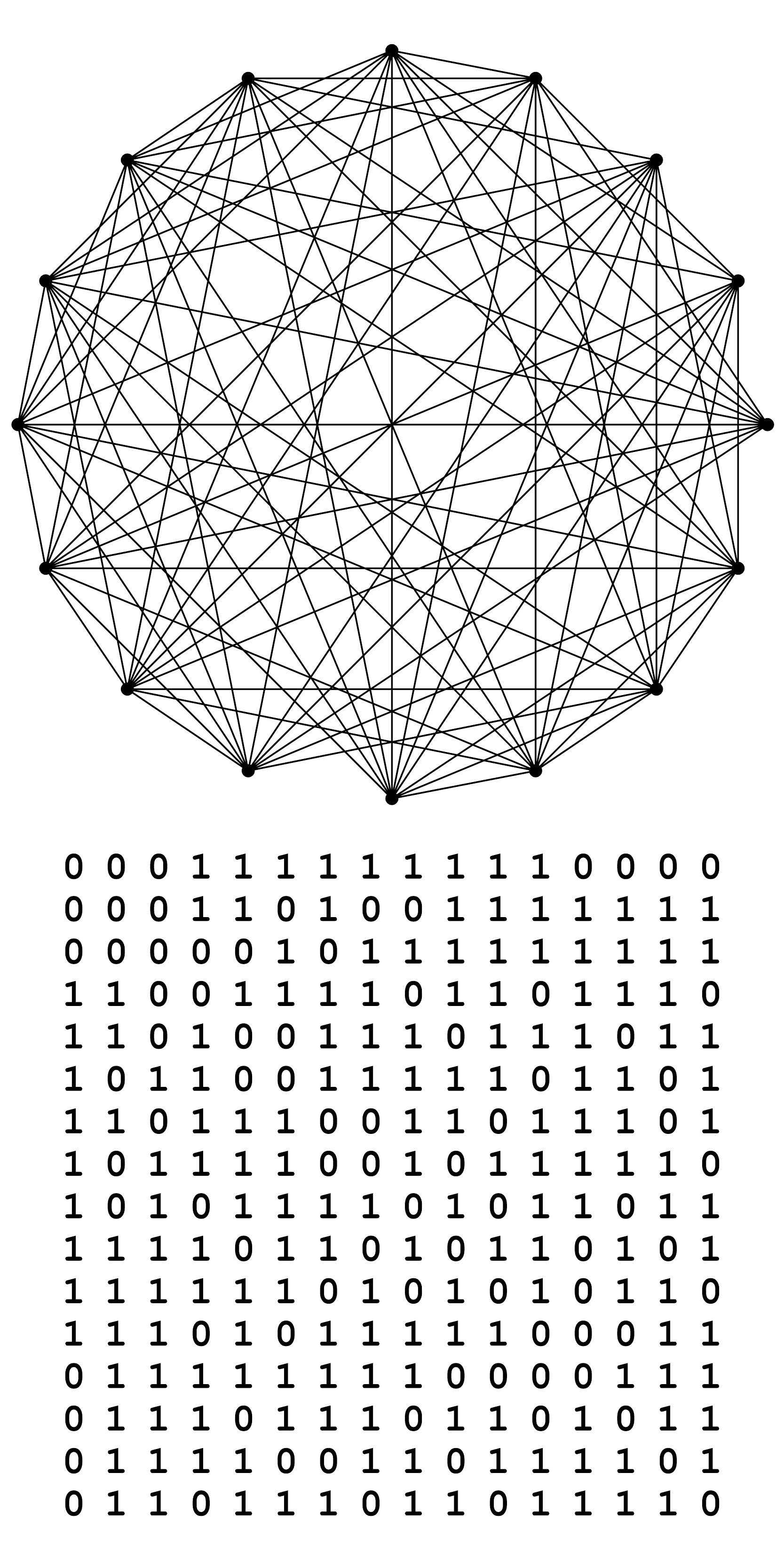}
		\caption*{\emph{$G_{51}$}}
		\label{figure: G_51}
	\end{subfigure}
	
	\begin{subfigure}[b]{0.5\textwidth}
		\centering
		\includegraphics[height=256px,width=128px]{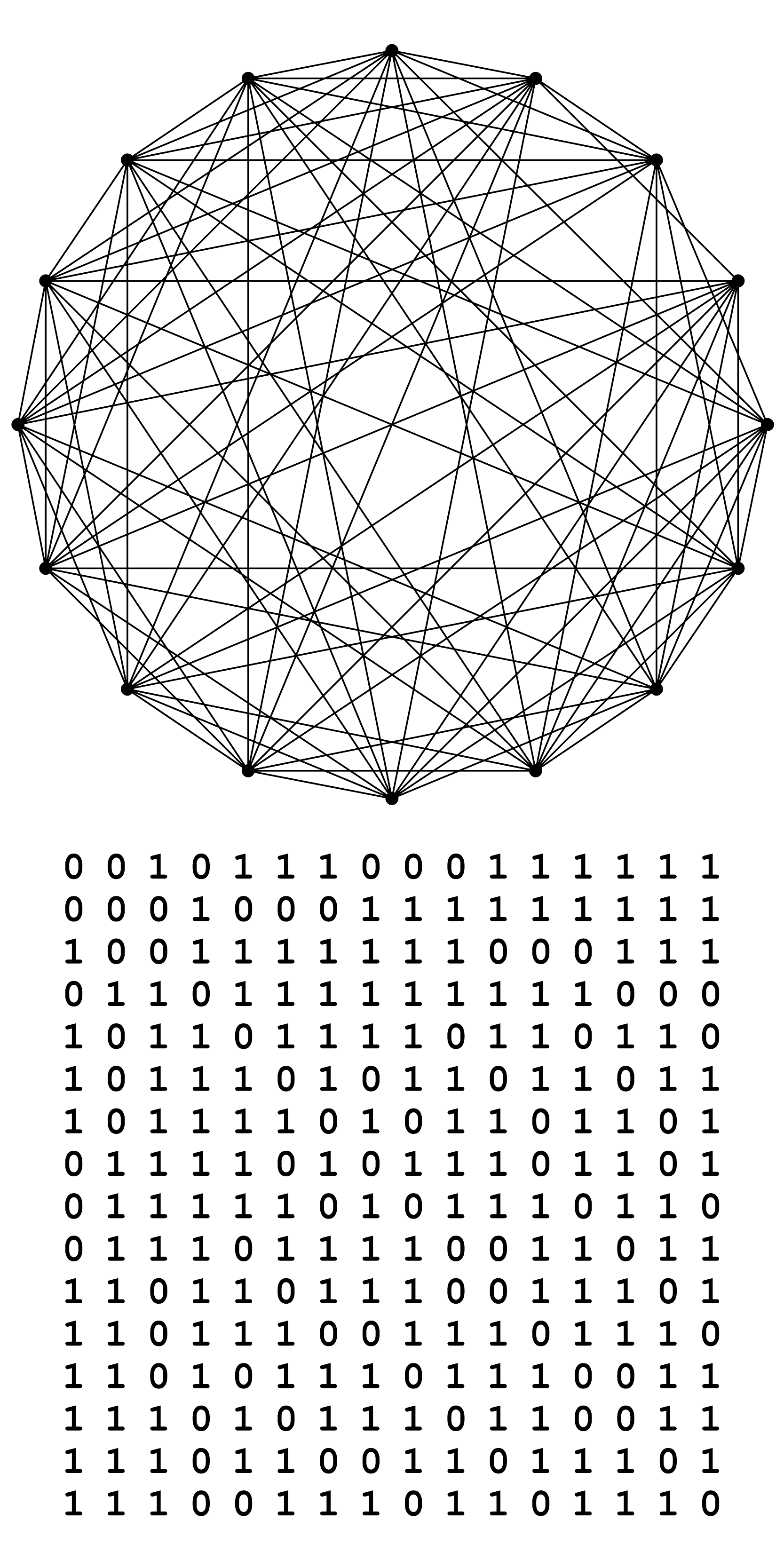}
		\caption*{\emph{$G_{81}$}}
		\label{figure: G_81}
	\end{subfigure}%
	\begin{subfigure}[b]{0.5\textwidth}
		\centering
		\includegraphics[height=256px,width=128px]{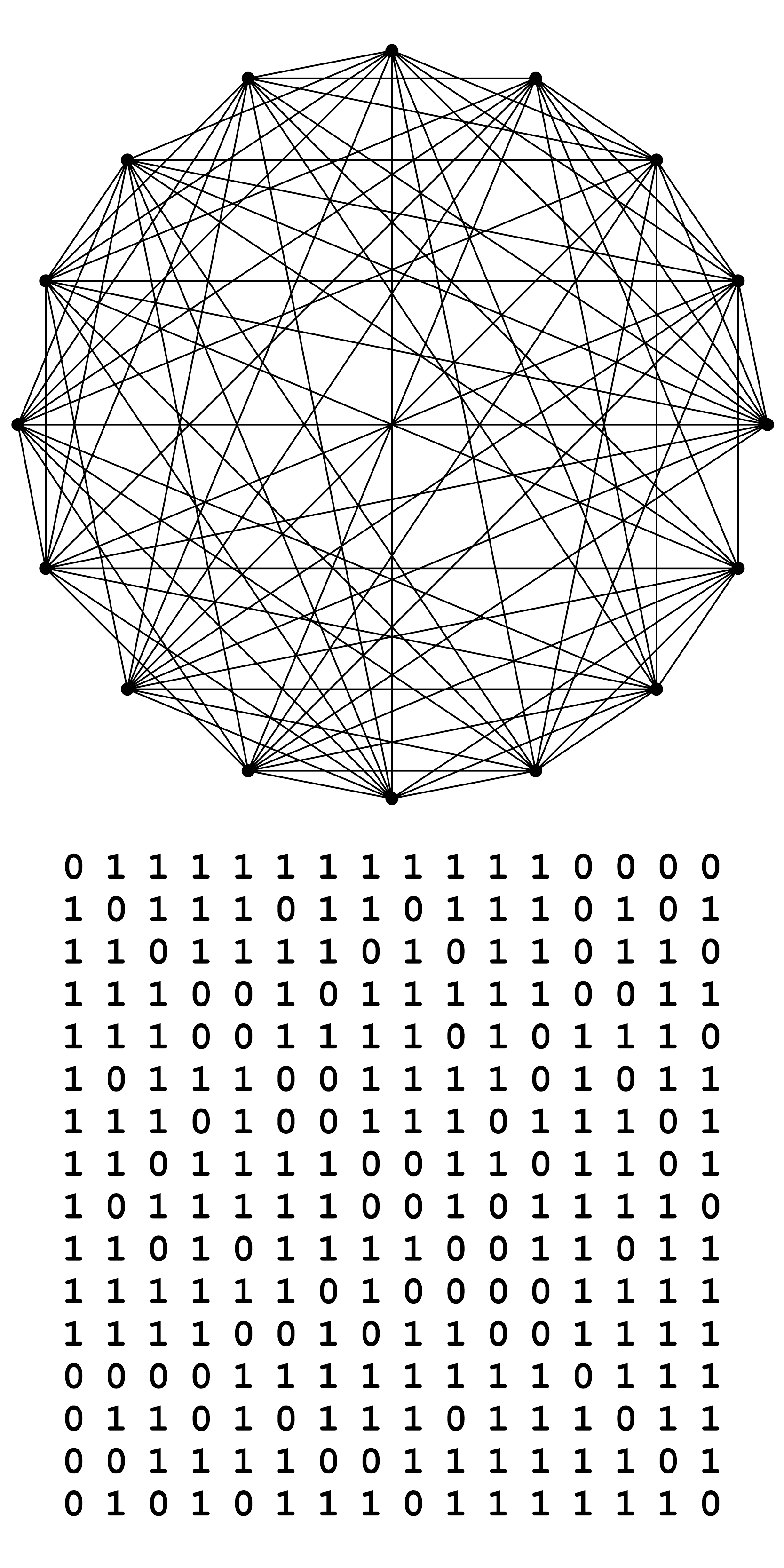}
		\caption*{\emph{$G_{146}$}}
		\label{figure: G_146}
	\end{subfigure}
	
	\vspace{1em}
	\caption{All 4 graphs in $\mH_v(3, 5; 6; 16)$}
	\label{figure: H_v(3, 5; 6; 16)}
\end{figure}

\begin{theorem}
	\label{theorem: abs(mH_v(3, 5; 6; 16)) = 4}
	$\mH_v(3, 5; 6; 16) = \set{G_{50}, G_{51}, G_{81}, G_{146}}$, see Figure \ref{figure: H_v(3, 5; 6; 16)}.
\end{theorem}

\begin{proof}
	Using $\mH_v(3, 5; 6; 16) \subseteq \mH_v(2, 2, 5; 6; 16)$, by checking with a computer which of the graphs in $\mH_v(2, 2, 5; 6; 16)$ belong to $\mH_v(3, 5; 6; 16)$, we obtain $\abs{\mH_v(3, 5; 6; 16)} = 4$. The graphs in $\mH_v(3, 5; 6; 16)$ are shown in Figure \ref{figure: H_v(3, 5; 6; 16)}.
\end{proof}

\begin{corollary}
	\label{corollary: F_v(3, 5; 6) = 16}
	\cite{SLPX12}
	$F_v(3, 5; 6) = 16$.
\end{corollary}

\begin{proof}
	From Theorem \ref{theorem: F_v(2, 2, 5; 6) = 16} and (\ref{equation: F_v(2, 2, p; p + 1) leq F_v(3, p; p + 1)}) we obtain $F_v(3, 5; 6) \geq 16$. Since $\mH_v(3, 5; 6; 16) \neq \emptyset$, it follows that $F_v(3, 5; 6) \leq 16$.
\end{proof}

Some properties of the graphs in $\mH_v(3, 5; 6; 16)$ are listed in Table \ref{table: H_v(3, 5; 6; 16) properties}.

It is interesting to note that for all these graphs the inequality (\ref{equation: G arrowsv (a_1, ..., a_s) Rightarrow chi(G) geq m}) is strict. Interesting results for Folkman graphs for which we have equality in (\ref{equation: G arrowsv (a_1, ..., a_s) Rightarrow chi(G) geq m}) are obtained in \cite{XLR18}.

The graphs $G_{50}$ and $G_{146}$ are maximal and the other two graphs $G_{51}$ and $G_{81}$ are their subgraphs and are obtained by removing one edge. In \cite{SXP09} the inequality $F_v(3, 5; 6) \leq 16$ is proved with the help of the graph $G_{146}$. Let us note that $\abs{Aut(G_{146})} = 96$, and among all graphs in $\mH_v(2, 2, 5; 6; 16)$ it has the most automorphisms.

\begin{table}
	\centering
	\resizebox{\textwidth}{!}{
		\begin{tabular}{ | l | r | r | r | r | r | r | }
			\hline
			{\parbox{4em}{Graph}}&
			{\parbox{4em}{\hfill$\abs{\E(G)}$}}&
			{\parbox{4em}{\hfill$\delta(G)$}}&
			{\parbox{4em}{\hfill$\Delta(G)$}}&
			{\parbox{4em}{\hfill$\alpha(G)$}}&
			{\parbox{4em}{\hfill$\chi(G)$}}&
			{\parbox{4em}{\hfill$\abs{Aut(G)}$}}\\
			\hline
		\hline
		$G_{50}$		&  87		& 10		& 12		& 3			& 8			& 6			\\
		$G_{51}$		&  86		& 9			& 11		& 3			& 8			& 6			\\
		$G_{81}$		&  87		& 10		& 11		& 2			& 8			& 6			\\
		$G_{146}$		&  88		& 11		& 11		& 2			& 8			& 96		\\
		\hline
	\end{tabular}
	}
	\caption{Some properties of the graphs in $\mH_v(3, 5; 6; 16)$}
	\label{table: H_v(3, 5; 6; 16) properties}
\end{table}

\vspace{1em}
Theorem \ref{theorem: abs(mH_v(5; 6; 10)) = 1 724 440}, Theorem \ref{theorem: abs(mH_v(2, 5; 6; 13)) = 20 013 726}, Theorem \ref{theorem: abs(mH_v(2, 2, 5; 6; 16)) = 147}, Theorem \ref{theorem: F_v(2, 2, 5; 6) = 16} and Theorem \ref{theorem: abs(mH_v(3, 5; 6; 16)) = 4} are published in \cite{BN15a}.

\vspace{1em}

\section{Bounds on the numbers $F_v(2_r, p; r + p - 1)$}

The bounds from Theorem \ref{theorem: F_v(2_(m - p), p; q) leq F_v(a_1, ..., a_s; q) leq wFv(m)(p)(q)} are useful because in general they are easier to estimate and compute than the numbers $F_v(a_1, ..., a_s)$ themselves. Further, we compute the exact value of the numbers $F_v(2_{m - 5}, 5; m - 1)$ (see Theorem \ref{theorem: rp(5) = 2}) and the numbers $\wFv{m}{5}{m - 1}$ (see Theorem \ref{theorem: wFv(m)(5)(m - 1) = ...}). Since $F_v(2_{m - 5}, 5; m - 1) = \wFv{m}{5}{m - 1}$, with the help of Theorem \ref{theorem: F_v(2_(m - p), p; q) leq F_v(a_1, ..., a_s; q) leq wFv(m)(p)(q)}, we prove Theorem \ref{theorem: F_v(a_1, ..., a_s; m - 1) = m + 9, max set(a_1, ..., a_s) = 5}.

In this section we prove that the computation of the lower bound in Theorem \ref{theorem: F_v(2_(m - p), p; q) leq F_v(a_1, ..., a_s; q) leq wFv(m)(p)(q)} in the case $q = m - 1$, i.e. the computation of the infinite sequence of numbers $F_v(2_r, p, r + p - 1)$ where $p$ is fixed, is reduced to the computation of a finite number of these numbers (Theorem \ref{theorem: rp}).\\
It is easy to prove that
\begin{equation*}
G \arrowsv (a_1, ..., a_s) \Rightarrow K_1 + G \arrowsv (2, a_1, ..., a_s).
\end{equation*}
Therefore,
\begin{equation}
\label{equation: G arrowsv (a_1, ..., a_s) Rightarrow K_t + G arrowsv (2_t, a_1, ..., a_s)}
G \arrowsv (a_1, ..., a_s) \Rightarrow K_t + G \arrowsv (2_t, a_1, ..., a_s).
\end{equation}

\begin{lemma}
	\label{lemma: F_v(2_r, p; r + p - 1) leq F_v(2_s, p; s + p - 1) + r - s}
	Let $2 \leq s \leq r$. Then,
	\begin{equation*}
	F_v(2_r, p; r + p - 1) \leq F_v(2_s, p; s + p - 1) + r - s.
	\end{equation*}
\end{lemma}

\begin{proof}
	Let $G$ be an extremal graph in $\mH_v(2_s, p; s + p - 1)$. Consider $G_1 = K_{r - s} + G$. According to (\ref{equation: G arrowsv (a_1, ..., a_s) Rightarrow K_t + G arrowsv (2_t, a_1, ..., a_s)}), $G_1 \arrowsv (2_r, p)$. Since $\omega(G_1) = r - s + \omega(G) < r + p - 1$, it follows that $G_1 \in \mH_v(2_r, p; r + p - 1)$. Therefore,
	
	$F_v(2_r, p; r + p - 1) \leq \abs{\V(G_1)} = F_v(2_s, p; s + p - 1) + r - s$.
\end{proof}

\vspace{1em}

\begin{theorem}
	\label{theorem: rp}
	Let $p$ be a fixed positive integer and $\rp(p) = \rp$ be the smallest positive integer for which
	$$\min_{r \geq 2} \set{F_v(2_r, p; r + p - 1) - r} = F_v(2_{\rp}, p; \rp + p - 1) - \rp.$$
	Then:
	\vspace{1em}\\
	(a) $F_v(2_r, p; r + p - 1) = F(2_{\rp}, p; \rp + p - 1) + r - \rp, \ r \geq \rp$.
	\vspace{1em}\\
	(b) If $\rp = 2$, then \hfill\break $F_v(2_r, p; r + p - 1) = F_v(2, 2, p; p + 1) + r - 2, \ r \geq 2$.
	\vspace{1em}\\
	(c) If $\rp > 2$ and $G$ is an extremal graph in $\mH (2_{\rp}, p; \rp + p - 1)$, then \hfill\break $G \arrowsv (2, \rp + p - 2)$.
	\vspace{1em}\\
	(d) $\rp < F_v(2, 2, p; p + 1) - 2p$.
\end{theorem}

\begin{proof}
	(a) According to the definition of $\rp = \rp(p)$, if $r \geq 2$, then
	
	$F_v(2_r, p; r + p - 1) - r \geq F_v(2_{\rp}, p; \rp + p - 1) - \rp$,\\
	i.e.
	
	$F_v(2_r, p; r + p - 1) \geq F_v(2_{\rp}, p; \rp + p - 1) + r - \rp$.\\
	If $r \geq \rp$, according to Lemma \ref{lemma: F_v(2_r, p; r + p - 1) leq F_v(2_s, p; s + p - 1) + r - s}, the opposite inequality is also true.
	
	(b) This equality is the special case $\rp = 2$ of the equality (a).
	
	(c) Suppose the opposite is true, and let $G$ be an extremal graph in $\mH_v(2_{\rp}, p; \rp + p - 1)$ and $V(G) = V_1 \cup V_2, \ V_1 \cap V_2 = \emptyset$, where $V_1$ is an independent set and $V_2$ does not contain an $(\rp + p - 2)$-clique. We can suppose that $V_1 \neq \emptyset$. Let $G_1 = \G[V_2]$. Then $\omega(G_1) < r + p - 2$, and from $G \arrowsv (2_{\rp}, p)$ and Proposition \ref{proposition: G - A arrowsv (a_1, ..., a_(i - 1), a_i - 1, a_(i + 1_, ..., a_s)} it follows $G_1 \arrowsv (2_{\rp - 1}, p)$. Therefore, $G_1 \in \mH_v(2_{\rp - 1}, p; \rp + p - 2)$ and
	
	$\abs{\V(G_1)} \geq F_v(2_{\rp - 1}, p; \rp + p - 2)$.\\
	Since $\abs{\V(G)} = F_v(2_{\rp}, p; \rp + p - 1)$ and $\abs{\V(G_1)} \leq \abs{\V(G)} - 1$, we obtain
	
	$F_v(2_{\rp - 1}, p; \rp + p - 2) - (\rp - 1) \leq F_v(2_{\rp}, p; \rp + p - 1) - \rp$,\\
	which contradicts the definition of $\rp = \rp(p)$.
	
	(d) According to (\ref{equation: m + p + 2 leq F_v(a_1, ..., a_s; m - 1) leq m + 3p}), $F_v(2, 2, p; p + 1) \geq 2p + 4$. Therefore, if $\rp = 2$, the inequality (d) is obvious.
	
	Let $\rp \geq 3$ and $G$ be an extremal graph in $\mH_v(2_{\rp}, p; \rp + p - 1)$. According to (c), $G \in \mH_v(2, \rp + p - 2; \rp + p - 1)$, and by Theorem \ref{theorem: F_v(a_1, ..., a_s; m) = m + p}, $\abs{\V(G)} \geq 2\rp + 2p - 3$. Let us notice that $\chi(\overline{C}_{2\rp + 2p - 3}) = \rp + p - 1$, but since $G \arrowsv (2_{\rp}, p)$, from (\ref{equation: G arrowsv (a_1, ..., a_s) Rightarrow chi(G) geq m}) it follows $\chi(G) \geq \rp + p = m$. Therefore, $G \neq \overline{C}_{2\rp + 2p - 3}$ and from Theorem \ref{theorem: F_v(a_1, ..., a_s; m) = m + p} we obtain
	
	$\abs{\V(G)} = F_v(2_{\rp}, p; \rp + p - 1) \geq 2\rp + 2p - 2$.\\
	Since $\rp \geq 3$, from the definition of $\rp$ we have
	
	$F_v(2_{\rp}, p; \rp + p - 1) < F_v(2, 2, p; p + 1) + \rp - 2$.\\
	Thus, we proved that
	
	$2\rp + 2p - 2 < F_v(2, 2, p; p + 1) + \rp - 2$,
	i.e.
	
	$\rp < F_v(2, 2, p; p + 1) - 2p$.
\end{proof}

\begin{remark}
	Since we suppose that $r \geq 2$, according to (\ref{equation: F_v(a_1, ..., a_s; q) exists}) all Folkman numbers in the proof of Theorem \ref{theorem: rp} exist.
\end{remark}

\begin{corollary}
	\label{corollary: F_v(a_1, ... a_s; m - 1) geq F_v(2_(rp), p; rp + p - 1) + r - rp}
	Let $a_1, ..., a_s$ be positive integers, let $m$ and $p$ be defined by (\ref{equation: m and p}), $m \geq p + 2$ and $r = m - p \geq \rp(p)$. Then,
	\begin{equation*}
	F_v(a_1, ... a_s; m - 1) \geq F_v(2_{\rp}, p; \rp + p - 1) + r - \rp.
	\end{equation*}
	In particular, if $\rp = 2$, then
	\begin{equation*}
	F_v(a_1, ..., a_s; m - 1) \geq F_v(2, 2, p; p + 1) + r - 2.
	\end{equation*}
\end{corollary}

\begin{proof}
	According to Theorem \ref{theorem: F_v(2_(m - p), p; q) leq F_v(a_1, ..., a_s; q) leq wFv(m)(p)(q)},
	
	$F_v(a_1, ..., a_s; m - 1) \geq F_v(2_r, p; r + p - 1)$.\\
	From this inequality and Theorem \ref{theorem: rp}(a) we obtain the desired inequality.
\end{proof}

\begin{example}
	\label{example: rp(2) = 4, rp(3) = 3, rp(4) = 2}
	From (\ref{equation: F_v(a_1, ..., a_s, m - 1) = ...}) and $F_v(2, 2, 2; 3) = F_v(2, 2, 2, 2; 4) = 11$ it follows $\rp(2) = 4$, and from (\ref{equation: F_v(a_1, ..., a_s, m - 1) = ...}) and $F_v(2, 2, 3; 4) = 14$ it follows $\rp(3) = 3$. Also, from (\ref{equation: F_v(a_1, ..., a_s, m - 1) = ...}) we see that $\rp(4) = 2$.
\end{example}

We suppose that the following is true:
\begin{conjecture}
	\label{conjecture: rp(p) = 2, p geq 4}
	Let $p \geq 4$ be a fixed integer. Then,
	\begin{equation*}
	\min_{r \geq 2} \set{F_v(2_r, p; r + p - 1) - r} = F_v(2, 2, p; p + 1) - 2,
	\end{equation*}
	i.e. $\rp(p) = 2$, and
	\begin{equation*}
	F_v(2_r, p; r + p - 1) = F_v(2, 2, p; p + 1) + r - 2, \ r \geq 2.
	\end{equation*}
\end{conjecture}
It is not difficult to see that Conjecture \ref{conjecture: rp(p) = 2, p geq 4} is true if and only if  for fixed $p$ the sequence $\set{F_v(2_r, p; r + p - 1)}$ is strictly increasing with respect to $r \geq 2$. From (\ref{equation: F_v(a_1, ..., a_s, m - 1) = ...}) we have $\rp(4) = 2$. Further, we prove that Conjecture \ref{conjecture: rp(p) = 2, p geq 4} is also true in the cases $p = 5$ (Theorem \ref{theorem: rp(5) = 2}), $p = 6$ (Theorem \ref{theorem: rp(6) = 2}) and $p = 7$ (Theorem \ref{theorem: rp(7) = 2}).

\vspace{1em}
Theorem \ref{theorem: rp} is published in \cite{BN15a}.

\section{Computation of the numbers $F_v(2_{m - 5}, 5; m - 1)$}

In this section we prove the following

\begin{theorem}
	\label{theorem: rp(5) = 2}
	$\rp(5) = 2$ and $F_v(2_{m - 5}, 5; m - 1) = m + 9, \ m \geq 7$.
\end{theorem}

\begin{proof}
	First, we will prove that $\rp(5) = 2$. From Theorem \ref{theorem: rp}(d) we have $\rp(5) \leq 5$. Therefore, we have to prove that $\rp(5) \neq 3$, $\rp(5) \neq 4$ and $\rp(5) \neq 5$, i.e. we have to prove the inequalities $F_v(2, 2, 2, 5; 7) > 16$, $F_v(2, 2, 2, 2, 5; 8) > 17$, $F_v(2, 2, 2, 2, 2, 5; 9) > 18$.
	
	According to Lemma \ref{lemma: F_v(2_r, p; r + p - 1) leq F_v(2_s, p; s + p - 1) + r - s}, it is enough to prove the last inequality. We will prove all three inequalities, because further we will use the results of these computations to check the correctness of Algorithm \ref{algorithm: A2}.\\

\vspace{-0.5em}
	
1. Proof of $F_v(2, 2, 2, 5; 7) > 16$.
	
	Using \emph{nauty} \cite{MP13} we generate all 10-vertex non-isomorphic graphs and among them we find all graphs in $\mH_{max}(5; 7; 10)$.
	
	We execute Algorithm \ref{algorithm: A1}($n = 12;\ r = 1;\ p = 5;\ q = 7;\ k = 2$) with input $\mA = \mH_{max}(5; 7; 10)$ to obtain all graphs in $\mB = \mH_{max}(2, 5; 7; 12)$. (see Remark \ref{remark: algorithm A1 k = 2})
		
	We execute Algorithm \ref{algorithm: A1}($n = 14;\ r = 2;\ p = 5;\ q = 7;\ k = 2$) with input $\mA = \mH_{max}(2, 5; 7; 12)$ to obtain all graphs in $\mB = \mH_{max}(2, 2, 5; 7; 14)$.
	
	By executing Algorithm \ref{algorithm: A1}($n = 16;\ r = 3;\ p = 5;\ q = 7;\ k = 2$) with input $\mA = \mH_{max}(2, 2, 5; 7; 14)$, we obtain $\mB = \emptyset$. According to Theorem \ref{theorem: algorithm A1}, $\mH_v(2, 2, 2, 5; 7; 16) = \emptyset$, and therefore $F_v(2, 2, 2, 5; 7) > 16$.
			
	The number of graphs obtained on each step of the proof is given in Table \ref{table: F_v(2, 2, 2, 5; 7) > 16}.\\
	
\vspace{-0.5em}

2. Proof of $F_v(2, 2, 2, 2, 5; 8) > 17$.
	
	Using \emph{nauty} \cite{MP13} we generate all 9-vertex non-isomorphic graphs and among them we find all graphs in $\mH_{max}(5; 8; 9)$. Actually, it is easy to see that $\mH_{max}(5; 8; 9) = \set{\overline{K}_3 + K_6, C_4 + K_5}$.
	
	We execute Algorithm \ref{algorithm: A1}($n = 11;\ r = 1;\ p = 5;\ q = 8;\ k = 2$) with input $\mA = \mH_{max}(5; 8; 9)$ to obtain all graphs in $\mB = \mH_{max}(2, 5; 8; 11)$ (see Remark \ref{remark: algorithm A1 k = 2}).
	
	We execute Algorithm \ref{algorithm: A1}($n = 13;\ r = 2;\ p = 5;\ q = 8;\ k = 2$) with input $\mA = \mH_{max}(2, 5; 8; 11)$ to obtain all graphs in $\mB = \mH_{max}(2, 2, 5; 8; 13)$.
		
	We execute Algorithm \ref{algorithm: A1}($n = 15;\ r = 3;\ p = 5;\ q = 8;\ k = 2$) with input $\mA = \mH_{max}(2, 2, 5; 8; 13)$ to obtain all graphs in $\mB = \mH_{max}(2, 2, 2, 5; 8; 15)$.
	
	By executing Algorithm \ref{algorithm: A1}($n = 17;\ r = 4;\ p = 5;\ q = 8;\ k = 2$) with input $\mA = \mH_{max}(2, 2, 2, 5; 8; 15)$, we obtain $\mB = \emptyset$. According to Theorem \ref{theorem: algorithm A1}, $\mH_v(2, 2, 2, 2, 5; 8; 17) = \emptyset$, and therefore $F_v(2, 2, 2, 2, 5; 8) > 17$.
	
	The number of graphs obtained on each step of the proof is given in Table \ref{table: F_v(2, 2, 2, 2, 5; 8) > 17}.\\
	
\vspace{-0.5em}

3. Proof of $F_v(2, 2, 2, 2, 2, 5; 9) > 18$.
	
	Using \emph{nauty} \cite{MP13} we generate all 10-vertex non-isomorphic graphs and among them we find all graphs in $\mH_{max}(2, 5; 9; 10)$. Actually, it is easy to see that $\mH_{max}(2, 5; 9; 10) = \set{\overline{K}_3 + K_7, C_4 + K_6}$.
	
	We execute Algorithm \ref{algorithm: A1}($n = 12;\ r = 2;\ p = 5;\ q = 9;\ k = 2$) with input $\mA = \mH_{max}(2, 5; 9; 10)$ to obtain all graphs in $\mB = \mH_{max}(2, 2, 5; 9; 12)$ (see Remark \ref{remark: algorithm A1 k = 2}).
	
	We execute Algorithm \ref{algorithm: A1}($n = 14;\ r = 3;\ p = 5;\ q = 9;\ k = 2$) with input $\mA = \mH_{max}(2, 2, 5; 9; 12)$ to obtain all graphs in $\mB = \mH_{max}(2, 2, 2, 5; 9; 14)$.
	
	We execute Algorithm \ref{algorithm: A1}($n = 16;\ r = 4;\ p = 5;\ q = 9;\ k = 2$) with input $\mA = \mH_{max}(2, 2, 2, 5; 9; 14)$ to obtain all graphs in $\mB = \mH_{max}(2, 2, 2, 2, 5; 9; 16)$.
	
	By executing Algorithm \ref{algorithm: A1}($n = 18;\ r = 5;\ p = 5;\ q = 9;\ k = 2$) with input $\mA = \mH_{max}(2, 2, 2, 2, 5; 9; 16)$, we obtain $\mB = \emptyset$. According to Theorem \ref{theorem: algorithm A1}, $\mH_v(2, 2, 2, 2, 2, 5; 9; 18) = \emptyset$, and therefore $F_v(2, 2, 2, 2, 2, 5; 9) > 18$.
	
	The number of graphs obtained on each step of the proof is given in Table \ref{table: F_v(2, 2, 2, 2, 2, 5; 9) > 18}.\\
	
We proved that $\rp(5) = 2$. Now, from Theorem \ref{theorem: rp}(b) we obtain
$$F_v(2_{m - 5}, 5; m - 1) = m + 9, \ m \geq 7.$$

Thus, the proof of Theorem \ref{theorem: rp(5) = 2} is finished.
\end{proof}

\begin{table}[h]
	\centering
	\resizebox{0.7\textwidth}{!}{
		\begin{tabular}{ | l | r | r | }
			\hline
			{\parbox{10em}{set}}&
			{\parbox{4em}{maximal\\ graphs}}&
			{\parbox{6em}{\hfill $(+K_6)$-graphs}}\\
			\hline
		$\mH_v(5; 7; 10)$						& 8					& 324		\\
		$\mH_v(2, 5; 7; 12)$					& 56				& 104 283	\\
		$\mH_v(2, 2, 5; 7; 14)$				& 420				& 2 614 547	\\
		$\mH_v(2, 2, 2, 5; 7; 16)$			& 0					& 0			\\
		\hline
	\end{tabular}
}
	\vspace{-0.5em}
	\caption{Steps in the proof of $\mH_v(2, 2, 2, 5; 7; 16) = \emptyset$}
	\label{table: F_v(2, 2, 2, 5; 7) > 16}
	
	\vspace{1em}
	\centering
	\resizebox{0.7\textwidth}{!}{
		\begin{tabular}{ | l | r | r | }
			\hline
			{\parbox{10em}{set}}&
			{\parbox{4em}{maximal\\ graphs}}&
			{\parbox{6em}{\hfill $(+K_7)$-graphs}}\\
			\hline
		$\mH_v(5; 8; 9)$					& 2					& 13		\\
		$\mH_v(2, 5; 8; 11)$				& 8					& 326		\\
		$\mH_v(2, 2, 5; 8; 13)$			& 56				& 105 138	\\
		$\mH_v(2, 2, 2, 5; 8; 15)$		& 423				& 2 616 723	\\
		$\mH_v(2, 2, 2, 2, 5; 8; 17)$		& 0					& 0			\\
		\hline
	\end{tabular}
}
	\vspace{-0.5em}
	\caption{Steps in the proof of $\mH_v(2, 2, 2, 2, 5; 8; 17) = \emptyset$}
	\label{table: F_v(2, 2, 2, 2, 5; 8) > 17}
	
	\vspace{1em}
	\centering
	\resizebox{0.7\textwidth}{!}{
		\begin{tabular}{ | l | r | r | }
			\hline
			{\parbox{10em}{set}}&
			{\parbox{4em}{maximal\\ graphs}}&
			{\parbox{6em}{\hfill $(+K_8)$-graphs}}\\
			\hline
		$\mH_v(2, 5; 9; 10)$				& 2					& 13		\\
		$\mH_v(2, 2, 5; 9; 12)$			& 8					& 327		\\
		$\mH_v(2, 2, 2, 5; 9; 14)$		& 56				& 105 294	\\
		$\mH_v(2, 2, 2, 2, 5; 9; 16)$		& 423				& 2 616 741	\\
		$\mH_v(2, 2, 2, 2, 2, 5; 9; 18)$	& 0					& 0			\\
		\hline
	\end{tabular}
}
	\vspace{-0.5em}
	\caption{Steps in the proof of $\mH_v(2, 2, 2, 2, 2, 5; 9; 18) = \emptyset$}
	\label{table: F_v(2, 2, 2, 2, 2, 5; 9) > 18}
\end{table}

\vspace{-0.5em}

All computations were done on a personal computer. The slowest part was the proof of $F_v(2, 2, 2, 2, 2, 5; 9) > 18$, which took several days to complete.

\vspace{-0.5em}

\vspace{1em}
Theorem \ref{theorem: rp(5) = 2} is published in \cite{BN15a}.

\vspace{-0.5em}

\section{Algorithm A2}

\vspace{-0.5em}

Algorithm \ref{algorithm: A1} is related to the special Folkman numbers of the form $F_v(2_r, p; q; n)$. In this section we present Algorithm \ref{algorithm: A2} with the help of which we can compute and obtain bounds on arbitrary Folkman numbers of the form $F_v(a_1, ..., a_s; q)$. This Algorithm A2 outputs all graphs $G \in \mH_{max}(a_1, ..., a_s; q; n)$ with $k \leq \alpha(G) \leq t$. Let us remind that the set of all graphs $G \in \mH_{max}(a_1, ..., a_s; q; n)$ with $\alpha(G) \leq t$ is denoted by $\mH_{max}^t(a_1, ..., a_s; q; n)$.

\begin{namedalgorithm}{A2}
	\label{algorithm: A2}
	The input of the algorithm is the set $\mA = \mH_{max}^t(a_1 - 1, a_2, ..., a_s; q; n - k)$, where $a_1, ..., a_s, q, n, k, t$ are fixed positive integers, $a_1 \geq 2$ and $k \leq t$.
	 
	The output of the algorithm is the set $\mB$ of all graphs $G \in \mH_{max}^t(a_1, ..., a_s; q; n)$ with $\alpha(G) \geq k$.
	
	\emph{1.} By removing edges from the graphs in $\mA$ obtain the set
	
	$\mA' = \mH_{+K_{q - 1}}^t(a_1 - 1, a_2, ..., a_s; q; n - k)$.
	
	\emph{2.} For each graph $H \in \mA'$:
	
	\emph{2.1.} Find the family $\mM(H) = \set{M_1, ..., M_l}$ of all maximal $K_{q - 1}$-free subsets of $\V(H)$.
	
	\emph{2.2.} For each $k$-element multiset $N = \set{M_{i_1}, ..., M_{i_k}}$ of elements of $\mM(H)$ construct the graph $G = G(N)$ by adding new independent vertices $v_1, ..., v_k$ to $\V(H)$ such that $N_G(v_j) = M_{i_j}, j = 1, ..., k$. If $\alpha(G) \leq t$ and $\omega(G + e) = q, \forall e \in \E(\overline{G})$, then add $G$ to $\mB$.
	
	\emph{3.} Remove the isomorphic copies of graphs from $\mB$.
	
	\emph{4.} Remove from $\mB$ all graphs $G$ for which $G \not\arrowsv (a_1, ..., a_s)$.
	
\end{namedalgorithm}

Clearly, the special case $a_1 = ... = a_{s - 1} = 2, a_s = p$ of Algorithm \ref{algorithm: A2} for sufficiently large $t$ coincides with Algorithm \ref{algorithm: A1}. In this sense, Algorithm \ref{algorithm: A2} is a generalization of Algorithm \ref{algorithm: A1}.

\begin{theorem}
	\label{theorem: algorithm A2}
	After the execution of Algorithm \ref{algorithm: A2}, the obtained set $\mB$ coincides with the set of all graphs $G \in \mH_{max}^t(a_1, ..., a_s; q; n)$ with $\alpha(G) \geq k$.
\end{theorem}

\begin{proof}
	Naturally, the proof follows the same pattern as the proof of Theorem \ref{theorem: algorithm A1}.
	Suppose that after the execution of Algorithm \ref{algorithm: A2} the graph $G \in \mB$. Then $G = G(N)$ where $N$ and the following notations are the same as in step 2.2. Since $H = G - \set{v_1, ..., v_k} \in \mA'$, we have $\omega(H) < q$. Since $N_G(v_i)$ are $K_{q - 1}$-free sets for all $i \in \set{1, ..., k}$, it follows that $\omega(G) < q$. The two checks at the end of step 2.2 guarantees that $G$ is a maximal $K_q$-free graph and $\alpha(G) \leq t$. The check in step 4 guarantees that $G \arrowsv (a_1, ..., a_s)$. We obtained $G \in \mH_{max}^t(a_1, ..., a_s; q; n)$. Since the vertices $v_1, ..., v_k$ are independent, it follows that $\alpha(G) \geq k$.
	
	Let $G \in \mH_{max}^t(a_1, ..., a_s; q; n)$ and $\alpha(G) \geq k$. We will prove that, after the execution of Algorithm \ref{algorithm: A2}, $G \in \mB$. Let $v_1, ... v_k$ be independent vertices in $G$ and $H = G - \set{v_1, ..., v_k}$. Using Proposition \ref{proposition: G - A in mH_(+K_(q - 1))(a_1 - 1, a_2, ..., a_s; q; n - abs(A))} we derive that $H \in \mA'$. Since $G$ is a maximal $K_q$-free graph, $N_G(v_i)$ are maximal $K_{q - 1}$-free subsets of $V(H)$ for all $i \in \set{1, ..., k}$, and therefore $N_G(v_i) \in \mM(H)$, see step 2.1. Thus, $G = G(N)$ where $N = \set{N_G(v_1), ..., N_G(v_k)}$ and in step 2.2 $G$ is added to $\mB$. Clearly, after step 4 the graph $G$ remains in $\mB$.
\end{proof}

\begin{remark}
	\label{remark: algorithm A2 k = 2}
	Note that if $G \in \mH_{max}^t(a_1, ..., a_s; q; n)$ and $n \geq q$, then $G$ is not a complete graph and $\alpha(G) \geq 2$. Therefore, if $n \geq q$ and $k = 2$, Algorithm \ref{algorithm: A2} finds all graphs in $\mH_{max}^t(a_1, ..., a_s; q; n)$.
\end{remark}

We will demonstrate the advantages of Algorithm \ref{algorithm: A2} by giving a second proof of Theorem \ref{theorem: rp(5) = 2}, which requires less computational time compared to Algorithm \ref{algorithm: A1}:

\subsection*{Second proof of Theorem \ref{theorem: rp(5) = 2}}
	
	As we already showed in the first proof of Theorem \ref{theorem: rp(5) = 2}, it is enough to prove the inequality $F_v(2, 2, 2, 2, 2, 5; 9) > 18$.
	
	Suppose that $\mH_{max}(2, 2, 2, 2, 2, 5; 9; 18) \neq \emptyset$ and let $G \in \mH_{max}(2, 2, 2, 2, 2, 5; 9; 18)$. It is clear that $\alpha(G) \geq 2$. From Theorem \ref{theorem: alpha(G) leq V(G) + m + p - 1}(b) it follows that $\alpha(G) \leq 3$. Using Algorithm \ref{algorithm: A2} we will prove separately that there are no graphs with independence number 2 and no graphs with independence number 3 in $\mH_{max}(2, 2, 2, 2, 2, 5; 9; 18)$.\\
	
	\vspace{-0.5em}
	
	First, we prove that there are no graphs in $\mH_{max}(2, 2, 2, 2, 2, 5; 9; 18)$ with independence number 3:
	
	The only graph in $\mH_{max}(5; 9; 7)$ is $K_7$.
	
	We execute Algorithm \ref{algorithm: A2}($n = 9;\ k = 2;\ t = 3$) with input $\mA = \mH_{max}^3(1, 5; 9; 7) = \mH_{max}^3(5; 9; 7) = \set{K_7}$ to obtain all graphs in $\mB = \mH_{max}^3(2, 5; 9; 9)$. (see Remark \ref{remark: algorithm A2 k = 2})
	
	We execute Algorithm \ref{algorithm: A2}($n = 11;\ k = 2;\ t = 3$) with input $\mA = \mH_{max}^3(1, 2, 5; 9; 9) = \mH_{max}^3(2, 5; 9; 9)$ to obtain all graphs in $\mB = \mH_{max}^3(2, 2, 5; 9; 11)$.
	
	We execute Algorithm \ref{algorithm: A2}($n = 13;\ k = 2;\ t = 3$) with input $\mA = \mH_{max}^3(1, 2, 2, 5; 9; 11) = \mH_{max}^3(2, 2, 5; 9; 11)$ to obtain all graphs in $\mB = \mH_{max}^3(2, 2, 2, 5; 9; 13)$.
		
	We execute Algorithm \ref{algorithm: A2}($n = 15;\ k = 2;\ t = 3$) with input $\mA = \mH_{max}^3(1, 2, 2, 2, 5; 9; 13) = \mH_{max}^3(2, 2, 2, 5; 9; 13)$ to obtain all graphs in $\mB = \mH_{max}^3(2, 2, 2, 2, 5; 9; 15)$.
	
	By executing Algorithm \ref{algorithm: A2}($n = 18;\ k = 3;\ t = 3$) with input $\mA = \mH_{max}^3(1, 2, 2, 2, 2, 5; 9; 15) = \mH_{max}^3(2, 2, 2, 2, 5; 9; 15)$, we obtain $\mB = \emptyset$. On the other hand, according to Theorem \ref{theorem: algorithm A2}, $\mB$ consists of all graphs with independence number 3 in $\mH_{max}(2, 2, 2, 2, 2, 5; 9; 18)$. We obtained that there are no graphs in $\mH_{max}(2, 2, 2, 2, 2, 5; 9; 18)$ with independence number 3.\\
					
	\vspace{-0.5em}
	
	It remains to be proved that there are no graphs in $\mH_{max}(2, 2, 2, 2, 2, 5; 9; 18)$ with independence number 2:
	
	The only graph in $\mH_{max}(5; 9; 8)$ is $K_8$.
	
	We execute Algorithm \ref{algorithm: A2}($n = 10;\ k = 2;\ t = 2$) with input $\mA = \mH_{max}^2(1, 5; 9; 8) = \mH_{max}^2(5; 9; 8) = \set{K_8}$ to obtain all graphs in $\mB = \mH_{max}^2(2, 5; 9; 10)$. (see Remark \ref{remark: algorithm A2 k = 2})
		
	We execute Algorithm \ref{algorithm: A2}($n = 12;\ k = 2;\ t = 2$) with input $\mA = \mH_{max}^2(1, 2, 5; 9; 10) = \mH_{max}^2(2, 5; 9; 10)$ to obtain all graphs in $\mB = \mH_{max}^2(2, 2, 5; 9; 12)$.
		
	We execute Algorithm \ref{algorithm: A2}($n = 14;\ k = 2;\ t = 2$) with input $\mA = \mH_{max}^2(1, 2, 2, 5; 9; 12) = \mH_{max}^2(2, 2, 5; 9; 12)$ to obtain all graphs in $\mB = \mH_{max}^2(2, 2, 2, 5; 9; 14)$.
		
	We execute Algorithm \ref{algorithm: A2}($n = 16;\ k = 2;\ t = 2$) with input $\mA = \mH_{max}^2(1, 2, 2, 2, 5; 9; 14) = \mH_{max}^2(2, 2, 2, 5; 9; 14)$ to obtain all graphs in $\mB = \mH_{max}^2(2, 2, 2, 2, 5; 9; 16)$.
		
	By executing Algorithm \ref{algorithm: A2}($n = 18;\ k = 2;\ t = 2$) with input $\mA = \mH_{max}^2(1, 2, 2, 2, 2, 5; 9; 16) = \mH_{max}^2(2, 2, 2, 2, 5; 9; 16)$, we obtain $\mB = \emptyset$. On the other hand, according to Theorem \ref{theorem: algorithm A2}, $\mB$ consists of all graphs with independence number 2 in $\mH_{max}(2, 2, 2, 2, 2, 5; 9; 18)$. We obtained that there are no graphs in $\mH_{max}(2, 2, 2, 2, 2, 5; 9; 18)$ with independence number 2.\\
					
	\vspace{-1em}
		
	Thus, we proved $\mH_{max}(2, 2, 2, 2, 2, 5; 9; 18) = \emptyset$, and therefore $F_v(2, 2, 2, 2, 2, 5; 9) > 18$. \qed
	
	\vspace{1em}
	
	The number of graphs obtained in each step of the proof is given in Table \ref{table: finding all graphs in H_v(2, 2, 2, 2, 2, 5; 9; 18)}.
	
	Let us note that the fact that there are no graphs in $\mH_{max}(2, 2, 2, 2, 2, 5; 9; 18)$ with independence number greater than 3 can be proved without using Theorem \ref{theorem: alpha(G) leq V(G) + m + p - 1}(b), but by applying Algorithm \ref{algorithm: A2} with $k > 3$ instead. We performed this check to test Algorithm \ref{algorithm: A2}.
	
	We also used Algorithm \ref{algorithm: A2} to give similar second proofs of the inequalities $F_v(2, 2, 2, 5; 7) > 16$ and $F_v(2, 2, 2, 2, 5; 8) > 17$, see Table \ref{table: finding all graphs in H_v(2, 2, 2, 5; 7; 16)} and Table \ref{table: finding all graphs in H_v(2, 2, 2, 2, 5; 8; 17)} respectively. We compared the results of the computations to the results obtained in the previous proofs of these inequalities to check the correctness of our implementation of Algorithm \ref{algorithm: A1} and Algorithm \ref{algorithm: A2}. To demonstrate the advantages of Algorithm \ref{algorithm: A2} we will note that the slowest step in the proof of Theorem \ref{theorem: rp(5) = 2} with Algorithm \ref{algorithm: A1} is finding all 2 616 741 graphs in $\mH_{+K_8}(2, 2, 2, 2, 5; 9; 16)$ (Table \ref{table: F_v(2, 2, 2, 2, 2, 5; 9) > 18}), while the slowest step using Algorithm \ref{algorithm: A3} is finding all 230 370 graphs with independence number 2 in $\mH_{+K_8}(2, 2, 2, 2, 5; 9; 16)$ (Table \ref{table: finding all graphs in H_v(2, 2, 2, 2, 2, 5; 9; 18)}). Thus, the total computational time needed for the proof of Theorem \ref{theorem: rp(5) = 2} is reduced from several days to several hours.
	
	We see that if the independence number of the emerging graphs is taken into account, we obtain a faster algorithm. In Algorithm \ref{algorithm: A2} we account for the independence number with the check $\alpha(G) \leq t$ in step 2.2. In some of the following problems this is not effective enough. Therefore, in the next Algorithm \ref{algorithm: A3} the check $\alpha(G) \leq t$ is replaced with other conditions (see step 2.2 of Algorithm \ref{algorithm: A3}). In this thesis Algorithm \ref{algorithm: A2} has no significance by itself, but it helps to obtain Algorithm \ref{algorithm: A3} in a natural way.
	
\begin{table}
	\centering
	\vspace{-2em}
	\resizebox{0.75\textwidth}{!}{
		\begin{tabular}{ | l | l | r | r | }
			\hline
			{\parbox{9em}{set}}&
			{\parbox{5.5em}{\small independence\\ number}}&
			{\parbox{4em}{maximal\\ graphs}}&
			{\parbox{6em}{\hfill $(+K_6)$-graphs}}\\
			\hline
		$\mH_v(3; 7; 5)$					& $\leq 3$				& 1						& 1					\\
		$\mH_v(4; 7; 7)$					& $\leq 3$				& 1						& 4					\\
		$\mH_v(5; 7; 9)$					& $\leq 3$				& 3						& 45				\\
		$\mH_v(2, 5; 7; 11)$				& $\leq 3$				& 12					& 3 036				\\
		$\mH_v(2, 2, 5; 7; 13)$			& $\leq 3$				& 14					& 1 120				\\
		$\mH_v(2, 2, 2, 5; 7; 16)$		& $= 3$					& 0						&					\\
		\hline
		$\mH_v(3; 7; 6)$					& $\leq 2$				& 1						& 1					\\
		$\mH_v(4; 7; 8)$					& $\leq 2$				& 1						& 8					\\
		$\mH_v(5; 7; 10)$					& $\leq 2$				& 3						& 82				\\
		$\mH_v(2, 5; 7; 12)$				& $\leq 2$				& 10					& 5 046				\\
		$\mH_v(2, 2, 5; 7; 14)$			& $\leq 2$				& 84					& 229 077			\\
		$\mH_v(2, 2, 2, 5; 7; 16)$		& $= 2$					& 0						&					\\
		\hline
		$\mH_v(2, 2, 2, 5; 7; 16)$		& 						& 0						&					\\
		\hline
	\end{tabular}
}
	\vspace{-0.5em}
	\caption{\small Steps in finding all maximal graphs in $\mH_v(2, 2, 2, 5; 7; 16)$}
	\label{table: finding all graphs in H_v(2, 2, 2, 5; 7; 16)}
	\vspace{1em}
	\centering
	\resizebox{0.75\textwidth}{!}{
		\begin{tabular}{ | l | l | r | r | }
			\hline
			{\parbox{9em}{set}}&
			{\parbox{5.5em}{\small independence\\ number}}&
			{\parbox{4em}{maximal\\ graphs}}&
			{\parbox{6em}{\hfill $(+K_7)$-graphs}}\\
			\hline
		$\mH_v(4; 8; 6)$					& $\leq 3$				& 1						& 1					\\
		$\mH_v(5; 8; 8)$					& $\leq 3$				& 1						& 4					\\
		$\mH_v(2, 5; 8; 10)$				& $\leq 3$				& 3						& 45				\\
		$\mH_v(2, 2, 5; 8; 12)$			& $\leq 3$				& 12					& 3 068				\\
		$\mH_v(2, 2, 2, 5; 8; 14)$		& $\leq 3$				& 14					& 1 121				\\
		$\mH_v(2, 2, 2, 2, 5; 8; 17)$		& $= 3$					& 0						&					\\
		\hline
		$\mH_v(4; 8; 7)$					& $\leq 2$				& 1						& 1					\\
		$\mH_v(5; 8; 9)$					& $\leq 2$				& 1						& 8					\\
		$\mH_v(2, 5; 8; 11)$				& $\leq 2$				& 3						& 84				\\
		$\mH_v(2, 2, 5; 8; 13)$			& $\leq 2$				& 10					& 5 380				\\
		$\mH_v(2, 2, 2, 5; 8; 15)$		& $\leq 2$				& 87					& 230 356			\\
		$\mH_v(2, 2, 2, 2, 5; 8; 17)$		& $= 2$					& 0						&					\\
		\hline
		$\mH_v(2, 2, 2, 2, 5; 8; 17)$		& 						& 0						&					\\
		\hline
	\end{tabular}
}
	\vspace{-0.5em}
	\caption{\small Steps in finding all maximal graphs in $\mH_v(2, 2, 2, 2, 5; 8; 17)$}
	\label{table: finding all graphs in H_v(2, 2, 2, 2, 5; 8; 17)}
	\vspace{1em}
	\centering
	\resizebox{0.75\textwidth}{!}{
		\begin{tabular}{ | l | l | r | r | }
			\hline
			{\parbox{5em}{set}}&
			{\parbox{5.5em}{\small independence\\ number}}&
			{\parbox{4em}{maximal\\ graphs}}&
			{\parbox{6em}{\hfill $(+K_8)$-graphs}}\\
			\hline
		$\mH_v(5; 9; 7)$					& $\leq 3$				& 1						& 1					\\
		$\mH_v(2, 5; 9; 9)$				& $\leq 3$				& 1						& 4					\\
		$\mH_v(2, 2, 5; 9; 11)$			& $\leq 3$				& 3						& 45				\\
		$\mH_v(2, 2, 2, 5; 9; 13)$		& $\leq 3$				& 12					& 3 077				\\
		$\mH_v(2, 2, 2, 2, 5; 9; 15)$		& $\leq 3$				& 14					& 1 121				\\
		$\mH_v(2, 2, 2, 2, 2, 5; 9; 18)$	& $= 3$					& 0						&					\\
		\hline
		$\mH_v(5; 9; 8)$					& $\leq 2$				& 1						& 1					\\
		$\mH_v(2, 5; 9; 10)$				& $\leq 2$				& 1						& 8					\\
		$\mH_v(2, 2, 5; 9; 12)$			& $\leq 2$				& 3						& 85				\\
		$\mH_v(2, 2, 2, 5; 9; 14)$		& $\leq 2$				& 10					& 5 459				\\
		$\mH_v(2, 2, 2, 2, 5; 9; 16)$		& $\leq 2$				& 87					& 230 370			\\
		$\mH_v(2, 2, 2, 2, 2, 5; 9; 18)$	& $= 2$					& 0						&					\\
		\hline
		$\mH_v(2, 2, 2, 2, 2, 5; 9; 18)$	& 						& 0						&					\\
		\hline
	\end{tabular}
}
	\vspace{-0.5em}
	\caption{\small Steps in finding all maximal graphs in $\mH_v(2, 2, 2, 2, 2, 5; 9; 18)$}
	\label{table: finding all graphs in H_v(2, 2, 2, 2, 2, 5; 9; 18)}
\end{table}

\vspace{1em}
Algorithm \ref{algorithm: A2} is a simplified version of Algorithm 3.7 in \cite{BN17a}. Theorem \ref{theorem: algorithm A2} follows from Theorem 3.8 published in \cite{BN17a}.

\section{Computation of the numbers $\wFv{m}{5}{m - 1}$}

Let us remind that $\wHv{m}{p}{q}$ and $\wFv{m}{p}{q}$ are defined in Section 1.4.

According to Proposition \ref{proposition: wFv(m)(p)(q) exists}, we have
\begin{equation}
\label{equation: wFv(m)(5)(m - 1) exists}
\wFv{m}{5}{m - 1} \mbox{ exists } \Leftrightarrow m \geq 7.
\end{equation}

We prove the following
\begin{theorem}
	\label{theorem: wFv(m)(5)(m - 1) = ...}
	The following equalities are true:
	\begin{equation*}
	\wFv{m}{5}{m - 1} = 
	\begin{cases}
	17, & \emph{if $m = 7$}\\
	m + 9, & \emph{if $m \geq 8$}.
	\end{cases}
	\end{equation*}
\end{theorem}

\begin{proof}
	\emph{Case 1.} $m = 7$. According to Theorem \ref{theorem: F_v(2_(m - p), p; q) leq F_v(a_1, ..., a_s; q) leq wFv(m)(p)(q)} and Theorem \ref{theorem: F_v(2, 2, 5; 6) = 16}, $\wFv{7}{5}{6} \geq F_v(2, 2, 5; 6) = 16$. With the help of the computer we check that none of the 4 graphs in $\mH_v(3, 5; 6; 16)$ (see Figure \ref{figure: H_v(3, 5; 6; 16)}) belongs to $\mH_v(4, 4; 6; 16)$. Therefore, $\wHvn{7}{5}{6}{16} = \emptyset$ and $\wFv{7}{5}{6} \geq 17$.
	
	By adding one vertex to the graphs in $\mH_v(2, 2, 5; 6; 16)$, and then removing edges from the obtained 17-vertex graphs, we find 353 graphs which belong to both $\mH_v(3, 5; 6; 17)$ and $\mH_v(4, 4; 6; 17)$. The graph $\Gamma_1$, given in Figure \ref{figure: H_v(3, 5; 6; 17) cap H_v(4, 4; 6; 17)}, is one of these graphs (it is the only one with independence number 4). We will prove that $\Gamma_1 \in \wHv{7}{5}{6}$. Since $\omega(\Gamma_1) = 5$, it remains to be proved that if $2 \leq b_1 \leq ... \leq b_s \leq 5$ (see (\ref{equation: 2 leq a_1 leq ... leq a_s})) are positive integers such that $\sum_{i = 1}^s (b_1 - 1) + 1 = 7$, then $\Gamma_1 \arrowsv (b_1, ..., b_s)$. The following cases are possible:
	
	$s = 2:\ b_1 = 3, b_2 = 5$.
	
	$s = 2:\ b_1 = 4, b_2 = 4$.
	
	$s = 3:\ b_1 = 2, b_2 = 2, b_3 = 5$.
	
	$s = 3:\ b_1 = 2, b_2 = 3, b_3 = 4$.
	
	$s = 3:\ b_1 = 3, b_2 = 3, b_3 = 3$.
	
	$s = 4:\ b_1 = 2, b_2 = 2, b_3 = 2, b_4 = 4$.
	
	$s = 4:\ b_1 = 2, b_2 = 2, b_3 = 3, b_4 = 3$.
	
	$s = 5:\ b_1 = 2, b_2 = 2, b_3 = 2, b_4 = 2, b_5 = 3$.
	
	$s = 6:\ b_1 = 2, b_2 = 2, b_3 = 2, b_4 = 2, b_5 = 2, b_6 = 2$.
	
	According to Proposition \ref{proposition: a_1 + a_2 - 1 > p}, from $\Gamma_1 \arrowsv (3, 5)$ and $\Gamma_1 \arrowsv (4, 4)$ it follows $\Gamma_1 \arrowsv \uni{7}{5}$.
	We proved that $\Gamma_1 \in \wHv{7}{5}{6}$. Therefore, $\wFv{7}{5}{6} \leq \abs{\V(\Gamma_1)} = 17$.\\
	
	\emph{Case 2.} $m = 8$. According to Theorem \ref{theorem: F_v(2_(m - p), p; q) leq F_v(a_1, ..., a_s; q) leq wFv(m)(p)(q)} and Theorem \ref{theorem: rp(5) = 2}, $\wFv{8}{5}{7} \geq F_v(2, 2, 2, 5; 7) = 17$. To prove the upper bound, consider the 17-vertex graph $\Gamma_2 \in \mH_v(4, 5; 7; 17)$, which is given in Figure \ref{figure: H_v(4, 5; 7; 17)}. The method to obtain this graph is described below. By construction, $\omega(\Gamma_2) = 6$ and $\Gamma_2 \arrowsv (4, 5)$. According to Proposition \ref{proposition: a_1 + a_2 - 1 > p}, from $\Gamma_2 \arrowsv (4, 5)$ it follows $\Gamma_2 \arrowsv \uni{8}{5}$. Therefore, $\Gamma_2 \in \wHv{8}{5}{7}$ and $\wFv{8}{5}{7} \leq \abs{\V(\Gamma_2)} = 17$.\\
	
	\emph{Case 3.} $m > 8$. From Theorem \ref{theorem: F_v(2_(m - p), p; q) leq F_v(a_1, ..., a_s; q) leq wFv(m)(p)(q)} and Theorem \ref{theorem: rp(5) = 2} it follows $\wFv{m}{5}{m - 1} \geq m + 9$. From Theorem \ref{theorem: wFv(m)(p)(m - m_0 + q) leq wFv(m_0)(p)(q) + m - m_0}($m_0 = 8, p = 5, q = 7$) and $\wFv{8}{5}{7} = 17$ it follows $\wFv{m}{5}{m - 1} \leq m + 9$.
\end{proof}
\vspace{1em}

\vspace{-0.5em}
\subsection*{Obtaining the graph $\Gamma_2 \in \mH_v(4, 5; 7; 17)$}

\vspace{1em}
Consider the 18-vertex graph $\Gamma_3$ (Figure \ref{figure: H_v(3, 6; 7; 18) cap H_v(4, 5; 7; 18)}). As mentioned, this is the graph with the help of which in \cite{SXP09} they prove the inequality $F_v(3, 6; 7) \leq 18$. With а computer we check that $\Gamma_3$ is a maximal graph in $\mH_v(4, 5; 7; 18)$. We will use the following procedure to obtain other maximal graphs in $\mH_v(4, 5; 7; 18)$:

\begin{procedure}
	\label{procedure: extending a set of maximal graphs in mH_v(a_1, ..., a_s; q; n)}
	Extending a set of maximal graphs in $\mH_v(a_1, ..., a_s; q; n)$.
	
	1. Let $\mA$ be a set of maximal graphs in $\mH_v(a_1, ..., a_s; q; n)$.
	
	2. By removing edges from the graphs in $\mA$, find all their subgraphs which are in $\mH_v(a_1, ..., a_s; q; n)$. This way a set of non-maximal graphs in $\mH_v(a_1, ..., a_s; q; n)$ is obtained.
	
	3. Add edges to the non-maximal graphs to find all their supergraphs which are maximal in $\mH_v(a_1, ..., a_s; q; n)$. Extend the set $\mA$ by adding the new maximal graphs.
\end{procedure}

Starting from a set containing a single element the graph $\Gamma_3$ and executing Procedure \ref{procedure: extending a set of maximal graphs in mH_v(a_1, ..., a_s; q; n)}, we find 12 new maximal graphs in $\mH_v(4, 5; 7; 18)$. Again, we execute Procedure \ref{procedure: extending a set of maximal graphs in mH_v(a_1, ..., a_s; q; n)} on the new set to find 110 more maximal graphs in $\mH_v(4, 5; 7; 18)$. By removing one vertex from these graphs, we obtain 17-vertex graphs, one of which is $\Gamma_2 \in \mH_v(4, 5; 7; 17)$ given in Figure \ref{figure: H_v(4, 5; 7; 17)}.

\vspace{1em}
Theorem \ref{theorem: wFv(m)(5)(m - 1) = ...} and Procedure \ref{procedure: extending a set of maximal graphs in mH_v(a_1, ..., a_s; q; n)} are published in \cite{BN15a}.

\begin{figure}
	\begin{minipage}{0.55\textwidth}
		\hspace{4em}
		\begin{subfigure}[b]{0.5\textwidth}
			\centering
			\includegraphics[height=210px,width=105px]{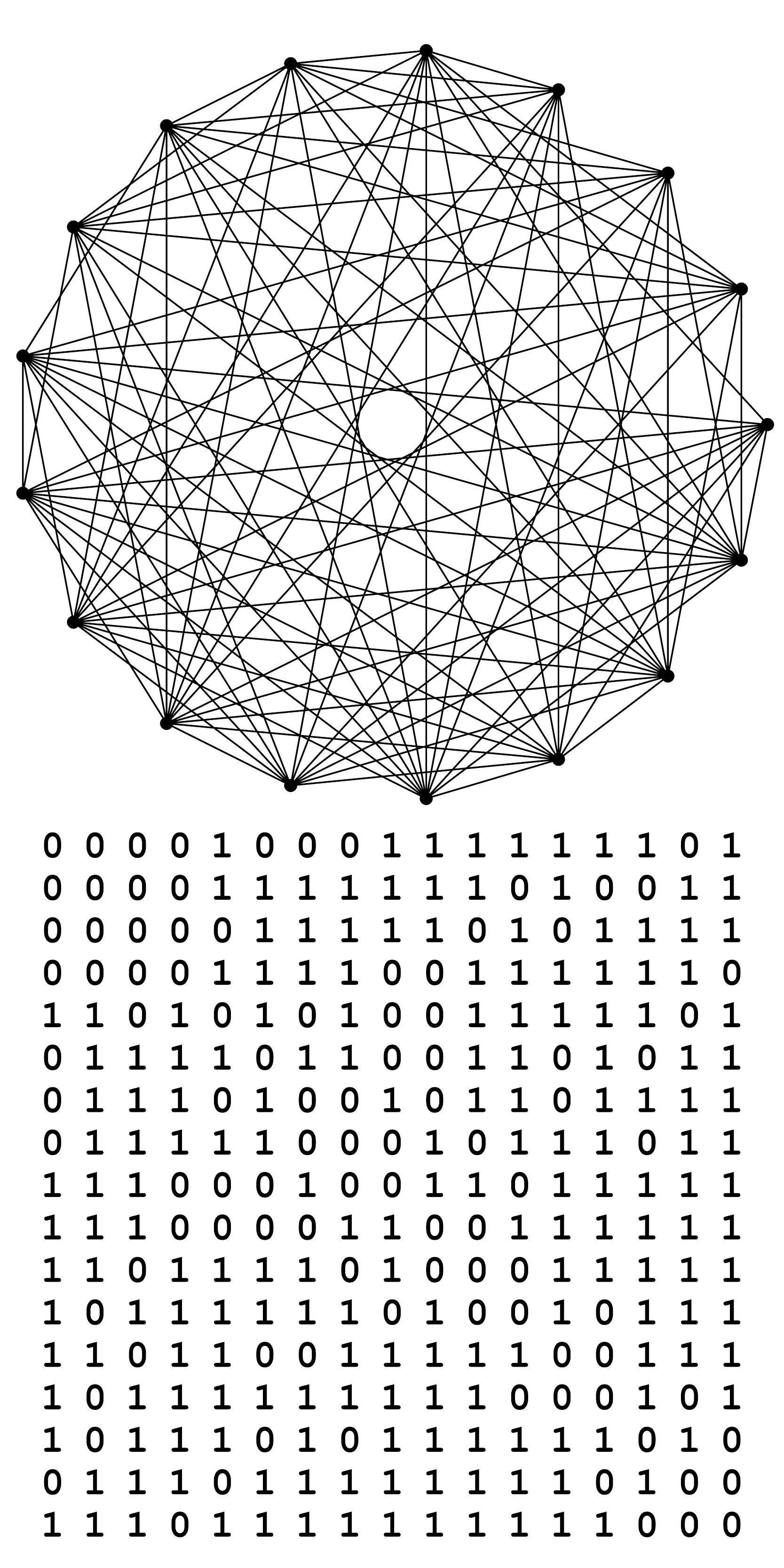}
		\end{subfigure}
		\caption*{\emph{$\Gamma_1$}}
		\label{figure: Gamma_1}
		
		\caption{\\\hspace{-4.5em}Graph $\Gamma_1 \in \mH(3, 5; 6; 17) \cap \mH(4, 4; 6; 17)$}
		\label{figure: H_v(3, 5; 6; 17) cap H_v(4, 4; 6; 17)}
	\end{minipage}
	\begin{minipage}{0.5\textwidth}
		\hspace{3em}
		\begin{subfigure}[b]{0.5\textwidth}
			\centering
			\includegraphics[height=210px,width=105px]{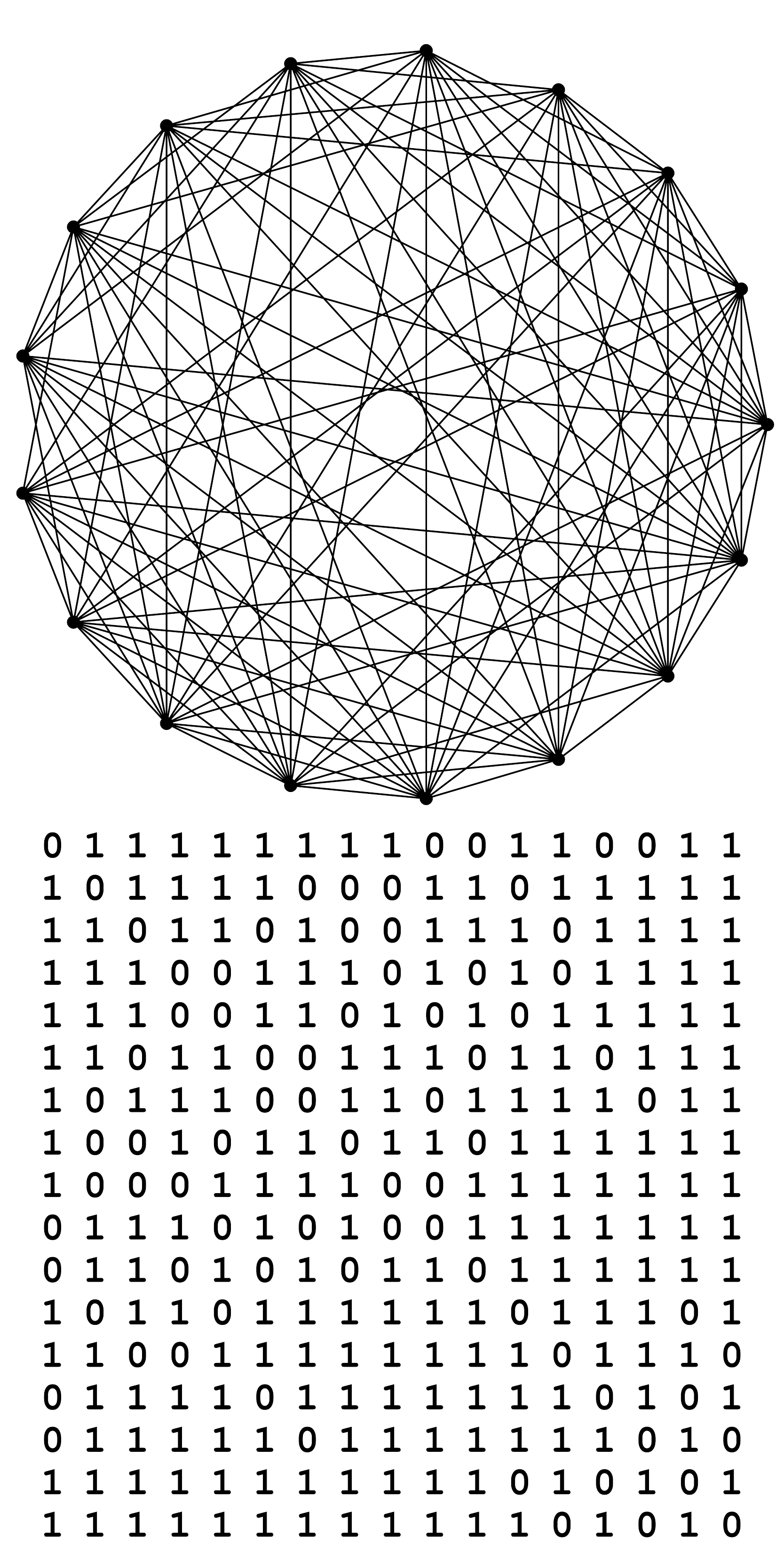}
		\end{subfigure}
		\caption*{\emph{$\Gamma_2$}}
		\label{figure: Gamma_2}
		
		\caption{\\\hspace{-0.5em}Graph $\Gamma_2 \in \mH_v(4, 5; 7; 17)$}
		\label{figure: H_v(4, 5; 7; 17)}
	\end{minipage}
	
	\centering
	\begin{subfigure}[b]{0.5\textwidth}
		\centering
		\includegraphics[height=240px,width=120px]{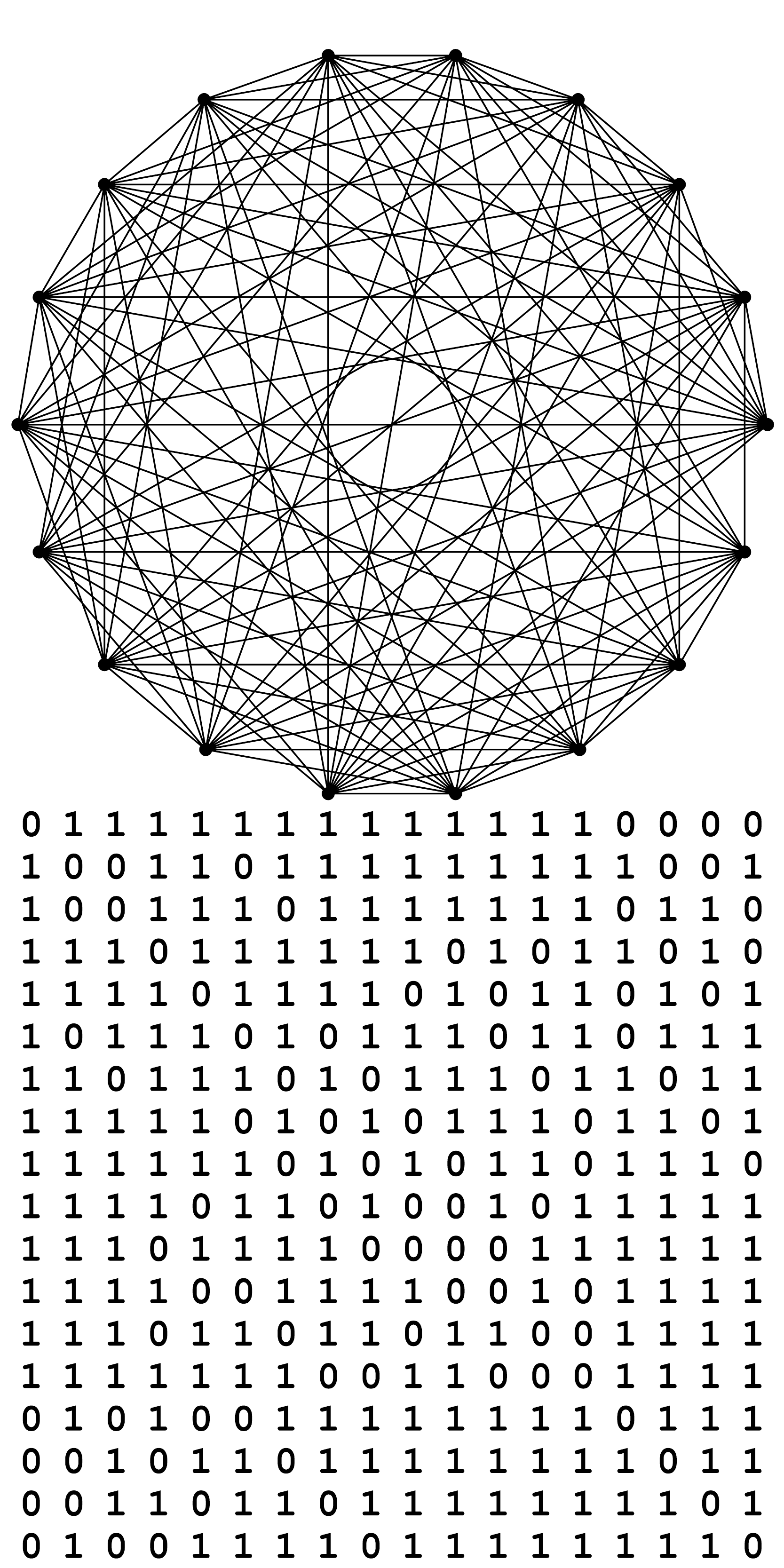}
	\end{subfigure}
	\caption*{\emph{$\Gamma_3$}}
	\label{figure: Gamma_3}
	
	\caption{\\Graph $\Gamma_3 \in \mH_v(3, 6; 7; 18) \cap \mH_v(4, 5; 7; 18)$ from \cite{SXP09}}
	\label{figure: H_v(3, 6; 7; 18) cap H_v(4, 5; 7; 18)}
\end{figure}

\section{Proof of Theorem \ref{theorem: F_v(a_1, ..., a_s; m - 1) = m + 9, max set(a_1, ..., a_s) = 5}}

Since $m \geq 7$, only the following two cases are possible:\\

\emph{Case 1.} $m = 7$. In this case $F_v(2, 2, 5; 6)$ and $F_v(3, 5; 6)$ are the only canonical vertex Folkman numbers of the form $F_v(a_1, ..., a_s; m - 1)$. The equality $F_v(2, 2, 5; 6) = 16$ is proved in Theorem \ref{theorem: F_v(2, 2, 5; 6) = 16}, and the equality $F_v(3, 5; 6) = 16$ is proved in \cite{SLPX12} (see also Corollary \ref{corollary: F_v(3, 5; 6) = 16}).\\

\emph{Case 2.} $m \geq 8$. According to Theorem \ref{theorem: wFv(m)(5)(m - 1) = ...} and Theorem \ref{theorem: F_v(2_(m - p), p; q) leq F_v(a_1, ..., a_s; q) leq wFv(m)(p)(q)}($q = m - 1$), $F_v(a_1, ..., a_s; m - 1) \leq m + 9$. From Theorem \ref{theorem: rp(5) = 2} and Theorem \ref{theorem: F_v(2_(m - p), p; q) leq F_v(a_1, ..., a_s; q) leq wFv(m)(p)(q)}($q = m - 1$) it follows that $F_v(a_1, ..., a_s; m - 1) \geq m + 9$.

\vspace{1em}
Theorem \ref{theorem: F_v(a_1, ..., a_s; m - 1) = m + 9, max set(a_1, ..., a_s) = 5} is published in \cite{BN15a}.

\chapter{Computation of the numbers $F_v(a_1, ..., a_s; m - 1)$ where $\max\{a_1, ..., a_s\} = 6$}

In this chapter we will prove the following main result:

\begin{theorem}
	\label{theorem: F_v(a_1, ..., a_s; m - 1) = ..., max set(a_1, ..., a_s) = 6}
	Let $a_1, ..., a_s$ be positive integers such that
	$$2 \leq a_1 \leq ... \leq a_s = 6,$$
	and $m = \sum\limits_{i=1}^s (a_i - 1) + 1 \geq 8$. Then:
	\vspace{1em}\\
	(a) $F_v(a_1, ..., a_s; m - 1) = m + 9$, if $a_1 = ... = a_{s - 1} = 2$. 
	\vspace{1em}\\
	(b) $F_v(a_1, ..., a_s; m - 1) = m + 10$, if $a_{s - 1} \geq 3$. 
\end{theorem}
According to (\ref{equation: F_v(a_1, ..., a_s; m - 1) exists}), the condition $m \geq 8$ in Theorem \ref{theorem: F_v(a_1, ..., a_s; m - 1) = ..., max set(a_1, ..., a_s) = 6} is necessary. Since $F_v(a_1, ..., a_s; q)$ is a symmetric function of $a_1, ..., a_s$, in Theorem \ref{theorem: F_v(a_1, ..., a_s; m - 1) = ..., max set(a_1, ..., a_s) = 6} we actually compute all numbers of the form $F_v(a_1, ..., a_s; m - 1)$ where $\max\set{a_1, ..., a_s} = 6$.

\vspace{1em}

\section{Algorithm A3}

We will present an optimization to Algorithm \ref{algorithm: A2}. The optimized algorithm is based on the following propositions:

\begin{proposition}
	\label{proposition: K_(q - 2) subseteq N_G(u) cap N_G(v)}
	Let $G$ be a maximal $K_q$-free graph and $u \neq v$ are non-adjacent vertices of $G$. Then,
	
	$K_{q - 2} \subseteq N_G(u) \cap N_G(v)$.
\end{proposition}

\begin{proof}
Otherwise, $G + [u,v]$ would not contain $K_q$.
\end{proof}

\begin{proposition}
	\label{proposition: alpha(G) leq t Leftrightarrow alpha(H - bigcup_(v in A') N_G((v)) leq t - abs(A')}
	Let $A$ be an independent set of vertices of $G$, $H = G - A$, and $t$ be a positive integer such that $t \geq \abs{A}$. Then,
	
	$\alpha(G) \leq t \Leftrightarrow \alpha(H - \bigcup_{v \in A'} N_G(v)) \leq t - \abs{A'}, \ \forall A' \subseteq A$.
\end{proposition}

\begin{proof}
	Let $\alpha(G) \leq t$. Suppose that for some $A' \subseteq A$ we have $\alpha(H - \bigcup_{v \in A'} N_G(v)) > t - \abs{A'}$. Consequently, there exists an independent set $A''$ of vertices of $H - \bigcup_{v \in A'} N_G(v)$ such that $\abs{A''} > t - \abs{A'}$. We obtained that the independent set $A' \cup A''$ has more than $t$ vertices, which is a contradiction.
	
	Now, let $\alpha(H - \bigcup_{v \in A'} N_G(v)) \leq t - \abs{A'}, \ \forall A' \subseteq A$. Let $\widetilde{A}$ be an independent set of vertices of $G$ and $\abs{\widetilde{A}} = \alpha(G)$. Then $\widetilde{A} = A_1 \cup A_2$ where $A_1 \subseteq A$ and $A_2$ is an independent set in $H - \bigcup_{v \in A_1} N_G(v)$. Since $\abs{A_2} \leq \alpha(H - \bigcup_{v \in A_1} N_G(v)) \leq t - \abs{A_1}$, we obtain $\alpha(G) = \abs{\widetilde{A}} = \abs{A_1} + \abs{A_2} \leq t$. 
\end{proof}

In some cases, a possible drawback of Algorithm \ref{algorithm: A2} is that in step 2.2 a large number ($l^k$) of multisets $N$ are considered, and therefore for a large number of graphs $G = G(N)$ it is checked if $\alpha(G) \leq t$. In the following Algorithm \ref{algorithm: A3} we modify step 2 in such a way that a smaller number of multisets $N$ are considered and it is guaranteed that the constructed graphs $G = G(N)$ satisfy the condition $\alpha(G) \leq t$.
\begin{namedalgorithm}{A3}
	\label{algorithm: A3}
	The input of the algorithm is the set $\mA = \mH_{max}^t(a_1 - 1, a_2, ..., a_s; q; n - k)$, where $a_1, ..., a_s, q, n, k, t$ are fixed positive integers, $a_1 \geq 2$ and $k \leq t$.
	
	The output of the algorithm is the set $\mB$ of all graphs $G \in \mH_{max}^t(a_1, ..., a_s; q; n)$ with $\alpha(G) \geq k$.
	
	\emph{1.} By removing edges from the graphs in $\mA$ obtain the set
	
	$\mA' = \mH_{+K_{q - 1}}^t(a_1 - 1, a_2, ..., a_s; q; n - k)$.
	
	\emph{2.} For each graph $H \in \mA'$:
	
	\emph{2.1.} Find the family $\mM(H) = \set{M_1, ..., M_l}$ of all maximal $K_{q - 1}$-free subsets of $\V(H)$.
	
	\emph{2.2.} Find all $k$-element multisets $N = \set{M_{i_1}, ..., M_{i_k}}$ of elements of $\mM(H)$ which fulfill the conditions:
	
	(a) $K_{q - 2} \subseteq M_{i_j} \cap M_{i_h}$ for every $M_{i_j}, M_{i_h} \in N, \ j \neq h$.
	
	(b) $\alpha(H - \bigcup_{M_{i_j} \in N'} M_{i_j}) \leq t - \abs{N'}$ for every subset $N'$ of $N$.
	
	\emph{2.3.} For each of the found in step 2.2 $k$-element multisets $N = \set{M_{i_1}, ..., M_{i_k}}$ of elements of $\mM(H)$ construct the graph $G = G(N)$ by adding new independent vertices $v_1, ..., v_k$ to $\V(H)$ such that $N_G(v_j) = M_{i_j}, j = 1, ..., k$. If $\omega(G + e) = q, \forall e \in \E(\overline{G})$, then add $G$ to $\mB$.
	
	\emph{3.} Remove the isomorphic copies of graphs from $\mB$.
	
	\emph{4.} Remove from $\mB$ all graphs $G$ for which $G \not\arrowsv (a_1, ..., a_s)$.
	
\end{namedalgorithm}

We will prove the correctness of Algorithm \ref{algorithm: A3} with the help of the following
\begin{lemma}
	\label{lemma: algorithm A3}
	After the execution of step 2.3 of Algorithm \ref{algorithm: A3}, the obtained set $\mB$ coincides with the set of all maximal $K_q$-free graphs $G$ with $k \leq \alpha(G) \leq t$ which have an independent set of vertices $A \subseteq \V(G), \abs{A} = k$ such that $G - A \in \mA'$.
\end{lemma}

\begin{proof}
	Suppose that in step 2.3 of Algorithm \ref{algorithm: A3} the graph $G$ is added to $\mB$. Then $G = G(N)$ and $G - \set{v_1, ..., v_k} = H \in \mA'$, where $N$, $v_1, ..., v_k$, and $H$ are the same as in step 2.3. Since $v_1, ..., v_k$ are independent, $\alpha(G) \geq k$. From the condition (b) in step 2.2 and Proposition \ref{proposition: alpha(G) leq t Leftrightarrow alpha(H - bigcup_(v in A') N_G((v)) leq t - abs(A')} it follows that $\alpha(G) \leq t$. By $H \in \mA'$ we have $\omega(H) < q$. Since $N_G(v_j), j = 1, ..., k$, are $K_{q - 1}$-free sets, it follows that $\omega(G) < q$. The check at the end of step 2.3 guarantees that $G$ is a maximal $K_q$-free graph. 
	
	Let $G$ be a maximal $K_q$-free graph, $k \leq \alpha(G) \leq t$, and $A = \set{v_1, ..., v_k}$ be an independent set of vertices of $G$ such that $H = G - A \in \mA'$. We will prove that, after the execution of step 2.3 of Algorithm \ref{algorithm: A3}, $G \in \mB$. Since $G$ is a maximal $K_q$-free graph, $N_G(v_j), j = 1, ..., k$, are maximal $K_{q - 1}$-free subsets of $V(H)$, and therefore $N_G(v_j) \in \mM(H), j = 1, ..., k$, see step 2.1. Let $N = \set{N_G(v_1), ..., N_G(v_k)}$. By Proposition \ref{proposition: K_(q - 2) subseteq N_G(u) cap N_G(v)}, $N$ fulfills the condition (a) in step 2.3, and by Proposition \ref{proposition: alpha(G) leq t Leftrightarrow alpha(H - bigcup_(v in A') N_G((v)) leq t - abs(A')}, $N$ also fulfills (b). Thus, we showed that $N$ fulfills all conditions in step 2.2, and since $G = G(N)$ is a maximal $K_q$-free graph, in step 2.3 $G$ is added to $\mB$.
\end{proof}

\begin{theorem}
	\label{theorem: algorithm A3}
	After the execution of Algorithm \ref{algorithm: A3}, the obtained set $\mB$ coincides with the set of all graphs $G \in \mH_{max}^t(a_1, ..., a_s; q; n)$ with $\alpha(G) \geq k$.
\end{theorem}

\begin{proof}
	Suppose that, after the execution of Algorithm \ref{algorithm: A3}, $G \in \mB$. According to Lemma \ref{lemma: algorithm A3}, $G$ is a maximal $K_q$-free graph and $k \leq \alpha(G) \leq t$. From step 4, $G \arrowsv(a_1, ..., a_s)$, therefore $G \in \mH_{max}^t(a_1, ..., a_s; q; n)$. 
	
	Conversely, let $G \in \mH_{max}^t(a_1, ..., a_s; q; n)$ and $\alpha(G) \geq k$. Let $A \subseteq \V(G)$ be an independent set of vertices of $G$, $\abs{A} = k$, and $H = G - A$. Using Proposition \ref{proposition: G - A in mH_(+K_(q - 1))(a_1 - 1, a_2, ..., a_s; q; n - abs(A))} we derive that $H \in \mA'$. From Lemma \ref{lemma: algorithm A3} it follows that after the execution of step 2.3, $G \in \mB$. Clearly, after step 4, $G$ remains in $\mB$.
\end{proof}

\begin{remark}
	\label{remark: algorithm A3 k = 2}
	Note that if $G \in \mH_{max}^t(a_1, ..., a_s; q; n)$ and $n \geq q$, then $G$ is not a complete graph and $\alpha(G) \geq 2$. Therefore, if $n \geq q$ and $k = 2$, Algorithm \ref{algorithm: A3} finds all graphs in $\mH_{max}^t(a_1, ..., a_s; q; n)$.
\end{remark}

We tested our implementation of Algorithm \ref{algorithm: A3} by reproducing the graphs obtained with the help of Algorithm \ref{algorithm: A2} in the second proof of Theorem \ref{theorem: rp(5) = 2} (see Table \ref{table: finding all graphs in H_v(2, 2, 2, 5; 7; 16)}, Table \ref{table: finding all graphs in H_v(2, 2, 2, 2, 5; 8; 17)}, and Table \ref{table: finding all graphs in H_v(2, 2, 2, 2, 2, 5; 9; 18)}). In this case, Algorithm \ref{algorithm: A3} is about 4 times faster than Algorithm \ref{algorithm: A2} and the total computational time for the proof of Theorem \ref{theorem: rp(5) = 2} is reduced about 2 times. Further, we will apply Algorithm \ref{algorithm: A3} to solve problems which cannot be solved in a reasonable amount of time using Algorithm \ref{algorithm: A2}. Some examples for such problems are Theorem \ref{theorem: abs(mH_v(2, 2, 6; 7; 18)) = 76515}, Theorem \ref{theorem: F_v(2, 2, 7; 8) = 20}, and Theorem \ref{theorem: F_v(4, 4; 5) geq F_v(2, 3, 4; 5) geq F_v(2, 2, 2, 4; 5) geq 19}.\\

\vspace{1em}

At the end of this section, we will propose a method to improve Algorithm \ref{algorithm: A3} which is based on the following proposition:
\begin{proposition}
	Let $G \in \mH_v(2, 2, p; p + 1)$ and $v \in \V(G)$. Then all non-neighbors of $v$ induce a graph with chromatic number greater than 2. In particular, from $G \in \mH_v(2, 2, p; p + 1)$ it follows that $\Delta(G) \leq \abs{\V(G)} - 4$.
\end{proposition}

\begin{proof}
Let $W$ be the set of non-neighbors of $v$. If we assume that $W = V_1 \cup V_2$ where $V_1$ and $V_2$ are independent sets, then since $\N(v)$ does not contain a $p$-clique, from $\V(G) = V_1 \cup V_2 \cup \N(v)$ it follows that $G \not\arrowsv (2, 2, p)$. 
\end{proof}
As we will see further (see Table \ref{table: H_v(2, 2, 6; 7; 17) properties} and Table \ref{table: H_v(2, 2, 6; 7; 18) properties}), the inequality $\Delta(G) \leq \abs{\V(G)} - 4$ is exact. In some special cases, for example the proofs of Theorem \ref{theorem: F_v(2, 2, 6; 7) = 17 and abs(mH_v(2, 2, 6; 7; 17)) = 3}, Theorem \ref{theorem: abs(mH_v(2, 2, 6; 7; 18)) = 76515} and Theorem \ref{theorem: F_v(2, 2, 7; 8) = 20}, we can use the inequality $\Delta(G) \leq \abs{\V(G)} - 4$ to speed up computations in some parts of the proofs. We used this inequality only to make sure that the obtained results are correct.

\vspace{1em}

\vspace{1em}
Theorem \ref{theorem: algorithm A3} and Algorithm \ref{algorithm: A3} are published in \cite{BN17a}.

\vspace{1em}

\section{Computation of the numbers $F_v(2, 2, 6; 7)$ and $F_v(3, 6; 7)$}

\vspace{1em}

Let $a_1, ..., a_s$ be positive integers and let $m$ and $p$ be defined by (\ref{equation: m and p}). According to (\ref{equation: F_v(a_1, ..., a_s; m - 1) exists}), $F_v(a_1, ..., a_s; m - 1)$ exists if and only if $m \geq p + 2$. In the border case $m = p + 2, \ p \geq 3$, there are only two canonical numbers in the form $F_v(a_1, ..., a_s; m - 1)$, namely $F_v(2, 2, p; p + 1)$ and $F_v(3, p; p + 1)$. The computation of the numbers $F_v(a_1, ..., a_s; m - 1)$ where $\max\set{a_1, ..., a_s} = 6$, i.e. the proof of Theorem \ref{theorem: F_v(a_1, ..., a_s; m - 1) = ..., max set(a_1, ..., a_s) = 6}, will be done with the help of the numbers $F_v(2, 2, 6; 7)$ and $F_v(3, 6; 7)$. Because of this, we will first compute these two numbers. In \cite{SXP09} it is proved $F_v(3, 6; 7) \leq 18$. According to Theorem \ref{theorem: F_v(a_1, ..., a_s; m - 1) = m + 9, max set(a_1, ..., a_s) = 5}, $F_v(2, 2, 2, 5; 7) = 17$. From the inclusion $\mH_v(3, 6; 7) \subseteq \mH_v(2, 2, 6; 7) \subseteq \mH_v(2, 2, 2, 5; 7)$ it follows that
	\begin{equation}
	\label{equation: 17 leq F_v(2, 2, 6; 7) leq F_v(3, 6; 7) leq 18}
	17 = F_v(2, 2, 2, 5; 7) \leq F_v(2, 2, 6; 7) \leq F_v(3, 6; 7) \leq 18.
	\end{equation}

\begin{figure}
	\centering
	\begin{subfigure}{0.5\textwidth}
		\centering
		\includegraphics[height=252px,width=126px]{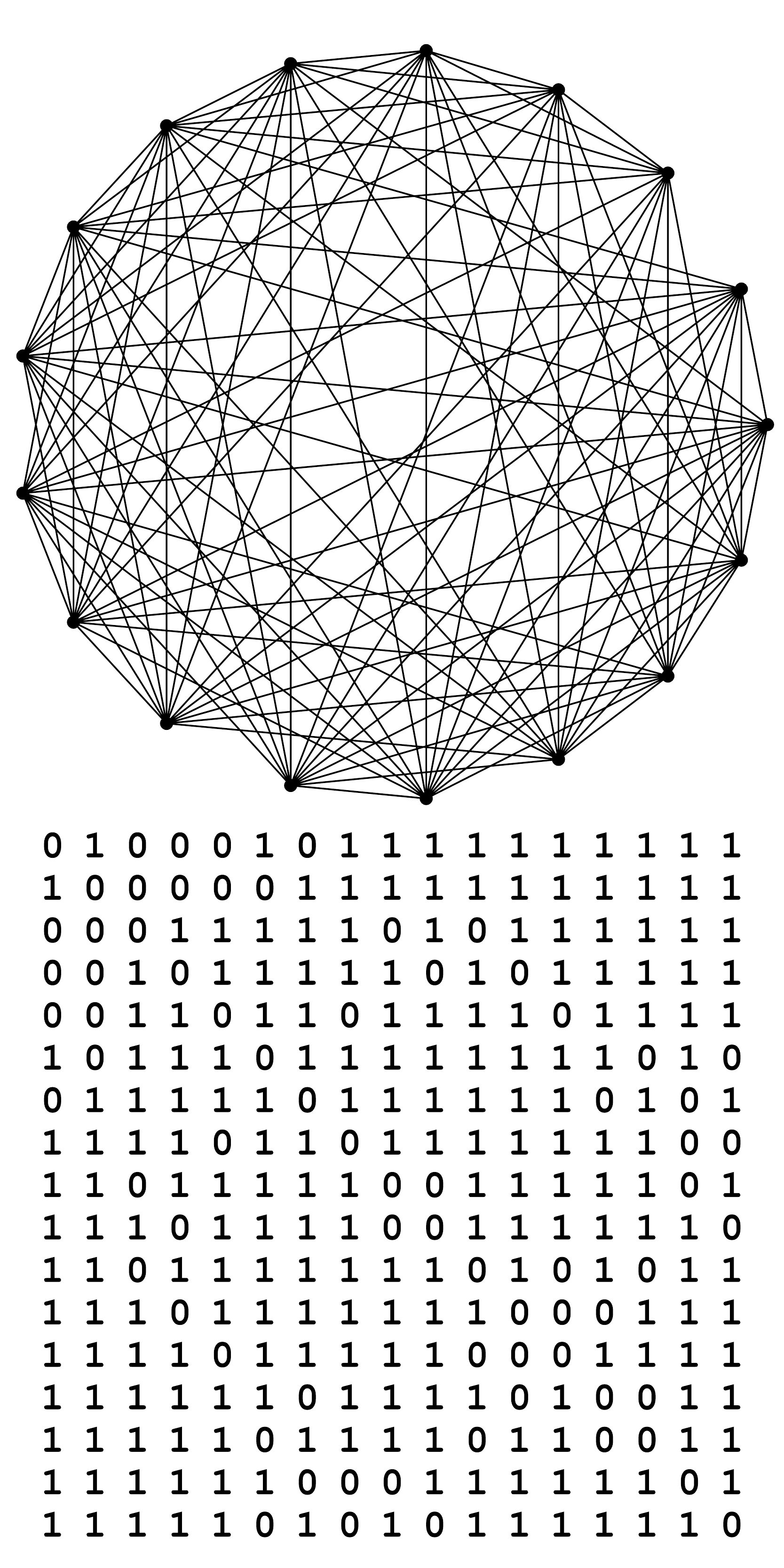}
		\caption*{\emph{$G_{1}$}}
		\label{figure: G_1}
	\end{subfigure}%
	\begin{subfigure}{0.5\textwidth}
		\centering
		\includegraphics[height=252px,width=126px]{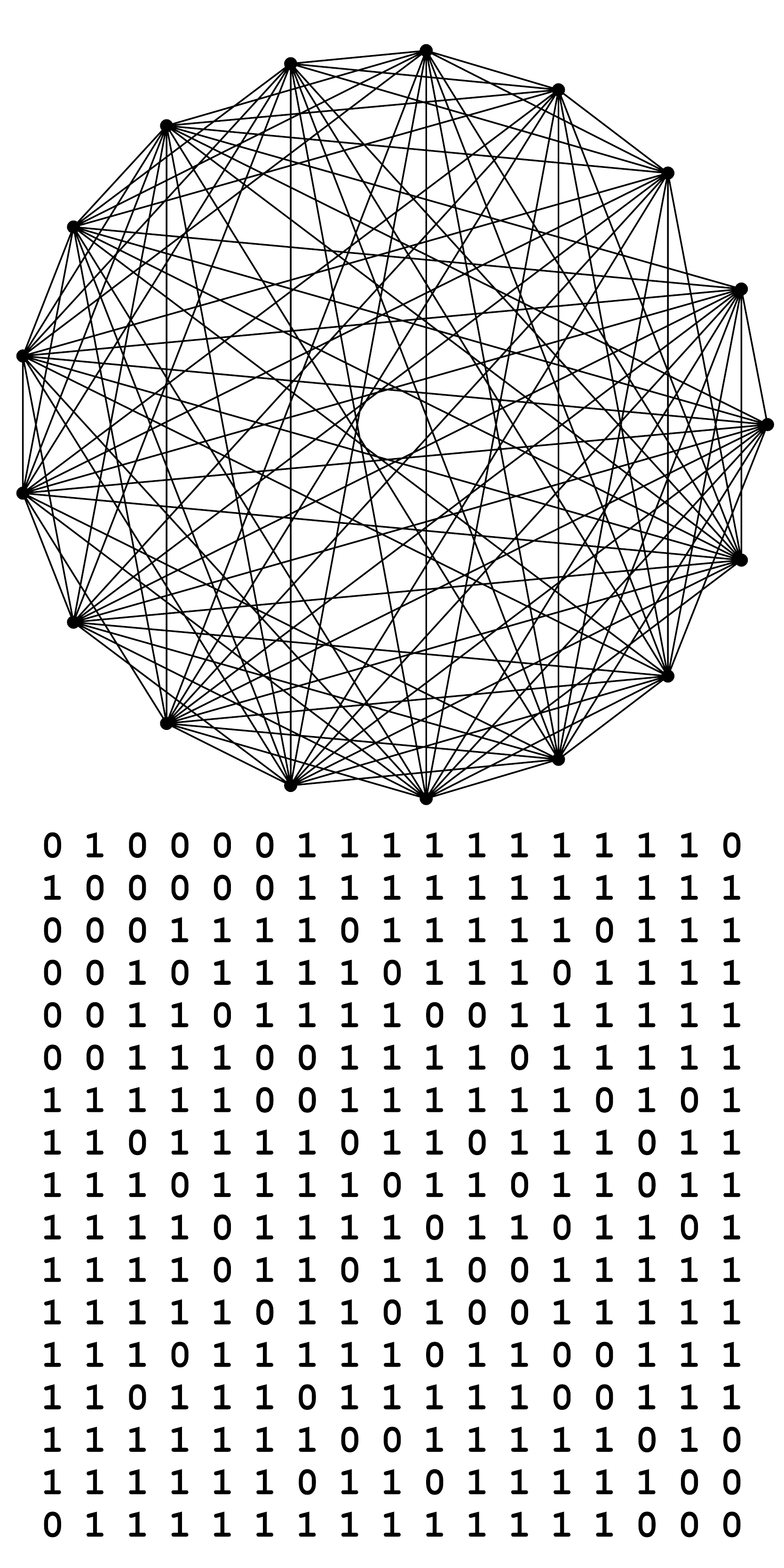}
		\caption*{\emph{$G_{2}$}}
		\label{figure: G_2}
	\end{subfigure}
	
	\begin{subfigure}{0.5\textwidth}
		\centering
		\includegraphics[height=252px,width=126px]{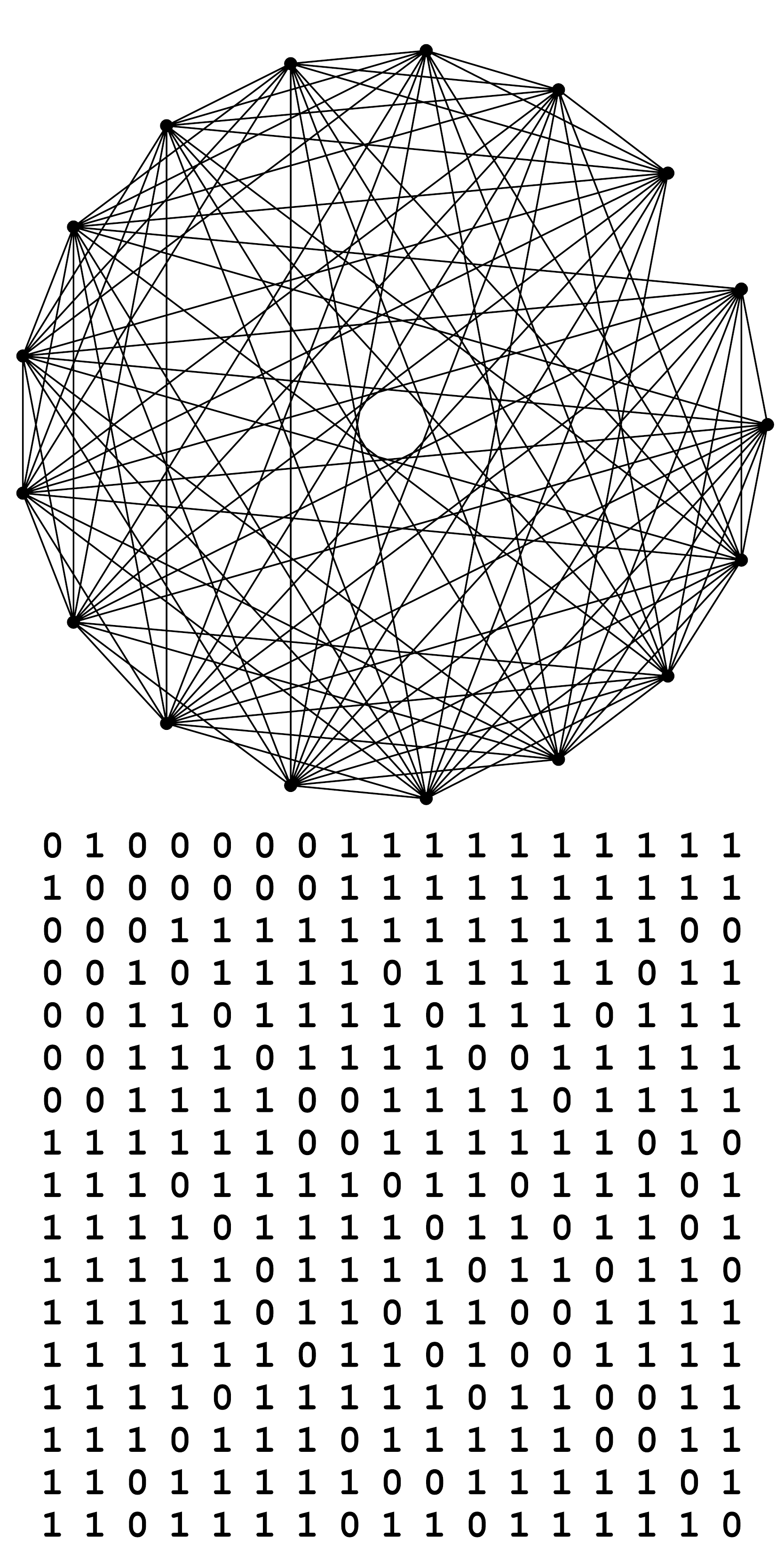}
		\caption*{\emph{$G_{3}$}}
		\label{figure: G_3}
	\end{subfigure}%
	\caption{All 3 graphs in $\mH_v(2, 2, 6; 7; 17)$}
	\label{figure: H_v(2, 2, 6; 7; 17)}
\end{figure}

\begin{theorem}
	\label{theorem: F_v(2, 2, 6; 7) = 17 and abs(mH_v(2, 2, 6; 7; 17)) = 3}
	$\mH_{max}(2, 2, 6; 7; 17) = \set{G_1}$ and $F_v(2, 2, 6; 7) = 17$. Furthermore, $\mH_{extr}(2, 2, 6; 7) = \set{G_1, G_2, G_3}$ (see Figure \ref{figure: H_v(2, 2, 6; 7; 17)}).
\end{theorem}

\begin{proof}
	We will find all graphs in $\mH_v(2, 2, 6; 7; 17)$ with the help of a computer. Let $G \in \mH_v(2, 2, 6; 7; 17)$. Clearly, $\alpha(G) \geq 2$, and according to Theorem \ref{theorem: alpha(G) leq V(G) + m + p - 1}(b), $\alpha(G) \leq 3$.\\

	\vspace{-1.5em}
	
	First, we prove that there are no graphs in $\mH_{max}(2, 2, 6; 7; 17)$ with independence number 3:
		
	It is clear that $K_6$ is the only graph in $\mH_{max}(3; 7; 6)$.
		
	We execute Algorithm \ref{algorithm: A3}($n = 8;\ k = 2;\ t = 3$) with input $\mA = \mH_{max}^3(3; 7; 6) = \set{K_6}$ to obtain all graphs in $\mB = \mH_{max}^3(4; 7; 8)$. (see Remark \ref{remark: algorithm A3 k = 2})
		
	We execute Algorithm \ref{algorithm: A3}($n = 10;\ k = 2;\ t = 3$) with input $\mA = \mH_{max}^3(4; 7; 8)$ to obtain all graphs in $\mB = \mH_{max}^3(5; 7; 10)$.
		
	We execute Algorithm \ref{algorithm: A3}($n = 12;\ k = 2;\ t = 3$) with input $\mA = \mH_{max}^3(5; 7; 10)$ to obtain all graphs in $\mB = \mH_{max}^3(6; 7; 12)$.
		
	We execute Algorithm \ref{algorithm: A3}($n = 14;\ k = 2;\ t = 3$) with input $\mA = \mH_{max}^3(1, 6; 7; 12) = \mH_{max}^3(6; 7; 12)$ to obtain all graphs in $\mB = \mH_{max}^3(2, 6; 7; 14)$.
		
	By executing Algorithm \ref{algorithm: A3}($n = 17;\ k = 3;\ t = 3$) with input $\mA = \mH_{max}^3(1, 2, 6; 7; 14) = \mH_{max}^3(2, 6; 7; 14)$, we obtain $\mB = \emptyset$. According to Theorem \ref{theorem: algorithm A3}, there are no graphs in $\mH_{max}(2, 2, 6; 7; 17)$ with independence number 3.\\
		
	\vspace{-1.5em}
	
	It remains to find all graphs in $\mH_{max}(2, 2, 6; 7; 17)$ with independence number 2:
	
	It is clear that $K_7 - e$ is the only graph in $\mH_{max}(3; 7; 7)$.
	
	We execute Algorithm \ref{algorithm: A3}($n = 9;\ k = 2;\ t = 2$) with input $\mA = \mH_{max}^2(3; 7; 7) = \set{K_7 - e}$ to obtain all graphs in $\mB = \mH_{max}^2(4; 7; 9)$. (see Remark \ref{remark: algorithm A3 k = 2})
	
	We execute Algorithm \ref{algorithm: A3}($n = 11;\ k = 2;\ t = 2$) with input $\mA = \mH_{max}^2(4; 7; 9)$ to obtain all graphs in $\mB = \mH_{max}^2(5; 7; 11)$.
	
	We execute Algorithm \ref{algorithm: A3}($n = 13;\ k = 2;\ t = 2$) with input $\mA = \mH_{max}^2(5; 7; 11)$ to obtain all graphs in $\mB = \mH_{max}^2(6; 7; 13)$.
	
	We execute Algorithm \ref{algorithm: A3}($n = 15;\ k = 2;\ t = 2$) with input $\mA = \mH_{max}^2(1, 6; 7; 13) = \mH_{max}^2(6; 7; 13)$ to obtain all graphs in $\mB = \mH_{max}^2(2, 6; 7; 15)$.
	
	By executing Algorithm \ref{algorithm: A3}($n = 17;\ k = 3;\ t = 2$) with input $\mA = \mH_{max}^2(1, 2, 6; 7; 15) = \mH_{max}^2(2, 6; 7; 15)$, we obtain $\mB = \set{G_1}$, where the graph $G_1$ is shown in Figure \ref{figure: H_v(2, 2, 6; 7; 17)}. According to Theorem \ref{theorem: algorithm A3}, $G_1$ is the only graph in $\mH_{max}(2, 2, 6; 7; 17)$ with independence number 2. Since there are no graphs in $\mH_{max}(2, 2, 6; 7; 17)$ with independence number greater than 2, we proved that $\mH_{max}(2, 2, 6; 7; 17) = \set{G_1}$.

	The number of maximal $K_7$-free graphs and $(+K_6)$-graphs obtained in each step of the proof is given in Table \ref{table: finding all graphs in H_v(2, 2, 6; 7; 17)}. By removing edges from $G_1$ we find that there are only two other graphs in $\mH_v(2, 2, 6; 7; 17)$, which we will denote by $G_2$ and $G_3$ (see Figure \ref{figure: H_v(2, 2, 6; 7; 17)}). We proved that $\mH_{extr}(2, 2, 6; 7) = \set{G_1, G_2, G_3}$. Thus, we finish the proof of Theorem \ref{theorem: F_v(2, 2, 6; 7) = 17 and abs(mH_v(2, 2, 6; 7; 17)) = 3}.
\end{proof}
	
	Let us also note, that $G_1 \supset G_2 \supset G_3$, and for the graphs $G_1$, $G_2$, and $G_3$ the inequality (\ref{equation: G arrowsv (a_1, ..., a_s) Rightarrow chi(G) geq m}) is strict (see Conjecture \ref{conjecture: chi(G) leq m + 1}). It is clear that $G_3$ is the only minimal graph in $\mH_v(2, 2, 6; 7; 17)$. Some properties of the graphs $G_1$, $G_2$, and $G_3$ are listed in Table \ref{table: H_v(2, 2, 6; 7; 17) properties}.

\begin{table}
	\centering
	\resizebox{\textwidth}{!}{
		\begin{tabular}{ | l | r | r | r | r | r | r | }
			\hline
			{\parbox{4em}{Graph}}&
			{\parbox{4em}{\hfill$\abs{\E(G)}$}}&
			{\parbox{4em}{\hfill$\delta(G)$}}&
			{\parbox{4em}{\hfill$\Delta(G)$}}&
			{\parbox{4em}{\hfill$\alpha(G)$}}&
			{\parbox{4em}{\hfill$\chi(G)$}}&
			{\parbox{4em}{\hfill$\abs{Aut(G)}$}}\\
			\hline
			$G_{1}$		&  108			& 12			& 13			& 2				& 9				& 2		\\
			$G_{2}$		&  107			& 11			& 13			& 2				& 9				& 4		\\
			$G_{3}$		&  106			& 11			& 13			& 2				& 9				& 40	\\
			\hline
		\end{tabular}
	}
	\caption{Some properties of the graphs in $\mH_v(2, 2, 6; 7; 17)$}
	\label{table: H_v(2, 2, 6; 7; 17) properties}
\end{table}

\begin{theorem}
	\label{theorem: F_v(3, 6; 7) = 18}
	$F_v(3, 6; 7) = 18$.
\end{theorem}

\begin{proof}
According to (\ref{equation: 17 leq F_v(2, 2, 6; 7) leq F_v(3, 6; 7) leq 18}), it remains to be proved that $F_v(3, 6; 7) \neq 17$. From (\ref{equation: G arrowsv (a_1, ..., a_s) Rightarrow G arrowsv (a_1, ..., a_(i - 1), t, a_i - t + 1, a_(i + 1), ..., a_s)}) we have $\mH_v(3, 6; 7) \subseteq \mH_v(2, 2, 6; 7)$. By Theorem \ref{theorem: F_v(2, 2, 6; 7) = 17 and abs(mH_v(2, 2, 6; 7; 17)) = 3}, $\mH_v(2, 2, 6; 7) = \set{G_1, G_2, G_3}$ (see Figure \ref{figure: H_v(2, 2, 6; 7; 17)}). We check with a computer that none of the graphs $G_1$, $G_2$, and $G_3$ belongs to $\mH_v(3, 6; 7; 17)$, and therefore we obtain $\mH_v(3, 6; 7; 17) = \emptyset$. Since the only maximal graph in $\mH_v(2, 2, 6; 7)$ is the graph $G_1$, it is enough to check that $G_1 \not\in \mH_v(3, 6; 7)$.
\end{proof}

Using Algorithm \ref{algorithm: A3}, we were also able to obtain all graphs in $\mH_v(2, 2, 6; 7; 18)$:
\begin{theorem}
	\label{theorem: abs(mH_v(2, 2, 6; 7; 18)) = 76515}
	$\abs{\mH_v(2, 2, 6; 7; 18)} = 76515$.
\end{theorem}

\begin{proof}
	Similarly to the proof of Theorem \ref{theorem: F_v(2, 2, 6; 7) = 17 and abs(mH_v(2, 2, 6; 7; 17)) = 3}, we will find all graphs in $\mH_v(2, 2, 6; 7; 18)$ with the help of a computer. Some of the graphs that we obtain in the steps of this proof were already obtained in the proof of Theorem \ref{theorem: F_v(2, 2, 6; 7) = 17 and abs(mH_v(2, 2, 6; 7; 17)) = 3} (compare Table \ref{table: finding all graphs in H_v(2, 2, 6; 7; 17)} to Table \ref{table: finding all graphs in H_v(2, 2, 6; 7; 18)}).
	
	Let $G \in \mH_v(2, 2, 6; 7; 18)$. Clearly, $\alpha(G) \geq 2$, and according to Theorem \ref{theorem: alpha(G) leq V(G) + m + p - 1}(b), $\alpha(G) \leq 4$.\\
	
	\vspace{-1em}
	
	First, we prove that there are no graphs in $\mH_{max}(2, 2, 6; 7; 18)$ with independence number 4:
	
	The only graph in $\mH_{max}(3; 7; 6)$ is $K_6$.
	
	We execute Algorithm \ref{algorithm: A3}($n = 8;\ k = 2;\ t = 4$) with input $\mA = \mH_{max}^4(3; 7; 6) = \set{K_6}$ to obtain all graphs in $\mB = \mH_{max}^4(4; 7; 8)$. (see Remark \ref{remark: algorithm A3 k = 2})
	
	We execute Algorithm \ref{algorithm: A3}($n = 10;\ k = 2;\ t = 4$) with input $\mA = \mH_{max}^4(4; 7; 8)$ to obtain all graphs in $\mB = \mH_{max}^4(5; 7; 10)$.
	
	We execute Algorithm \ref{algorithm: A3}($n = 12;\ k = 2;\ t = 4$) with input $\mA = \mH_{max}^4(5; 7; 10)$ to obtain all graphs in $\mB = \mH_{max}^4(6; 7; 12)$.
	
	We execute Algorithm \ref{algorithm: A3}($n = 14;\ k = 2;\ t = 4$) with input $\mA = \mH_{max}^4(1, 6; 7; 12) = \mH_{max}^4(6; 7; 12)$ to obtain all graphs in $\mB = \mH_{max}^4(2, 6; 7; 14)$.
	
	By executing Algorithm \ref{algorithm: A3}($n = 18;\ k = 4;\ t = 4$) with input $\mA = \mH_{max}^4(1, 2, 6; 7; 14) = \mH_{max}^4(2, 6; 7; 14)$, we obtain $\mB = \emptyset$. According to Theorem \ref{theorem: algorithm A3}, there are no graphs in $\mH_{max}(2, 2, 6; 7; 18)$ with independence number 4.\\
		
	\vspace{-1em}
	
	Next, we find all graphs in $\mH_{max}(2, 2, 6; 7; 18)$ with independence number 3:
	
	The only graph in $\mH_{max}(3; 7; 7)$ is $K_7 - e$
	
	We execute Algorithm \ref{algorithm: A3}($n = 9;\ k = 2;\ t = 3$) with input $\mA = \mH_{max}^3(3; 7; 7) = \set{K_7 - e}$ to obtain all graphs in $\mB = \mH_{max}^3(4; 7; 9)$. (see Remark \ref{remark: algorithm A3 k = 2})
	
	We execute Algorithm \ref{algorithm: A3}($n = 11;\ k = 2;\ t = 3$) with input $\mA = \mH_{max}^3(4; 7; 9)$ to obtain all graphs in $\mB = \mH_{max}^3(5; 7; 11)$.
	
	We execute Algorithm \ref{algorithm: A3}($n = 13;\ k = 2;\ t = 3$) with input $\mA = \mH_{max}^3(5; 7; 11)$ to obtain all graphs in $\mB = \mH_{max}^3(6; 7; 13)$.
	
	We execute Algorithm \ref{algorithm: A3}($n = 15;\ k = 2;\ t = 3$) with input $\mA = \mH_{max}^3(1, 6; 7; 13) = \mH_{max}^3(6; 7; 13)$ to obtain all graphs in $\mB = \mH_{max}^3(2, 6; 7; 15)$.
	
	By executing Algorithm \ref{algorithm: A3}($n = 18;\ k = 3;\ t = 3$) with input $\mA = \mH_{max}^3(1, 2, 6; 7; 15) = \mH_{max}^3(2, 6; 7; 15)$, we obtain all 308 graphs in $\mB$. According to Theorem \ref{theorem: algorithm A3}, the graphs in $\mB$ are all graphs in $\mH_{max}(2, 2, 6; 7; 18)$ with independence number 3.\\
	
	\vspace{-1em}
	
	The last step, is to find all graphs in $\mH_{max}(2, 2, 6; 7; 18)$ with independence number 2:
	
	It is easy to see that $\mH_{max}(3; 7; 8) = \set{\overline{K}_3 + K_5, C_4 + K_4}$, and therefore $C_4 + K_4$ is the only graph in $\mH_{max}(3; 7; 8)$ with independence number 2.

	We execute Algorithm \ref{algorithm: A3}($n = 10;\ k = 2;\ t = 2$) with input $\mA = \mH_{max}^2(3; 7; 8) = \set{C_4 + K_4}$ to obtain all graphs in $\mB = \mH_{max}^2(4; 7; 10)$. (see Remark \ref{remark: algorithm A3 k = 2})
	
	We execute Algorithm \ref{algorithm: A3}($n = 12;\ k = 2;\ t = 2$) with input $\mA = \mH_{max}^2(4; 7; 10)$ to obtain all graphs in $\mB = \mH_{max}^2(5; 7; 12)$.
	
	We execute Algorithm \ref{algorithm: A3}($n = 14;\ k = 2;\ t = 2$) with input $\mA = \mH_{max}^2(5; 7; 12)$ to obtain all graphs in $\mB = \mH_{max}^2(6; 7; 14)$.
	
	We execute Algorithm \ref{algorithm: A3}($n = 16;\ k = 2;\ t = 2$) with input $\mA = \mH_{max}^2(1, 6; 7; 14) = \mH_{max}^2(6; 7; 14)$ to obtain all graphs in $\mB = \mH_{max}^2(2, 6; 7; 16)$.
	
	By executing Algorithm \ref{algorithm: A3}($n = 18;\ k = 2;\ t = 2$) with input $\mA = \mH_{max}^2(1, 2, 6; 7; 16) = \mH_{max}^2(2, 6; 7; 16)$, we obtain all 84 graphs in $\mB$. According to Theorem \ref{theorem: algorithm A3}, the graphs in $\mB$ are all graphs in $\mH_{max}(2, 2, 6; 7; 18)$ with independence number 2.\\
	
	\vspace{-1em}
	
	Thus, we obtained all 392 graphs in $\mH_{max}(2, 2, 6; 7; 18)$. By removing edges from these graphs we find all 76 515 graphs in $\mH_v(2, 2, 6; 7; 18)$. Some properties of these graphs are listed in Table \ref{table: H_v(2, 2, 6; 7; 18) properties}.	The number of maximal $K_7$-free graphs and $(+K_6)$-graphs obtained in each step of the proof is given in Table \ref{table: finding all graphs in H_v(2, 2, 6; 7; 18)}. Because of the large number of graphs in $\mH_{+K_6}^2(2, 6; 7; 16)$, we needed about two weeks to complete the computations.
\end{proof}

\begin{table}
	\centering
	\resizebox{\textwidth}{!}{
		\begin{tabular}{ | l r | l r | l r | l r | l r | l r | }
			\hline
			\multicolumn{2}{|c|}{\parbox{5em}{$\abs{\E(G)}$ \hfill $\#$}}&
			\multicolumn{2}{|c|}{\parbox{5em}{$\delta(G)$ \hfill $\#$}}&
			\multicolumn{2}{|c|}{\parbox{5em}{$\Delta(G)$ \hfill $\#$}}&
			\multicolumn{2}{|c|}{\parbox{5em}{$\alpha(G)$ \hfill $\#$}}&
			\multicolumn{2}{|c|}{\parbox{5em}{$\chi(G)$ \hfill $\#$}}&
			\multicolumn{2}{|c|}{\parbox{5em}{$\abs{Aut(G)}$ \hfill $\#$}}\\
			\hline
			106	& 1				& 0		& 3			& 13	& 65			& 2	& 290				& 8	& 84				& 1		& 72 335	\\
			107	& 4				& 1		& 20			& 14	& 76 450		& 3	& 76 225			& 9	& 76 431			& 2		& 3 699	\\
			108	& 19				& 2		& 124			& 		& 				& 	& 					& 	& 					& 4		& 430		\\
			109	& 88				& 3		& 571			& 		& 				& 	& 					& 	& 					& 8		& 33		\\
			110	& 369				& 4		& 1 943		& 		& 				& 	& 					& 	& 					& 10	& 2		\\
			111	& 1 240			& 5		& 4 986		& 		& 				& 	& 					& 	& 					& 16	& 2		\\
			112	& 3 303			& 6		& 9 826		& 		& 				& 	& 					& 	& 					& 20	& 6		\\
			113	& 6 999			& 7		& 14 896		& 		& 				& 	& 					& 	& 					& 24	& 1		\\
			114	& 11 780			& 8		& 17 057		& 		& 				& 	& 					& 	& 					& 36	& 1		\\
			115	& 15 603			& 9		& 14 288		& 		& 				& 	& 					& 	& 					& 40	& 6		\\
			116	& 15 956			& 10	& 8 397		& 		& 				& 	& 					& 	& 					& 		&			\\
			117	& 12 266			& 11	& 3 504		& 		& 				& 	& 					& 	& 					& 		&			\\
			118	& 6 575			& 12	& 876			& 		& 				& 	& 					& 	& 					& 		&			\\
			119	& 2 044			& 13	& 24			& 		& 				& 	& 					& 	& 					& 		&			\\
			120	& 261				& 		& 				& 		& 				& 	& 					& 	& 					& 		&			\\
			121	& 7				& 		& 				& 		& 				& 	& 					& 	& 					& 		&			\\
			\hline
		\end{tabular}
	}
	\caption{Some properties of the graphs in $\mH_v(2, 2, 6; 7; 18)$}
	\label{table: H_v(2, 2, 6; 7; 18) properties}
	\vspace{-1em}
\end{table}

We check with a computer that among the 76 515 graphs in $\mH_v(2, 2, 6; 7; 18)$, only the graph $\Gamma_3$ (see Figure \ref{figure: H_v(3, 6; 7; 18) cap H_v(4, 5; 7; 18)}) belongs to $\mH_v(3, 6; 7; 18)$. This is the graph that gives the upper bound $F_v(3, 6; 7) \leq 18$ in \cite{SXP09}. Now, from Theorem \ref{theorem: F_v(3, 6; 7) = 18} it follows
\begin{theorem}
	\label{theorem: abs(mH_v(3, 6; 7; 18)) = 1}
	$\mH_{extr}(3, 6; 7) = \mH_v(3, 6; 7; 18) = \set{\Gamma_3}$ (see Figure \ref{figure: H_v(3, 6; 7; 18) cap H_v(4, 5; 7; 18)}).
\end{theorem}

Let us note that $\chi(\Gamma_3) = 9$, and for this graph the inequality (\ref{equation: G arrowsv (a_1, ..., a_s) Rightarrow chi(G) geq m}) is strict. However, from Theorem \ref{theorem: abs(mH_v(3, 6; 7; 18)) = 1} it follows that in this special case Conjecture \ref{conjecture: chi(G) leq m + 1} is true.

There are two 13-regular graphs in $\mH_v(2, 2, 6; 7; 18)$, one of them being $\Gamma_3$. The graph $\Gamma_3$ is the only vertex transitive graph in $\mH_v(2, 2, 6; 7; 18)$ and it has 36 automorphisms. The other 13-regular graph has 24 automorphisms.

Let us also note that 2 467 of the graphs in $\mH_v(2, 2, 6; 7; 18)$ do not contain a subgraph in $\mH_v(2, 2, 6; 7; 17)$. We obtained the remaining 74048 graphs in $\mH_v(2, 2, 6; 7; 18)$ in another way by adding one vertex to the graphs in $\mH_v(2, 2, 6; 7; 17)$. This also testifies to the correctness of our programs.

\vspace{1em}
Theorem \ref{theorem: F_v(2, 2, 6; 7) = 17 and abs(mH_v(2, 2, 6; 7; 17)) = 3}, Theorem \ref{theorem: F_v(3, 6; 7) = 18}, Theorem \ref{theorem: abs(mH_v(2, 2, 6; 7; 18)) = 76515} and Theorem \ref{theorem: abs(mH_v(3, 6; 7; 18)) = 1} are published in \cite{BN17a}.

\begin{table}
	\centering
	\resizebox{0.8\textwidth}{!}{
		\begin{tabular}{ | l | l | r | r | }
			\hline
			{\parbox{10em}{set}}&
			{\parbox{6em}{\small independence\\ number}}&
			{\parbox{4em}{maximal\\ graphs}}&
			{\parbox{6em}{\hfill $(+K_6)$-graphs}}\\
			\hline
			$\mH_v(3; 7; 6)$				& $\leq 3$				& 1						& 2					\\
			$\mH_v(4; 7; 8)$				& $\leq 3$				& 2						& 12				\\
			$\mH_v(5; 7; 10)$				& $\leq 3$				& 6						& 274				\\
			$\mH_v(6; 7; 12)$				& $\leq 3$				& 37					& 78 926			\\
			$\mH_v(2, 6; 7; 14)$			& $\leq 3$				& 20					& 5 291				\\
			$\mH_v(2, 2, 6; 7; 17)$		& $= 3$					& 0						&					\\
			\hline
			$\mH_v(3; 7; 7)$				& $\leq 2$				& 1						& 3					\\
			$\mH_v(4; 7; 9)$				& $\leq 2$				& 2						& 22				\\
			$\mH_v(5; 7; 11)$				& $\leq 2$				& 5						& 468				\\
			$\mH_v(6; 7; 13)$				& $\leq 2$				& 24					& 97 028			\\
			$\mH_v(2, 6; 7; 15)$			& $\leq 2$				& 473					& 10 018 539		\\
			$\mH_v(2, 2, 6; 7; 17)$		& $= 2$					& 1						&					\\
			\hline
			$\mH_v(2, 2, 6; 7; 17)$		& 						& 1						&					\\
			\hline
		\end{tabular}
	}
	\caption{Steps in finding all maximal graphs in $\mH_v(2, 2, 6; 7; 17)$}
	\label{table: finding all graphs in H_v(2, 2, 6; 7; 17)}
	\vspace{2em}
	\centering
	\resizebox{0.8\textwidth}{!}{
		\begin{tabular}{ | l | l | r | r | }
			\hline
			{\parbox{10em}{set}}&
			{\parbox{6em}{\small independence\\ number}}&
			{\parbox{4em}{maximal\\ graphs}}&
			{\parbox{6em}{\hfill $(+K_6)$-graphs}}\\
			\hline
			$\mH_v(3; 7; 6)$				& $\leq 4$				& 1						& 2					\\
			$\mH_v(4; 7; 8)$				& $\leq 4$				& 2						& 13				\\
			$\mH_v(5; 7; 10)$				& $\leq 4$				& 7						& 317				\\
			$\mH_v(6; 7; 12)$				& $\leq 4$				& 50					& 102 387			\\
			$\mH_v(2, 6; 7; 14)$			& $\leq 4$				& 20					& 5 293				\\
			$\mH_v(2, 2, 6; 7; 18)$		& $= 4$					& 0						&					\\
			\hline
			$\mH_v(3; 7; 7)$				& $\leq 3$				& 1						& 4					\\
			$\mH_v(4; 7; 9)$				& $\leq 3$				& 3						& 45				\\
			$\mH_v(5; 7; 11)$				& $\leq 3$				& 12					& 3 071				\\
			$\mH_v(6; 7; 13)$				& $\leq 3$				& 168					& 4 691 237			\\
			$\mH_v(2, 6; 7; 15)$			& $\leq 3$				& 1627					& 70 274 176		\\
			$\mH_v(2, 2, 6; 7; 18)$		& $= 3$					& 308					&					\\
			\hline
			$\mH_v(3; 7; 8)$				& $\leq 2$				& 1						& 8					\\
			$\mH_v(4; 7; 10)$				& $\leq 2$				& 3						& 82				\\
			$\mH_v(5; 7; 12)$				& $\leq 2$				& 10					& 5 057				\\
			$\mH_v(6; 7; 14)$				& $\leq 2$				& 96					& 2 799 416			\\
			$\mH_v(2, 6; 7; 16)$			& $\leq 2$				& 7509					& 920 112 878		\\
			$\mH_v(2, 2, 6; 7; 18)$		& $= 2$					& 84					&					\\
			\hline
			$\mH_v(2, 2, 6; 7; 18)$		& 						& 392					&					\\
			\hline
		\end{tabular}
	}
	\caption{Steps in finding all maximal graphs in $\mH_v(2, 2, 6; 7; 18)$}
	\label{table: finding all graphs in H_v(2, 2, 6; 7; 18)}
\end{table}

\section{Computation of the numbers $F_v(2_{m - 6}, 6; m - 1)$}

In support to Conjecture \ref{conjecture: rp(p) = 2, p geq 4} we will prove the following theorem:
\begin{theorem}
	\label{theorem: F_v(2, 2, p; p + 1) leq 2p + 5, then ...}
	Let $F_v(2, 2, p; p + 1) \leq 2p + 5$. Then $\rp(p) = 2$ and
	\begin{equation*}
	F_v(2_r, p; r + p - 1) = F_v(2, 2, p; p + 1) + r - 2, \ r \geq 2.
	\end{equation*}
\end{theorem}
\begin{proof}
	From Theorem \ref{theorem: rp}(b) it follows that it is enough to prove the equality $\rp(p) = 2$. According to (\ref{equation: m + p + 2 leq F_v(a_1, ..., a_s; m - 1) leq m + 3p}), $F_v(2, 2, p; p + 1) \geq 2p + 4$. Therefore, only the following two cases are possible:\\
	
	\emph{Case 1.} $F_v(2, 2, p; p + 1) = 2p + 4$. According to (\ref{equation: m + p + 2 leq F_v(a_1, ..., a_s; m - 1) leq m + 3p}),
	
	$F_v(2_r, p; r + p - 1) \geq m + p + 2 = r + 2p + 2$.\\
	Therefore,
	
	$F_v(2_r, p; r + p - 1) - r \geq 2p + 2 = F_v(2, 2, p; p + 1) - 2, \ r \geq 2$,\\
	and we have $\rp(p) = 2$.\\
	
	\emph{Case 2.} $F_v(2, 2, p; p + 1) = 2p + 5$. From Theorem \ref{theorem: F_v(a_1, ..., a_s; m - 1) geq m + p + 3} we have
	
	$F_v(2_r, p; r + p - 1) \geq r + 2p + 3, \ r \geq 2$.\\
	From this inequality we obtain
	
	$F_v(2_r, p; r + p - 1) - r \geq 2p + 3 = F_v(2, 2, p; p + 1) - 2, \ r \geq 2$.\\
	Therefore, in this case we also have $\rp(p) = 2$.
\end{proof}

\begin{remark}
	\label{remark: F_v(2, 2, p; p + 1) = 2p + 4?}
	It is unknown whether the first case is possible, i.e. if $F_v(2, 2, p; p + 1) = 2p + 4$ for some $p$. If $p \leq 7$ this equality is not true.
\end{remark}

\begin{theorem}
	\label{theorem: rp(6) = 2}
	$\rp(6) = 2$ and $F_v(2_{m - 6}, 6; m - 1) = m + 9, \ m \geq 8$.
\end{theorem}

\begin{proof}
By Theorem \ref{theorem: F_v(2, 2, 6; 7) = 17 and abs(mH_v(2, 2, 6; 7; 17)) = 3}, $F_v(2, 2, 6; 7) = 17$. From this fact and Theorem \ref{theorem: F_v(2, 2, p; p + 1) leq 2p + 5, then ...} we obtain $\rp(6) = 2$ and the equality $F_v(2_{m - 6}, 6; m - 1) = m + 9, \ m \geq 8$.
\end{proof}

\vspace{1em}
Theorem \ref{theorem: F_v(2, 2, p; p + 1) leq 2p + 5, then ...} and Theorem \ref{theorem: rp(6) = 2} are published in \cite{BN17a}.

\section{Proof of Theorem \ref{theorem: F_v(a_1, ..., a_s; m - 1) = ..., max set(a_1, ..., a_s) = 6}(a)}

\vspace{1em}

Since $a_1 = ... = a_{s-1} = 2$ and $a_s = 6$, we have $m = s + 5$, and therefore

$F_v(a_1, ..., a_s; m - 1) = F_v(2_{s - 1}, 6; m - 1) = F_v(2_{m - 6}, 6; m - 1)$.\\
From Theorem \ref{theorem: rp(6) = 2} it now follows that $F_v(a_1, ..., a_s; m - 1) = m + 9$.
\qed

\vspace{1em}
Theorem \ref{theorem: F_v(a_1, ..., a_s; m - 1) = ..., max set(a_1, ..., a_s) = 6} is published in \cite{BN17a}.

The lower bound $F_v(a_1, ..., a_s; m - 1) \geq m + 9$ in Theorem \ref{theorem: F_v(a_1, ..., a_s; m - 1) = ..., max set(a_1, ..., a_s) = 6} is first proved and published in \cite{BN15a}.

\vspace{1em}

\section{Bounds on the numbers $F_v(2_r, 3, p; r + p + 1)$}

\vspace{1em}

The proof of Theorem \ref{theorem: F_v(a_1, ..., a_s; m - 1) = ..., max set(a_1, ..., a_s) = 6}(b) is more complex. First, we will compute the numbers $F_v(2_{m - 8}, 3, 6; m - 1), \ m \geq 8$ (Theorem \ref{theorem: rpp(6) = 0}). For the computation of the numbers $F_v(2_{m - 8}, 3, 6; m - 1), \ m \geq 8$ we will need the following theorem:

\begin{theorem}
	\label{theorem: rpp}
	Let $p$ be a fixed positive integer and $\rpp(p) = \rpp$ be the smallest positive integer for which
	$$\min_{r \geq 0} \set{F_v(2_r, 3, p; r + p + 1) - r} = F_v(2_{\rpp}, 3, p; \rpp + p + 1) - \rpp.$$
	Then:
	\vspace{1em}\\
	(a) $F_v(2_r, 3, p; r + p + 1) = F(2_{\rpp}, 3, p; \rpp + p + 1) + r - \rpp, \ r \geq \rpp$.
	\vspace{1em}\\
	(b) If $\rpp = 0$, then \hfill\break
	$F_v(2_r, 3, p; r + p + 1) = F_v(3, p; p + 1) + r, \ r \geq 0$.
	\vspace{1em}\\
	(c) If $\rpp > 0$ and $G$ is an extremal graph in $\mH (2_{\rpp}, 3, p; \rpp + p + 1)$, then \hfill\break
	$G \arrowsv (2, \rpp + p).$
	\vspace{1em}\\
	(d) $\rpp < F_v(3, p; p + 1) - 2p - 2$.
\end{theorem}

\begin{proof}
	(a) According to the definition of $\rpp = \rpp(p)$ we have
	
	$F_v(2_r, 3, p; r + p + 1) \geq F_v(2_{\rpp}, 3, p; \rpp + p + 1) + r - \rpp, \ r \geq 0$.\\
	Now we will prove that if $r \geq \rpp$ the opposite inequality is also true. 
	Let $G \in \mH_{extr}(2_{\rpp}, 3, p; \rpp + p + 1)$. Then from (\ref{equation: G arrowsv (a_1, ..., a_s) Rightarrow K_t + G arrowsv (2_t, a_1, ..., a_s)}) it follows that $K_{r - \rpp} + G \in \mH_v(2_r, 3, p; r + p + 1), \ r \geq \rpp$. Therefore
	
	$F_v(2_r, 3, p; r + p + 1) \leq \abs{\V(K_{r - \rpp} + G)} = F_v(2_{\rpp}, 3, p; \rpp + p + 1) + r - \rpp, \ r \geq \rpp$.\\
	Thus, (a) is proved.\\
	
	(b) If $\rpp(p) = 0$, then obviously the equality (b) follows from (a).\\
	
	(c) Assume the opposite is true and let $G$ be an extremal graph in $\mH_v(2_{\rpp}, 3, p; \rpp + p + 1)$ such that $\V(G) = V_1 \cup V_2$ where $V_1$ is an independent set and $V_2$ does not contain $(\rpp + p)$-clique. We can assume that $V_1 \neq \emptyset$. Let $G_1 = \G[V_2] = G - V_1$. Then $\omega(G_1) < \rpp + p$ and since $\rpp \geq 1$, from Proposition \ref{proposition: G - A arrowsv (a_1, ..., a_(i - 1), a_i - 1, a_(i + 1_, ..., a_s)} it follows that $G_1 \arrowsv (2_{\rpp - 1}, 3, p)$. Therefore, $G_1 \in \mH_v(2_{\rpp - 1}, 3, p; \rpp + p)$ and
	
	$\abs{\V(G)} - 1 \geq \abs{\V(G_1)} \geq F_v(2_{\rpp - 1}, 3, p; \rpp + p)$.\\
	Since $\abs{\V(G)} = F_v(2_{\rpp}, 3, p; \rpp + p + 1)$, we obtain
	
	$F_v(2_{\rpp - 1}, 3, p; \rpp + p) - (\rpp - 1) \leq F_v(2_{\rpp}, 3, p; \rpp + p + 1) - \rpp$,\\
	which contradicts the definition of $\rpp$.\\
	
	(d) According to (\ref{equation: m + p + 2 leq F_v(a_1, ..., a_s; m - 1) leq m + 3p}) $F_v(3, p; p + 1) \geq 2p + 4$, and therefore in the case $\rpp = 0$ the inequality holds. Let $\rpp > 0$ and $G$ be an extremal graph in $\mH_v(2_{\rpp}, 3, p; \rpp + p + 1)$. According to (\ref{equation: G arrowsv (a_1, ..., a_s) Rightarrow chi(G) geq m}),
	\begin{equation}
	\label{equation: chi(G) geq rpp + p + 2}
	\chi(G) \geq \rpp + p + 2.
	\end{equation}
	According to (c), $G \in \mH_v(2, \rpp + p; \rpp + p + 1)$, and by Theorem \ref{theorem: F_v(a_1, ..., a_s; m) = m + p}
	
	$\abs{\V(G)} \geq 2\rpp + 2p + 1$.\\
	Since $\chi(\overline{C}_{2\rpp + 2p + 1}) = \rpp + p + 1$, from (\ref{equation: chi(G) geq rpp + p + 2}) it follows $G \neq \overline{C}_{2\rpp + 2p + 1}$. By Theorem \ref{theorem: F_v(a_1, ..., a_s; m) = m + p}(b),
	
	$\abs{\V(G)} = F_v(2_{\rpp}, 3, p; \rpp + p + 1) \geq 2\rpp + 2p + 2$.\\
	Since $\rpp > 0$, we have
	
	$F_v(2_{\rpp}, 3, p; \rpp + p + 1) - \rpp < F_v(3, p; p + 1).$\\
	From the last two inequalities it follows that
	
	$\rpp < F_v(3, p; p + 1) - 2p - 2$.
\end{proof}

\vspace{1em}

Since $F_v(3, 3; 4) = 14$, from (\ref{equation: F_v(a_1, ..., a_s, m - 1) = ...}) we obtain $\rpp(3) = 1$. Also from (\ref{equation: F_v(a_1, ..., a_s, m - 1) = ...}) we see that $\rpp(4) = 0$. From Theorem \ref{theorem: F_v(a_1, ..., a_s; m - 1) = m + 9, max set(a_1, ..., a_s) = 5} it follows that $\rpp(5) = 0$. We suppose the following conjecture is true
\begin{conjecture}
	\label{conjecture: rpp(p) = 0, p geq 4}
	Let $p \geq 4$ be a fixed integer. Then,
	\begin{equation*}
	\min_{r \geq 0} \set{F_v(2_r, 3, p; r + p - 1) - r} = F_v(3, p; p + 1),
	\end{equation*}
	i.e. $\rpp(p) = 0$, and
	\begin{equation*}
	F_v(2_r, 3, p; r + p + 1) = F_v(3, p; p + 1) + r.
	\end{equation*}
\end{conjecture}

\vspace{1em}

It is not difficult to see that Conjecture \ref{conjecture: rpp(p) = 0, p geq 4} is true if and only if for fixed $p$ the sequence $\set{F_v(2_r, 3, p; r + p + 1)}$ is strictly increasing with respect to $r$. We will prove that when $p = 6$ Conjecture \ref{conjecture: rpp(p) = 0, p geq 4} is also true. Theorem \ref{theorem: F_v(a_1, ..., a_s; m - 1) = ..., max set(a_1, ..., a_s) = 6} (b) follows easily from this fact.

\vspace{1em}

\vspace{1em}
Theorem \ref{theorem: rpp} is published in \cite{BN17a}.

\vspace{3em}

\section{Computation of the numbers $F_v(2_{m - 8}, 3, 6; m - 1)$}

\vspace{1em}

\begin{theorem}
	\label{theorem: rpp(6) = 0}
	$\rpp(6) = 0$ and $F_v(2_{m - 8}, 3, 6; m - 1) = m + 10, \ m \geq 8$.
\end{theorem}

\begin{proof}
	From Theorem \ref{theorem: rpp} (d) we obtain $\rpp(6) < 4$. Therefore, we have to prove $\rpp(6) \neq 1$, $\rpp(6) \neq 2$, and $\rpp(6) \neq 3$. Since $F_v(3, 6; 7) = 18$, we have to prove the inequalities $F_v(2, 3, 6; 8) > 18$, $F_v(2, 2, 3, 6; 9) > 19$, and $F_v(2, 2, 2, 3, 6; 10) > 20$. We will prove these inequalities with the help of a computer. From (\ref{equation: G arrowsv (a_1, ..., a_s) Rightarrow K_t + G arrowsv (2_t, a_1, ..., a_s)}) (t = 1) it is easy to see that $F_v(2_{r - 1}, 3; p) + 1 \geq F_v(2_r, 3; p + 1)$, and therefore it is enough to prove $F_v(2, 2, 2, 3, 6; 10) > 20$. We will present the proof of this inequality only, but we also proved the other two inequalities in the same way with a computer, since the obtained additional information is interesting and useful.
	
	\vspace{1em}

	Similarly to the proof of Theorem \ref{theorem: F_v(2, 2, 6; 7) = 17 and abs(mH_v(2, 2, 6; 7; 17)) = 3}, we will use Algorithm \ref{algorithm: A3} to prove that $\mH_v(2, 2, 2, 3, 6; 10; 20) = \emptyset$. According to Theorem \ref{theorem: alpha(G) leq V(G) + m + p - 1}(b), there are no graphs in $\mH_v(2, 2, 2, 3, 6; 10; 20)$ with independence number greater than 3.\\
	
	First, we prove that there are no graphs in $\mH_{max}(2, 2, 2, 3, 6; 10; 20)$ with independence number 3:
	
	The only graph in $\mH_{max}(6; 10; 9)$ is $K_9$.
	
	We execute Algorithm \ref{algorithm: A3}($n = 11;\ k = 2;\ t = 3$) with input $\mA = \mH_{max}^3(1, 6; 10; 9) = \mH_{max}^3(6; 10; 9) = \set{K_9}$ to obtain all graphs in $\mB = \mH_{max}^3(2, 6; 10; 11)$. (see Remark \ref{remark: algorithm A3 k = 2})
	
	We execute Algorithm \ref{algorithm: A3}($n = 13;\ k = 2;\ t = 3$) with input $\mA = \mH_{max}^3(2, 6; 10; 11)$ to obtain all graphs in $\mB = \mH_{max}^3(3, 6; 10; 13)$.
	
	We execute Algorithm \ref{algorithm: A3}($n = 15;\ k = 2;\ t = 3$) with input $\mA = \mH_{max}^3(1, 3, 6; 10; 13) = \mH_{max}^3(3, 6; 10; 13)$ to obtain all graphs in $\mB = \mH_{max}^3(2, 3, 6; 10; 15)$.
	
	We execute Algorithm \ref{algorithm: A3}($n = 17;\ k = 2;\ t = 3$) with input $\mA = \mH_{max}^3(1, 2, 3, 6; 10; 15) = \mH_{max}^3(2, 3, 6; 10; 15)$ to obtain all graphs in $\mB = \mH_{max}^3(2, 2, 3, 6; 10; 17)$.
	
	By executing Algorithm \ref{algorithm: A3}($n = 20;\ k = 3;\ t = 3$) with input $\mA = \mH_{max}^3(1, 2, 2, 3, 6; 10; 17) = \mH_{max}^3(2, 2, 3, 6; 10; 17)$, we obtain $\mB = \emptyset$. According to Theorem \ref{theorem: algorithm A3}, there are no graphs in $\mH_{max}(2, 2, 2, 3, 6; 10; 20)$ with independence number 3.\\
	
	It remains to be proved that there are no graphs in $\mH_{max}(2, 2, 2, 3, 6; 10; 20)$ with independence number 2:
	
	The only graph in $\mH_{max}(6; 10; 10)$ is $K_{10} - e$.

	We execute Algorithm \ref{algorithm: A3}($n = 12;\ k = 2;\ t = 2$) with input $\mA = \mH_{max}^2(1, 6; 10; 10) = \mH_{max}^2(6; 10; 10) = \set{K_{10} - e}$ to obtain all graphs in $\mB = \mH_{max}^2(2, 6; 10; 12)$. (see Remark \ref{remark: algorithm A3 k = 2})
	
	We execute Algorithm \ref{algorithm: A3}($n = 14;\ k = 2;\ t = 2$) with input $\mA = \mH_{max}^2(2, 6; 10; 12)$ to obtain all graphs in $\mB = \mH_{max}^2(3, 6; 10; 14)$.
	
	We execute Algorithm \ref{algorithm: A3}($n = 16;\ k = 2;\ t = 2$) with input $\mA = \mH_{max}^2(1, 3, 6; 10; 14) = \mH_{max}^2(3, 6; 10; 14)$ to obtain all graphs in $\mB = \mH_{max}^2(2, 3, 6; 10; 16)$.
	
	We execute Algorithm \ref{algorithm: A3}($n = 18;\ k = 2;\ t = 2$) with input $\mA = \mH_{max}^2(1, 2, 3, 6; 10; 16) = \mH_{max}^2(2, 3, 6; 10; 16)$ to obtain all graphs in $\mB = \mH_{max}^2(2, 2, 3, 6; 10; 18)$.
	
	By executing Algorithm \ref{algorithm: A3}($n = 20;\ k = 2;\ t = 2$) with input $\mA = \mH_{max}^2(1, 2, 2, 3, 6; 10; 18) = \mH_{max}^2(2, 2, 3, 6; 10; 18)$, we obtain $\mB = \emptyset$. According to Theorem \ref{theorem: algorithm A3}, there are no graphs in $\mH_{max}(2, 2, 2, 3, 6; 10; 20)$ with independence number 2.

	Thus, we proved $\mH_{max}(2, 2, 2, 3, 6; 10; 20) = \emptyset$, and therefore $F_v(2, 2, 2, 3, 6; 10) > 20$ and $\rpp(6) = 0$.\\
	
From $\rpp(6) = 0$ and Theorem \ref{theorem: rpp}(b) we obtain
$$F_v(2_{m - 8}, 3, 6; m - 1) = m + 10, \ m \geq 8.$$

Thus, Theorem \ref{theorem: rpp(6) = 0} is proved.
\end{proof}

\vspace{1em}

The number of graphs obtained in each step is given in Table \ref{table: finding all graphs in H_v(2, 2, 2, 3, 6; 10; 20)} (see also Table \ref{table: finding all graphs in H_v(2, 3, 6; 8; 18)} and Table \ref{table: finding all graphs in H_v(2, 2, 3, 6; 9; 19)}).

\vspace{1em}
Theorem \ref{theorem: rpp(6) = 0} is published in \cite{BN17a}.

\vspace{1em}

\begin{table}
	\centering
	\vspace{-2em}
	\resizebox{0.75\textwidth}{!}{
		\begin{tabular}{ | l | l | r | r | }
			\hline
			{\parbox{10em}{set}}&
			{\parbox{5.5em}{\small independence\\ number}}&
			{\parbox{4em}{maximal\\ graphs}}&
			{\parbox{6em}{\hfill $(+K_7)$-graphs}}\\
			\hline
		$\mH_v(4; 8; 7)$				& $\leq 3$				& 1						& 2					\\
		$\mH_v(5; 8; 9)$				& $\leq 3$				& 2						& 12				\\
		$\mH_v(6; 8; 11)$				& $\leq 3$				& 6						& 276				\\
		$\mH_v(2, 6; 8; 13)$			& $\leq 3$				& 37					& 79 749			\\
		$\mH_v(3, 6; 8; 15)$			& $\leq 3$				& 21					& 3 458				\\
		$\mH_v(2, 3, 6; 8; 18)$		& $= 3$					& 0						&					\\
		\hline
		$\mH_v(4; 8; 8)$				& $\leq 2$				& 1						& 3					\\
		$\mH_v(5; 8; 10)$				& $\leq 2$				& 2						& 22				\\
		$\mH_v(6; 8; 12)$				& $\leq 2$				& 5						& 489				\\
		$\mH_v(2, 6; 8; 14)$			& $\leq 2$				& 25					& 119 126			\\
		$\mH_v(3, 6; 8; 16)$			& $\leq 2$				& 509					& 3 582 157			\\
		$\mH_v(2, 3, 6; 8; 18)$		& $= 2$					& 0						&					\\
		\hline
		$\mH_v(2, 3, 6; 8; 18)$		& 						& 0						&					\\
		\hline
	\end{tabular}
}
	\vspace{-0.5em}
	\caption{\small Steps in finding all maximal graphs in $\mH_v(2, 3, 6; 8; 18)$}
	\label{table: finding all graphs in H_v(2, 3, 6; 8; 18)}
	\vspace{1em}
	\centering
	\resizebox{0.75\textwidth}{!}{
		\begin{tabular}{ | l | l | r | r | }
			\hline
			{\parbox{10em}{set}}&
			{\parbox{5.5em}{\small independence\\ number}}&
			{\parbox{4em}{maximal\\ graphs}}&
			{\parbox{6em}{\hfill $(+K_8)$-graphs}}\\
			\hline
		$\mH_v(5; 9; 8)$				& $\leq 3$				& 1						& 2					\\
		$\mH_v(6; 9; 10)$				& $\leq 3$				& 2						& 12				\\
		$\mH_v(2, 6; 9; 12)$			& $\leq 3$				& 6						& 277				\\
		$\mH_v(3, 6; 9; 14)$			& $\leq 3$				& 37					& 79 901			\\
		$\mH_v(2, 3, 6; 9; 16)$		& $\leq 3$				& 21					& 3 459				\\
		$\mH_v(2, 2, 3, 6; 9; 19)$	& $= 3$					& 0						&					\\
		\hline
		$\mH_v(5; 9; 9)$				& $\leq 2$				& 1						& 3					\\
		$\mH_v(6; 9; 11)$				& $\leq 2$				& 2						& 22				\\
		$\mH_v(2, 6; 9; 13)$			& $\leq 2$				& 5						& 496				\\
		$\mH_v(3, 6; 9; 15)$			& $\leq 2$				& 25					& 121 499			\\
		$\mH_v(2, 3, 6; 9; 17)$		& $\leq 2$				& 512					& 3 585 530			\\
		$\mH_v(2, 2, 3, 6; 9; 19)$	& $= 2$					& 0						&					\\
		\hline
		$\mH_v(2, 2, 3, 6; 9; 19)$	& 						& 0						&					\\
		\hline
	\end{tabular}
}
	\vspace{-0.5em}
	\caption{\small Steps in finding all maximal graphs in $\mH_v(2, 2, 3, 6; 9; 19)$}
	\label{table: finding all graphs in H_v(2, 2, 3, 6; 9; 19)}
	\vspace{1em}
	\centering
	\resizebox{0.75\textwidth}{!}{
		\begin{tabular}{ | l | l | r | r | }
			\hline
			{\parbox{10em}{set}}&
			{\parbox{5.5em}{\small independence\\ number}}&
			{\parbox{4em}{maximal\\ graphs}}&
			{\parbox{6em}{\hfill $(+K_9)$-graphs}}\\
			\hline
		$\mH_v(6; 10; 9)$					& $\leq 3$				& 1						& 2					\\
		$\mH_v(2, 6; 10; 11)$				& $\leq 3$				& 2						& 12				\\
		$\mH_v(3, 6; 10; 13)$				& $\leq 3$				& 6						& 277				\\
		$\mH_v(2, 3, 6; 10; 15)$			& $\leq 3$				& 37					& 79 934			\\
		$\mH_v(2, 2, 3, 6; 10; 17)$		& $\leq 3$				& 21					& 3 459				\\
		$\mH_v(2, 2, 2, 3, 6; 10; 20)$	& $= 3$					& 0						&					\\
		\hline
		$\mH_v(6; 10; 10)$				& $\leq 2$				& 1						& 3					\\
		$\mH_v(2, 6; 10; 12)$				& $\leq 2$				& 2						& 22				\\
		$\mH_v(3, 6; 10; 14)$				& $\leq 2$				& 5						& 498				\\
		$\mH_v(2, 3, 6; 10; 16)$			& $\leq 2$				& 25					& 121 864			\\
		$\mH_v(2, 2, 3, 6; 10; 18)$		& $\leq 2$				& 512					& 3 585 546			\\
		$\mH_v(2, 2, 2, 3, 6; 10; 20)$	& $= 2$					& 0						&					\\
		\hline
		$\mH_v(2, 2, 2, 3, 6; 10; 20)$	& 						& 0						&					\\
		\hline
	\end{tabular}
}
	\vspace{-0.5em}
	\caption{\small Steps in finding all maximal graphs in $\mH_v(2, 2, 2, 3, 6; 10; 20)$}
	\label{table: finding all graphs in H_v(2, 2, 2, 3, 6; 10; 20)}
\end{table}

\section{Computation of the numbers $\wFv{m}{6}{m - 1}$}

Let us remind that $\wHv{m}{p}{q}$ and $\wFv{m}{p}{q}$ are defined in Section 1.4.\\

\vspace{-1em}

According to Proposition \ref{proposition: wFv(m)(p)(q) exists}, we have
\begin{equation}
	\label{equation: wFv(m)(6)(m - 1) exists}
	\wFv{m}{6}{m - 1} \mbox{ exists } \Leftrightarrow m \geq 8.
\end{equation}

We will prove the following
\begin{theorem}
	\label{theorem: wFv(m)(6)(m - 1) = m + 10}
	$\wFv{m}{6}{m - 1} = m + 10, \ m \geq 8.$
\end{theorem}

\begin{proof}
The lower bound $\wFv{m}{6}{m - 1} \geq m + 10$ follows from Theorem \ref{theorem: rpp(6) = 0} and Theorem \ref{theorem: F_v(2_(m - p), p; q) leq F_v(a_1, ..., a_s; q) leq wFv(m)(p)(q)}. 

To prove the upper bound consider the 18-vertex graph $\Gamma_3$ (Figure \ref{figure: H_v(3, 6; 7; 18) cap H_v(4, 5; 7; 18)}) with the help of which in \cite{SXP09} the authors proved the inequality $F_v(3, 6; 7) \leq 18$. In addition to the property $\Gamma_3 \arrowsv (3, 6)$, the graph $\Gamma_3$ also has the property $\Gamma_3 \arrowsv (4, 5)$. According to Proposition \ref{proposition: a_1 + a_2 - 1 > p}, from $\Gamma_3 \arrowsv (3, 6)$ and $\Gamma_3 \arrowsv (4, 5)$ it follows $\Gamma_3 \arrowsv \uni{8}{6}$. Since $\omega(\Gamma_3) = 6$, we obtain $\Gamma_3 \in \wHv{8}{6}{7}$ and $\wFv{8}{6}{7} \leq \abs{\V(\Gamma_3)} = 18$. From this inequality and Theorem \ref{theorem: wFv(m)(p)(m - m_0 + q) leq wFv(m_0)(p)(q) + m - m_0} ($m_0 = 8, p = 6; q = 7$) it follows $\wFv{m}{6}{m - 1} \leq m + 10, \ m \geq 8$.
\end{proof}

The numbers $\wFv{m}{6}{m - 1}$ can be computed directly with the help of the following Theorem \ref{theorem: m_0}. This is done in \cite{BN15b}.
\begin{theorem}
\label{theorem: m_0}
Let $m_0(p) = m_0$ be the smallest positive integer for which
$$\min_{m \geq p + 2}\set{\wFv{m}{p}{m - 1} - m} = \wFv{m_0}{p}{m_0 - 1} - m_0.$$
Then:
\vspace{1em}\\
(a) $\wFv{m}{p}{m - 1} = \wFv{m_0}{p}{m_0 - 1} + m - m_0, \ m \geq m_0$.
\vspace{1em}\\
(b) If $m_0 > p + 2$ and $G$ is an extremal graph in $\wHv{m_0}{p}{m_0 - 1}$, then \hfill\break
$G \overset{v}{\rightarrow} (2, m_0 - 2)$.
\vspace{1em}\\
(c) $m_0 < \wFv{(p + 2)}{p}{p + 1} - p$.
\end{theorem}
Theorem \ref{theorem: m_0} is published in \cite{BN15b}. The proof is similar to the proof of Theorem \ref{theorem: rp} and Theorem \ref{theorem: rpp}. According to Theorem \ref{theorem: m_0}, the computation of the numbers $\wFv{m}{6}{m - 1}$ is reduced to finding the exact values of the first several of these numbers. In \cite{BN15b}, we use a modification of Algorithm \ref{algorithm: A2} to prove that $\wFv{8}{6}{7} = 18$, $\wFv{9}{6}{8} > 18$, $\wFv{10}{6}{9} > 19$, and $\wFv{11}{6}{10} > 20$. Thus, with the help of Theorem \ref{theorem: m_0}, we obtain $m_0(6) = 8$ and $\wFv{m}{6}{m - 1} = m + 10, \ m \geq 8$. We will not present this proof in detail.

\vspace{1em}

\section{Proof of Theorem \ref{theorem: F_v(a_1, ..., a_s; m - 1) = ..., max set(a_1, ..., a_s) = 6}(b)}

According to Theorem \ref{theorem: F_v(2_(m - p), p; q) leq F_v(a_1, ..., a_s; q) leq wFv(m)(p)(q)} and Theorem \ref{theorem: rpp(6) = 0}

$F_v(a_1, ..., a_s; m - 1) \geq F_v(2_{m - 8}, 3, 6; m - 1) = m + 10$. 

From Theorem \ref{theorem: F_v(2_(m - p), p; q) leq F_v(a_1, ..., a_s; q) leq wFv(m)(p)(q)} and Theorem \ref{theorem: wFv(m)(6)(m - 1) = m + 10} we obtain

$F_v(a_1, ..., a_s; m - 1) \leq \wFv{m}{6}{m - 1} = m + 10$.
\qed

\vspace{1em}
Theorem \ref{theorem: F_v(a_1, ..., a_s; m - 1) = ..., max set(a_1, ..., a_s) = 6} is published in \cite{BN17a}.

The upper bound $F_v(a_1, ..., a_s; m - 1) \leq m + 10$ in Theorem \ref{theorem: F_v(a_1, ..., a_s; m - 1) = ..., max set(a_1, ..., a_s) = 6} is first proved and published in \cite{BN15a}.

\chapter{Numbers of the form $F_v(a_1, ..., a_s; m - 1)$ where\\ $\max\{a_1, ..., a_s\} = 7$}

\vspace{-2em}

In this chapter we obtain the bounds:
\vspace{-0.5em}
\begin{theorem}
	\label{theorem: F_v(a_1, ..., a_s; m - 1) leq m + 12, max set(a_1, ..., a_s) = 7}
	Let $a_1, ..., a_s$ be positive integers such that $\max\set{a_1, ..., a_s} = 7$ and $m = \sum\limits_{i=1}^s (a_i - 1) + 1 \geq 9$. Then:
	$$m + 11 \leq F_v(a_1, ..., a_s; m - 1) \leq m + 12.$$
\end{theorem}
\vspace{-0.5em}
According to (\ref{equation: F_v(a_1, ..., a_s; m - 1) exists}), the condition $m \geq 9$ in Theorem \ref{theorem: F_v(a_1, ..., a_s; m - 1) leq m + 12, max set(a_1, ..., a_s) = 7} is necessary.

First, with the help of Algorithm \ref{algorithm: A3} we will find the exact value of the number $F_v(2, 2, 7; 8)$. Then, we compute the numbers $F_v(2_{m - 7}, 7; m - 1)$. For the computation of these numbers we will need a new improved Algorithm \ref{algorithm: A4}.
\vspace{1em}

\vspace{-2.5em}

\section{Computation of the number $F_v(2, 2, 7; 8)$}

\vspace{-0.5em}
Regarding the number $F_v(2, 2, 7; 8)$ we know that $18 \leq F_v(2, 2, 7; 8) \leq 22$. The lower bound follows from (\ref{equation: m + p + 2 leq F_v(a_1, ..., a_s; m - 1) leq m + 3p}, and the upper bound follows from $F_v(3, 7; 8) \leq 22$, \cite{SXP09}.

\begin{theorem}
	\label{theorem: F_v(2, 2, 7; 8) = 20}
	$F_v(2, 2, 7; 8) = 20$.
\end{theorem}

\begin{proof}

1. Proof of the lower bound $F_v(2, 2, 7; 8) \geq 20$.

We will prove that $\mH_v(2, 2, 7; 8; 19) = \emptyset$ using a similar method to the proof of Theorem \ref{theorem: F_v(2, 2, 6; 7) = 17 and abs(mH_v(2, 2, 6; 7; 17)) = 3}. Suppose that $G \in \mH_v(2, 2, 7; 8; 19)$. Clearly, $\alpha(G) \geq 2$, and according to Theorem \ref{theorem: alpha(G) leq V(G) + m + p - 1}(b), $\alpha(G) \leq 3$. It remains to be proved that there are no graphs in $\mH_{max}(2, 2, 7; 8; 19)$ with independence number 2 or 3.

\vspace{1em}

First, we prove that there are no graphs in $\mH_{max}(2, 2, 7; 8; 19)$ with independence number 3:
	
The only graph in $\mH_{max}(4; 8; 8)$ is $K_8 - e$.
	
We execute Algorithm \ref{algorithm: A3}($n = 10;\ k = 2;\ t = 3$) with input $\mA = \mH_{max}^3(4; 8; 8) = \set{K_8 - e}$ to obtain all graphs in $\mB = \mH_{max}^3(5; 8; 10)$. (see Remark \ref{remark: algorithm A3 k = 2})
	
We execute Algorithm \ref{algorithm: A3}($n = 12;\ k = 2;\ t = 3$) with input $\mA = \mH_{max}^3(5; 8; 10)$ to obtain all graphs in $\mB = \mH_{max}^3(6; 8; 12)$.
	
We execute Algorithm \ref{algorithm: A3}($n = 14;\ k = 2;\ t = 3$) with input $\mA = \mH_{max}^3(6; 8; 12)$ to obtain all graphs in $\mB = \mH_{max}^3(7; 8; 14)$.
	
We execute Algorithm \ref{algorithm: A3}($n = 16;\ k = 2;\ t = 3$) with input $\mA = \mH_{max}^3(1, 7; 8; 14) = \mH_{max}^3(7; 8; 14)$ to obtain all graphs in $\mB = \mH_{max}^3(2, 7; 8; 16)$.
	
By executing Algorithm \ref{algorithm: A3}($n = 19;\ k = 3;\ t = 3$) with input $\mA = \mH_{max}^3(1, 2, 7; 8; 16) = \mH_{max}^3(2, 7; 8; 16)$, we obtain $\mB = \emptyset$. According to Theorem \ref{theorem: algorithm A3}, there are no graphs in $\mH_{max}(2, 2, 7; 8; 19)$ with independence number 3.\\

\vspace{-1.5em}

It remains to prove that there are no graphs in $\mH_{max}(2, 2, 7; 8; 19)$ with independence number 2:

It is easy to see that $\mH_{max}(4; 8; 9) = \set{\overline{K}_3 + K_6, C_4 + K_5}$, and therefore $C_4 + K_5$ is the only graph in $\mH_{max}(4; 8; 9)$ with independence number 2.	
	
We execute Algorithm \ref{algorithm: A3}($n = 11;\ k = 2;\ t = 2$) with input $\mA = \mH_{max}^2(4; 8; 9) = \set{C_4 + K_5}$ to obtain all graphs in $\mB = \mH_{max}^2(5; 8; 11)$. (see Remark \ref{remark: algorithm A3 k = 2})

We execute Algorithm \ref{algorithm: A3}($n = 13;\ k = 2;\ t = 2$) with input $\mA = \mH_{max}^2(5; 8; 11)$ to obtain all graphs in $\mB = \mH_{max}^2(6; 8; 13)$.

We execute Algorithm \ref{algorithm: A3}($n = 15;\ k = 2;\ t = 2$) with input $\mA = \mH_{max}^2(6; 8; 13)$ to obtain all graphs in $\mB = \mH_{max}^2(7; 8; 15)$.

We execute Algorithm \ref{algorithm: A3}($n = 17;\ k = 2;\ t = 2$) with input $\mA = \mH_{max}^2(1, 7; 8; 15) = \mH_{max}^2(7; 8; 15)$ to obtain all graphs in $\mB = \mH_{max}^2(2, 7; 8; 17)$.

By executing Algorithm \ref{algorithm: A3}($n = 19;\ k = 2;\ t = 2$) with input $\mA = \mH_{max}^2(1, 2, 7; 8; 17) = \mH_{max}^2(2, 7; 8; 17)$, we obtain $\mB = \emptyset$. According to Theorem \ref{theorem: algorithm A3}, there are no graphs in $\mH_{max}(2, 2, 7; 8; 19)$ with independence number 2.\\
	
\vspace{-1.5em}

Thus, we proved $\mH_{max}(2, 2, 7; 8; 19) = \emptyset$, and therefore $F_v(2, 2, 7; 8) \geq 20$. The number of graphs obtained in each step of the proof is given in Table \ref{table: finding all graphs in H_v(2, 2, 7; 8; 19)}. The computer needed almost two weeks to complete the computations, because of the large number of graphs in $\mH_{+K_7}^2(2, 7; 8; 17)$.\\

\begin{table}
	\centering
	\resizebox{0.8\textwidth}{!}{
		\begin{tabular}{ | l | l | r | r | }
			\hline
			{\parbox{10em}{set}}&
			{\parbox{6em}{\small independence\\ number}}&
			{\parbox{4em}{maximal\\ graphs}}&
			{\parbox{6em}{\hfill $(+K_7)$-graphs}}\\
			\hline
		$\mH_v(4; 8; 8)$				& $\leq 3$				& 1						& 4					\\
		$\mH_v(5; 8; 10)$				& $\leq 3$				& 3						& 45				\\
		$\mH_v(6; 8; 12)$				& $\leq 3$				& 12					& 3 104				\\
		$\mH_v(7; 8; 14)$				& $\leq 3$				& 169					& 4 776 518			\\
		$\mH_v(2, 7; 8; 16)$			& $\leq 3$				& 34					& 22 896			\\
		$\mH_v(2, 2, 7; 8; 19)$		& $= 3$					& 0						&					\\
		\hline
		$\mH_v(4; 8; 9)$				& $\leq 2$				& 1						& 8					\\
		$\mH_v(5; 8; 11)$				& $\leq 2$				& 3						& 84				\\
		$\mH_v(6; 8; 13)$				& $\leq 2$				& 10					& 5 394				\\
		$\mH_v(7; 8; 15)$				& $\leq 2$				& 102					& 4 984 994			\\
		$\mH_v(2, 7; 8; 17)$			& $\leq 2$				& 2760					& 380 361 736		\\
		$\mH_v(2, 2, 7; 8; 19)$		& $= 2$					& 0						&					\\
		\hline
		$\mH_v(2, 2, 7; 8; 19)$		& 						& 0						&					\\
		\hline
	\end{tabular}
	}
	\vspace{-0.5em}
	\caption{Steps in finding all maximal graphs in $\mH_v(2, 2, 7; 8; 19)$}
	\label{table: finding all graphs in H_v(2, 2, 7; 8; 19)}
\end{table}

\vspace{-1em}

2. Proof of the upper bound $F_v(2, 2, 7; 8) \leq 20$.

We need to construct a 20-vertex graph in $\mH_v(2, 2, 7, 8; 20)$.
All vertex transitive graphs with up to 31 vertices are known and can be found in \cite{Roy_c}. With the help of a computer we check which of these graphs belong to $\mH_v(2, 2, 7; 8)$. This way, we find 4 24-vertex graphs in $\mH_v(2, 2, 7; 8)$.

By removing one vertex from the 24-vertex transitive graphs in $\mH_v(2, 2, 7; 8)$, we obtain 3 23-vertex graphs in $\mH_v(2, 2, 7; 8)$, and by removing two vertices, we obtain 8 22-vertex graphs in $\mH_v(2, 2, 7; 8)$. We add edges to one of the 8 22-vertex graphs to obtain one graph in $\mH_{max}(2, 2, 7; 8; 22)$. Using Procedure \ref{procedure: extending a set of maximal graphs in mH_v(a_1, ..., a_s; q; n)} we find 1696 more graphs in $\mH_{max}(2, 2, 7; 8; 22)$.

By removing one vertex from the obtained graphs in $\mH_{max}(2, 2, 7; 8; 22)$, we find 22 21-vertex graphs in $\mH_v(2, 2, 7; 8)$. We add edges to these graphs to obtain 22 graphs in $\mH_{max}(2, 2, 7; 8; 21)$. Then we apply Procedure \ref{procedure: extending a set of maximal graphs in mH_v(a_1, ..., a_s; q; n)} twice to obtain 15259 more graphs in $\mH_{max}(2, 2, 7; 8; 21)$.

By removing one vertex from the obtained graphs in $\mH_{max}(2, 2, 7; 8; 21)$, we find 9 20-vertex graphs in $\mH_v(2, 2, 7; 8)$. Again, by successively applying Procedure \ref{procedure: extending a set of maximal graphs in mH_v(a_1, ..., a_s; q; n)}, we obtain 39 graphs in $\mH_{max}(2, 2, 7; 8; 20)$. One of these graphs is the graph $G_{2,2,7}$ shown in Figure \ref{figure: H_v(2, 2, 7; 8; 20)}. Further, we will use the graph $G_{2,2,7}$ in the proof of Theorem \ref{theorem: F_v(a_1, ..., a_s; m - 1) leq m + 12, max set(a_1, ..., a_s) = 7}. 
The 20-vertex graphs in $\mH_{max}(2, 2, 7; 8; 20)$ that we found are obtained from one of the 24-vertex transitive graphs. However, at the beginning of the proof we use all 4 graphs, since we do not know in advance which one will be useful. In \cite{BN17a} we omitted to note that all 4 24-vertex transitive graphs are used.

We proved that $F_v(2, 2, 7; 20) \leq 20$, which finishes the proof of Theorem \ref{theorem: F_v(2, 2, 7; 8) = 20}.
\end{proof}

\vspace{-1em}

\vspace{1em}
Theorem \ref{theorem: F_v(2, 2, 7; 8) = 20} is published in \cite{BN17a}.

\begin{figure}
	\centering
	\begin{subfigure}{\textwidth}
		\centering
		\includegraphics[height=320px,width=160px]{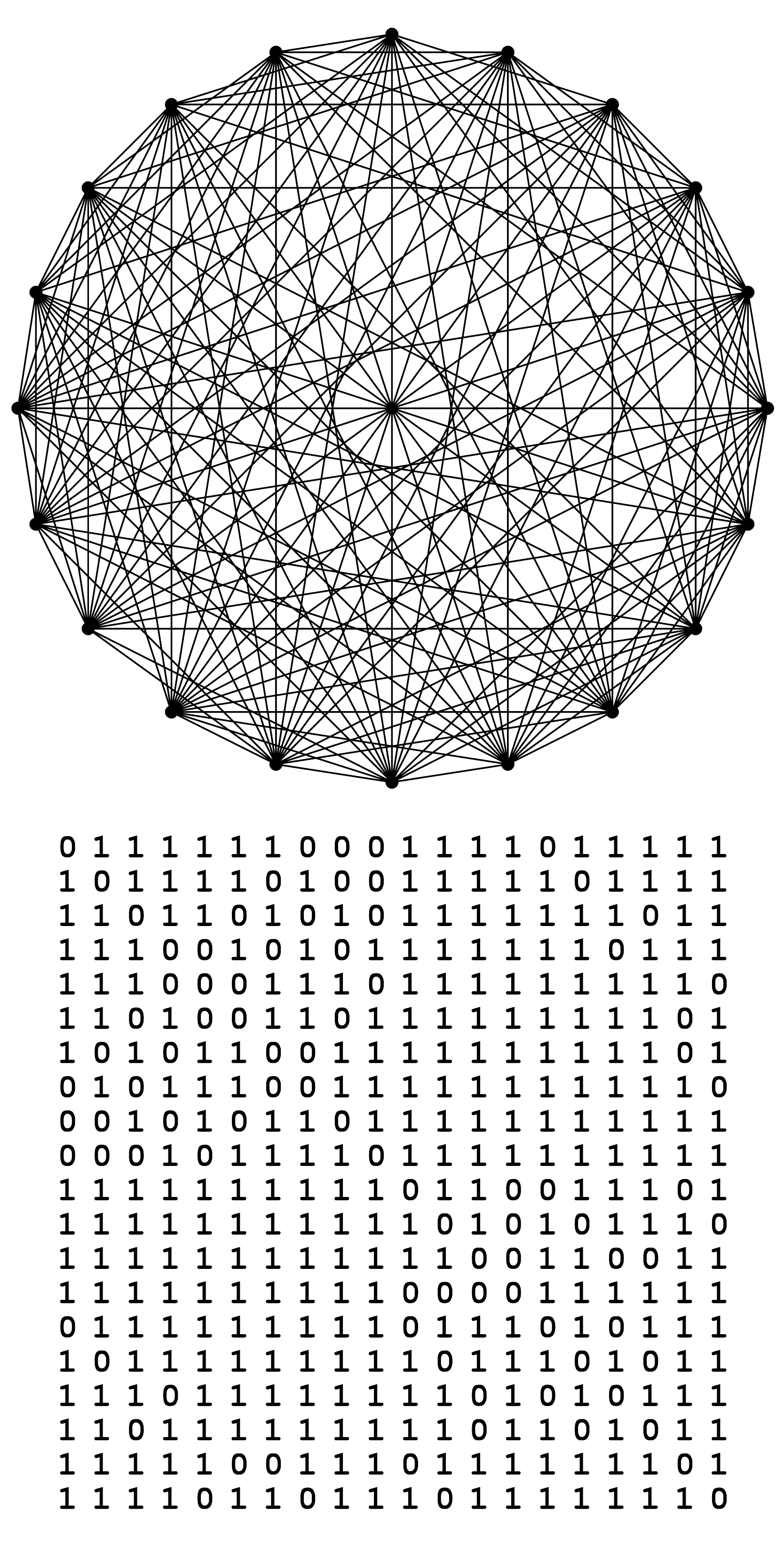}
		\vspace{-1em}
		\caption*{\emph{$G_{2,2,7}$}}
		\label{figure: G_227}
	\end{subfigure}%
	\vspace{-0.5em}
	\caption{20-vertex graph $G_{2,2,7} \in \mH_v(2, 2, 7; 8)$}
	\label{figure: H_v(2, 2, 7; 8; 20)}
\end{figure}

\clearpage

\section{Algorithm A4}

In this section we present one optimization to Algorithm \ref{algorithm: A3}.

We will use the following term:
\begin{definition}
	\label{definition: cone vertex}
	We say that $v$ is a cone vertex in the graph $G$ if $v$ is adjacent to all other vertices in $G$, i.e. $G = K_1 + H$.
\end{definition}

\begin{proposition}
\label{proposition: H in mH_(max)(a_1 - 1, a_2, ..., a_s; q - 1; n - 1)}
Let $G = K_1 + H \in \mH_{max}(a_1, ..., a_s; q; n)$ and $a_1 \geq 2$. Then $H \in \mH_{max}(a_1 - 1, a_2, ..., a_s; q - 1; n - 1)$.
\end{proposition}

\begin{proof}
According to Proposition \ref{proposition: G - A arrowsv (a_1, ..., a_(i - 1), a_i - 1, a_(i + 1_, ..., a_s)}, $H \arrowsv (a_1 - 1, a_2, ..., a_s)$. It is clear that $\omega(H) = \omega(G) - 1$, and therefore $H \in \mH_v(a_1 - 1, a_2, ..., a_s; q - 1; n - 1)$. Since $G \in \mH_{max}(a_1, ..., a_s; q; n)$, it follows that $H \in \mH_{max}(a_1 - 1, a_2, ..., a_s; q - 1; n - 1)$.
\end{proof}

According to Proposition \ref{proposition: H in mH_(max)(a_1 - 1, a_2, ..., a_s; q - 1; n - 1)}, if we know all the graphs in $\mH_{max}(a_1 - 1, a_2, ..., a_s; q - 1; n - 1)$ we can easily obtain the graphs in $\mH_{max}(a_1, ..., a_s; q; n)$ which have a cone vertex. We will use this fact to modify Algorithm \ref{algorithm: A3} and make it faster in the cases where all graphs in $\mH_{max}(a_1 - 1, a_2, ..., a_s; q - 1; n - 1)$ are already known. To find the graphs in $\mH_{max}(a_1, ..., a_s; q; n)$ which have no cone vertices we will need the following:

\begin{proposition}
	\label{proposition: G in mH_(max)(a_1, ..., a_s; q; n) and alpha(G) > k}
	Let $G \in \mH_{max}(a_1, ..., a_s; q; n)$ be a graph without cone vertices and $A$ be an independent set in $G$ such that $G - A$ has a cone vertex, i.e. $G - A$ = $K_1 + H$. Then $G = \overline{K}_{k + 1} + H$, where $k = \abs{A}$, $H$ has no cone vertices, and $K_1 + H \in \mH_{max}(a_1, ..., a_s; q; n - k)$.
\end{proposition}

\begin{proof}
	Let $A = \set{v_1, ..., v_k}$ be an independent set in $G$ and $G - A$ = $K_1 + H = \set{u} + H$. Since $G$ has no cone vertices, there exist $v_i \in A$ such that $v_i$ is not adjacent to $u$. Then $N_G(v_i) \subseteq N_G(u)$, and since $G$ is a maximal $K_q$-free graph, we obtain $N_G(v_i) = N_G(u) = \V(H)$. Hence, $u$ is not adjacent to any of the vertices in $A$, and therefore $N_G(v_j) = N_G(u) = \V(H), j = 1, ..., k$. We derived $G = \overline{K}_{k + 1} + H$. The graph $H$ has no cone vertices, since any cone vertex in $H$ would be a cone vertex in $G$. It is easy to see that from $\overline{K}_{k + 1} + H \arrowsv (a_1, ..., a_s)$ it follows $K_1 + H \arrowsv (a_1, ..., a_s)$. It is clear that $K_1 + H \in \mH_{max}(a_1, ..., a_s; q; n - k)$.
\end{proof}

Now we present the optimized algorithm:
\begin{namedalgorithm}{A4}
	\label{algorithm: A4}
	The input of the algorithm are the set $\mA_1 = \mH_{max}^t(a_1 - 1, a_2, ..., a_s; q; n - k)$ and the set $\mA_2 = \mH_{max}^t(a_1 - 1, a_2, ..., a_s; q - 1; n - 1)$, where $a_1, ..., a_s, q, n, k, t$ are fixed positive integers, $a_1 \geq 2$ and $k \leq t$.
	
	The output of the algorithm is the set $\mB$ of all graphs $G \in \mH_{max}^t(a_1, ..., a_s; q; n)$ with $\alpha(G) \geq k$.
	
	\emph{1.} By removing edges from the graphs in $\mA_1$ obtain the set
	
	$\mA_1' = \set{H \in \mH_{+K_{q - 1}}^t(a_1 - 1, a_2, ..., a_s; q; n - k) : \mbox{ $H$ has no cone vertices}}$.
	
	\emph{2.} Repeat step 2 of Algorithm \ref{algorithm: A3}.
	
	\emph{3.} Repeat step 3 of Algorithm \ref{algorithm: A3}.
	
	\emph{4.} Repeat step 4 of Algorithm \ref{algorithm: A3}.
	
	\emph{5.} If $t > k$, find the subset $\mA_1''$ of $\mA_1$ containing all graphs with exactly one cone vertex. For each graph $K_1 + H \in \mA_1''$, if $K_1 + H \arrowsv (a_1, ..., a_s)$, then add $\overline{K}_{k + 1} + H$ to $\mB$.
	
	\emph{6.} For each graph $H$ in $\mA_2$ such that $\alpha(H) \geq k$, if $K_1 + H \arrowsv (a_1, ..., a_s)$, then add $K_1 + H$ to $\mB$.
\end{namedalgorithm}

\begin{theorem}
	\label{theorem: algorithm A4}
	After the execution of Algorithm \ref{algorithm: A4}, the obtained set $\mB$ coincides with the set of all graphs $G \in \mH_{max}^t(a_1, ..., a_s; q; n)$ with $\alpha(G) \geq k$.
\end{theorem}

\begin{proof}
	Suppose that after the execution of Algorithm \ref{algorithm: A4}, $G \in \mB$. If after step 4 $G \in \mB$, then according to the proof of Theorem \ref{theorem: algorithm A3}, $G \in \mH_{max}^t(a_1, ..., a_s; q; n)$ and $\alpha(G) \geq k$.
	
	If $G$ is added to $\mB$ in step 5, then $G = \overline{K}_{k + 1} + H$ and from $K_1 + H \arrowsv(a_1, ..., a_s)$ it follows $G \arrowsv (a_1, ..., a_s)$. Since $\alpha(H) \leq t$ and $t > k$, we have $\alpha(G) \leq t$. In this case it is clear that $\alpha(G) \geq k + 1$. Since $K_1 + H \in \mH_{max}(a_1, ..., a_s; q; n - k)$ and $K_1 + H$ is not a complete graph, it follows that $G = \overline{K}_{k + 1} + H \in \mH_{max}(a_1, ..., a_s; q; n)$.
	
	If $G$ is added to $\mB$ in step 6, then $G = K_1 + H \arrowsv (a_1, ..., a_s)$, where $H \in \mA_2$ and $\alpha(H) \geq k$. Since $k \leq \alpha(H) \leq t$, we have $k \leq \alpha(G) \leq t$. From $H \in \mH_{max}(a_1 - 1, a_2, ... a_s; q - 1; n - 1)$ it follows that $G = K_1 + H \in \mH_{max}(a_1, ..., a_s; q; n)$.
	
	Now let $G \in \mH_{max}^t(a_1, ..., a_s; q; n)$ and $\alpha(G) \geq k$. If $G = K_1 + H$ for some graph $H$, then, according to Proposition \ref{proposition: H in mH_(max)(a_1 - 1, a_2, ..., a_s; q - 1; n - 1)}, $H \in \mA_2$ and in step 6 $G$ is added to $\mB$.
	
	Suppose that $G$ has no cone vertices. Let $A \subseteq \V(G)$ be an independent set such that $\abs{A} = k$.
	
	If $G - A$ has a cone vertex, i.e. $G - A = K_1 + H$, then, according to Proposition \ref{proposition: G in mH_(max)(a_1, ..., a_s; q; n) and alpha(G) > k}, $G = \overline{K}_{k + 1} + H$, where $K_1 + H \in \mA_1''$. In this case it is clear that $t > k$, and therefore in step 5 $G$ is added to $\mB$.
	
	If $G - A$ has no cone vertices, then according to Proposition \ref{proposition: G - A in mH_(+K_(q - 1))(a_1 - 1, a_2, ..., a_s; q; n - abs(A))}, $G - A \in \mA_1'$. From the proof of Theorem \ref{theorem: algorithm A3} it follows that after the execution of step 4, $G \in \mB$.
\end{proof}

\vspace{1em}

\begin{remark}
	\label{remark: algorithm A4 k = 2}
	Note that if $G \in \mH_{max}^t(a_1, ..., a_s; q; n)$ and $n \geq q$, then $G$ is not a complete graph and $\alpha(G) \geq 2$. Therefore, if $n \geq q$ and $k = 2$, Algorithm \ref{algorithm: A4} finds all graphs in $\mH_{max}^t(a_1, ..., a_s; q; n)$.
\end{remark}

We will use Algorithm \ref{algorithm: A4} to prove Conjecture \ref{conjecture: rp(p) = 2, p geq 4} in the case $p = 7$ (Theorem \ref{theorem: rp(7) = 2}). We could prove Theorem \ref{theorem: rp(7) = 2} using Algorithm \ref{algorithm: A3}, but it would take us more than a month of computational time, while the presented proof was completed in just one day. To test our implementation of Algorithm \ref{algorithm: A4}, we proved again in a different way Conjecture \ref{conjecture: rp(p) = 2, p geq 4} in the case $p = 5$ (Theorem \ref{theorem: rp(5) = 2}) and Conjecture \ref{conjecture: rpp(p) = 0, p geq 4} in the case $p = 6$ (Theorem \ref{theorem: rp(6) = 2}), and compared the graphs obtained in each step of the proof to the graphs obtained in the proofs of these theorems.

\vspace{1em}

\vspace{1em}
Theorem \ref{theorem: algorithm A4} is published in \cite{BN17b}. Algorithm \ref{algorithm: A4} is published in \cite{BN17b} as Algorithm 3.9 (due to a typographical error, in step 5 of Algorithm 3.9 in \cite{BN17b} it is written \emph{$H \in \mA''_1$} instead of the correct \emph{$K_1 + H \in \mA''_1$}).

\vspace{2em}

\section{Computation of the numbers $F_v(2_{m - 7}, 7; m - 1)$}

\vspace{2em}

Let us remind that the numbers $\rp(p)$ are defined in the formulation of Theorem \ref{theorem: rp}.

\vspace{2em}

We will prove that Conjecture \ref{conjecture: rp(p) = 2, p geq 4} is true in the case $p = 7$.

\begin{theorem}
	\label{theorem: rp(7) = 2}
	$\rp(7) = 2$ and $F_v(2_{m - 7}, 7; m - 1) = m + 11, \ m \geq 9$.
\end{theorem}

\vspace{1em}

\begin{proof}
Since $F_v(2, 2, 7; 8) = 20$, according to Theorem \ref{theorem: rp}(d), to prove that $\rp(7) = 2$ we have to prove the inequalities $F_v(2, 2, 2, 7; 9) > 20$, $F_v(2, 2, 2, 2, 7; 10) > 21$, and $F_v(2, 2, 2, 2, 2, 7; 11) > 22$. By Lemma \ref{lemma: F_v(2_r, p; r + p - 1) leq F_v(2_s, p; s + p - 1) + r - s}, it is enough to prove only the last of the three inequalities. Using Algorithm \ref{algorithm: A3}, it is possible to prove that $F_v(2, 2, 2, 2, 2, 7; 11) > 22$, but it would take a lot of time. Instead, we will prove the three inequalities successively with Algorithm \ref{algorithm: A4}. Only the proof of the first inequality is presented in details, since the proofs of the other two are very similar. We will show that $\mH_v(2, 2, 7; 8; 19) = \emptyset$. The proof uses the graphs $\mH_{max}^3(4; 8; 8)$, $\mH_{max}^3(5; 8; 10)$, $\mH_{max}^3(6; 8; 12)$, $\mH_{max}^3(7; 8; 14)$, $\mH_{max}^3(2, 7; 8; 16)$, $\mH_{max}^3(2, 2, 7; 8; 19)$, $\mH_{max}^2(4; 8; 9)$, $\mH_{max}^2(5; 8; 11)$, $\mH_{max}^2(6; 8; 13)$, $\mH_{max}^2(7; 8; 15)$, $\mH_{max}^2(2, 7; 8; 17)$, $\mH_{max}^2(2, 2, 7; 8; 19)$ obtained in the proof of the lower bound $F_v(2, 2, 7; 8) \geq 20$ (see Table \ref{table: finding all graphs in H_v(2, 2, 7; 8; 19)}).

Suppose that $G \in \mH_v(2, 2, 2, 7; 9; 20)$. According to Theorem \ref{theorem: alpha(G) leq V(G) + m + p - 1}(b), $\alpha(G) \leq 3$. It is clear that $\alpha(G) \geq 2$. Therefore, it is enough to prove that there are no graphs with independence number 2 or 3 in $\mH_{max}(2, 2, 2, 7; 9; 20)$.\\

First we prove that there are no graphs in $\mH_{max}(2, 2, 2, 7; 9; 20)$ with independence number 3:

It is clear that $K_7$ is the only graph in $\mH_{max}(4; 9; 7)$.

We execute Algorithm \ref{algorithm: A4}($n = 9;\ k = 2;\ t = 3$) with inputs $\mA_1 = \mH_{max}^3(4; 9; 7) = \set{K_7}$ and $\mA_2 = \mH_{max}^3(4; 8; 8)$ to obtain all graphs in $\mB = \mH_{max}^3(5; 9; 9)$ (see Remark \ref{remark: algorithm A4 k = 2}).

We execute Algorithm \ref{algorithm: A4}($n = 11;\ k = 2;\ t = 3$) with inputs $\mA_1 = \mH_{max}^3(5; 9; 9)$ and $\mA_2 = \mH_{max}^3(5; 8; 10)$ to obtain all graphs in $\mB = \mH_{max}^3(6; 9; 11)$.

We execute Algorithm \ref{algorithm: A4}($n = 13;\ k = 2;\ t = 3$) with inputs $\mA_1 = \mH_{max}^3(6; 9; 11)$ and $\mA_2 = \mH_{max}^3(6; 8; 12)$ to obtain all graphs in $\mB = \mH_{max}^3(7; 9; 13)$.

We execute Algorithm \ref{algorithm: A4}($n = 15;\ k = 2;\ t = 3$) with inputs $\mA_1 = \mH_{max}^3(1, 7; 9; 13) = \mH_{max}^3(7; 9; 13)$ and $\mA_2 = \mH_{max}^3(7; 8; 14)$ to obtain all graphs in $\mB = \mH_{max}^3(2, 7; 9; 15)$.

We execute Algorithm \ref{algorithm: A4}($n = 17;\ k = 2;\ t = 3$) with inputs $\mA_1 = \mH_{max}^3(1, 2, 7; 9; 15) = \mH_{max}^3(2, 7; 9; 15)$ and $\mA_2 = \mH_{max}^3(1, 2, 7; 8; 16) = \mH_{max}^3(2, 7; 8; 16)$ to obtain all graphs in $\mB = \mH_{max}^3(2, 2, 7; 9; 17)$.

By executing Algorithm \ref{algorithm: A4}($n = 20;\ k = 3;\ t = 3$) with inputs $\mA_1 = \mH_{max}^3(1, 2, 2, 7; 9; 17) = \mH_{max}^3(2, 2, 7; 9; 17)$ and $\mA_2 = \mH_{max}^3(1, 2, 2, 7; 8; 19) = \mH_{max}^3(2, 2, 7; 8; 19)$, we obtain $\mB = \emptyset$. According to Theorem \ref{theorem: algorithm A4}, there are no graphs in $\mH_{max}(2, 2, 2, 7; 9; 20)$ with independence number 3.\\

It remains to be proved that there are no graphs in $\mH_{max}(2, 2, 2, 7; 9; 20)$ with independence number 2:

It is clear that $K_8$ is the only graph in $\mH_{max}(4; 9; 8)$.

We execute Algorithm \ref{algorithm: A4}($n = 10;\ k = 2;\ t = 2$) with inputs $\mA_1 = \mH_{max}^2(4; 9; 8) = \set{K_8}$ and $\mA_2 = \mH_{max}^2(4; 8; 9)$ to obtain all graphs in $\mB = \mH_{max}^3(5; 9; 10)$ (see Remark \ref{remark: algorithm A4 k = 2}).

We execute Algorithm \ref{algorithm: A4}($n = 12;\ k = 2;\ t = 2$) with inputs $\mA_1 = \mH_{max}^2(5; 9; 10)$ and $\mA_2 = \mH_{max}^2(5; 8; 11)$ to obtain all graphs in $\mB = \mH_{max}^3(6; 9; 12)$.

We execute Algorithm \ref{algorithm: A4}($n = 14;\ k = 2;\ t = 2$) with inputs $\mA_1 = \mH_{max}^2(6; 9; 12)$ and $\mA_2 = \mH_{max}^2(6; 8; 13)$ to obtain all graphs in $\mB = \mH_{max}^3(7; 9; 14)$.

We execute Algorithm \ref{algorithm: A4}($n = 16;\ k = 2;\ t = 2$) with inputs $\mA_1 = \mH_{max}^2(1, 7; 9; 14) = \mH_{max}^2(7; 9; 14)$ and $\mA_2 = \mH_{max}^2(7; 8; 15)$ to obtain all graphs in $\mB = \mH_{max}^3(2, 7; 9; 16)$.

We execute Algorithm \ref{algorithm: A4}($n = 18;\ k = 2;\ t = 2$) with inputs $\mA_1 = \mH_{max}^2(1, 2, 7; 9; 16) = \mH_{max}^2(2, 7; 9; 16)$ and $\mA_2 = \mH_{max}^2(1, 2, 7; 8; 17) = \mH_{max}^2(2, 7; 8; 17)$ to obtain all graphs in $\mB = \mH_{max}^2(2, 2, 7; 9; 18)$.

By executing Algorithm \ref{algorithm: A4}($n = 20;\ k = 2;\ t = 2$) with inputs $\mA_1 = \mH_{max}^2(1, 2, 2, 7; 9; 18) = \mH_{max}^2(2, 2, 7; 9; 18)$ and $\mA_2 = \mH_{max}^2(1, 2, 2, 7; 8; 19) = \mH_{max}^2(2, 2, 7; 8; 19)$, we obtain $\mB = \emptyset$. According to Theorem \ref{theorem: algorithm A4}, there are no graphs in $\mH_{max}(2, 2, 2, 7; 9; 20)$ with independence number 2.\\

We proved that $\mH_{max}(2, 2, 2, 7; 9; 20) = \emptyset$ and $F_v(2, 2, 2, 7; 9) > 20$.\\

In the same way, the graphs obtained in the proof of the inequality $F_v(2, 2, 2, 7; 9) > 20$ are used to prove $F_v(2, 2, 2, 2, 7; 10) > 21$, and the graphs obtained in the proof of the inequality $F_v(2, 2, 2, 2, 7; 10) > 21$ are used to prove $F_v(2, 2, 2, 2, 2, 7; 11) > 22$. The number of graphs obtained in each step of the proofs is given in Table \ref{table: finding all graphs in H_v(2, 2, 2, 7; 9; 20)}, Table \ref{table: finding all graphs in H_v(2, 2, 2, 2, 7; 10; 21)}, and Table \ref{table: finding all graphs in H_v(2, 2, 2, 2, 2, 7; 11; 22)}. Notice that the number of graphs without cone vertices is relatively small, which reduces the computation time significantly.\\

Thus, we proved that $\rp(7) = 2$. From Theorem \ref{theorem: rp}(b) we obtain
$$F_v(2_{m - 7}, 7; m - 1) = m + 11, \ m \geq 9.$$
\end{proof}

\vspace{1em}

\vspace{1em}
Theorem \ref{theorem: rp(7) = 2} is published in \cite{BN17b}.

\vspace{2em}

	\begin{table}
		\centering
	\resizebox{0.85\textwidth}{!}{
		\begin{tabular}{ | l | l | r | r | r | r | }
			\hline
			{\parbox{9.5em}{set}}&
			{\parbox{3.5em}{\small indepen-\\dence\\ number}}&
			{\parbox{4em}{\small maximal\\ graphs}}&
			{\parbox{4em}{\small maximal\\ graphs\\ no cone v.}}&
			{\parbox{6em}{\small $(+K_8)$-graphs}}&
			{\parbox{6em}{\small $(+K_8)$-graphs\\ no cone v.}}\\
			\hline
				$\mH_v(4; 9; 7)$				& $\leq 3$				& 1					& 0		& 1					& 0			\\
				$\mH_v(5; 9; 9)$				& $\leq 3$				& 1					& 0		& 4					& 0			\\
				$\mH_v(6; 9; 11)$				& $\leq 3$				& 3					& 0		& 45				& 0			\\
				$\mH_v(7; 9; 13)$				& $\leq 3$				& 12				& 0		& 3 113				& 9			\\
				$\mH_v(2, 7; 9; 15)$			& $\leq 3$				& 169				& 0		& 4 783 615			& 7 097		\\
				$\mH_v(2, 2, 7; 9; 17)$		& $\leq 3$				& 36				& 2		& 22 918			& 22		\\
				$\mH_v(2, 2, 2, 7; 9; 20)$	& $= 3$					& 0					& 0		&					&			\\
				\hline
				$\mH_v(4; 9; 8)$				& $\leq 2$				& 1					& 0		& 1					& 0			\\
				$\mH_v(5; 9; 10)$				& $\leq 2$				& 1					& 0		& 8					& 0			\\
				$\mH_v(6; 9; 12)$				& $\leq 2$				& 3					& 0		& 85				& 1			\\
				$\mH_v(7; 9; 14)$				& $\leq 2$				& 10				& 0		& 5 474				& 80		\\
				$\mH_v(2, 7; 9; 16)$			& $\leq 2$				& 103				& 1		& 5 346 982			& 361 988	\\
				$\mH_v(2, 2, 7; 9; 18)$		& $\leq 2$				& 2845				& 85	& 387 948 338  		& 7 586 602	\\
				$\mH_v(2, 2, 2, 7; 9; 20)$	& $= 2$					& 0					& 0		&					&			\\
				\hline
				$\mH_v(2, 2, 2, 7; 9; 20)$	& 						& 0					& 0		&					&			\\
				\hline
			\end{tabular}
		}
		\vspace{-0.5em}
		\caption{\small Steps in finding all maximal graphs in $\mH_v(2, 2, 2, 7; 9; 20)$}
		\label{table: finding all graphs in H_v(2, 2, 2, 7; 9; 20)}
		\vspace{1em}

		\centering
	\resizebox{0.85\textwidth}{!}{
		\begin{tabular}{ | l | l | r | r | r | r | }
			\hline
			{\parbox{9.5em}{set}}&
			{\parbox{3.5em}{\small indepen-\\dence\\ number}}&
			{\parbox{4em}{\small maximal\\ graphs}}&
			{\parbox{4em}{\small maximal\\ graphs\\ no cone v.}}&
			{\parbox{6em}{\small $(+K_9)$-graphs}}&
			{\parbox{6em}{\small $(+K_9)$-graphs\\ no cone v.}}\\
			\hline
				$\mH_v(5; 10; 8)$					& $\leq 3$				& 1					& 0		& 1					& 0			\\
				$\mH_v(6; 10; 10)$				& $\leq 3$				& 1					& 0		& 4					& 0			\\
				$\mH_v(7; 10; 12)$				& $\leq 3$				& 3					& 0		& 45				& 0			\\
				$\mH_v(2, 7; 10; 14)$				& $\leq 3$				& 12				& 0		& 3 115				& 2			\\
				$\mH_v(2, 2, 7; 10; 16)$			& $\leq 3$				& 169				& 0		& 4 784 483			& 868		\\
				$\mH_v(2, 2, 2, 7; 10; 18)$		& $\leq 3$				& 36				& 0		& 22 919			& 1			\\
				$\mH_v(2, 2, 2, 2, 7; 10; 21)$	& $= 3$					& 0					& 0		&					&			\\
				\hline
				$\mH_v(5; 10; 9)$					& $\leq 2$				& 1					& 0		& 1					& 0			\\
				$\mH_v(6; 10; 11)$				& $\leq 2$				& 1					& 0		& 8					& 0			\\
				$\mH_v(7; 10; 13)$				& $\leq 2$				& 3					& 0		& 85				& 0			\\
				$\mH_v(2, 7; 10; 15)$				& $\leq 2$				& 10				& 0		& 5 495				& 21		\\
				$\mH_v(2, 2, 7; 10; 17)$			& $\leq 2$				& 103				& 0		& 5 371 651			& 24 669	\\
				$\mH_v(2, 2, 2, 7; 10; 19)$		& $\leq 2$				& 2848				& 3		& 387 968 658  		& 20 320	\\
				$\mH_v(2, 2, 2, 2, 7; 10; 21)$	& $= 2$					& 0					& 0		&					&			\\
				\hline
				$\mH_v(2, 2, 2, 2, 7; 10; 21)$	& 						& 0					& 0		&					&			\\
				\hline
			\end{tabular}
		}
		\vspace{-0.5em}
		\caption{\small Steps in finding all maximal graphs in $\mH_v(2, 2, 2, 2, 7; 10; 21)$}
		\label{table: finding all graphs in H_v(2, 2, 2, 2, 7; 10; 21)}
		\vspace{1em}

		\centering
	\resizebox{0.85\textwidth}{!}{
		\begin{tabular}{ | l | l | r | r | r | r | }
			\hline
			{\parbox{9.5em}{set}}&
			{\parbox{3.5em}{\small indepen-\\dence\\ number}}&
			{\parbox{4em}{\small maximal\\ graphs}}&
			{\parbox{4em}{\small maximal\\ graphs\\ no cone v.}}&
			{\parbox{6em}{\small $(+K_{10})$-graphs}}&
			{\parbox{6em}{\small $(+K_{10})$-graphs\\ no cone v.}}\\
			\hline
				$\mH_v(6; 11; 9)$						& $\leq 3$				& 1					& 0		& 1					& 0			\\
				$\mH_v(7; 11; 11)$					& $\leq 3$				& 1					& 0		& 4					& 0			\\
				$\mH_v(2, 7; 11; 13)$					& $\leq 3$				& 3					& 0		& 45				& 0			\\
				$\mH_v(2, 2, 7; 11; 15)$				& $\leq 3$				& 12				& 0		& 3 116				& 1			\\
				$\mH_v(2, 2, 2, 7; 11; 17)$			& $\leq 3$				& 169				& 0		& 4 784 638			& 155		\\
				$\mH_v(2, 2, 2, 2, 7; 11; 19)$		& $\leq 3$				& 36				& 0		& 22 919			& 0			\\
				$\mH_v(2, 2, 2, 2, 2, 7; 11; 22)$		& $= 3$					& 0					& 0		&					&			\\
				\hline
				$\mH_v(6; 11; 10)$					& $\leq 2$				& 1					& 0		& 1					& 0			\\
				$\mH_v(7; 11; 12)$					& $\leq 2$				& 1					& 0		& 8					& 0			\\
				$\mH_v(2, 7; 11; 14)$					& $\leq 2$				& 3					& 0		& 85				& 0			\\
				$\mH_v(2, 2, 7; 11; 16)$				& $\leq 2$				& 10				& 0		& 5 502				& 7			\\
				$\mH_v(2, 2, 2, 7; 11; 18)$			& $\leq 2$				& 103				& 0		& 5 374 143			& 2 492		\\
				$\mH_v(2, 2, 2, 2, 7; 11; 20)$		& $\leq 2$				& 2848				& 0		& 387 968 676  		& 18		\\
				$\mH_v(2, 2, 2, 2, 2, 7; 11; 22)$		& $= 2$					& 0					& 0		&					&			\\
				\hline
				$\mH_v(2, 2, 2, 2, 2, 7; 11; 22)$		& 						& 0					& 0		&					&			\\
				\hline
			\end{tabular}
		}
		\vspace{-0.5em}
		\caption{\small Steps in finding all maximal graphs in $\mH_v(2, 2, 2, 2, 2, 7; 11; 22)$}
		\label{table: finding all graphs in H_v(2, 2, 2, 2, 2, 7; 11; 22)}
	\end{table}

\section{Proof of Theorem \ref{theorem: F_v(a_1, ..., a_s; m - 1) leq m + 12, max set(a_1, ..., a_s) = 7}}

The lower bound follows from Theorem \ref{theorem: rp(7) = 2} and Theorem \ref{theorem: F_v(2_(m - p), p; q) leq F_v(a_1, ..., a_s; q) leq wFv(m)(p)(q)}.

To prove the upper bound, we will use the graph $\Gamma_4 \in \mH_v(3, 7; 8) \cap \mH_v(4, 6; 8) \cap \mH_v(5, 5; 8)$ (see Figure \ref{figure: wHv(9)(7)(8)(21)}) obtained by adding one vertex to the graph $G_{2,2,7} \in \mH_v(2, 2, 7; 8; 20)$ (see Figure \ref{figure: H_v(2, 2, 7; 8; 20)}). According to Proposition \ref{proposition: a_1 + a_2 - 1 > p}, from $\Gamma_4 \arrowsv (3, 7)$, $\Gamma_4 \arrowsv (4, 6)$, and $\Gamma_4 \arrowsv (5, 5)$ it follows $\Gamma_4 \arrowsv \uni{9}{7}$. Therefore, $\Gamma_4 \in \wHv{9}{7}{8}$ and $\wFv{9}{7}{8} \leq 21$. Now, from Theorem \ref{theorem: F_v(2_(m - p), p; q) leq F_v(a_1, ..., a_s; q) leq wFv(m)(p)(q)} and Theorem \ref{theorem: wFv(m)(p)(m - m_0 + q) leq wFv(m_0)(p)(q) + m - m_0} we derive
$$F_v(a_1, ..., a_s; m - 1) \leq \wFv{m}{7}{m - 1} \leq \wFv{9}{7}{8} + m - 9 \leq m + 12.$$
Thus, the theorem is proved. \qed

\vspace{1em}

Regarding the number $F_v(3, 7; 8)$, the following bounds were known:
$$18 \leq F_v(3, 7; 8) \leq 22.$$

The lower bound is true according to (\ref{equation: m + p + 2 leq F_v(a_1, ..., a_s; m - 1) leq m + 3p}) and the upper bound was proved in \cite{SXP09}. We improve these bounds by proving the following

\begin{theorem}
	\label{theorem: 20 leq F_v(3, 7; 8) leq 21}
	$20 \leq F_v(3, 7; 8) \leq 21$.
\end{theorem}

\begin{proof}
	The upper bound is true according to Theorem \ref{theorem: F_v(a_1, ..., a_s; m - 1) leq m + 12, max set(a_1, ..., a_s) = 7} and the lower bound follows from $F_v(3, 7; 8) \geq F_v(2, 2, 7; 8) = 20$.
\end{proof}

\vspace{1em}
Theorem \ref{theorem: F_v(a_1, ..., a_s; m - 1) leq m + 12, max set(a_1, ..., a_s) = 7} and Theorem \ref{theorem: 20 leq F_v(3, 7; 8) leq 21} are published in \cite{BN17b}.

\begin{figure}
	\centering
	\begin{subfigure}{\textwidth}
		\centering
		\includegraphics[height=336px,width=168px]{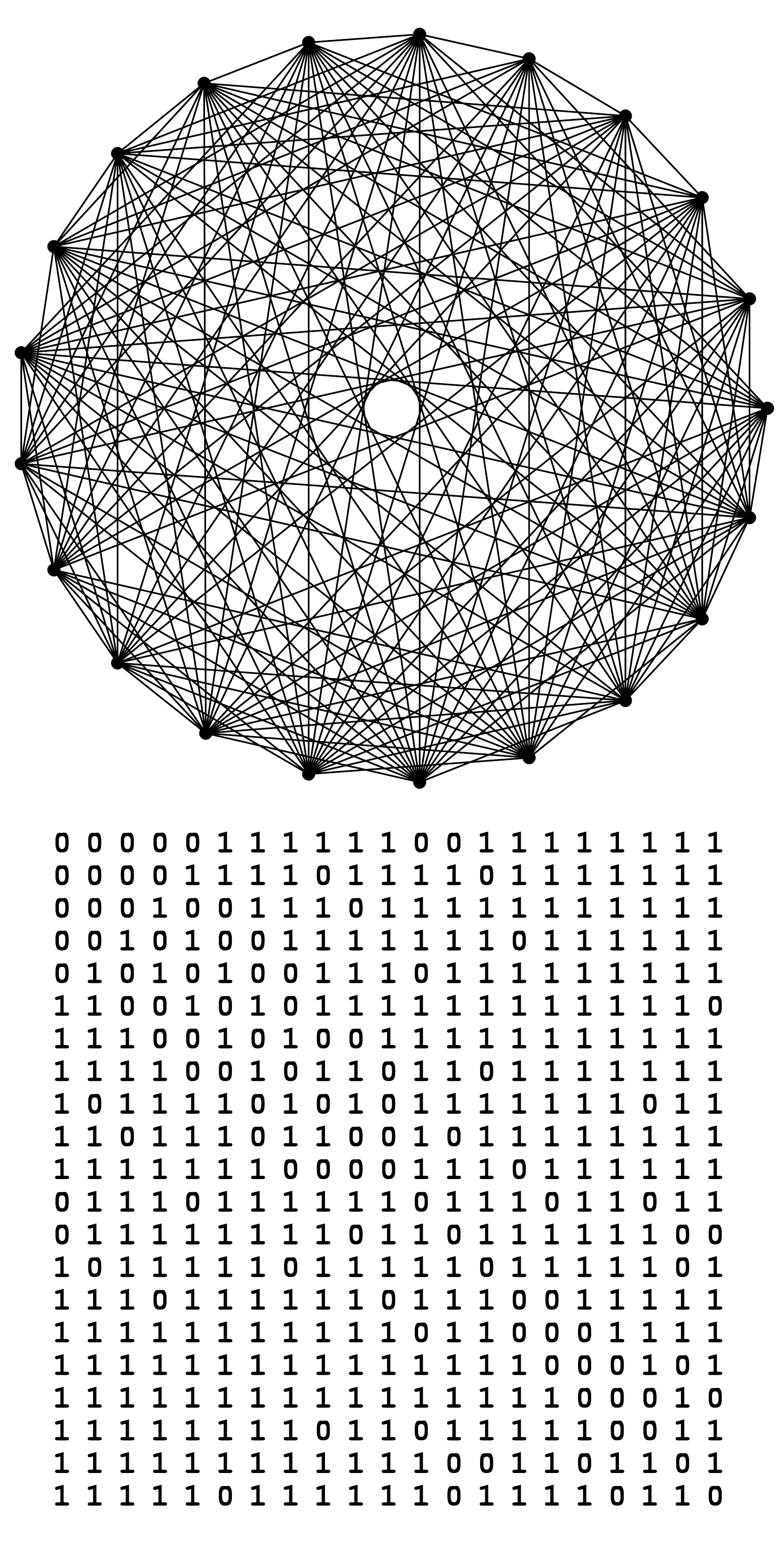}
		\caption*{\emph{$\Gamma_4$}}
		\label{figure: Gamma_4}
	\end{subfigure}%
	\caption{Graph $\Gamma_4 \in \mH_v(3, 7; 8; 21) \cap \mH_v(4, 6; 8; 21) \cap \mH_v(5, 5; 8; 21)$}
	\label{figure: wHv(9)(7)(8)(21)}
\end{figure}

\chapter{Numbers of the form $F_v(a_1, ..., a_s; m - 2)$}

Not much is known about the numbers $F_v(a_1, ..., a_s; q)$ when $q \leq m - 2$. In the case $q = m - 2$ we know all the numbers where $p = \max\set{a_1, ..., a_s} = 2$, i.e. all the numbers of the form $F_v(2_r; r - 1)$. Nenov proved in \cite{Nen83} that $F_v(2_r; r - 1) = r + 7$ if $r \geq 8$, and later in \cite{Nen07} he proved that the same equality holds when $r \geq 6$. Other proof of the same equality was given in \cite{Nen09}. By a computer aided research, Jensen and Royle \cite{JR95} obtained the result $F_v(2_4; 3) = 22$. All extremal graphs in $\mH_v(2_4; 3)$, as well as other interesting graphs, are available in \cite{BCGM_HoG}. The last remaining number of this form is $F_v(2_5; 4) = 16$. The upper bound was proved by Nenov in \cite{Nen07} and the lower bound was proved by Lathrop and Radziszowski in \cite{LR11} with the help of a computer. Also in \cite{LR11} it is proved that there are exactly two extremal graphs in $\mH_v(2_5; 4)$. The number $F_v(2_5; 3)$ is also of significant interest. The newest bounds for this number are $32 \leq F_v(2_5; 3) \leq 40$ obtained by Goedgebeur \cite{Goe17}.

The exact values of the numbers $F_v(a_1, ..., a_s; m - 2)$ for which $p \geq 3$ are unknown and no good general bounds have been obtained. According to (\ref{equation: F_v(a_1, ..., a_s; q) exists}), $F_v(a_1, ..., a_s; m - 2)$ exists if and only if $m \geq p + 3$. If $p = 3$, it is easy to see that in the border case $m = 6$ there are only two numbers of the form $F_v(a_1, ..., a_s; m - 2)$, namely $F_v(2, 2, 2, 3; 4)$ and $F_v(2, 3, 3; 4)$. Since $G \arrowsv (2, 3, 3) \Rightarrow G \arrowsv (2, 2, 2, 3)$, it follows that
\begin{equation*}
\label{equation: F_v(2, 2, 2, 3; 4) leq F_v(2, 3, 3; 4)}
F_v(2, 2, 2, 3; 4) \leq F_v(2, 3, 3; 4).
\end{equation*}
Computing and obtaining bounds on the numbers $F_v(2, 2, 2, 3; 4)$ and $F_v(2, 3, 3; 4)$ has an important role in relation to computing and obtaining bounds on the other numbers of the form $F_v(a_1, ..., a_s; m - 2)$ for which $p = 3$ (see \cite{Bik17} for more details).

Shao, Xu, and Luo \cite{SXL09} proved in 2009 that
\begin{equation*}
18 \leq F_v(2, 2, 2, 3; 4) \leq F_v(2, 3, 3; 4) \leq 30.
\end{equation*}
In 2011 Shao, Liang, He, and Xu \cite{SLHX11} raised the lower bound on $F_v(2, 3, 3; 4)$ to 19.

Our interest in the number $F_v(2, 3, 3; 4)$ is also motivated by our conjecture that the inequality $F_v(2, 3, 3; 4) \leq F_e(3, 3; 4)$ is true. If this inequality holds, then there would be another possible way to prove a lower bound on the number $F_e(3, 3; 4)$ by bounding the number $F_v(2, 3, 3; 4)$.

We improve the bounds on the numbers $F_v(2, 2, 2, 3; 4)$ and $F_v(2, 3, 3; 4)$ by proving the following theorems:
\begin{theorem}
	\label{theorem: 20 leq F_v(2, 2, 2, 3; 4) leq 22}
	$20 \leq F_v(2, 2, 2, 3; 4) \leq 22$.
\end{theorem}
\begin{theorem}
	\label{theorem: 20 leq F_v(2, 3, 3; 4) leq 24}
	$20 \leq F_v(2, 3, 3; 4) \leq 24.$
\end{theorem}

\vspace{1em}

\section{Algorithm A5}

According to Proposition \ref{proposition: G - A in mH_(+K_(q - 1))(a_1 - 1, a_2, ..., a_s; q; n - abs(A))}, if $G \arrowsv (2, 3, 3)$ and the graph $H$ is obtained by removing an independent set of vertices from $G$, then $H \arrowsv (3, 3)$. Shao, Liang, He, and Xu showed in \cite{SLHX11} that if $G \in \mH_v(2, 3, 3; 4; 18)$, then $G$ can be obtained by adding 4 independent vertices to 111 of all the 153 graphs in $\mH_v(3, 3; 4; 14)$ obtained in \cite{PRU99}. In this way, with the help of a computer, the authors proved that $\mH_v(2, 3, 3; 4; 18) = \emptyset$, and therefore $F_v(2, 3, 3; 4) \geq 19$. Similarly, if $G \in \mH_v(2, 2, 2, 3; 4; 18)$, then $G$ can be obtained by adding 4 independent vertices to 12064 of the 12227 graphs in $\mH_v(2, 2, 3; 4; 14)$, which were obtained in \cite{CR06}. Proving the bound $F_v(2, 2, 2, 3; 4) \geq 19$ is harder because of the larger number of graphs that have to be extended.

\begin{definition}
	\label{definition: Sperner graph}
	We say that $G$ is a Sperner graph if $N_G(u) \subseteq N_G(v)$ for some pair of vertices $u, v \in \V(G)$.
\end{definition}

\begin{proposition}
\label{proposition: Sperner graphs in mH_v(2_r, 3; 4; n)}
If $G \in \mH_v(2_r, 3; 4; n)$ is a Sperner graph and $N_G(u) \subseteq N_G(v)$, then $G - u \in \mH_v(2_r, 3; 4; n - 1)$.
\end{proposition}

In the special case when $G$ is a maximal Sperner graph in $\mH_v(2_r, 3; 4; n)$, from $N_G(u) \subseteq N_G(v)$ we derive $N_G(u) = N_G(v)$, and it follows that $G - u$ is a maximal graph in $\mH_v(2_r, 3; 4; n - 1)$. Thus, we proved the following proposition:

\begin{proposition}
	\label{proposition: maximal Sperner graphs}
	Every maximal Sperner graph in $\mH_v(2_r, 3; 4; n)$ is obtained by duplicating a vertex in some of the maximal graphs in $\mH_v(2_r, 3; 4; n - 1)$.
\end{proposition}

\vspace{0.5em}

\begin{proposition}
	\label{proposition: N_G(u) neq N_H(v)}
	Let $G$ be a non-Sperner graph, $A \subseteq \V(G)$ be an independent set of vertices of $G$, and $H = G - A$. Then,
	
	$N_G(u) \not\subseteq N_H(v), \ \forall u \in A \mbox{ and } \forall v \in \V(H)$.
\end{proposition}

\begin{proof}
	If we suppose the opposite is true, it follows that $N_G(v) \supseteq N_G(u)$, which is a contradiction.
\end{proof}

\vspace{0.5em}

We will use the following specialized version of Algorithm \ref{algorithm: A3} to efficiently generate all non-Sperner graphs $G \in \mH_{max}(2_r, 3; 4; n)$ with $\alpha(G) = k$:
\begin{namedalgorithm}{A5}
	\label{algorithm: A5}
	The input of the algorithm is the set $\mA = \mH_{max}^k(2_{r - 1}, 3; 4; n - k)$, where $r, n, k$ are fixed positive integers.
	
	The output of the algorithm is the set $\mB$ of all non-Sperner graphs $G \in \mH_{max}(2_r, 3; 4; n)$ with $\alpha(G) = k$.
	
	\emph{1.} By removing edges from the graphs in $\mA$ obtain the set
	
	$\mA' = \mH_{+K_3}^k(2_{r - 1}, 3; 4; n - k)$.
	
	\emph{2.} For each graph $H \in \mA'$:
	
	\emph{2.1.} Find the family $\mM(H) = \set{M_1, ..., M_l}$ of all maximal $K_3$-free subsets of $\V(H)$.
	
	\emph{2.2.} Find all $k$-element subsets $N = \set{M_{i_1}, ..., M_{i_k}}$ of $\mM(H)$ which fulfill the conditions:
	
	(a) $M_{i_j} \neq N_H(v)$ for every $v \in \V(H)$ and for every $M_{i_j} \in N$.
	
	(b) $K_2 \subseteq M_{i_j} \cap M_{i_h}$ for every $M_{i_j}, M_{i_h} \in N, \ j \neq h$.
	
	(c) $\alpha(H - \bigcup_{M_{i_j} \in N'} M_{i_j}) \leq k - \abs{N'}$ for every $N' \subseteq N$.
	
	\emph{2.3.} For each of the found in step 2.2 $k$-element subsets $N = \set{M_{i_1}, ..., M_{i_k}}$ of $\mM(H)$ construct the graph $G = G(N)$ by adding new independent vertices $v_1, ..., v_k$ to $\V(H)$ such that $N_G(v_j) = M_{i_j}, j = 1, ..., k$. If $G$ is not a Sperner graph and $\omega(G + e) = 4, \forall e \in \E(\overline{G})$, then add $G$ to $\mB$.
	
	\emph{3.} Remove the isomorphic copies of graphs from $\mB$.
	
	\emph{4.} Remove from $\mB$ all graphs $G$ for which $G \not\arrowsv (2_r, 3)$.
\end{namedalgorithm}

We will prove the correctness of Algorithm \ref{algorithm: A5} with the help of the following
\begin{lemma}
	\label{lemma: algorithm A5}
	After the execution of step 2.3 of Algorithm \ref{algorithm: A5}, the obtained set $\mB$ coincides with the set of all maximal $K_4$-free non-Sperner graphs $G$ with $\alpha(G) = k$ which have an independent set of vertices $A \subseteq \V(G), \abs{A} = k$ such that $G - A \in \mA'$.
\end{lemma}

\begin{proof}
	The proof follows the proof of Lemma \ref{lemma: algorithm A3}.
	Suppose that in step 2.3 of Algorithm \ref{algorithm: A3} the graph $G$ is added to $\mB$. In the same way as in the proof of Lemma \ref{lemma: algorithm A3}, we obtain $G \in \mH_{max}(2_r, 3; 4; n)$ and $\alpha(G) = k$. The check at the end of step 2.3 guarantees that $G$ is not a Sperner graph.
	
	Let $G$ be a maximal $K_4$-free non-Sperner graph, $\alpha(G) = k$, and $A = \set{v_1, ..., v_k}$ be an independent set of vertices of $G$ such that $H = G - A \in \mA'$. We will prove that, after the execution of step 2.3 of Algorithm \ref{algorithm: A5}, $G \in \mB$. Let $N = \set{N_G(v_1), ..., N_G(v_k)}$. In the same way as in the proof of Lemma \ref{lemma: algorithm A3} we show that $N$ fulfills the conditions (b) and (c) in step 2.2. Since $G$ is not a Sperner graph, from Proposition \ref{proposition: N_G(u) neq N_H(v)} it follows that $N$ fulfills (a). Thus, we showed that $N$ fulfills all conditions in step 2.2, and since $G = G(N)$ is a maximal $K_4$-free non-Sperner graph, in step 2.3 $G$ is added to $\mB$.
\end{proof}

\begin{theorem}
	\label{theorem: algorithm A5}
	After the execution of Algorithm \ref{algorithm: A5}, the obtained set $\mB$ coincides with the set of all non-Sperner graphs with independence number $k$ in $\mH_{max}(2_r, 3; 4; n)$.
\end{theorem}

\begin{proof}
	Suppose that, after the execution of Algorithm \ref{algorithm: A5}, $G \in \mB$. According to Lemma \ref{lemma: algorithm A5}, $G$ is a maximal $K_4$-free non-Sperner graph with independence number $k$. From step 4 it follows that $G \in \mH_{max}(2_r, 3; 4; n)$. 
	
	Conversely, let $G$ be an arbitrary non-Sperner graph with independence number $k$ in $\mH_{max}(2_r, 3; 4; n)$. Let $A \subseteq \V(G)$ be an independent set of vertices of $G$, $\abs{A} = k$ and $H = G - A$. According to Proposition \ref{proposition: G - A in mH_(+K_(q - 1))(a_1 - 1, a_2, ..., a_s; q; n - abs(A))}, $H \in \mA'$, and from Lemma \ref{lemma: algorithm A5} it follows that, after the execution of step 2.3, $G \in \mB$. Clearly, after step 4, $G$ remains in $\mB$.
\end{proof}

We performed various tests to our implementation of Algorithm \ref{algorithm: A5}. For example, we used the Algorithm \ref{algorithm: A5} to reproduce all 12227 graphs in $\mH_v(2, 2, 3; 4; 14)$, which were obtained in \cite{CR06}.

\vspace{1em}
Theorem \ref{theorem: algorithm A5} is published in \cite{Bik17}. Algorithm \ref{algorithm: A5} is a slightly modified version of Algorithm 2.4 in \cite{Bik17}.

\section{Proof of Theorem \ref{theorem: 20 leq F_v(2, 2, 2, 3; 4) leq 22} and Theorem \ref{theorem: 20 leq F_v(2, 3, 3; 4) leq 24}}

\vspace{1em}
\subsection*{Proof of Theorem \ref{theorem: 20 leq F_v(2, 2, 2, 3; 4) leq 22}}

1. Proof of the inequality $F_v(2, 2, 2, 3; 4) \geq 19$.

It is enough to prove that $\mH_v(2, 2, 2, 3; 4; 18) = \emptyset$. From $R(4, 4) = 18$ it follows that there are no graphs with independence number less than 4 in $\mH_v(2, 2, 2, 3; 4; 18)$, and from Proposition \ref{proposition: G - A in mH_(+K_(q - 1))(a_1 - 1, a_2, ..., a_s; q; n - abs(A))} and $F_v(2, 2, 3; 4) = 14$ \cite{CR06} we derive that there are no graphs with independence number more than 4 in this set. From $F_v(2, 2, 2, 3; 4) \geq 18$ and Proposition \ref{proposition: Sperner graphs in mH_v(2_r, 3; 4; n)} it follows that there are no Sperner graphs in $\mH_v(2, 2, 2, 3; 4; 18)$. It remains to be proved that there are no non-Sperner graphs with independence number 4 in $\mH_{max}(2, 2, 2, 3; 4; 18)$:

We execute Algorithm \ref{algorithm: A5} ($n = 18;\ r = 3;\ k = 4$) with input the set $\mA$ of all 584 graphs in $\mH_{max}^4(2, 2, 3; 4; 14)$. Let us remind, that all 12227 graphs in $\mH_v(2, 2, 3; 4; 14)$ were obtained in \cite{CR06}. After the execution of step 3, 130923 graphs remain in the set $\mB$. None of these graphs belong to $\mH_v(2, 2, 2, 3; 4)$, and therefore after step 4 we have $\mB = \emptyset$. Now, from Theorem \ref{theorem: algorithm A5} we conclude that there are no non-Sperner graphs with independence number 4 in $\mH_{max}(2, 2, 2, 3; 4; 18)$, which finishes the proof.\\

2. Proof of the inequality $F_v(2, 2, 2, 3; 4) \geq 20$.

It is enough to prove that $\mH_v(2, 2, 2, 3; 4; 19) = \emptyset$. Again, from $R(4, 4) = 18$ it follows that there are no graphs with independence number less than 4 in $\mH_v(2, 2, 2, 3; 4; 19)$, and from Proposition \ref{proposition: G - A in mH_(+K_(q - 1))(a_1 - 1, a_2, ..., a_s; q; n - abs(A))} and $F_v(2, 2, 3; 4) = 14$ \cite{CR06} we derive that there are no graphs with independence number more than 5 in this set. From $F_v(2, 2, 2, 3; 4) \geq 19$ and Proposition \ref{proposition: Sperner graphs in mH_v(2_r, 3; 4; n)} it follows that there are no Sperner graphs in $\mH_v(2, 2, 2, 3; 4; 19)$. It remains to be proved that there are no non-Sperner graphs with independence number 4 or 5 in $\mH_{max}(2, 2, 2, 3; 4; 19)$:\\

First we prove that there are no non-Sperner graphs with independence number 5 in $\mH_{max}(2, 2, 2, 3; 4; 19)$:

We execute Algorithm \ref{algorithm: A5} ($n = 19;\ r = 3;\ k = 5$) with input the set $\mA$ of all 591 graphs in $\mH_{max}^5(2, 2, 3; 4; 14)$ known from \cite{CR06}. After the execution of step 3, 27433657 graphs remain in the set $\mB$. Since none of these graphs are in $\mH_v(2, 2, 2, 3; 4)$, after step 4 we have $\mB = \emptyset$. Now, from Theorem \ref{theorem: algorithm A5} we conclude that there are no non-Sperner graphs with independence number 5 in $\mH_{max}(2, 2, 2, 3; 4; 19)$.\\

It remains to be proved that there are no non-Sperner graphs with independence number 4 in $\mH_{max}(2, 2, 2, 3; 4; 19)$:

Using the \emph{nauty} program \cite{MP13} we generate all 11-vertex non-isomorphic graphs and among them we find all 353 graphs in $\mH_{max}^4(2, 3; 4; 11)$.

We execute Algorithm \ref{algorithm: A5} ($n = 15;\ r = 2;\ k = 4$) with input $\mA = \mH_{max}^4(2, 3; 4; 11)$. By Theorem \ref{theorem: algorithm A5}, we obtain all 165614 non-Sperner graphs in $\mH_{max}(2, 2, 3; 4; 15)$ with independence number 4.

According to Proposition \ref{proposition: maximal Sperner graphs}, all Sperner graphs in $\mH_{max}(2, 2, 3; 4; 15)$ are obtained by duplicating a vertex in the graphs in $\mH_{max}(2, 2, 3; 4; 14)$. In this way, we find all 4603 Sperner graphs with independence number 4 in $\mH_{max}(2, 2, 3; 4; 15)$.

There are exactly 640 15-vertex $K_4$-free graphs with independence number less than 4, which are available on \cite{McK_r}. Among them there are 35 graphs in $\mH_{max}(2, 2, 3; 4; 15)$.

Thus, we found all 170252 graphs in $\mH_{max}^4(2, 2, 3; 4; 15)$.

We execute Algorithm \ref{algorithm: A5} ($n = 19;\ r = 3;\ k = 4$) with input $\mA = \mH_{max}^4(2, 2, 3; 4; 15)$. After the execution of step 3, 347307340 graphs remain in the set $\mB$. Similarly to the previous case, none of these graphs are in $\mH_v(2, 2, 2, 3; 4)$, and after step 4 we have $\mB = \emptyset$. By Theorem \ref{theorem: algorithm A5}, there are no non-Sperner graphs with independence number 4 in $\mH_{max}(2, 2, 2, 3; 4; 19)$, which finishes the proof.

All computations were completed in about a week on a personal computer.\\

3. Proof of the inequality $F_v(2, 2, 2, 3; 4) \leq 22$.

We need to construct a 22-vertex graph in $\mH_v(2, 2, 2, 3; 4)$. First, we find a large number of 24-vertex and 23-vertex maximal graphs in $\mH_v(2, 2, 2, 3; 4)$.

All 352366 24-vertex graphs $G$ with $\omega(G) \leq 3$ and $\alpha(G) \leq 4$ were found by McKay, Radziszowski and Angeltveit (see \cite{McK_r}). Among these graphs there are 3903 maximal graphs in $\mH_v(2, 2, 2, 3; 4; 24)$.

By removing one vertex from the obtained 24-vertex maximal graphs we find 6 graphs in $\mH_v(2, 2, 2, 3; 4; 23)$. Let us note that the removal of two vertices from the 24-vertex maximal graphs does not produce any graphs in $\mH_v(2, 2, 2, 3; 4; 22)$. Out of the 6 obtained graphs in $\mH_v(2, 2, 2, 3; 4; 23)$, 5 are maximal. One more maximal graph is obtained by adding one edge to the 6th graph, thus we have 6 maximal graphs in $\mH_v(2, 2, 2, 3; 4; 23)$. By applying Procedure \ref{procedure: extending a set of maximal graphs in mH_v(a_1, ..., a_s; q; n)} to these graphs, we find 192 more maximal graphs in $\mH_v(2, 2, 2, 3; 4; 23)$ 

By removing one vertex from the obtained 23-vertex maximal graphs we find a maximal graph in $\mH_v(2, 2, 2, 3; 4; 22)$, which is shown in Figure \ref{figure: H_v(2, 2, 2, 3; 4; 22)}. By removing one edge from the maximal graph in $\mH_v(2, 2, 2, 3; 4; 22)$, we obtain two more graphs in $\mH_v(2, 2, 2, 3; 4; 22)$.

\subsection*{Proof of Theorem \ref{theorem: 20 leq F_v(2, 3, 3; 4) leq 24}} 

1. Proof of the inequality $F_v(2, 3, 3; 4) \geq 20$.

Since $F_v(2, 3, 3; 4) \geq F_v(2, 2, 2, 3; 4)$, the lower bound $F_v(2, 3, 3; 4) \geq 20$ follows from Theorem \ref{theorem: 20 leq F_v(2, 2, 2, 3; 4) leq 22}.\\

2. Proof of the inequality $F_v(2, 3, 3; 4) \leq 24$.

All vertex transitive graphs with up to 31 vertices are known and can be found in \cite{Roy_c}. With the help of a computer we check which of these graphs belong to $\mH_v(2, 3, 3; 4)$. In this way, we find one 24-vertex graph, one 28-vertex graph and six 30-vertex graphs in $\mH_v(2, 3, 3; 4)$. The only 24-vertex transitive graph in $\mH_v(2, 3, 3; 4)$ is given in Figure \ref{figure: G_233}. It does not have proper subgraphs in $\mH_v(2, 3, 3; 4)$, but by adding edges to this graph we obtain 18 more graphs in $\mH_v(2, 3, 3; 4; 24)$, two of which are maximal $K_4$-free graphs. \qed

\vspace{1em}
Theorem \ref{theorem: 20 leq F_v(2, 2, 2, 3; 4) leq 22} is published in \cite{Bik17}. Theorem \ref{theorem: 20 leq F_v(2, 3, 3; 4) leq 24} is published in \cite{BN16}. The lower bound $F_v(2, 3, 3; 4) \geq 20$ in Theorem \ref{theorem: 20 leq F_v(2, 3, 3; 4) leq 24} follows from Theorem \ref{theorem: 20 leq F_v(2, 2, 2, 3; 4) leq 22}, but it is proved first in \cite{BN16} without using Theorem \ref{theorem: 20 leq F_v(2, 2, 2, 3; 4) leq 22}.

\begin{figure}
	\centering
	\begin{subfigure}[b]{\textwidth}
		\centering
		\includegraphics[trim={0 0 0 490},clip,height=220px,width=220px]{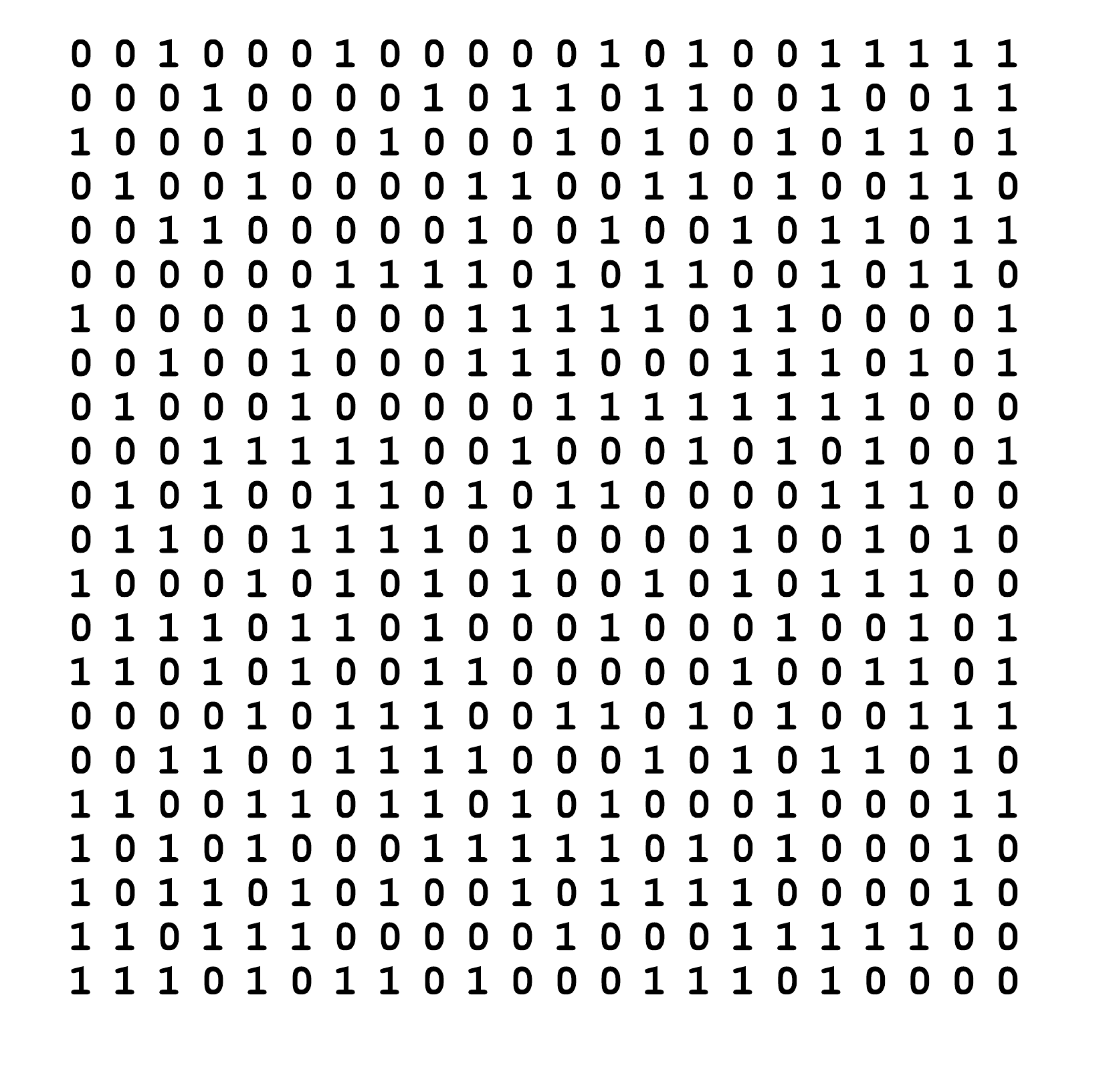}
		\label{figure: G_2223}
	\end{subfigure}%
	\vspace{-0.5em}
	\caption{22-vertex graph in $\mH_v(2, 2, 2, 3; 4)$}
	\label{figure: H_v(2, 2, 2, 3; 4; 22)}
	
	\vspace{1em}
	
	\centering
	\includegraphics[trim={0 0 0 490},clip,height=240px,width=240px]{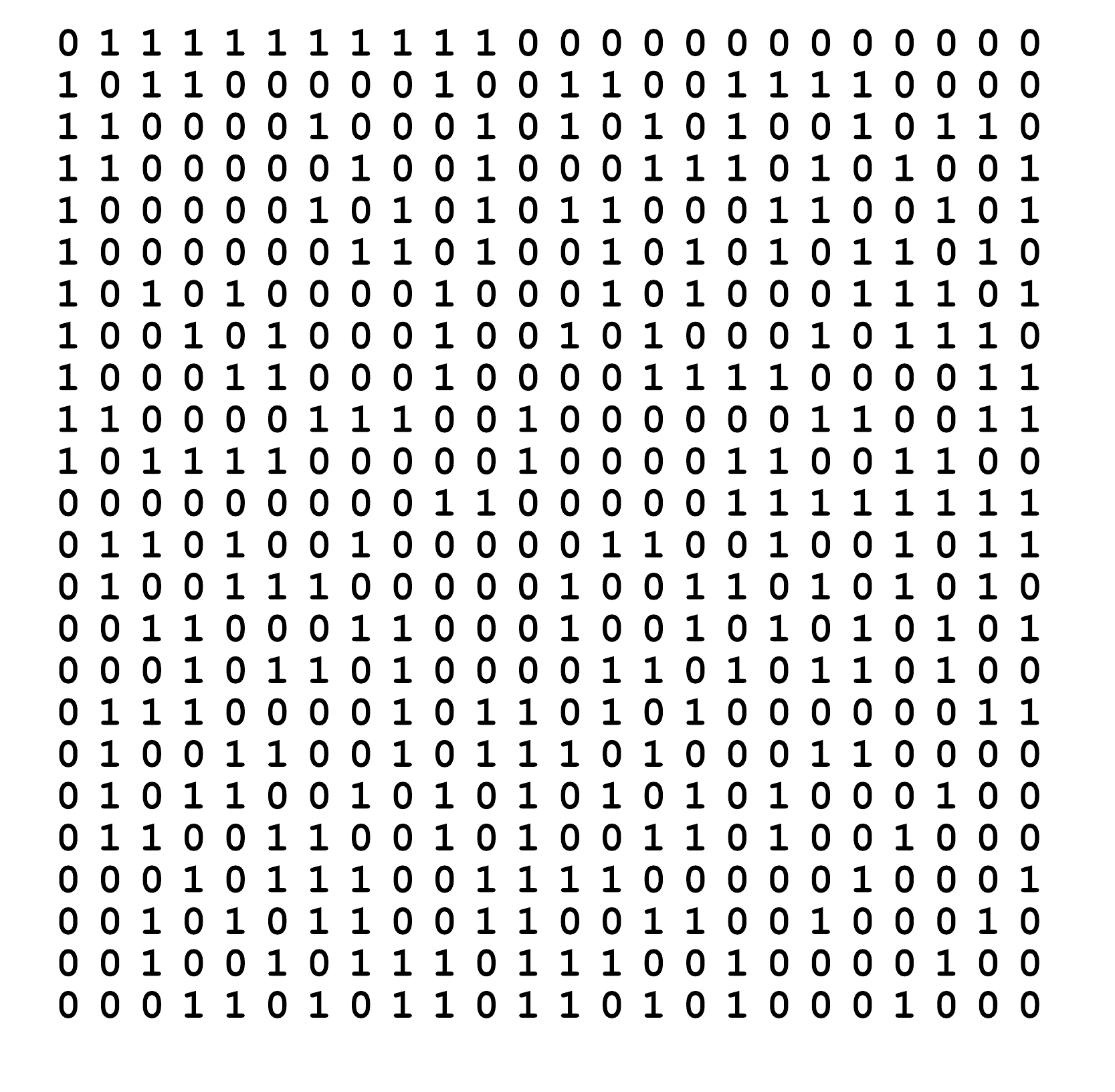}
	\vspace{-0.5em}
	\caption{24-vertex transitive graph in $\mH_v(2, 3, 3; 4)$}
	\label{figure: G_233}
\end{figure}

\chapter{Bounds on some numbers of the form $F_v(a_1, ..., a_s; q)$ where \hbox{$q = \max\set{a_1, ..., a_s} + 1$}}

\vspace{-1em}

The computation of the numbers of this form is very hard. The numbers $F_v(2, 2, p; p + 1), \ p \leq 4$, and $F_v(3, p; p + 1), \ p \leq 5$, are already computed (see the text after (\ref{equation: F_v(a_1, ..., a_s, m - 1) = ...})). We proved that $F_v(2, 2, 5; 6) = 16$ (Theorem \ref{theorem: F_v(2, 2, 5; 6) = 16}), $F_v(2, 2, 6; 7) = 17$ (Theorem \ref{theorem: F_v(2, 2, 6; 7) = 17 and abs(mH_v(2, 2, 6; 7; 17)) = 3}), $F_v(3, 6; 7) = 18$ (Theorem \ref{theorem: F_v(3, 6; 7) = 18}), and $F_v(2, 2, 7; 8) = 20$ (Theorem \ref{theorem: F_v(2, 2, 7; 8) = 20}). The only other known number of this form is $F_v(2, 2, 2, 2; 3) = 22$, \cite{JR95}.

We also obtained the bounds $20 \leq F_v(3, 7; 8) \leq 21$ (Theorem \ref{theorem: 20 leq F_v(3, 7; 8) leq 21}). In the previous chapter we proved that $20 \leq F_v(2, 2, 2, 3; 4) \leq 22$ (Theorem \ref{theorem: 20 leq F_v(2, 2, 2, 3; 4) leq 22}) и $20 \leq F_v(2, 3, 3; 4) \leq 24$ (Theorem \ref{theorem: 20 leq F_v(2, 3, 3; 4) leq 24}).

The numbers $F_v(p, p; p + 1)$ are of significant interest. Obviously, $F_v(2, 2; 3) = 5$. It is known that $F_v(3, 3; 4) = 14$, \cite{Nen81a} and \cite{PRU99}, and in \cite{XLS10} it is proved that $17 \leq F_v(4, 4; 5) \leq 23$. For now, there is no good upper bound for $F_v(5, 5; 6)$. In \cite{XS10} it is proved that
\begin{equation}
\label{equation: F_v(p, p; p + 1) geq 4p - 1}
F_v(p, p; p + 1) \geq 4p - 1.
\end{equation}

Let us pose the following question:

\emph{Is it true, that the sequence $F_v(p, p; p + 1), \ p \geq 2$, is increasing ?}

\vspace{0.5em}

In this chapter we will obtain new lower bounds on the numbers $F_v(4, 4; 5)$, $F_v(5, 5; 6)$, $F_v(6, 6; 7)$, and $F_v(7, 7; 8)$.

Let $G \in \mH_v(2_r, p; p + 1)$ and $A \subseteq \V(G)$ be an independent set. Then obviously, $G - A \in \mH_v(2_{r - 1}, p; p + 1)$, and therefore
\begin{equation}
\label{equation: F_v(2_r, p; p + 1) geq F_v(2_(r - 1), p; p + 1) + alpha(r, p), r geq 2}
F_v(2_r, p; p + 1) \geq F_v(2_{r - 1}, p; p + 1) + \alpha(r, p), \ r \geq 2,
\end{equation}
where $\alpha(r, p) = \max\set{\alpha(G) : G \in \mH_{extr}(2_r, p; p + 1)}$.\\
From (\ref{equation: F_v(2_r, p; p + 1) geq F_v(2_(r - 1), p; p + 1) + alpha(r, p), r geq 2}) it follows easily
\begin{equation}
\label{equation: F_v(2_r, p; p + 1) geq F_v(2, 2, p; p + 1) + sum_(i = 3)^(r)alpha(i, p), r geq 3}
F_v(2_r, p; p + 1) \geq F_v(2, 2, p; p + 1) + \sum_{i = 3}^{r}\alpha(i, p), \ r \geq 3.
\end{equation}
Since $\alpha(i, p) \geq 2$, from (\ref{equation: F_v(2_r, p; p + 1) geq F_v(2, 2, p; p + 1) + sum_(i = 3)^(r)alpha(i, p), r geq 3}) we obtain
\begin{equation}
\label{equation: F_v(2_r, p; p + 1) geq F_v(2, 2, p; p + 1) + 2(r - 2), r geq 3}
F_v(2_r, p; p + 1) \geq F_v(2, 2, p; p + 1) + 2(r - 2), \ r \geq 3.
\end{equation}
From (\ref{equation: F_v(2_r, p; p + 1) geq F_v(2, 2, p; p + 1) + 2(r - 2), r geq 3}) and Theorem \ref{theorem: F_v(2_(m - p), p; q) leq F_v(a_1, ..., a_s; q) leq wFv(m)(p)(q)} we see that
\begin{equation}
\label{equation: F_v(p, p; p + 1) geq F_v(2_(p - 1), p; p + 1) geq F_v(2, 2, p; p + 1) + 2p - 6, p geq 3}
F_v(p, p; p + 1) \geq F_v(2_{p - 1}, p; p + 1) \geq F_v(2, 2, p; p + 1) + 2p - 6, \ p \geq 3.
\end{equation}
According to (\ref{equation: m + p + 2 leq F_v(a_1, ..., a_s; m - 1) leq m + 3p}), $F_v(2, 2, p; p + 1) \geq 2p + 4$. If $F_v(2, 2, p; p + 1) = 2p + 4$, then the inequality (\ref{equation: F_v(p, p; p + 1) geq 4p - 1}) gives a better bound for $F_v(p, p; p + 1)$ than the inequality (\ref{equation: F_v(p, p; p + 1) geq F_v(2_(p - 1), p; p + 1) geq F_v(2, 2, p; p + 1) + 2p - 6, p geq 3}). It is interesting to note that it is not known whether the equality $F_v(2, 2, p; p + 1) = 2p + 4$ holds for any $p$. If $F_v(2, 2, p; p + 1) = 2p + 5$, then the bounds for $F_v(p, p; p + 1)$ from (\ref{equation: F_v(p, p; p + 1) geq 4p - 1}) and (\ref{equation: F_v(p, p; p + 1) geq F_v(2_(p - 1), p; p + 1) geq F_v(2, 2, p; p + 1) + 2p - 6, p geq 3}) coincide, and if $F_v(2, 2, p; p + 1) > 2p + 5$, then the inequality (\ref{equation: F_v(p, p; p + 1) geq F_v(2_(p - 1), p; p + 1) geq F_v(2, 2, p; p + 1) + 2p - 6, p geq 3}) gives a better bound for $F_v(p, p; p + 1)$.

In the case $p = 5$, using the graphs from Theorem \ref{theorem: abs(mH_v(2, 2, 5; 6; 16)) = 147}, an even better bound for $F_v(5, 5; 6)$ can be obtained. From $F_v(2, 2, 5; 6) = 16$ and (\ref{equation: F_v(2_r, p; p + 1) geq F_v(2, 2, p; p + 1) + 2(r - 2), r geq 3}) it follows that $F_v(2_r, 5; 6) \geq  18, \ r \geq 3$. Since the Ramsey number $R(3, 6) = 18$, we have $\alpha(r, 5) \geq 3, \ r \geq 3$. Now, from (\ref{equation: F_v(2_r, p; p + 1) geq F_v(2, 2, p; p + 1) + sum_(i = 3)^(r)alpha(i, p), r geq 3}) we derive
\begin{equation*}
F_v(2_r, 5; 6) \geq F_v(2, 2, 5; 6) + 3(r - 2) = 10 + 3r.
\end{equation*}
From this inequality we see that $F_v(2, 2, 2, 5; 6) \geq 19$. We will prove that $F_v(2, 2, 2, 5; 6) \geq 20$. Suppose that $G \in \mH_v(2, 2, 2, 5; 6; 19)$. Since $R(3, 6) = 18$, we have $\alpha(G) \geq 3$. From $F_v(2, 2, 5; 6) = 16$ and Proposition \ref{proposition: G - A arrowsv (a_1, ..., a_(i - 1), a_i - 1, a_(i + 1_, ..., a_s)} it follows that $\alpha(G) \leq 3$. It remains to be proved that there are no graphs with independence number 3 in $\mH_{max}(2, 2, 2, 5; 6; 19)$. Let us remind, that in the proof of Theorem \ref{theorem: abs(mH_v(2, 2, 5; 6; 16)) = 147} we found all 37 graphs in $\mH_{max}(2, 2, 5; 6; 16)$. From Table \ref{table: H_v(2, 2, 5; 6; 16) properties} we see that $\mH_{max}(2, 2, 5; 6; 16) = \mH_{max}^3(2, 2, 5; 6; 16)$. By executing Algorithm \ref{algorithm: A3} ($n = 19;\ k = 3;\ t = 3$) with input $\mA = \mH_{max}^3(2, 2, 5; 6; 16)$, we obtain $\mB = \emptyset$. According to Theorem \ref{theorem: algorithm A3}, there are no graphs in $\mH_{max}(2, 2, 2, 5; 6; 19)$ with independence number 3. Thus, we proved $\mH_v(2, 2, 2, 5; 6; 19) = \emptyset$, and therefore $F_v(2, 2, 2, 5; 6) \geq 20$. Since $\alpha(4, 5) \geq 3$, from (\ref{equation: F_v(2_r, p; p + 1) geq F_v(2_(r - 1), p; p + 1) + alpha(r, p), r geq 2}) ($k = 4, p = 5$) and Theorem \ref{theorem: F_v(2_(m - p), p; q) leq F_v(a_1, ..., a_s; q) leq wFv(m)(p)(q)} we obtain
\begin{theorem}
	\label{theorem: F_v(5, 5; 6) geq F_v(2, 2, 2, 2, 5; 6) geq 23}
	$F_v(5, 5; 6) \geq F_v(2, 2, 2, 2, 5; 6) \geq 23$.
\end{theorem}

Now we proceed to bound the number $F_v(6, 6; 7)$. We will prove the following more general result:
\begin{theorem}
	\label{theorem: F_v(a_1, ..., a_s; 7) geq F_v(2_(m - 6), 6; 7) geq 3m - 5}
	Let $a_1, ..., a_s$ be positive integers such that $\max\set{a_1, ..., a_s} = 6$ and $m = \sum\limits_{i=1}^s (a_i - 1) + 1 \geq 9$. Then,
	$$F_v(a_1, ..., a_s; 7) \geq F_v(2_{m - 6}, 6; 7) \geq 3m - 5.$$
	In particular, $F_v(6, 6; 7) \geq 28$.
\end{theorem}

\begin{proof}
According to Theorem \ref{theorem: F_v(2_(m - p), p; q) leq F_v(a_1, ..., a_s; q) leq wFv(m)(p)(q)} from this paper, $F_v(a_1, ..., a_s; 7) \geq F_v(2_{m - 6}, 6; 7)$. We will prove by induction that $F_v(2_{m - 6}, 6; 7) \geq 3m - 5, m \geq 9$. 

The base case is $m = 9$, i.e. we have to prove that $F_v(2, 2, 2, 6; 7) \geq 22$. We will show that $\mH_v(2, 2, 2, 6; 7; 21) = \emptyset$. From $F_v(2, 2, 6; 7) = 17$ and Proposition \ref{proposition: G - A arrowsv (a_1, ..., a_(i - 1), a_i - 1, a_(i + 1_, ..., a_s)} it follows that there are no graphs in $\mH_v(2, 2, 2, 6; 7; 21)$ with independence number greater than 4. 

Let us remind that all graphs in $\mH_v(2, 2, 6; 7; 17)$ were obtained in the proof of Theorem \ref{theorem: F_v(2, 2, 6; 7) = 17 and abs(mH_v(2, 2, 6; 7; 17)) = 3} and all graphs in $\mH_v(2, 2, 6; 7; 18)$ were obtained in the proof of Theorem \ref{theorem: abs(mH_v(2, 2, 6; 7; 18)) = 76515}.

According to Theorem \ref{theorem: F_v(2, 2, 6; 7) = 17 and abs(mH_v(2, 2, 6; 7; 17)) = 3}, $\mH_{max}(2, 2, 6; 7; 17) = \set{G_1}$ from Figure \ref{figure: H_v(2, 2, 6; 7; 17)}. By executing Algorithm \ref{algorithm: A3} ($n = 21;\ k = 4;\ t = 4$) with input $\mA = \mH_{max}^4(1, 2, 2, 6; 7; 17) = \mH_{max}^4(2, 2, 6; 7; 17) = \set{G_1}$, we obtain $\mB = \emptyset$. According to Theorem \ref{theorem: algorithm A3}, there are no graphs in $\mH_{max}(2, 2, 2, 6; 7; 21)$ with independence number 4.

In the proof of Theorem \ref{theorem: abs(mH_v(2, 2, 6; 7; 18)) = 76515} we found all 392 graphs in $\mH_{max}(2, 2, 6; 7; 18)$. From Table \ref{table: H_v(2, 2, 6; 7; 18) properties} we see that $\mH_{max}(2, 2, 6; 7; 18) = \mH_{max}^3(2, 2, 6; 7; 18)$. By executing Algorithm \ref{algorithm: A3} ($n = 21;\ k = 3;\ t = 3$) with input $\mA = \mH_{max}^3(1, 2, 2, 6; 7; 18) = \mH_{max}^3(2, 2, 6; 7; 18)$, we obtain $\mB = \emptyset$. According to Theorem \ref{theorem: algorithm A3}, there are no graphs in $\mH_{max}(2, 2, 2, 6; 7; 21)$ with independence number 3.

It remains to be proved that there are no graphs in $\mH_{max}(2, 2, 2, 6; 7; 21)$ with independence number 2. All 21-vertex graphs $G$ for which $\alpha(G) < 3$ and $\omega(G) < 7$ are known and are available on \cite{McK_r}. There are 1 118 436 such graphs $G$, and with the help of the computer we check that none of these graphs belong to $\mH_v(2, 2, 2, 6; 7)$.

Thus, we proved $\mH_v(2, 2, 2, 6; 7; 21) = \emptyset$ and $F_v(2, 2, 2, 6; 7) \geq 22$.

Now, suppose that for all $m'$ such that $9 \leq m' < m$ we have $F_v(2_{m' - 6}, 6; 7) \geq 3m' - 5$. Let $G \in \mH_v(2_{m - 6}, 6; 7)$ and $\abs{\V(G)} = F_v(2_{m - 6}, 6; 7)$. From the base case it follows that $F_v(2_{m - 6}, 6; 7) > 22$. Since the Ramsey number $R(3, 7) = 23$, we have $\alpha(G) \geq 3$. According to (\ref{equation: F_v(2_r, p; p + 1) geq F_v(2_(r - 1), p; p + 1) + alpha(r, p), r geq 2}),
$$F_v(2_{m - 6}, 6; 7) = \abs{\V(G)} \geq F_v(2_{m - 7}, 6; 7) + 3 \geq 3(m - 1) - 5 + 3 = 3m - 5.$$
\end{proof}

\vspace{-1.5em}

\begin{theorem}
	\label{theorem: F_v(a_1, ..., a_s; 7) leq F_v(6, 6; 7) leq 60}
	Let $a_1, ..., a_s$ be positive integers such that $\max\set{a_1, ..., a_s} = 6$ and $m = \sum\limits_{i=1}^s (a_i - 1) + 1$. Then:
	\vspace{1em}\\
	(a) $22 \leq F_v(a_1, ..., a_s; 7) \leq F_v(4, 6; 7) \leq 35$, if $m = 9$.
	\vspace{1em}\\
	(b) $28 \leq F_v(a_1, ..., a_s; 7) \leq F_v(6, 6; 7) \leq 70$, if $m = 11$.
\end{theorem}

\begin{proof}
The lower bounds in (a) and (b) follow from Theorem \ref{theorem: F_v(a_1, ..., a_s; 7) geq F_v(2_(m - 6), 6; 7) geq 3m - 5}. It is easy to see that $F_v(a_1, ..., a_s; 7) \leq F_v(4, 6; 7)$ if $m = 9$, and $F_v(a_1, ..., a_s; 7) \leq F_v(6, 6; 7)$ if $m = 11$. To finish the proof we will use the following inequality proved by Kolev in \cite{Kol08}:
$$F_v(a_1, ..., a_s; q + 1) . F_v(b_1, ..., b_s; t + 1) \geq F_v(a_1.b_1, ..., a_s.b_s; qt + 1),$$
where $q \geq \max\set{a_1, ..., a_s}$ and $t \geq \max\set{b_1, ..., b_s}$.

From $F_v(2, 3; 4) = 7$ (Theorem \ref{theorem: F_v(a_1, ..., a_s; m) = m + p}) and $F_v(2, 2; 3) = 5$ it follows

$F_v(4, 6; 7) \leq F_v(2, 2; 3) . F_v(2, 3; 4) = 35.$

Since $F_v(2, 2; 3) = 5$ and $F_v(3, 3; 4) = 14$, \cite{Nen81a} and \cite{PRU99}, it follows that

$F_v(6, 6; 7) \leq F_v(2, 2; 3) . F_v(3, 3; 4) = 70.$
\end{proof}

Let $a_1, ..., a_s$ be positive integers and $m$ and $p$ be defined by (\ref{equation: m and p}). If $p = 7$, by Theorem \ref{theorem: F_v(2_(m - p), p; q) leq F_v(a_1, ..., a_s; q) leq wFv(m)(p)(q)} we have 
\begin{equation}
\label{equation: F_v(a_1, ..., a_s; q) geq F_v(2_(m - 8), 7; 8)}
F_v(a_1, ..., a_s; 8) \geq F_v(2_{m - 8}, 7; 8), \ m \geq 8.
\end{equation}
Since $F_v(2, 2, 7; 8) = 20$, from (\ref{equation: F_v(a_1, ..., a_s; q) geq F_v(2_(m - 8), 7; 8)}) and (\ref{equation: F_v(2_r, p; p + 1) geq F_v(2, 2, p; p + 1) + 2(r - 2), r geq 3}) we obtain
\begin{equation}
\label{equation: F_v(a_1, ..., a_s; 8) geq 2m + 2}
F_v(a_1, ..., a_s; 8) \geq 2m + 2.
\end{equation}
In particular, when $m = 13$ we have $F_v(a_1, ..., a_s; 8) \geq 28$. Since the Ramsey number $R(3, 8) = 28$, it follows that $\alpha(i, 7) \geq 3$, when $i \geq 6$. Now, from (\ref{equation: F_v(2_r, p; p + 1) geq F_v(2, 2, p; p + 1) + sum_(i = 3)^(r)alpha(i, p), r geq 3}) it follows easily that
\begin{theorem}
	\label{theorem: F_v(a_1, ..., a_s; 8) geq 3m - 10}
	If $m \geq 13$, and $\max\set{a_1, ..., a_s} = 7$, then
	$$F_v(a_1, ..., a_s; 8) \geq 3m - 10.$$
	In particular, $F_v(7, 7; 8) \geq 29$.
\end{theorem}
It is clear that when $3m - 10 \geq R(4, 8)$ these bounds for $F_v(a_1, ..., a_s; 8)$ can be improved considerably. In comparison, note that by (\ref{equation: F_v(p, p; p + 1) geq 4p - 1}), we have $F_v(7, 7; 8) \geq 27$, and according to (\ref{equation: F_v(a_1, ..., a_s; 8) geq 2m + 2}), $F_v(7, 7; 8) \geq 28$.\\

At the end of this chapter we will prove the following theorem for the number $F_v(4, 4; 5)$:
\begin{theorem}
	\label{theorem: F_v(4, 4; 5) geq F_v(2, 3, 4; 5) geq F_v(2, 2, 2, 4; 5) geq 19}
	$F_v(4, 4; 5) \geq F_v(2, 3, 4; 5) \geq F_v(2, 2, 2, 4; 5) \geq 19$.
\end{theorem}

\begin{proof}
The inequalities $F_v(4, 4; 5) \geq F_v(2, 3, 4; 5) \geq F_v(2, 2, 2, 4; 5)$ follow from (\ref{equation: G arrowsv (a_1, ..., a_s) Rightarrow G arrowsv (a_1, ..., a_(i - 1), t, a_i - t + 1, a_(i + 1), ..., a_s)}). It remains to be proved that $F_v(2, 2, 2, 4; 5) \geq 19$. Suppose that $\mH_{max}(2, 2, 2, 4; 5; 18) \neq \emptyset$ and let $G \in \mH_{max}(2, 2, 2, 4; 5; 18)$. Since the Ramsey number $R(3, 5) = 14$, $\alpha(G) \geq 3$. In \cite{Nen01c} it is proved that $F_v(2, 2, 4; 5) = 13$.  From Proposition \ref{proposition: G - A arrowsv (a_1, ..., a_(i - 1), a_i - 1, a_(i + 1_, ..., a_s)} and the equality $F_v(2, 2, 4; 5) = 13$ it follows that $\alpha(G) \leq 5$. It remains to be proved that there are no graphs with independence number 3, 4, or 5 in $\mH_{max}(2, 2, 2, 4; 5; 18)$.

In \cite{XLS10} it is proved $\mH_v(2, 2, 4; 5; 13) = \set{Q}$, where $Q$ is the unique 13-vertex $K_5$-free graph with independence number 2. 

By executing Algorithm \ref{algorithm: A3}($n = 18;\ k = 5;\ t = 5$) with input $\mA = \mH_{max}^5(1, 2, 2, 4; 5; 13) = \mH_{max}^5(2, 2, 4; 5; 13) = \set{Q}$, we obtain $\mB = \emptyset$. According to Theorem \ref{theorem: algorithm A3}, there are no graphs in $\mH_{max}(2, 2, 2, 4; 5; 18)$ with independence number 5.\\

\vspace{-1em}

Now we will prove that there are no graphs in $\mH_{max}(2, 2, 2, 4; 5; 18)$ with independence number 4:

Using the \emph{nauty} program \cite{MP13} we generate all 8-vertex non-isomorphic graphs and among them we find all 7 graphs in $\mH_{max}^4(3; 5; 8)$ 

We execute Algorithm \ref{algorithm: A3}($n = 10;\ k = 2;\ t = 4$) with input $\mA = \mH_{max}^4(3; 5; 8)$ to obtain all graphs in $\mB = \mH_{max}^4(4; 5; 10)$. (see Remark \ref{remark: algorithm A3 k = 2})

We execute Algorithm \ref{algorithm: A3}($n = 12;\ k = 2;\ t = 4$) with input $\mA = \mH_{max}^4(1, 4; 5; 10) = \mH_{max}^4(4; 5; 10)$ to obtain all graphs in $\mB = \mH_{max}^4(2, 4; 5; 12)$.

We execute Algorithm \ref{algorithm: A3}($n = 14;\ k = 2;\ t = 4$) with input $\mA = \mH_{max}^4(1, 2, 4; 5; 12) = \mH_{max}^4(2, 4; 5; 12)$ to obtain all graphs in $\mB = \mH_{max}^4(2, 2, 4; 5; 14)$.

By executing Algorithm \ref{algorithm: A3}($n = 18;\ k = 4;\ t = 4$) with input $\mA = \mH_{max}^4(1, 2, 2, 4; 5; 14) = \mH_{max}^4(2, 2, 4; 5; 14)$, we obtain $\mB = \emptyset$. According to Theorem \ref{theorem: algorithm A3}, there are no graphs in $\mH_{max}(2, 2, 2, 4; 5; 18)$ with independence number 4.\\

\vspace{-1em}

The last step is to prove that there are no graphs in $\mH_{max}(2, 2, 2, 4; 5; 18)$ with independence number 3:

Using the \emph{nauty} program \cite{MP13} we generate all 11-vertex non-isomorphic graphs and among them we find all 11 graphs in $\mH_{max}^3(3; 5; 9)$ 

We execute Algorithm \ref{algorithm: A3}($n = 11;\ k = 2;\ t = 3$) with input $\mA = \mH_{max}^3(3; 5; 9)$ to obtain all graphs in $\mB = \mH_{max}^3(4; 5; 11)$. (see Remark \ref{remark: algorithm A3 k = 2})

We execute Algorithm \ref{algorithm: A3}($n = 13;\ k = 2;\ t = 3$) with input $\mA = \mH_{max}^3(1, 4; 5; 11) = \mH_{max}^3(4; 5; 11)$ to obtain all graphs in $\mB = \mH_{max}^3(2, 4; 5; 13)$.

We execute Algorithm \ref{algorithm: A3}($n = 15;\ k = 2;\ t = 3$) with input $\mA = \mH_{max}^3(1, 2, 4; 5; 13) = \mH_{max}^3(2, 4; 5; 13)$ to obtain all graphs in $\mB = \mH_{max}^3(2, 2, 4; 5; 15)$.

By executing Algorithm \ref{algorithm: A3}($n = 18;\ k = 3;\ t = 3$) with input $\mA = \mH_{max}^3(1, 2, 2, 4; 5; 15) = \mH_{max}^3(2, 2, 4; 5; 15)$, we obtain $\mB = \emptyset$. According to Theorem \ref{theorem: algorithm A3}, there are no graphs in $\mH_{max}(2, 2, 2, 4; 5; 18)$ with independence number 3.

We proved that $\mH_{max}(2, 2, 2, 4; 5; 18) = \emptyset$, and therefore $F_v(2, 2, 2, 4; 5) \geq 19$.
\end{proof}

The number of graphs obtained in each step of the proof is given in Table \ref{table: finding all graphs in H_v(2, 2, 2, 4; 5; 18)}. Because of the large number of graphs in $\mH_{+K_4}^3(2, 4; 5; 13)$, the computer needed about a month to complete the computations.
	
	\begin{table}
		\centering
	\resizebox{0.8\textwidth}{!}{
		\begin{tabular}{ | l | l | r | r | }
			\hline
			{\parbox{10em}{set}}&
			{\parbox{6em}{\small independence\\ number}}&
			{\parbox{4em}{maximal\\ graphs}}&
			{\parbox{6em}{\hfill $(+K_4)$-graphs}}\\
			\hline
				$\mH_v(3; 5; 8)$				& $\leq 4$				& 7						& 274				\\
				$\mH_v(4; 5; 10)$				& $\leq 4$				& 44					& 65 422			\\
				$\mH_v(2, 4; 5; 12)$			& $\leq 4$				& 1 059					& 18 143 174		\\
				$\mH_v(2, 2, 4; 5; 14)$		& $\leq 4$				& 13					& 71				\\
				$\mH_v(2, 2, 2, 4; 5; 18)$	& $= 4$					& 0						&					\\
				\hline
				$\mH_v(3; 5; 9)$				& $\leq 3$				& 11					& 2 252				\\
				$\mH_v(4; 5; 11)$				& $\leq 3$				& 135					& 1 678 802			\\
				$\mH_v(2, 4; 5; 13)$			& $\leq 3$				& 11 439				& 2 672 047 607		\\
				$\mH_v(2, 2, 4; 5; 15)$		& $\leq 3$				& 1 103					& 78 117			\\
				$\mH_v(2, 2, 2, 4; 5; 18)$	& $= 3$					& 0						&					\\
				\hline
				$\mH_v(2, 2, 2, 4; 5; 18)$	& 						& 0						&					\\
				\hline
			\end{tabular}
		}
		\caption{Steps in finding all maximal graphs in $\mH_v(2, 2, 2, 4; 5; 18)$}
		\label{table: finding all graphs in H_v(2, 2, 2, 4; 5; 18)}
	\end{table}
	
The upper bound $F_v(4, 4; 5) \leq 23$ is proved in \cite{XLS10} with the help of a 23-vertex transitive graph. We were not able to obtain any other graphs in $\mH_v(4, 4; 5; 23)$, which leads us to believe that $F_v(4, 4; 5) = 23$. We did find a large number of 23-vertex graphs in $\mH_v(2, 2, 2, 4; 5)$, but so far we have not obtained smaller graphs in this set.

\vspace{1em}
Theorem \ref{theorem: F_v(5, 5; 6) geq F_v(2, 2, 2, 2, 5; 6) geq 23} is published in \cite{BN15a}. Theorem \ref{theorem: F_v(a_1, ..., a_s; 7) geq F_v(2_(m - 6), 6; 7) geq 3m - 5} and Theorem \ref{theorem: F_v(a_1, ..., a_s; 7) leq F_v(6, 6; 7) leq 60} are published in \cite{BN17a}. Theorem \ref{theorem: F_v(a_1, ..., a_s; 8) geq 3m - 10} and Theorem \ref{theorem: F_v(4, 4; 5) geq F_v(2, 3, 4; 5) geq F_v(2, 2, 2, 4; 5) geq 19} are published in \cite{BN17b}.

\part{Edge Folkman Numbers}

\chapter{Definition of the edge Folkman numbers and some known results}

\vspace{-1em}

The expression $G \arrowse (3, 3)$ means that in every coloring of the edges of the graph $G$ in two colors there is a monochromatic triangle.

It is well known that $K_6 \arrowse (3, 3)$.

Denote:

$\mH_e(3, 3) = \set{G : G \arrowse (3, 3)}$.

$\mH_e(3, 3; q) = \set{G : G \arrowse (3, 3) \mbox{ and } \omega(G) < q}$.

$\mH_e(3, 3; q; n) = \set{G : G \in \mH_e(3, 3; q) \mbox{ and } \abs{\V(G)} = n}$.

The edge Folkman number $F_e(3, 3; q)$ is defined with

$F_e(3, 3; q) = \min{\set{\abs{\V(G)} : G \in \mH_e(3, 3; q)}}$\\

\vspace{-1em}

The equality $R(3, 3) = 6$ means that $K_6 \arrowse (3, 3)$ and $K_5 \not\arrowse (3, 3)$. It follows that $F_e(3, 3; q) = 6, q \geq 7$.

In \cite{EH67} Erd\"os and Hajnal posed the following problem:
\begin{center}
	\label{question: Erdos and Hajnal}
	\emph{Does there exist a graph $G \arrowse (3, 3)$ with $\omega(G) < 6$ ?}
\end{center}

The first example of a graph which gives a positive answer to this question was given by van Lint. The complement of this graph is shown in Figure \ref{figure: vanLint}. Van Lint did not publish this result himself, but the graph was included in \cite{GS71}. Later, Graham \cite{Gra68} constructed the smallest possible example of such a graph, namely $K_3+C_5$. Thus, he proved that $F_e(3, 3; 6) = 8$. It is easy to see that the van Lint graph contains $K_3+C_5$ (it is the subgraph induced by the black vertices in Figure \ref{figure: vanLint}).

\begin{figure}
	\centering
	\includegraphics[height=160px,width=160px]{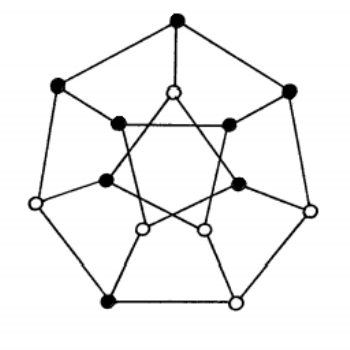}
	\caption{The van Lint graph from \cite{GS71}}\label{figure: vanLint}
\end{figure}

Nenov \cite{Nen81a} constructed a 15-vertex graph in $\mH_e(3, 3; 5)$ in 1981, thus proving $F_e(3, 3; 5) \leq 15$. This graph is obtained from the graph $\Gamma$ shown in Figure \ref{figure: Nenov_14} by adding a new vertex which is adjacent to all vertices of $\Gamma$. In 1999 Piwakowski, Radziszowski, and Urbanski \cite{PRU99} completed the computation of the number $F_e(3, 3; 5)$ by proving with the help of a computer that $F_e(3, 3; 5) \geq 15$. In \cite{PRU99} they also obtained all graphs in $\mH_e(3, 3; 5; 15)$.

\vspace{1em}

\begin{figure}[h]
	\centering
	\includegraphics[height=210px,width=210px]{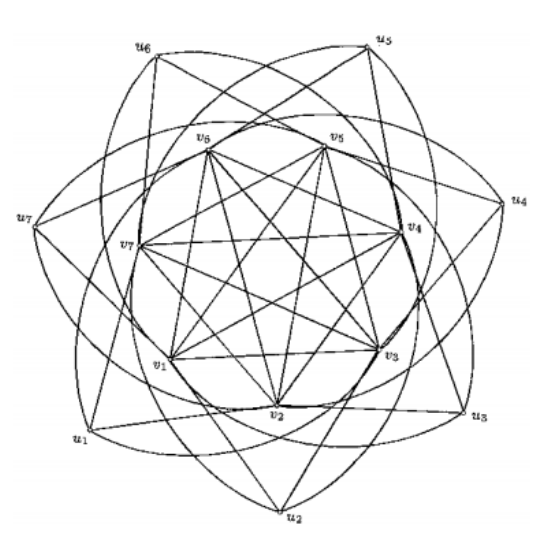}
	\caption{The Nenov graph $\Gamma$ from \cite{Nen81a}}
	\label{figure: Nenov_14}
\end{figure}

Folkman constructed a graph $G \arrowse (3, 3)$ with $\omega(G) = 3$ \cite{Fol70}. The exact value of the number $F_e(3, 3; 4)$ is not known. It is known that $19 \leq F_e(3, 3; 4) \leq 786$, \cite{RX07} \cite{LRX14}. In Chapter 9 we improve the lower bound on this number by proving $F_e(3, 3; 4) \geq 20$.

A more detailed view on results related to the numbers $F_e(3, 3; q)$ is given in the book \cite{Soi08}, and also in the papers \cite{RX07}, \cite{Gra12}, \cite{LRX14} and \cite{RX16}.\\

Let $a_1, ..., a_s$ be positive integers. The expression $G \arrowse (a_1, ..., a_s)$ means that in every coloring of $\E(G)$ in $s$ colors ($s$-coloring) there exists $i \in \set{1, ..., s}$ such that there is a monochromatic $a_i$-clique of color $i$.

Define:

$\mH_e(a_1, ..., a_s) = \set{ G : G \arrowse (a_1, ..., a_s) }.$

$\mH_e(a_1, ..., a_s; q) = \set{ G : G \arrowse (a_1, ..., a_s) \mbox{ and } \omega(G) < q }.$

$\mH_e(a_1, ..., a_s; q; n) = \set{ G : G \in \mH_e(a_1, ..., a_s; q) \mbox{ and } \abs{\V(G)} = n }.$

The edge Folkman numbers $F_e(a_1, ..., a_s; q)$ are defined by the equality:
\begin{equation*}
F_e(a_1, ..., a_s; q) = \min\set{\abs{\V(G)} : G \in \mH_e(a_1, ..., a_s; q)}.
\end{equation*}

In general, very little is known about the numbers $F_e(a_1, ..., a_s; q)$.

In Chapter 8 we study the minimal (inclusion-wise) graphs in $\mH_e(3, 3)$. In Chapter 9 we obtain the new lower bound $F_e(3, 3; 4) \geq 20$.

\chapter{Minimal graphs in $\mH_e(3, 3)$}

Obviously, if $H \arrowse (p, q)$, then its every supergraph $G \arrowse (p, q)$.
\begin{definition}
	\label{definition: minimal graph in mH_e(p, q)}
	We say that $G$ is a minimal graph in $\mH_e(p, q)$ if $G \arrowse (p, q)$ and $H \not\arrowse (p, q)$ for each proper subgraph $H$ of $G$.
\end{definition}

It is easy to see that $K_6$ is a minimal graph in $\mH_e(3, 3)$ and there are no minimal graphs in $\mH_e(3, 3)$ with 7 vertices. The only minimal 8-vertex graph is the Graham graph $K_3+C_5$, and there is only one minimal 9-vertex graph, obtained by Nenov \cite{Nen79} (see Figure \ref{figure: Nenov_9}).

\vspace{1em}

\begin{figure}[h]
	\begin{minipage}{.45\textwidth}
		\centering
		\includegraphics[trim={0 470 0 0},clip,height=153px,width=153px]{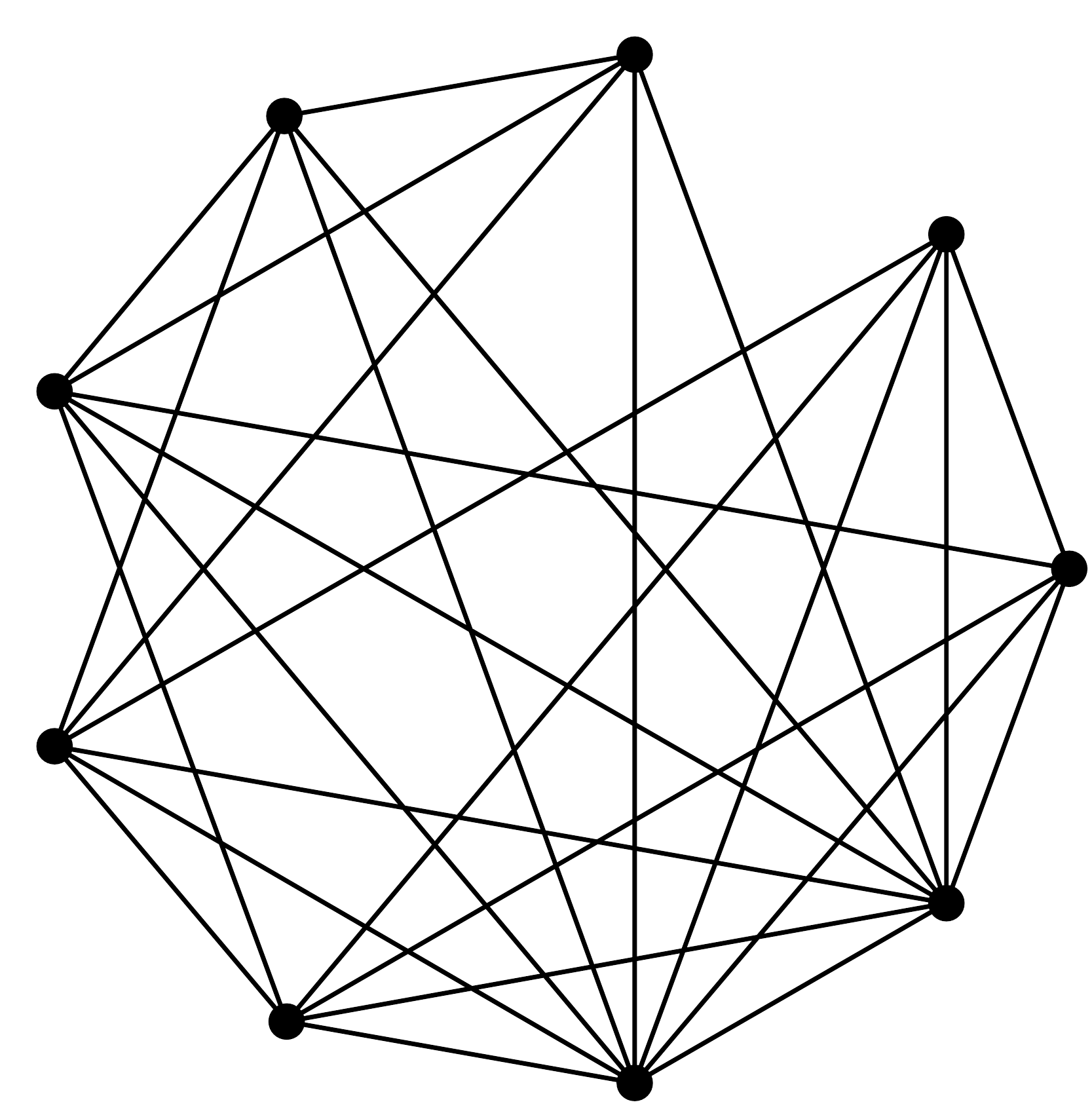}
		\vspace{1.5em}
		\caption{9-vertex minimal\hfill\break graph in $\mH_e(3, 3)$}
		\label{figure: Nenov_9}
	\end{minipage}\hfill
	\begin{minipage}{.45\textwidth}
		\centering
		\includegraphics[trim={0 470 0 0},clip,height=170px,width=170px]{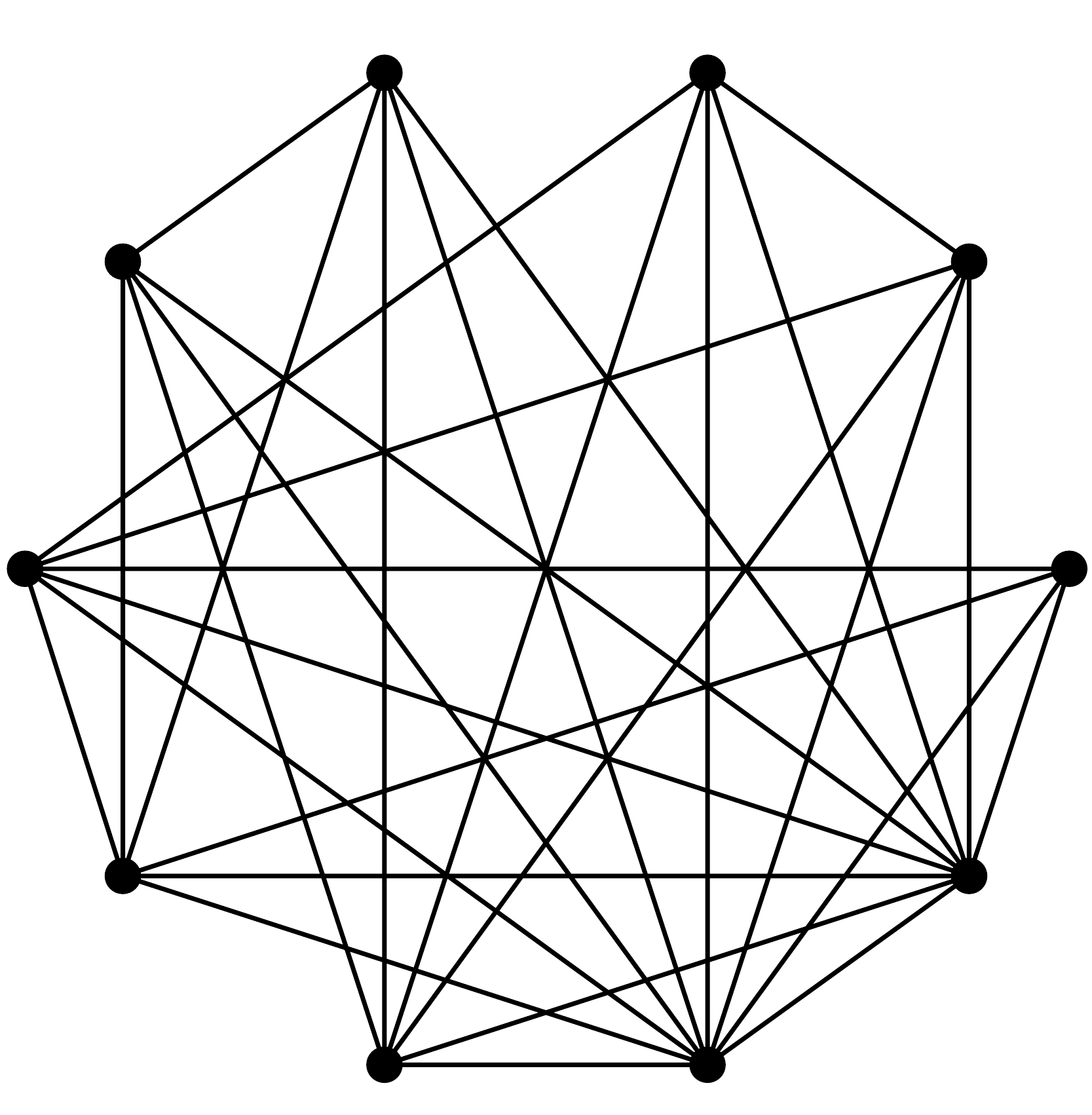}
		\vspace{0.5em}
		\caption{10-vertex minimal\hfill\break graph in $\mH_e(3, 3)$}
		\label{figure: Nenov_10}
	\end{minipage}
\end{figure}

For each pair of positive integers $p \geq 3$, $q \geq 3$ there exist infinitely many minimal graphs in $\mH_e(p, q)$ \cite{BEL76}, \cite{FL06}. The simplest infinite sequence of minimal graphs in $\mH_e(3, 3)$ are the graphs $K_3+C_{2r+1}, r \geq 1$. This sequence contains the already mentioned graphs $K_6$ and $K_3+C_5$. It was obtained by Nenov and Khadzhiivanov in \cite{NK79}. Later, this sequence was reobtained in \cite{BR80}, \cite{GSS95}, and \cite{Sza77}.

Three 10-vertex minimal graphs in $\mH_e(3, 3)$ are known. One of them is $K_3+C_7$ from the sequence $K_3+C_{2r+1}, r \geq 1$. The other two were obtained by Nenov in \cite{Nen80c} (the second graph is given in Figure \ref{figure: Nenov_10} and the third is a subgraph of $K_1+\overline{C_9}$).

In the following sections, we obtain new minimal graphs in $\mH_e(3, 3)$. We also obtain some general bounds for the graphs in $\mH_e(3, 3)$.

We will need the following results:

\begin{theorem}
	\label{theorem: delta(G) geq (p-1)^2}
	\cite{BEL76}\cite{FL06}
	Let $G$ be a minimal graph in $\mH_e(p, p)$. Then $\delta(G) \geq (p-1)^2$. In particular, when $p = 3$, we have $\delta(G) \geq 4$.
\end{theorem}

\begin{proposition}
	\label{proposition: minimal graphs in mH_e(p, q) are not Sperner}
	If $G$ is a minimal graph in $\mH_e(p, q)$, then $G$ is not a Sperner graph.
\end{proposition}

\begin{proof}
	Suppose the opposite is true, and let $u, v \in \V(G)$ be such that $N_G(u) \subseteq N_G(v)$. We color the edges of $G - u$ with two colors in such a way that there is no monochromatic $p$-clique of the first color and no monochromatic $q$-clique of the second color. After that, for each vertex $w \in N_G(u)$ we color the edge $[u, w]$ with the same color as the edge $[v, w]$. We obtain a 2-coloring of the edges of $G$ with no monochromatic $p$-cliques of the first color and no monochromatic $q$-cliques of the second color.
\end{proof}

Since $F_e(3, 3; 5) = 15$ \cite{Nen81a}\cite{PRU99}, every graph $G \in \mH_e(3, 3)$ with no more than 14 vertices contains a 5-clique. There exist 14-vertex graphs in $\mH_e(3, 3)$ containing only a single 5-clique, an example of such a graph is given in Figure \ref{figure: 14_1xk5}. The graph in Figure \ref{figure: 14_1xk5} is obtained with the help of the only 15-vertex bicritical graph in $\mH_e(3, 3)$ with clique number 4 from \cite{PRU99}. First, by removing a vertex from the bicritical graph, we obtain 14-vertex graphs without 5 cliques. After that, by adding edges to the obtained graphs, we find a 14-vertex graph in $\mH_e(3, 3)$ with a single 5-clique whose subgraph is the minimal graph in $\mH_e(3, 3)$ in Figure \ref{figure: 14_1xk5}. Let us note that in \cite{PRU99} the authors obtain all 15-vertex graphs in $\mH_e(3, 3)$ with clique number 4, and with the help of these graphs, one can find more examples of 14-vertex minimal graphs in $\mH_e(3, 3)$.

\begin{figure}[h]
	\centering
	\includegraphics[trim={0 470 0 0},clip,height=160px,width=160px]{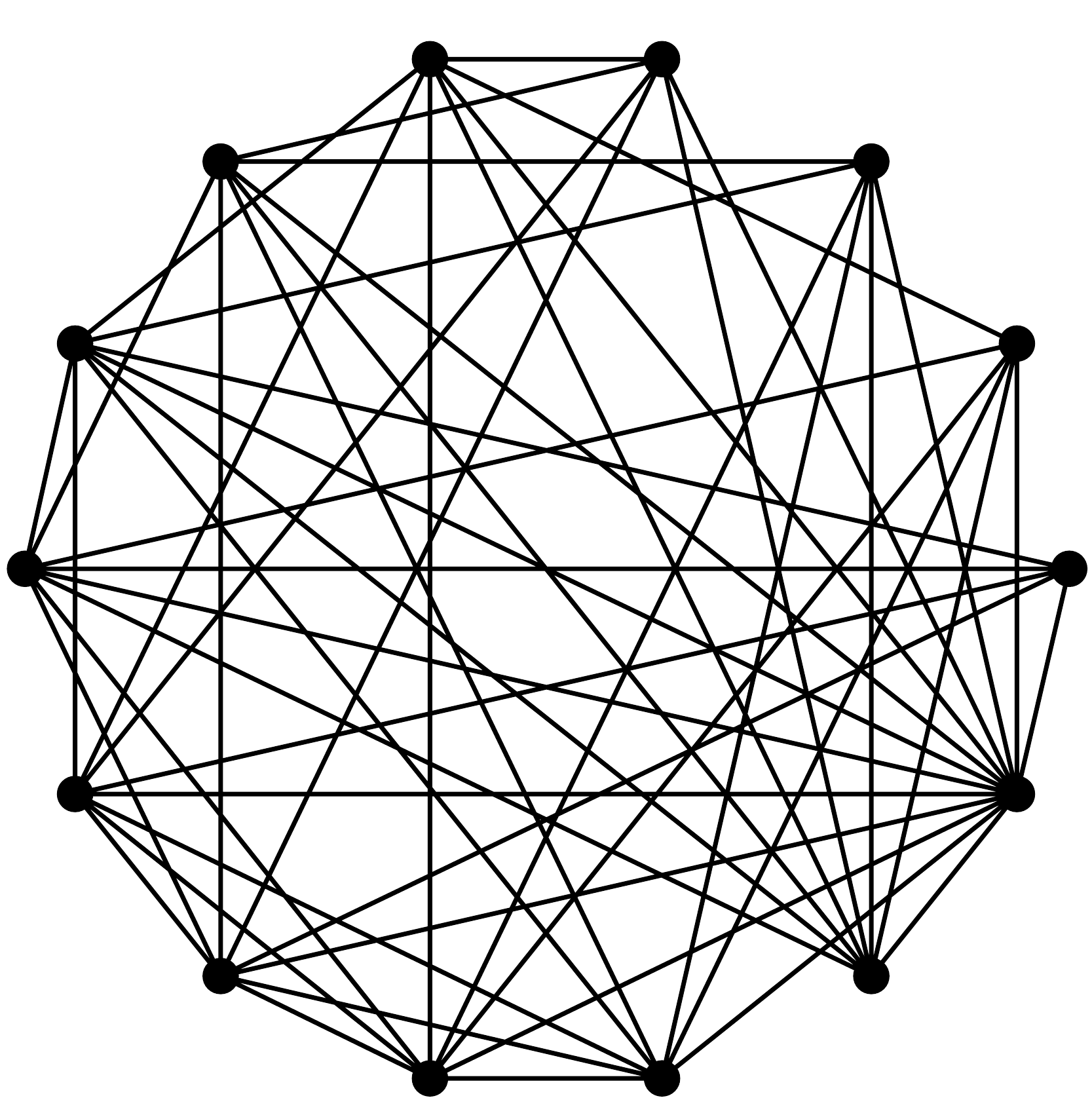}
	\caption{14-vertex minimal graph in $\mH_e(3, 3)$ with a single 5 clique}
	\label{figure: 14_1xk5}
\end{figure}

\begin{theorem}
	\label{theorem: chi(G) geq R(p, q)}
	\cite{Lin72}
	Let $G \arrowse (p, q)$. Then $\chi(G) \geq R(p, q)$. In particular, if $G \arrowse (3, 3)$, then $\chi(G) \geq 6$.
\end{theorem}

\begin{corollary}
	\label{corollary: chi(H) geq 5}
	Let $G \arrowse (3, 3)$, $v_1, ..., v_s$ be independent vertices of $G$, and $H = G - \set{v_1, ..., v_s}$. Then $\chi(H) \geq 5$.
\end{corollary}

\begin{theorem}
	\label{theorem: alpha(G(v)) leq d(v) - 3}
	Let $G$ be a minimal graph in $\mH_e(3, 3)$. Then for each vertex $v \in \V(G)$ we have $\alpha(G(v)) \leq d(v) - 3$.
\end{theorem}

\begin{proof}
	Suppose the opposite is true, and let $A \subseteq N_G(v)$ be an independent set in $G(v)$ such that $|A| = d(v) - 2$. Let $a, b \in N_G(v) \setminus A$. Consider a 2-coloring of the edges of $G - v$ in which there are no monochromatic triangles. We color the edges $[v, a]$ and $[v, b]$ with the same color in such a way that there is no monochromatic triangle (if $a$ and $b$ are adjacent, we chose the color of $[v, a]$ and $[v, b]$ to be different from the color of $[a, b]$, and if $a$ and $b$ are not adjacent, then we chose an arbitrary color for $[v, a]$ and $[v, b]$). We color the remaining edges incident to $v$ with the other color, which is different from the color of $[v, a]$ and $[v, b]$. Since $N_G(v) \setminus \{a, b\} = A$ is and independent set, we obtain a 2-coloring of the edges of $G$ without monochromatic triangles, which is a contradiction.
\end{proof}

\begin{corollary}
	\label{corollary: if d(v) = 4, then G(v) = K_4} 
	Let $G$ be a minimal graph in $\mH_e(3, 3)$ and $d(v) = 4$ for some vertex $v \in \V(G)$. Then $G(v) = K_4$.
\end{corollary}

\vspace{1em}
Theorem \ref{theorem: alpha(G(v)) leq d(v) - 3} is published in \cite{Bik16}.

\section{Algorithm A6}

The following algorithm is appropriate for finding all minimal graphs in $\mH_e(3, 3)$ with a small number of vertices. 

\begin{namedalgorithm}{A6}
	\label{algorithm: A6}
	Let $n$ be a fixed positive integer such that $7 \leq n \leq 14$.
	
	The output of the algorithm is the set $\mB$ of all $n$-vertex minimal graphs in $\mH_e(3, 3)$, 
	
	1. Generate all $n$-vertex non-isomorphic graphs with minimum degree at least 4, and denote the obtained set by $\mB$.
	
	2. Remove from $\mB$ all Sperner graphs.
	
	3. Remove from $\mB$ all graphs with clique number not equal to 5.
	
	4. Remove from $\mB$ all graphs with chromatic number less than 6.
	
	5. Remove from $\mB$ all graphs which are not in $\mH_e(3, 3)$.
	
	6. Remove from $\mB$ all graphs which are not minimal graphs in $\mH_e(3, 3)$.
\end{namedalgorithm}

\begin{theorem}
	\label{theorem: algorithm A6}
	Fix $n \in \set{7, ..., 14}$. Then after executing Algorithm \ref{algorithm: A6}, $\mB$ coincides with the set of all $n$-vertex minimal graphs in $\mH_e(3, 3)$.
\end{theorem}

\begin{proof}
	From step 1 it becomes clear, that the set $\mB$ contains only $n$-vertex graphs. Step 6 guaranties that $\mB$ contains only minimal graphs in $\mH_e(3, 3)$.
	
	Let $G$ be an arbitrary $n$-vertex minimal graph in $\mH_e(3, 3)$. We will prove that, after the execution of Algorithm \ref{algorithm: A6}, $G \in \mB$. By Theorem \ref{theorem: delta(G) geq (p-1)^2}, $\delta(G) \geq 4$, and by Proposition \ref{proposition: minimal graphs in mH_e(p, q) are not Sperner}, $G$ is not a Sperner graph. From $\abs{V(G)} \geq 7$ it follows that $G \not\supseteq K_6$. Since $\abs{V(G)} \leq 14$, by $F_e(3, 3; 5) = 15$, \cite{Nen81a}\cite{PRU99}, we obtain $\omega(G) = 5$. By Theorem \ref{theorem: chi(G) geq R(p, q)}, $\chi(G) \geq 6$. Therefore, after step 4, $G \in \mB$.
\end{proof}

We will use Algorithm \ref{algorithm: A6} to obtain all minimal graphs in $\mH_e(3, 3)$ with up to 12 vertices. Algorithm \ref{algorithm: A6} is not appropriate in the cases $n \geq 13$, because the number of graphs generated in step 1 is too large. To find the 13-vertex minimal graphs in $\mH_e(3, 3)$, we will use Algorithm \ref{algorithm: A7-M}, which is defined below.

\vspace{1em}
Theorem \ref{theorem: algorithm A6} and Algorithm \ref{algorithm: A6} are published in \cite{Bik16}.

\section{Minimal graphs in $\mH_e(3, 3; n)$ for $n \leq 12$}

\begin{table}
	\centering
	\vspace{-2em}
	\resizebox{0.85\textwidth}{!}{
		\begin{tabular}{ | l r | l r | l r | l r | l r | l r | }
			\hline
			\multicolumn{2}{|c|}{\parbox{5em}{$\abs{\E(G)}$ \hfill $\#$}}&
			\multicolumn{2}{|c|}{\parbox{5em}{$\delta(G)$ \hfill $\#$}}&
			\multicolumn{2}{|c|}{\parbox{5em}{$\Delta(G)$ \hfill $\#$}}&
			\multicolumn{2}{|c|}{\parbox{5em}{$\alpha(G)$ \hfill $\#$}}&
			\multicolumn{2}{|c|}{\parbox{5em}{$\chi(G)$ \hfill $\#$}}&
			\multicolumn{2}{|c|}{\parbox{5em}{$\abs{Aut(G)}$ \hfill $\#$}}\\
			\hline
			30	& 1				& 4	& 1				& 9	& 6				& 2	& 3				& 6	& 6				& 4		& 2		\\
			31	& 1				& 5	& 4				& 	& 					& 3	& 3				& 	& 					& 8		& 2		\\
			32	& 2				& 6	& 1				& 	& 					& 	& 					& 	& 					& 16	& 1		\\
			33	& 1				& 	& 					& 	& 					& 	& 					& 	& 					& 84	& 1		\\
			34	& 1				& 	& 					& 	& 					& 	& 					& 	& 					& 		&			\\
			\hline
		\end{tabular}
	}
	\vspace{-0.5em}
	\caption{Some properties of the 10-vertex minimal graphs in $\mH_e(3, 3)$}
	\label{table: 10-vertex graphs properties}
	\vspace{1em}
	\resizebox{0.85\textwidth}{!}{
		\begin{tabular}{ | l r | l r | l r | l r | l r | l r | }
			\hline
			\multicolumn{2}{|c|}{\parbox{5em}{$\abs{\E(G)}$ \hfill $\#$}}&
			\multicolumn{2}{|c|}{\parbox{5em}{$\delta(G)$ \hfill $\#$}}&
			\multicolumn{2}{|c|}{\parbox{5em}{$\Delta(G)$ \hfill $\#$}}&
			\multicolumn{2}{|c|}{\parbox{5em}{$\alpha(G)$ \hfill $\#$}}&
			\multicolumn{2}{|c|}{\parbox{5em}{$\chi(G)$ \hfill $\#$}}&
			\multicolumn{2}{|c|}{\parbox{5em}{$\abs{Aut(G)}$ \hfill $\#$}}\\
			\hline
			35	& 6				& 4	& 5				& 8	& 1				& 2	& 4				& 6	& 73				& 1		& 20		\\
			36	& 13				& 5	& 58				& 10& 72				& 3	& 66				& 	& 					& 2		& 29		\\
			37	& 23				& 6	& 10				& 	& 					& 4	& 3				& 	& 					& 4		& 14		\\
			38	& 25				& 	& 					& 	& 					& 	& 					& 	& 					& 6		& 1		\\
			39	& 5				& 	& 					& 	& 					& 	& 					& 	& 					& 8		& 4		\\
			41	& 1				& 	& 					& 	& 					& 	& 					& 	& 					& 12	& 1		\\
			& 					& 	& 					& 	& 					& 	& 					& 	& 					& 16	& 3		\\
			& 					& 	& 					& 	& 					& 	& 					& 	& 					& 24	& 1		\\
			\hline
		\end{tabular}
	}
	\vspace{-0.5em}
	\caption{Some properties of the 11-vertex minimal graphs in $\mH_e(3, 3)$}
	\label{table: 11-vertex graphs properties}
	\vspace{1em}
	\resizebox{0.85\textwidth}{!}{
		\begin{tabular}{ | l r | l r | l r | l r | l r | l r | }
			\hline
			\multicolumn{2}{|c|}{\parbox{5em}{$\abs{\E(G)}$ \hfill $\#$}}&
			\multicolumn{2}{|c|}{\parbox{5em}{$\delta(G)$ \hfill $\#$}}&
			\multicolumn{2}{|c|}{\parbox{5em}{$\Delta(G)$ \hfill $\#$}}&
			\multicolumn{2}{|c|}{\parbox{5em}{$\alpha(G)$ \hfill $\#$}}&
			\multicolumn{2}{|c|}{\parbox{5em}{$\chi(G)$ \hfill $\#$}}&
			\multicolumn{2}{|c|}{\parbox{5em}{$\abs{Aut(G)}$ \hfill $\#$}}\\
			\hline
			38	& 5				& 4	& 129				& 8	& 43				& 2	& 124				& 6	& 3 041			& 1		& 1 792	\\
			39	& 27				& 5	& 2 178			& 9	& 1 196			& 3	& 2 431			& 	& 					& 2		& 851		\\
			40	& 144				& 6	& 611				& 11& 1 802			& 4	& 485				& 	& 					& 4		& 286		\\
			41	& 418				& 7	& 123				& 	& 					& 5	& 1				& 	& 					& 6		& 1		\\
			42	& 1 014			& 	& 					& 	& 					& 	& 					& 	& 					& 8		& 67		\\
			43	& 459				& 	& 					& 	& 					& 	& 					& 	& 					& 12	& 16		\\
			44	& 224				& 	& 					& 	& 					& 	& 					& 	& 					& 16	& 18		\\
			45	& 351				& 	& 					& 	& 					& 	& 					& 	& 					& 24	& 6		\\
			46	& 299				& 	& 					& 	& 					& 	& 					& 	& 					& 32	& 1		\\
			47	& 84				& 	& 					& 	& 					& 	& 					& 	& 					& 36	& 1		\\
			48	& 16				& 	& 					& 	& 					& 	& 					& 	& 					& 96	& 1		\\
			& 					& 	& 					& 	& 					& 	& 					& 	& 					& 108	& 1		\\
			\hline
		\end{tabular}
	}
	\vspace{-0.5em}
	\caption{Some properties of the 12-vertex minimal graphs in $\mH_e(3, 3)$}
	\label{table: 12-vertex graphs properties}
	\vspace{1em}
	\resizebox{0.85\textwidth}{!}{
		\begin{tabular}{ | l r | l r | l r | l r | l r | l r | }
			\hline
			\multicolumn{2}{|c|}{\parbox{5em}{$\abs{\E(G)}$ \hfill $\#$}}&
			\multicolumn{2}{|c|}{\parbox{5em}{$\delta(G)$ \hfill $\#$}}&
			\multicolumn{2}{|c|}{\parbox{5em}{$\Delta(G)$ \hfill $\#$}}&
			\multicolumn{2}{|c|}{\parbox{5em}{$\alpha(G)$ \hfill $\#$}}&
			\multicolumn{2}{|c|}{\parbox{5em}{$\chi(G)$ \hfill $\#$}}&
			\multicolumn{2}{|c|}{\parbox{5em}{$\abs{Aut(G)}$ \hfill $\#$}}\\
			\hline
			41	& 4				& 4	& 13 725			& 8	& 16				& 2	& 13				& 6	& 306 622			& 1		& 251 976	\\
			42	& 44				& 5	& 191 504			& 9	& 61 678			& 3	& 218 802			& 7	& 13				& 2		& 46 487	\\
			43	& 220				& 6	& 85 932			& 10& 175 108			& 4	& 86 721			& 	& 					& 3		& 10		\\
			44	& 1 475			& 7	& 15 391			& 12& 69 833			& 5	& 1 097			& 	& 					& 4		& 6 851	\\
			45	& 7 838			& 8	& 83				& 	& 					& 6	& 2				& 	& 					& 6		& 83		\\
			46	& 28 805			& 	& 					& 	& 					& 	& 					& 	& 					& 8		& 916		\\
			47	& 33 810			& 	& 					& 	& 					& 	& 					& 	& 					& 12	& 129		\\
			48	& 26 262			& 	& 					& 	& 					& 	& 					& 	& 					& 16	& 106		\\
			49	& 39 718			& 	& 					& 	& 					& 	& 					& 	& 					& 24	& 44		\\
			50	& 62 390			& 	& 					& 	& 					& 	& 					& 	& 					& 32	& 12		\\
			51	& 59 291			& 	& 					& 	& 					& 	& 					& 	& 					& 36	& 3		\\
			52	& 34 132			& 	& 					& 	& 					& 	& 					& 	& 					& 40	& 1		\\
			53	& 10 878			& 	& 					& 	& 					& 	& 					& 	& 					& 48	& 11		\\
			54	& 1 680			& 	& 					& 	& 					& 	& 					& 	& 					& 72	& 3		\\
			55	& 86				& 	& 					& 	& 					& 	& 					& 	& 					& 96	& 2		\\
			56	& 2				& 	& 					& 	& 					&	& 					& 	& 					& 144	& 1		\\
			\hline
		\end{tabular}
	}
	\vspace{-0.5em}
	\caption{Some properties of the 13-vertex minimal graphs in $\mH_e(3, 3)$}
	\label{table: 13-vertex graphs properties}
\end{table}

We execute Algorithm \ref{algorithm: A6} for $n = 7, 8, 9, 10, 11, 12$, and we find all minimal graphs in $\mH_e(3, 3)$ with up to 12 vertices except $K_6$. In this way, we obtain the known results: there is no minimal graph in $\mH_e(3, 3)$ with 7 vertices, the Graham graph $K_3+C_5$ is the only minimal 8-vertex graph, and there exists only one minimal 9-vertex graph, the Nenov graph from \cite{Nen79} (see Figure \ref{figure: Nenov_9}). We also obtain the following new results:

\begin{theorem}
	\label{theorem: 10-vertex minimal graphs in mH_e(3, 3)}
	There are exactly 6 minimal 10-vertex graphs in $\mH_e(3, 3)$. These graphs are given in Figure \ref{figure: 10}, and some of their properties are listed in Table \ref{table: 10-vertex graphs properties}.
\end{theorem}

\begin{theorem}
	\label{theorem: 11-vertex minimal graphs in mH_e(3, 3)}
	There are exactly 73 minimal 11-vertex graphs in $\mH_e(3, 3)$. Some of their properties are listed in Table \ref{table: 11-vertex graphs properties}. Examples of 11-vertex minimal graphs in $\mH_e(3, 3)$ are given in Figure \ref{figure: 11_a4} and Figure \ref{figure: 11_a2}.
\end{theorem}

\begin{theorem}
	\label{theorem: 12-vertex minimal graphs in mH_e(3, 3)}
	There are exactly 3041 minimal 12-vertex graphs in $\mH_e(3, 3)$. Some of their properties are listed in Table \ref{table: 12-vertex graphs properties}. Examples of 12-vertex minimal graphs in $\mH_e(3, 3)$ are given in Figure \ref{figure: 12_a5} and Figure \ref{figure: 12_aut}.
\end{theorem}

We will use the following enumeration for the obtained minimal graphs in $\mH_e(3, 3)$:\\
- $G_{10.1}$, ..., $G_{10.6}$ are the 10-vertex graphs;\\
- $G_{11.1}$, ..., $G_{11.73}$ are the 11-vertex graphs;\\
- $G_{12.1}$, ..., $G_{12.3041}$ are the 12-vertex graphs;

The indexes correspond to the order of the graphs' canonical labels defined in \emph{nauty} \cite{MP13}.

Detailed data for the number of graphs obtained at each step of the execution of Algorithm \ref{algorithm: A6} is given in Table \ref{table: steps in finding all minimal graphs in mH_e(3, 3) with up to 12 vertices}.

\begin{table}
	\centering
	\resizebox{\textwidth}{!}{
		\begin{tabular}{ | l | r | r | r | r | r | }
			\hline
			{\parbox{6em}{Step of\\ Algorithm \ref{algorithm: A6}}}&
			{\parbox{3em}{\hfill$n = 8$}}&
			{\parbox{4em}{\hfill$n = 9$}}&
			{\parbox{5em}{\hfill$n = 10$}}&
			{\parbox{6em}{\hfill$n = 11$}}&
			{\parbox{7em}{\hfill$n = 12$}}\\
			\hline
			1				&  424			& 15 471		& 1 249 973		& 187 095 840	& 48 211 096 031\\
			2				&  59			& 2 365			& 206 288		& 33 128 053	& 9 148 907 379	\\
			3				&  9			& 380			& 41 296		& 8 093 890		& 2 763 460 021	\\
			4				&  1			& 7				& 356			& 78 738		& 44 904 195	\\
			5				&  1			& 3				& 126			& 23 429		& 11 670 079	\\
			6				&  1			& 1				& 6				& 73			& 3041			\\
			\hline
		\end{tabular}
	}
	\caption{\small Steps in finding all minimal graphs in $\mH_e(3, 3)$ with up to 12 vertices}
	\label{table: steps in finding all minimal graphs in mH_e(3, 3) with up to 12 vertices}
\end{table}

\vspace{1em}
Theorem \ref{theorem: 10-vertex minimal graphs in mH_e(3, 3)}, Theorem \ref{theorem: 11-vertex minimal graphs in mH_e(3, 3)}, and Theorem \ref{theorem: 12-vertex minimal graphs in mH_e(3, 3)} are published in \cite{Bik16}.

\section{Algorithm A7}

In order to present the next algorithms, we will need the following definitions and auxiliary propositions:

We say that a 2-coloring of the edges of a graph is $(3, 3)$-free if it has no monochromatic triangles.

\begin{definition}
	\label{definition: marked vertex set}
	Let $G$ be a graph and $M \subseteq \V(G)$. Let $G_1$ be a graph which is obtained by adding a new vertex $v$ to $G$ such that $N_{G_1}(v) = M$. We say that $M$ is a marked vertex set in $G$ if there exists a $(3, 3)$-free $2$-coloring of the edges of $G$ which cannot be extended to a $(3, 3)$-free $2$-coloring of the edges of $G_1$.
\end{definition}

It is clear that if $G \arrowse (3, 3)$, then there are no marked vertex sets in $G$. The following proposition is true:

\begin{proposition}
	\label{proposition: marked vertex sets in graphs in mH_e(3, 3)}
	Let $G$ be a minimal graph in $\mH_e(3, 3)$, let $v_1, ..., v_s$ be independent vertices of $G$, and $H = G - \set{v_1, ..., v_s}$. Then $N_G(v_i), i = 1, ..., s$, are marked vertex sets in $H$.
\end{proposition}

\begin{proof}
	Suppose the opposite is true, i.e. $N_G(v_i)$ is not a marked vertex set in $H$ for some $i \in \set{1, ..., s}$. Since $G$ is a minimal graph in $\mH_e(3, 3)$, there exists a $(3, 3)$-free $2$-coloring of the edges of $G - v_i$, which induces a $(3, 3)$-free $2$-coloring of the edges of $H$. By supposition, we can extend this $2$-coloring to a $(3, 3)$-free $2$-coloring of the edges of the graph $H_i = G - \set{v_1, ..., v_{i - 1}, v_{i + 1}, ..., v_s}$. Thus, we obtain a $(3, 3)$-free $2$-coloring of the edges of $G$, which is a contradiction.
\end{proof}

\begin{definition}
	\label{definition: complete family of marked vertex sets}
	Let $\set{M_1, ..., M_s}$ be a family of marked vertex sets in the graph $G$. Let $G_i$ be a graph which is obtained by adding a new vertex $v_i$ to $G$ such that $N_{G_i}(v_i) = M_i, i = 1, ..., s$. We say that $\set{M_1, ..., M_s}$ is a complete family of marked vertex sets in $G$, if for each $(3, 3)$-free $2$-coloring of the edges of $G$ there exists $i \in \set{1, ..., s}$ such that this $2$-coloring can not be extended to a $(3, 3)$-free $2$-coloring of the edges of $G_i$.
\end{definition}

\begin{proposition}
	\label{proposition: if set(N_G(v_1), ..., N_G(v_s)) is a complete family of marked vertex sets in H, then G rightarrow (3, 3)}
	Let $v_1, ..., v_s$ be independent vertices of the graph $G$, and $H = G - \set{v_1, ..., v_s}$. If $\set{N_G(v_1), ..., N_G(v_s)}$ is a complete family of marked vertex sets in $H$, then $G \arrowse (3, 3)$.
\end{proposition}

\begin{proof}
	Consider a 2-coloring of the edges of $G$ which induces a 2-coloring with no monochromatic triangles in $H$. According to Definition \ref{definition: complete family of marked vertex sets}, this 2-coloring of the edges of $H$ can not be extended in $G$ without forming a monochromatic triangle.
\end{proof}

It is easy to prove the following strengthening of Proposition \ref{proposition: marked vertex sets in graphs in mH_e(3, 3)}:
\begin{proposition}
	\label{proposition: if G is a minimal graph in mH_e(3, 3), then set(N_G(v_1), ..., N_G(v_s)) is a complete family of marked vertex sets}
	Let $G$ be a minimal graph in $\mH_e(3, 3)$, let $v_1, ..., v_s$ be independent vertices of $G$, and $H = G - \set{v_1, ..., v_s}$. Then $\set{N_G(v_1), ..., N_G(v_s)}$ is a complete family of marked vertex sets in $H$. Furthermore, this family is a minimal complete family, in the sense that it does not contain a proper complete subfamily.
\end{proposition}

Before presenting Algorithm \ref{algorithm: A7}, we will prove
\begin{proposition}
\label{proposition: alpha(G) geq abs(V(G)) - k geq 1}
For every fixed positive integer $k$ there exist at most a finite number of minimal graphs $G \in \mH_e(3, 3)$ such that $\alpha(G) \geq \abs{\V(G)} - k \geq 1$.
\end{proposition}
\begin{proof}
Let $G$ be a minimal graph in $\mH_e(3, 3)$ such that $\alpha(G) \geq \abs{\V(G)} - k \geq 1$. Let $s = \abs{\V(G)} - k$, $v_1, ..., v_s$ be independent vertices of $G$, and $H = G - \set{v_1, ..., v_s}$. According to Proposition \ref{proposition: minimal graphs in mH_e(p, q) are not Sperner}, $N_G(v_i) \not\subseteq N_G(v_j), \ i \neq j$. Thus, the vertices $v_1, ..., v_s$ are connected to different subsets of $\V(H)$. We derive that each $k$-vertex graph $H$ can be extended to at most finitely many possible graphs $G$.
\end{proof}

\vspace{1em}

The following algorithm finds all minimal graphs $G \in \mH_e(3, 3)$ for which $\alpha(G) \geq \abs{\V(G)} - k \geq 1$, where $k$ is fixed (but $|\V(G)|$ is not fixed).

\begin{namedalgorithm}{A7}
	\label{algorithm: A7}
	Let $q$ and $k$ be fixed positive integers.
	
	The input of the algorithm is the set $\mA$ of all $k$-vertex graphs $H$ for which $\omega(H) < q$ and $\chi(H) \geq 5$. The output of the algorithm is the set $\mB$ of all minimal graphs $G \in \mH_e(3, 3; q)$ for which $\alpha(G) \geq \abs{\V(G)} - k \geq 1$.
	
	1. For each graph $H \in \mA$:
	
	1.1. Find all subsets $M$ of $\V(H)$ which have the properties:
	
	(a) $K_{q - 1} \not\subseteq M$, i.e. $M$ is a $K_{(q-1)}$-free subset.
	
	(b) $M \not\subseteq N_H(v), \forall v \in \V(H)$.
	
	(c) $M$ is a marked vertex set in $H$(see Definition \ref{definition: marked vertex set}).
	
	Denote by $\mM(H)$ the family of subsets of $\V(H)$ which have the properties (a), (b) and (c). Enumerate the elements of $\mM(H)$: $\mM(H) = \set{M_1, ..., M_l}$.
	
	1.2. Find all subfamilies $\set{M_{i_1}, ..., M_{i_s}}$ of $\mM(H)$ which are minimal complete families of marked vertex sets in $H$ (see Definition \ref{definition: complete family of marked vertex sets}). For each such found subfamily $\set{M_{i_1}, ..., M_{i_s}}$ construct the graph $G = G(M_{i_1}, ..., M_{i_s})$ by adding new independent vertices $v_1, ..., v_s$ to $\V(H)$ such that $N_G(v_j) = M_{i_j}, j = 1, ..., s$. Add $G$ to $\mB$.
	
	2. Remove the isomorphic copies of graphs from $\mB$.
	
	3. Remove from $\mB$ all graphs which are not minimal graphs in $\mH_e(3, 3)$.
\end{namedalgorithm}

\vspace{1em}

\begin{remark}
	\label{remark: q in set(4, 5, 6)}
	It is clear, that if $G$ is a minimal graph in $\mH_e(3, 3)$ and $\omega(G) \geq 6$, then $G = K_6$. Obviously there are no graphs in $\mH_e(3, 3)$ with clique number less than 3. Therefore, we will use Algorithm \ref{algorithm: A7} only for $q \in \set{4, 5, 6}$.
\end{remark}

\begin{theorem}
	\label{theorem: algorithm A7}
	After executing Algorithm \ref{algorithm: A7}, the set $\mB$ coincides with the set of all minimal graphs $G \in \mH_e(3, 3; q)$ for which $\alpha(G) \geq \abs{\V(G)} - k \geq 1$.
\end{theorem}

\begin{proof}
	From step 1.2 it becomes clear that every graph $G$ which is added to $\mB$ is obtained by adding the independent vertices $v_1, ..., v_s$ to a graph $H \in \mA$. Therefore, $\alpha(G) \geq s = \abs{\V(G)} - \abs{\V(H)} = \abs{\V(G)} - k$. From $\omega(H) < q$ and $K_{q - 1} \not\subseteq N_G(v_i), i = 1, ..., s$, it follows $\omega(G) < q$. According to Proposition \ref{proposition: if set(N_G(v_1), ..., N_G(v_s)) is a complete family of marked vertex sets in H, then G rightarrow (3, 3)}, after step 1.2, $\mB$ contains only graphs in $\mH_e(3, 3; q)$, and after step 3, $\mB$ contains only minimal graphs in $\mH_e(3, 3; q)$.
	
	Consider an arbitrary minimal graph $G \in \mH_e(3, 3; q)$ for which $\alpha(G) \geq \abs{\V(G)} - k \geq 1$. We will prove that $G \in \mB$.
	
	Denote $s = \abs{\V(G)} - k \geq 1$. Let $v_1, ..., v_s$ be independent vertices of $G$, and $H = G - \set{v_1, ..., v_s}$. By Corollary \ref{corollary: chi(H) geq 5}, $\chi(H) \geq 5$. Therefore, $H \in \mA$.
	
	From $\omega(G) < q$ it follows that $K_{q - 1} \not\subseteq N_G(v_i), i = 1, ..., s$. By Proposition \ref{proposition: minimal graphs in mH_e(p, q) are not Sperner}, $G$ is not a Sperner graph, and therefore $N_G(v_i) \not\subseteq N_H(v), \forall v \in \V(H)$. According to Proposition \ref{proposition: marked vertex sets in graphs in mH_e(3, 3)}, $N_G(v_i)$ are marked vertex sets in $H$. Therefore, after executing step 1.1, $N_G(v_i) \in \mM(H), i = 1, ..., s$.
	
	From Proposition \ref{proposition: if G is a minimal graph in mH_e(3, 3), then set(N_G(v_1), ..., N_G(v_s)) is a complete family of marked vertex sets} it becomes clear that $\set{N_G(v_1), ..., N_G(v_s)}$ is a minimal complete family of marked vertex sets in $H$. Therefore, in step 1.2 the graph $G$ is added to $\mB$.
\end{proof}

In order to find the 13-vertex minimal graphs in $\mH_e(3, 3)$ we will use the following modification of Algorithm \ref{algorithm: A7} in which $n = \abs{\V(G)}$ is fixed:

\begin{namedalgorithm}{A7-M}
	\label{algorithm: A7-M}
	Modification of Algorithm \ref{algorithm: A7} for finding all minimal graphs $G \in \mH_e(3, 3; q; n)$ for which $\alpha(G) \geq n - k \geq 1$, where $q$, $k$, and $n$ are fixed positive integers.
	
	In step 1.2 of Algorithm \ref{algorithm: A7} add the condition to consider only minimal complete subfamilies $\set{M_{i_1}, ..., M_{i_s}}$ of $\mM(H)$ in which $s = n - k$.
\end{namedalgorithm}

\vspace{1em}
Theorem \ref{theorem: algorithm A7}, Algorithm \ref{algorithm: A7}, and Algorithm \ref{algorithm: A7-M} are published in \cite{Bik16}.

\section{Minimal graphs in $\mH_e(3, 3; 13)$}

\vspace{1em}

The method that we will use to find all minimal graphs in $\mH_e(3, 3; 13)$ consists of two parts:

\vspace{0.5em}

1. All minimal graphs in $\mH_e(3, 3; 13)$ with independence number 2 are a subset of $\mR(3, 6; 13)$. All 275 086 graphs in $\mR(3, 6; 13)$ are known and are available on \cite{McK_r}. With a computer we find that exactly 13 of these graphs are minimal graphs in $\mH_e(3, 3; 13)$.

\vspace{0.5em}

2. It remains to find the 13-vertex minimal graphs in $\mH_e(3, 3)$ with independence number at least 3. We generate all 10-vertex non-isomorphic graphs using the \emph{nauty} program \cite{MP13}, and among them we find the set $\mA$ of all 1 923 103 graphs $H$ with 10 vertices for which $\omega(H) \leq 5$ and $\chi(H) \geq 5$. By executing Algorithm \ref{algorithm: A7-M}($n = 13; k = 10; q = 6$) with input $\mA$ we obtain the set $\mB$ of all 306 622 minimal 13-vertex graphs in $\mH_e(3, 3)$ with independence number at least 3.

\vspace{0.5em}

Finally, we obtain
\begin{theorem}
	\label{theorem: 13-vertex minimal graphs in mH_e(3, 3)}
	There are exactly 306 635 minimal 13-vertex graphs in $\mH_e(3, 3)$. Some of their properties are listed in Table \ref{table: 13-vertex graphs properties}. Examples of 13-vertex minimal graphs in $\mH_e(3, 3)$ are given in Figure \ref{figure: 13_regular}, Figure \ref{figure: 13_aut} and Figure \ref{figure: 13_a2}.
\end{theorem}

We enumerate the obtained minimal 13-vertex graphs in $\mH_e(3, 3)$:\\
$G_{13.1}$, ..., $G_{13.306635}$.\\

All graphs in $\mR(3, 6)$ are known, and since $R(3, 6) = 18$, these graphs have at most 17 vertices. With the help of a computer we check that there are no minimal graphs in $\mH_e(3, 3)$ with independence number 2 and more than 13 vertices. Thus, we prove

\begin{theorem}
	\label{theorem: minimal graphs in mH_e(3, 3) G for which alpha(G) = 2}
	Let $G$ be a minimal graph in $\mH_e(3, 3)$ and $\alpha(G) = 2$. Then $\abs{\V(G)} \leq 13$. There are exactly 145 minimal graphs $G$ in $\mH_e(3, 3)$ for which $\alpha(G) = 2$:
	
	- 8-vertex: 1 ($K_3+C_5$);
	
	- 9-vertex: 1 (see Figure \ref{figure: Nenov_9});
	
	- 10-vertex: 3 ($G_{10.3}$, $G_{10.5}$, $G_{10.6}$, see Figure \ref{figure: 10});
	
	- 11-vertex: 4 ($G_{11.46}$, $G_{11.47}$, $G_{11.54}$, $G_{11.69}$, see Figure \ref{figure: 11_a2});
	
	- 12-vertex: 124;
	
	- 13-vertex: 13 (see Figure \ref{figure: 13_a2});
\end{theorem}

By executing Algorithm \ref{algorithm: A7-M}($n = 10, 11, 12; k = 7, 8, 9; q = 6$), we find all minimal graphs in $\mH_e(3, 3)$ with 10, 11, and 12 vertices and independence number greater than 2. In this way, with the help of Theorem \ref{theorem: minimal graphs in mH_e(3, 3) G for which alpha(G) = 2}, we obtain a new proof of Theorem \ref{theorem: 10-vertex minimal graphs in mH_e(3, 3)}, Theorem \ref{theorem: 11-vertex minimal graphs in mH_e(3, 3)}, and Theorem \ref{theorem: 12-vertex minimal graphs in mH_e(3, 3)}.

\vspace{1em}
Theorem \ref{theorem: 13-vertex minimal graphs in mH_e(3, 3)} and Theorem \ref{theorem: minimal graphs in mH_e(3, 3) G for which alpha(G) = 2} are published in \cite{Bik16}.

\section{Properties of the minimal graphs in $\mH_e(3, 3; n)$ for $n \leq 13$}

\subsection*{Minimum and maximum degree}

By Theorem \ref{theorem: delta(G) geq (p-1)^2}, if $G$ is a minimal graph in $\mH_e(3, 3)$, then $\delta(G) \geq 4$. Via very elegant constructions, in \cite{BEL76} and \cite{FL06} it is proved that the bound $\delta(G) \geq (p - 1)^2$ from Theorem \ref{theorem: delta(G) geq (p-1)^2} is exact. However, these constructions are not very economical in the case $p = 3$. For example, the minimal graph $G \in \mH_e(3, 3)$ from \cite{FL06} with $\delta(G) = 4$ is not presented explicitly, but it is proved that it is a subgraph of a graph with 17577 vertices. From the next theorem we see that the smallest minimal graph $G \in \mH_e(3, 3)$ with $\delta(G) = 4$ has 10 vertices:

\begin{theorem}
	\label{theorem: 10-vertex minimal graph in mH_e(3, 3) G for which delta(G) = 4}
	Let $G$ be a minimal graph in $\mH_e(3, 3)$ and $\delta(G) = 4$. Then $\abs{\V(G)} \geq 10$. There is only one 10-vertex minimal graph $G \in \mH_e(3, 3)$ with $\delta(G) = 4$, namely $G_{10.2}$ (see Figure \ref{figure: 10}). Furthermore, $G_{10.2}$ has only a single vertex of degree 4. For all other 10-vertex minimal graphs $G \in \mH_e(3, 3)$, $\delta(G) = 5$.
\end{theorem}

Let $G$ be a graph in $\mH_e(3, 3)$. By Theorem \ref{theorem: chi(G) geq R(p, q)}, $\chi(G) \geq 6$ and from the inequality $\chi(G) \leq \Delta(G) + 1$ (see \cite{Har69}) we obtain $\Delta(G) \geq 5$. From the Brooks' Theorem (see \cite{Har69}) it follows that if $G \neq K_6$, then $\Delta(G) \geq 6$. The following related question arises naturally:
\begin{align*}
\label{question: 6-regular minimal graph in mH_e(3, 3)}
&\emph{Are there minimal graphs in $\mH_e(3, 3)$ which are 6-regular?}\\
&\emph{(i.e. $d(v) = 6, \forall v \in \V(G)$)}
\end{align*}
From the obtained minimal graphs in $\mH_e(3, 3)$ we see that the following theorem is true:
\begin{theorem}
	\label{theorem: regular minimal graphs in mH_e(3, 3)}
	Let $G$ be a regular minimal graph in $\mH_e(3, 3)$ and $G \neq K_6$. Then $\abs{\V(G)} \geq 13$. There is only one regular minimal graph in $\mH_e(3, 3)$ with 13 vertices, and this is the graph given in Figure \ref{figure: 13_regular}, which is 8-regular.
\end{theorem}	

\begin{figure}
	\centering
	\includegraphics[trim={0 470 0 0},clip,height=160px,width=160px]{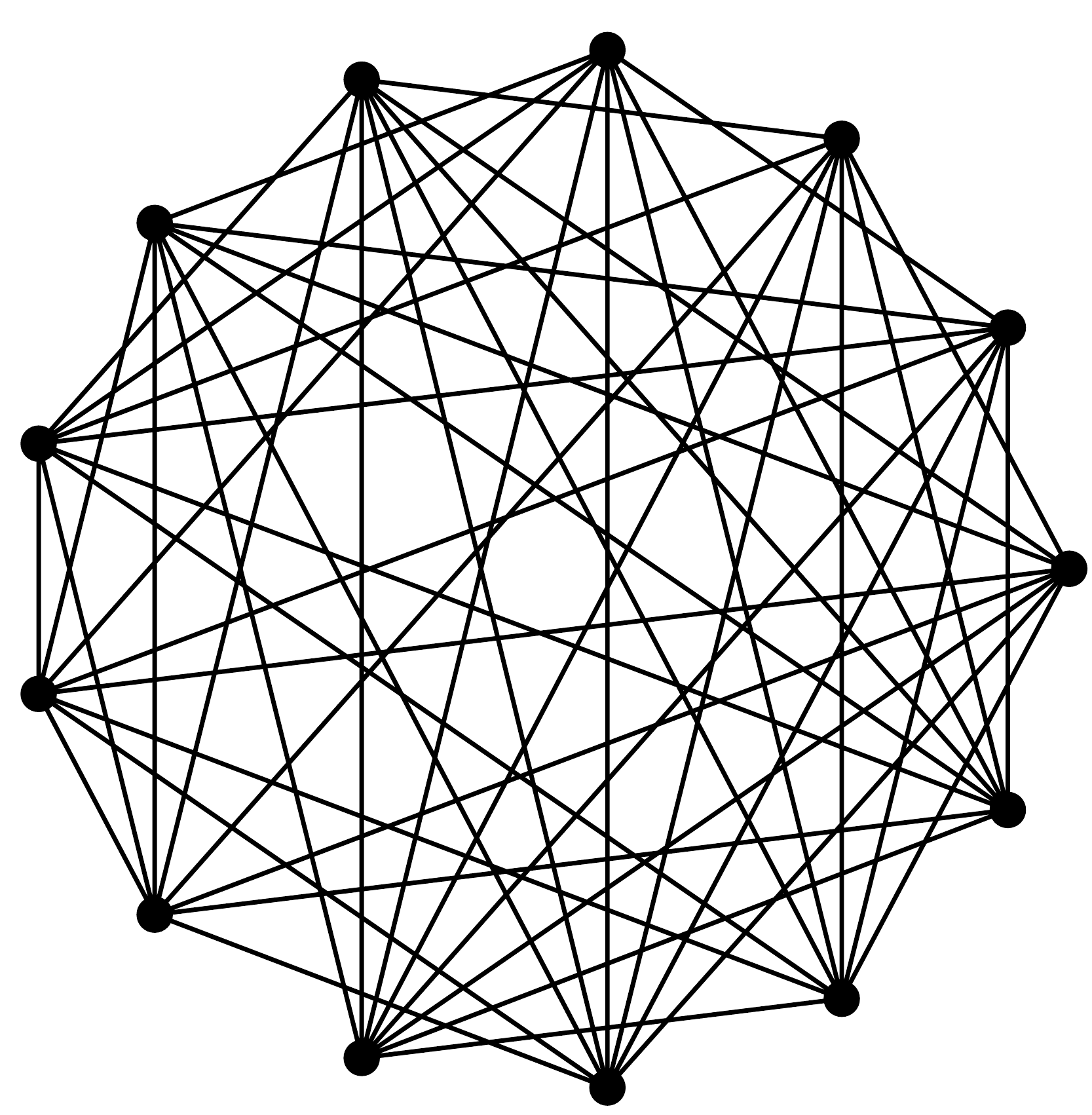}
	\caption{13-vertex 8-regular minimal graph in $\mH_e(3, 3)$}
	\label{figure: 13_regular}
\end{figure}

Regarding the maximum degree of the minimal graphs in $\mH_e(3, 3)$ we obtain the following result:
\begin{theorem}
	\label{theorem: maximum degree of minimal graphs in mH_e(3, 3)}
	Let $G$ be a minimal graph in $\mH_e(3, 3)$. Then:
	\vspace{1em}\\
	(a) $\Delta(G) = \abs{\V(G)} - 1$, if $\abs{\V(G)} \leq 10$.
	\vspace{1em}\\
	(b) $\Delta(G) \geq 8$, if $\abs{\V(G)} = 11$, $12$, or $13$.
\end{theorem}

\subsection*{Chromatic number}

By Theorem \ref{theorem: chi(G) geq R(p, q)}, if $G \in \mH_e(3, 3)$, then $\chi(G) \geq 6$.

From the obtained minimal graphs in $\mH_e(3, 3)$ we derive the following results:

\begin{theorem}
	\label{theorem: if abs(V(G)) leq 12, then chi(G) leq 7}
	Let $G$ be a minimal graph in $\mH_e(3, 3)$ and $\abs{\V(G)} \leq 12$. Then $\chi(G) = 6$.
\end{theorem}

\begin{theorem}
	\label{theorem: 7-chromatic graphs in mH_e(3, 3)}
	Let $G$ be a minimal graph in $\mH_e(3, 3)$ and $|V(G)| \leq 14$. Then $\chi(G) \leq 7$. The smallest 7-chromatic minimal graphs in $\mH_e(3, 3)$ are the 13 minimal graphs on 13 vertices and independence number 2, given in Figure \ref{figure: 13_a2}.
\end{theorem}

\begin{proof}
	Suppose the opposite is true, i.e. $\chi(G) \geq 8$. Then, according to \cite{Nen10}, $G = K_1 + Q$, where $\overline{Q}$ is the graph shown in Figure \ref{figure: Q_13}. The graph $K_1 + Q$ is a graph in $\mH_e(3, 3)$, but it is not minimal. By Theorem \ref{theorem: if abs(V(G)) leq 12, then chi(G) leq 7}, there are no 7-chromatic minimal graphs in $\mH_e(3, 3)$ with less than 13 vertices. The graphs in Figure \ref{figure: 13_a2} are 13-vertex minimal graphs in $\mH_e(3, 3)$ with independence number 2, and therefore these graphs are 7-chromatic. With a computer we find that among the 13-vertex graphs in $\mH_e(3, 3)$ with independence number greater than 2 there are no 7-chromatic graphs.
\end{proof}

\vspace{-1em}

\begin{figure}[h]
	\centering
	\includegraphics[height=195px,width=195px]{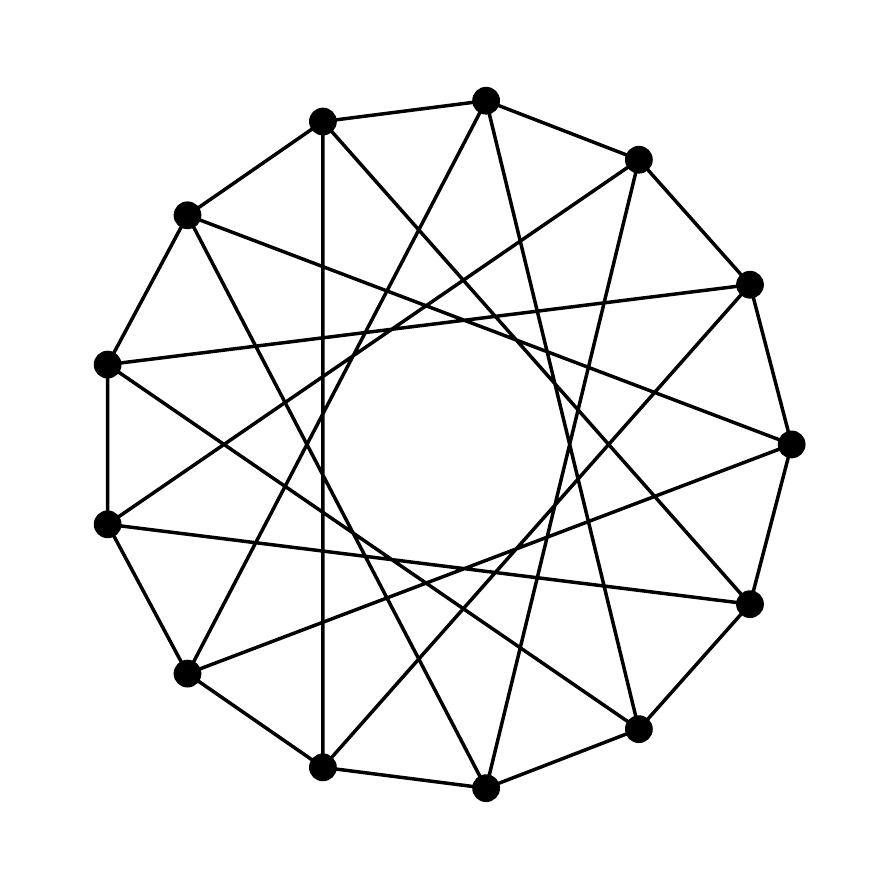}
	\caption{Graph $\overline{Q}$}
	\label{figure: Q_13}
\end{figure}

\subsection*{Multiplicities}

\begin{definition}
	\label{definition: K_3-multiplicity}
	Denote by $M(G)$ the minimum number of monochromatic triangles in all 2-colorings of $\E(G)$. The number $M(G)$ is called a $K_3$-multiplicity of the graph $G$.
\end{definition}

In \cite{Goo59} the $K_3$-multiplicities of all complete graphs are computed, i.e. $M(K_n)$ is computed for all positive integers $n$. Similarly, the $K_p$-multiplicity of a graph is defined \cite{HP74}. The following works are dedicated to the computation of the multiplicities of some concrete graphs: \cite{Jac80}, \cite{Jac82}, \cite{RS77}, \cite{BR90}, \cite{PR01}.

With the help of a computer, we check the $K_3$-multiplicities of the obtained minimal graphs in $\mH_e(3, 3)$ and we derive the following results:

\begin{theorem}
	\label{theorem: if \abs(V(G)) leq 13 and G neq K_6, then M(G) = 1}
	If $G$ is a minimal graph in $\mH_e(3, 3)$, $\abs{\V(G)} \leq 13$, and $G \neq K_6$, then $M(G) = 1$.
\end{theorem}

We suppose the following hypothesis is true:
\begin{hypothesis}
	\label{hypothesis: if G is a minimal graph in mH_e(3, 3) and G neq K_6, then M(G) = 1}
	If $G$ is a minimal graph in $\mH_e(3, 3)$ and $G \neq K_6$, then $M(G) = 1$.
\end{hypothesis}

In support to this hypothesis we prove the following:

\begin{proposition}
	\label{proposition: if G is a minimal graph in mH_e(3, 3), G neq K_6 and delta(G) leq 5, then M(G) = 1}
	If $G$ is a minimal graph in $\mH_e(3, 3)$, $G \neq K_6$ and $\delta(G) \leq 5$, then $M(G) = 1$.
\end{proposition}

\begin{proof}
	Let $v \in \V(G)$ and $d(v) \leq 5$. Consider a 2-coloring of $\E(G-v)$ without monochromatic triangles. We will color the edges incident to $v$ with two colors in such a way that we will obtain a 2-coloring of $\E(G)$ with exactly one monochromatic triangle. To achieve this, we consider the following two cases:
	
	\emph{Case 1.} $d(v) = 4$. By Corollary \ref{corollary: if d(v) = 4, then G(v) = K_4}, $G(v) = K_4$. Let $N_v = \{a, b, c, d\}$ and suppose that $[a, b]$ is colored with the first color. Then $[c, d]$ is also colored with the first color (otherwise, by coloring $[v, a]$ and $[v, b]$ with the second color and $[v, c]$ and $[v, d]$ with the fist color, we obtain a 2-coloring of $\E(G)$ without monochromatic triangles). Thus, $[a, b]$ and $[c, d]$ are colored in the first color. We color $[v, a]$ and $[v, b]$ with the first color and $[v, c]$ and $[v, d]$ with the second color. We obtain a 2-coloring of $\E(G)$ with exactly one monochromatic triangle $[v, a, b]$.
	
	\emph{Case 2.} $d(v) = 5$. Since $\omega(G) \leq 5$, in $N_G(v)$ there are two non-adjacent vertices $a$ and $b$. From $G \arrowse (3, 3)$ it follows easily that in $G(v) - \{a, b\}$ there is an edge of the first color and an edge of the second color. Therefore, we can suppose that in $G(v) - \{a, b\}$ there is exactly one edge of one of the colors, say the first color. We color $[v, a]$ and $[v, b]$ with the second color and the other three edges incident to $v$ with the first color. We obtain a 2-coloring of $\E(G)$ with exactly one monochromatic triangle.
\end{proof}

In the end, also in support to the hypothesis, let us note that $M(K_3 + C_{2r + 1}) = 1, \ r \geq 2$ \cite{NK79}.

\subsection*{Automorphism groups}

Denote by $Aut(G)$ the automorphism group of the graph $G$. We use the \emph{nauty} programs \cite{MP13} to find the number of automorphisms of the obtained minimal graphs in $\mH_e(3, 3)$ with 10, 11, 12, and 13 vertices. Most of the obtained graphs have small automorphism groups (see Table \ref{table: 10-vertex graphs properties}, Table \ref{table: 11-vertex graphs properties}, Table \ref{table: 12-vertex graphs properties}, and Table \ref{table: 13-vertex graphs properties}). We list the graphs with at least 60 automorphisms:

- The graphs in the form $K_3+C_{2r+1}$: $\abs{Aut(K_3+C_5)} = 60$. $\abs{Aut(K_3+C_7)} = 84$, $\abs{Aut(K_3+C_9)} = 108$;

- $\abs{Aut(G_{12.2240})} = 96$ (see Figure \ref{figure: 12_aut});

- $\abs{Aut(G_{13.255653})} = 144$, $\abs{Aut(G_{13.248305})} = 96$, $\abs{Aut(G_{13.304826})} = 96$, $\abs{Aut(G_{13.113198})} = 72$, $\abs{Aut(G_{13.175639})} = 72$, $\abs{Aut(G_{13.302168})} = 72$ (see Figure \ref{figure: 13_aut});

\vspace{1em}
Theorem \ref{theorem: 10-vertex minimal graph in mH_e(3, 3) G for which delta(G) = 4}, Theorem \ref{theorem: regular minimal graphs in mH_e(3, 3)}, Theorem \ref{theorem: maximum degree of minimal graphs in mH_e(3, 3)}, Theorem \ref{theorem: if abs(V(G)) leq 12, then chi(G) leq 7}, Theorem \ref{theorem: 7-chromatic graphs in mH_e(3, 3)}, and Theorem \ref{theorem: if \abs(V(G)) leq 13 and G neq K_6, then M(G) = 1} are published in \cite{Bik16}.

\section{Upper bounds on the independence number of the minimal graphs in $\mH_e(3, 3)$}

Regarding the maximal possible value of the independence number of the minimal graphs in $\mH_e(3, 3)$, the following theorem holds:

\begin{theorem}
	\label{theorem: alpha(G) leq abs(V(G)) - 7}
	\cite{Nen80b}
	If $G$ is a minimal graph in $\mH_e(3, 3)$, $G \neq K_6$ and $G \neq K_3+C_5$, then $\alpha(G) \leq |V(G)| - 7$. There is a finite number of graphs for which equality is reached.
\end{theorem}

According to Theorem \ref{theorem: algorithm A7}, by executing Algorithm \ref{algorithm: A7}($q = 6; k = 8$), we obtain the set $\mB$ of all minimal graphs $G \in \mH_e(3, 3; 6)$ with $\alpha(G) \geq \abs{\V(G)} - 8 \geq 1$. Since the only minimal graphs in $\mH_e(3, 3)$ with less than 9 vertices are $K_6$ and $K_3 + C_5$, from Theorem \ref{theorem: alpha(G) leq abs(V(G)) - 7} it follows that $\mB$ consists of all minimal graphs $G \in \mH_e(3, 3)$ for which $\alpha(G) = \abs{\V(G)} - 7$ or $\alpha(G) = \abs{\V(G)} - 8$. Thus, we derive the following additions to Theorem \ref{theorem: alpha(G) leq abs(V(G)) - 7}:

\begin{theorem}
	\label{theorem: graphs in mH_e(3, 3) G for which alpha(G) = abs(V(G)) - 7}
	There are exactly 11 minimal graphs $G \in \mH_e(3, 3)$ for which $\alpha(G) = \abs{\V(G)} - 7$:
	
	-  9-vertex: 1 (Figure \ref{figure: Nenov_9});
	
	- 10-vertex: 3 ($G_{10.1}$, $G_{10.2}$, $G_{10.4}$, see Figure \ref{figure: 10});
	
	- 11-vertex: 3 ($G_{11.1}$, $G_{11.2}$, $G_{11.21}$, see Figure \ref{figure: 11_a4});
	
	- 12-vertex: 1 ($G_{12.163}$, see Figure \ref{figure: 12_a5});
	
	- 13-vertex: 2 ($G_{13.}$, $G_{13.}$, see Figure \ref{figure: 13_a6}) ;
	
	- 14-vertex: 1 (see Figure \ref{figure: 14_a7});
\end{theorem}

\begin{theorem}
	\label{theorem: graphs in mH_e(3, 3) G for which alpha(G) = abs(V(G)) - 8}
	There are exactly 8633 minimal graphs $G \in \mH_e(3, 3)$ for which $\alpha(G) = \abs{\V(G)} - 8$. The largest of these graphs has 26 vertices, and it is given in Figure \ref{figure: 26_a18}. There is only one minimal graph $G \in \mH_e(3, 3)$ for which $\alpha(G) = \abs{\V(G)} - 8$ and $\omega(G) < 5$, and it is the 15-vertex graph $K_1 + \Gamma$ from \cite{Nen81a} (see Figure \ref{figure: Nenov_14}).
\end{theorem}

\begin{corollary}
	\label{corollary: alpha(G) leq abs(V(G)) - 9, G in mH_e(3, 3; 6)}
	Let $G$ be a minimal graph in $\mH_e(3, 3)$ and $\abs{\V(G)} \geq 27$. Then $\alpha(G) \leq \abs{\V(G)} - 9$.
\end{corollary}

\begin{figure}
	\centering
	\includegraphics[trim={0 0 0 490},clip,height=140px,width=140px]{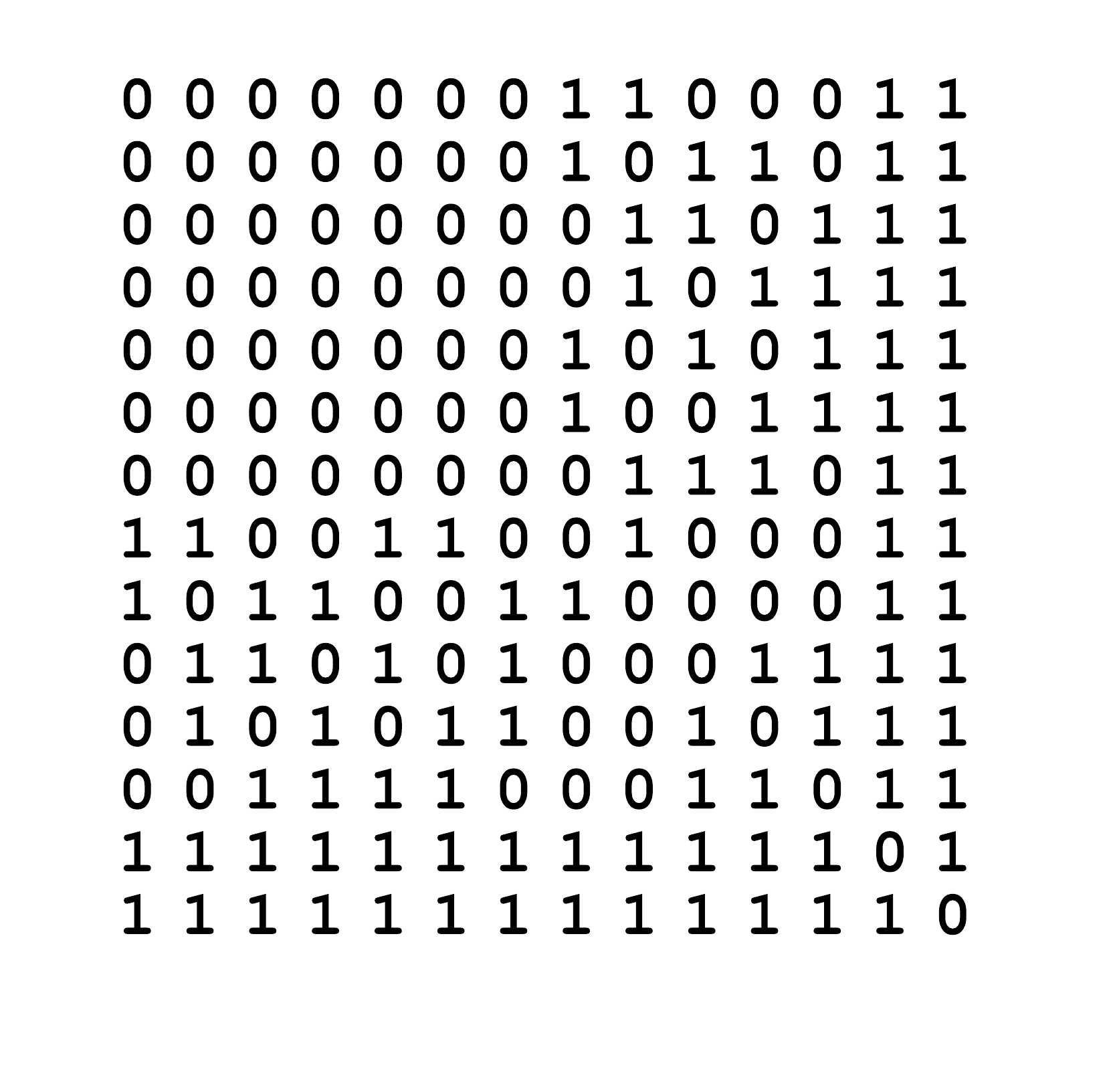}
	\caption{\small 14-vertex minimal graph in $\mH_e(3, 3)$ with independence number 7}
	\label{figure: 14_a7}
	\vspace{2em}
	\centering
	\includegraphics[trim={0 0 0 490},clip,height=208px,width=208px]{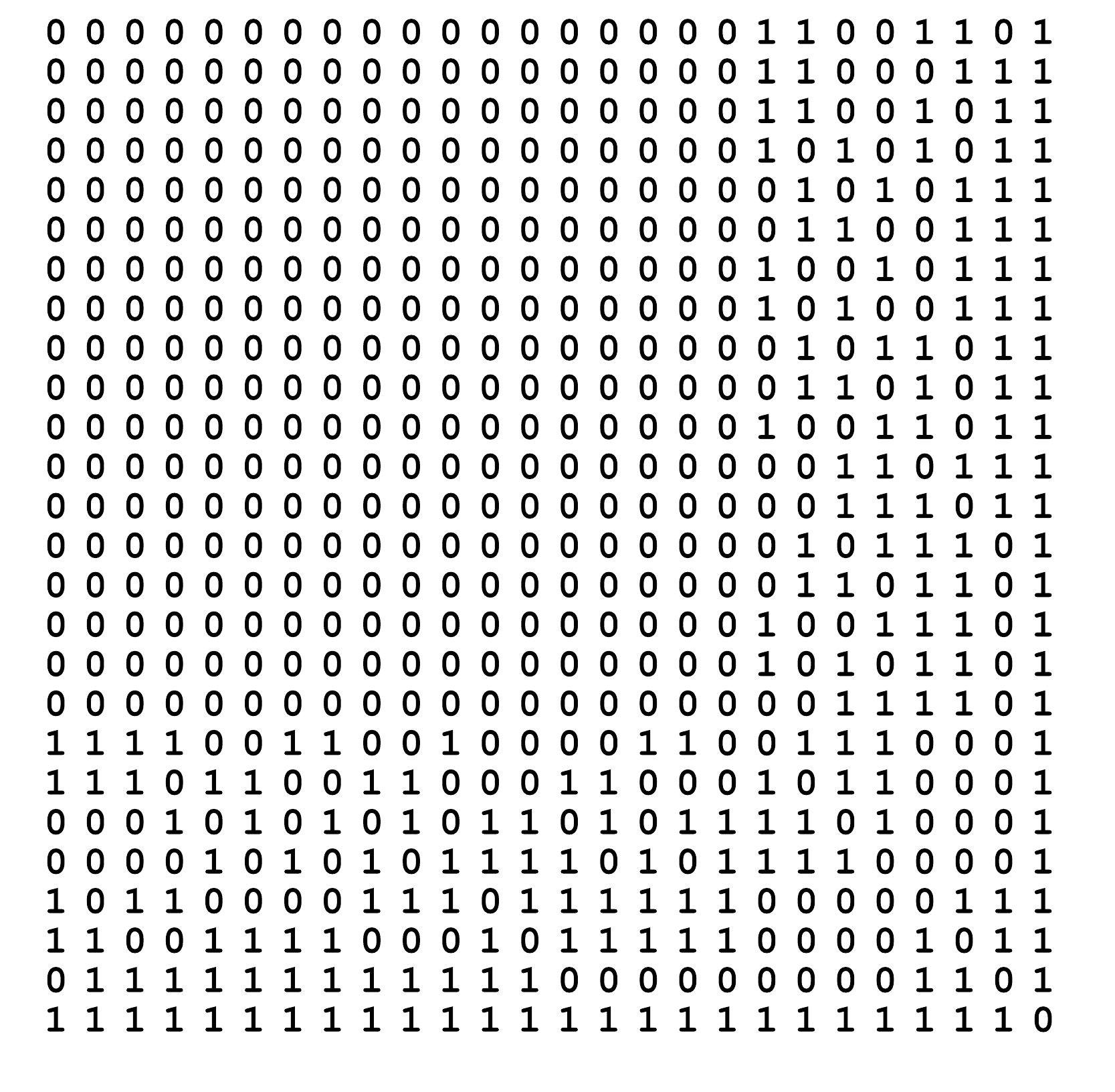}
	\caption{\small 26-vertex minimal graph in $\mH_e(3, 3)$ with independence number 18}
	\label{figure: 26_a18}
\end{figure}

According to Theorem \ref{theorem: algorithm A7}, by executing Algorithm \ref{algorithm: A7}($q = 5; k = 9$), we obtain the set $\mB$ of all minimal graphs $G \in \mH_e(3, 3; 5)$ for which $\alpha(G) \geq \abs{\V(G)} - 9 \geq 1$. Since $F_e(3, 3; 5) = 15$ \cite{Nen81a} \cite{PRU99}, all graphs in $\mH_e(3, 3; 5)$ have at least 15 vertices. We already proved that the only minimal graph $G \in \mH_e(3, 3; 5)$ with $\alpha(G) \geq \abs{\V(G)} - 8$ is the graph $K_1 + \Gamma$ from \cite{Nen81a} (see Theorem \ref{theorem: graphs in mH_e(3, 3) G for which alpha(G) = abs(V(G)) - 8}). Therefore, $\mB$ consists of $K_1 + \Gamma$ and all minimal graphs $G \in \mH_e(3, 3; 5)$ for which $\alpha(G) = \abs{\V(G)} - 9$. Thus, we proved the following theorem:

\begin{theorem}
	\label{theorem: graphs in mH_e(3, 3) G for which omega (G) < 5 and alpha(G) = abs(V(G)) - 9}
	There are exactly 8903 minimal graphs $G \in \mH_e(3, 3)$ for which $\omega(G) < 5$ and $\alpha(G) = \abs{\V(G)} - 9$. The largest of these graphs has 29 vertices, and it is given in Figure \ref{figure: 29_w4_a20}.
\end{theorem}

\begin{corollary}
	\label{corollary: alpha(G) leq abs(V(G)) - 10, G in mH_e(3, 3; 5)}
	Let $G$ be a minimal graph in $\mH_e(3, 3)$ such that $\omega(G) < 5$ and $\abs{\V(G)} \geq 30$. Then $\alpha(G) \leq \abs{\V(G)} - 10$.
\end{corollary}

\begin{figure}
	\centering
	\includegraphics[trim={0 0 0 490},clip,height=232px,width=232px]{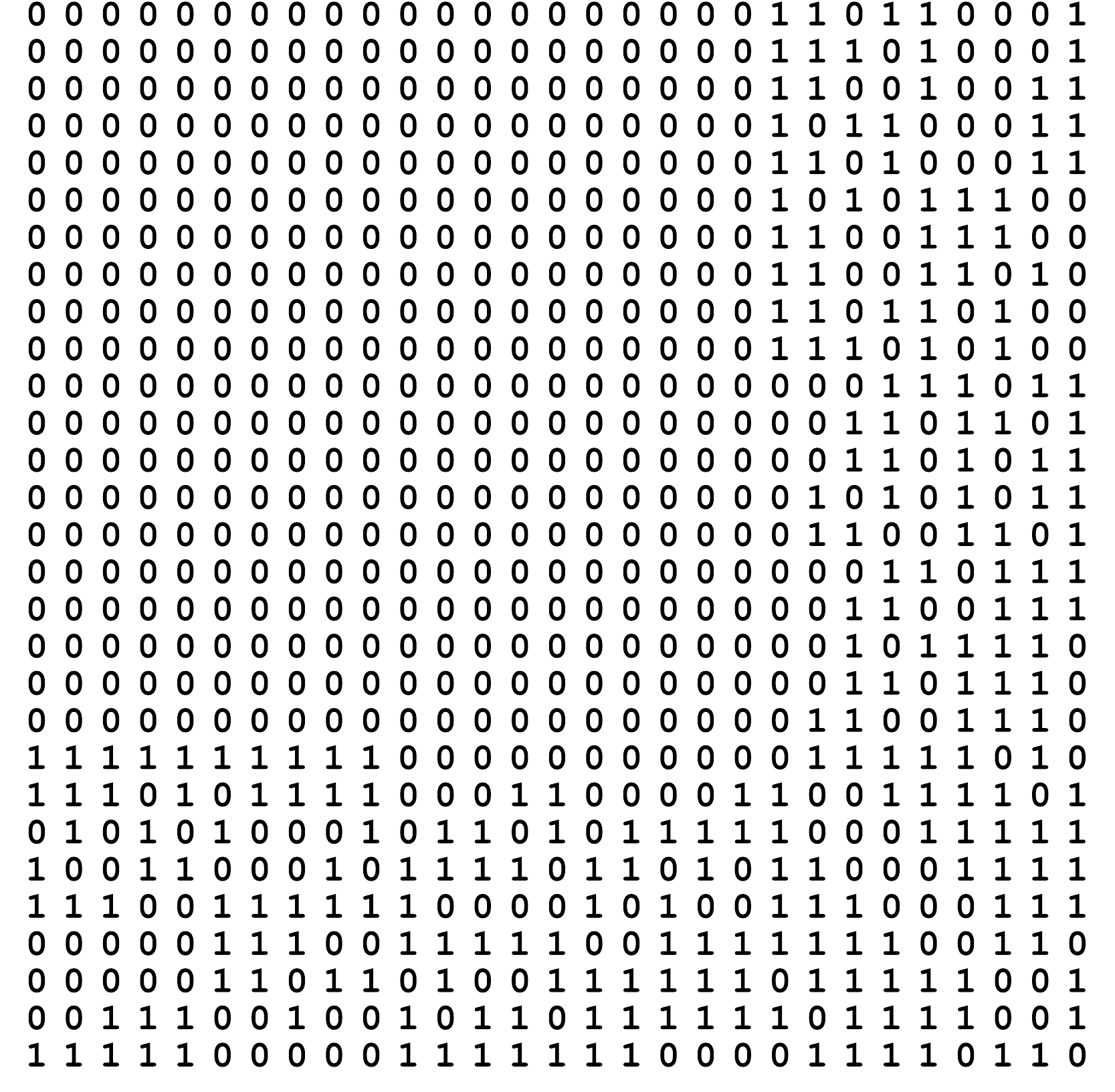}
	\caption{\small 29-vertex minimal graph in $\mH_e(3, 3)$ with independence number 20}
	\label{figure: 29_w4_a20}
\end{figure}

\vspace{1em}
Theorem \ref{theorem: graphs in mH_e(3, 3) G for which alpha(G) = abs(V(G)) - 7}, Theorem \ref{theorem: graphs in mH_e(3, 3) G for which alpha(G) = abs(V(G)) - 8}, and Theorem \ref{theorem: graphs in mH_e(3, 3) G for which omega (G) < 5 and alpha(G) = abs(V(G)) - 9} are published in \cite{Bik16}.

\section{Lower bounds on the minimum degree of the minimal graphs in $\mH_e(3, 3)$}

According to Proposition \ref{proposition: marked vertex sets in graphs in mH_e(3, 3)}, if $G$ is a minimal graph in $\mH_e(3, 3)$, then for each vertex $v$ of $G$, $N_G(v)$ is a marked vertex set in $G - v$, and therefore $N_G(v)$ is a marked vertex set in $G(v)$.

It is easy to see that if $W \subseteq \V(G)$ and $\abs{W} \leq 3$, or $\abs{W} = 4$ and $G[W] \neq K_4$, then $W$ is not a marked vertex set in $G$. A $(3, 3)$-free 2-coloring of $K_4$ which cannot be extended to a $(3, 3)$-free 2-coloring of $K_5$ is shown in Figure \ref{figure: N_4_coloring}. Therefore, the only 4-vertex graph $N$ such that $\V(N)$ is a marked vertex set in $N$ is $K_4$.

With the help of a computer, we derive that there are exactly 3 graphs $N$ with 5 vertices such that $K_4 \not\subset N$ and $\V(N)$ is a marked vertex set in $N$. Namely, they are the graphs $N_{5.1}$, $N_{5.2}$, and $N_{5.3}$ given in Figure \ref{figure: N_5_1 N_5_2 N_5_3}. Let us note that $N_{5.1} \subset N_{5.2} \subset N_{5.3}$. From these results we derive
\begin{theorem}
	\label{theorem: delta(G) geq 5, G in mH_e(3, 3; 5)}
	Let $G$ be a minimal graph in $\mH_e(3, 3; 5)$. Then $\delta(G) \geq 5$. If $v \in \V(G)$ and $d(v) = 5$, then $G(v) = N_{5.i}$ for some $i \in \set{1, 2, 3}$ (see Figure \ref{figure: N_5_1 N_5_2 N_5_3}).
\end{theorem}
The bound $\delta(G) \geq 5$ from Theorem \ref{theorem: delta(G) geq 5, G in mH_e(3, 3; 5)} is exact. For example, the graph $K_1 + \Gamma$ from \cite{Nen81a} (see Figure \ref{figure: Nenov_14}) has 7 vertices $v$ such that $d(v) = 5$ and $G(v) = N_{5.3}$.

Also with the help of a computer, we derive that the smallest graphs $N$ such that $K_3 \not\subset N$ and $\V(N)$ is a marked vertex set in $N$ have 8 vertices, and there are exactly 7 such graphs. Namely, they are the graphs $N_{8.i}, i = 1, ..., 7$ given in Figure \ref{figure: N_8_1 N_8_2 N_8_3 N_8_4 N_8_5 N_8_6 N_8_7}. Among them the minimal graphs are $N_{8.1}$, $N_{8.2}$, and $N_{8.3}$, and the remaining 4 graphs are their supergraphs. Thus, we derive the following
\begin{theorem}
	\label{theorem: delta(G) geq 8, G in mH_e(3, 3; 4)}
	Let $G$ be a minimal graph in $\mH_e(3, 3; 4)$. Then $\delta(G) \geq 8$. If $v \in \V(G)$ and $d(v) = 8$, then $G(v) = N_{8.i}$ for some $i \in \set{1, ..., 7}$ (see Figure \ref{figure: N_8_1 N_8_2 N_8_3 N_8_4 N_8_5 N_8_6 N_8_7}).
\end{theorem}

\begin{figure}
	\centering
	\includegraphics[trim={0 470 0 0},clip,height=120px,width=120px]{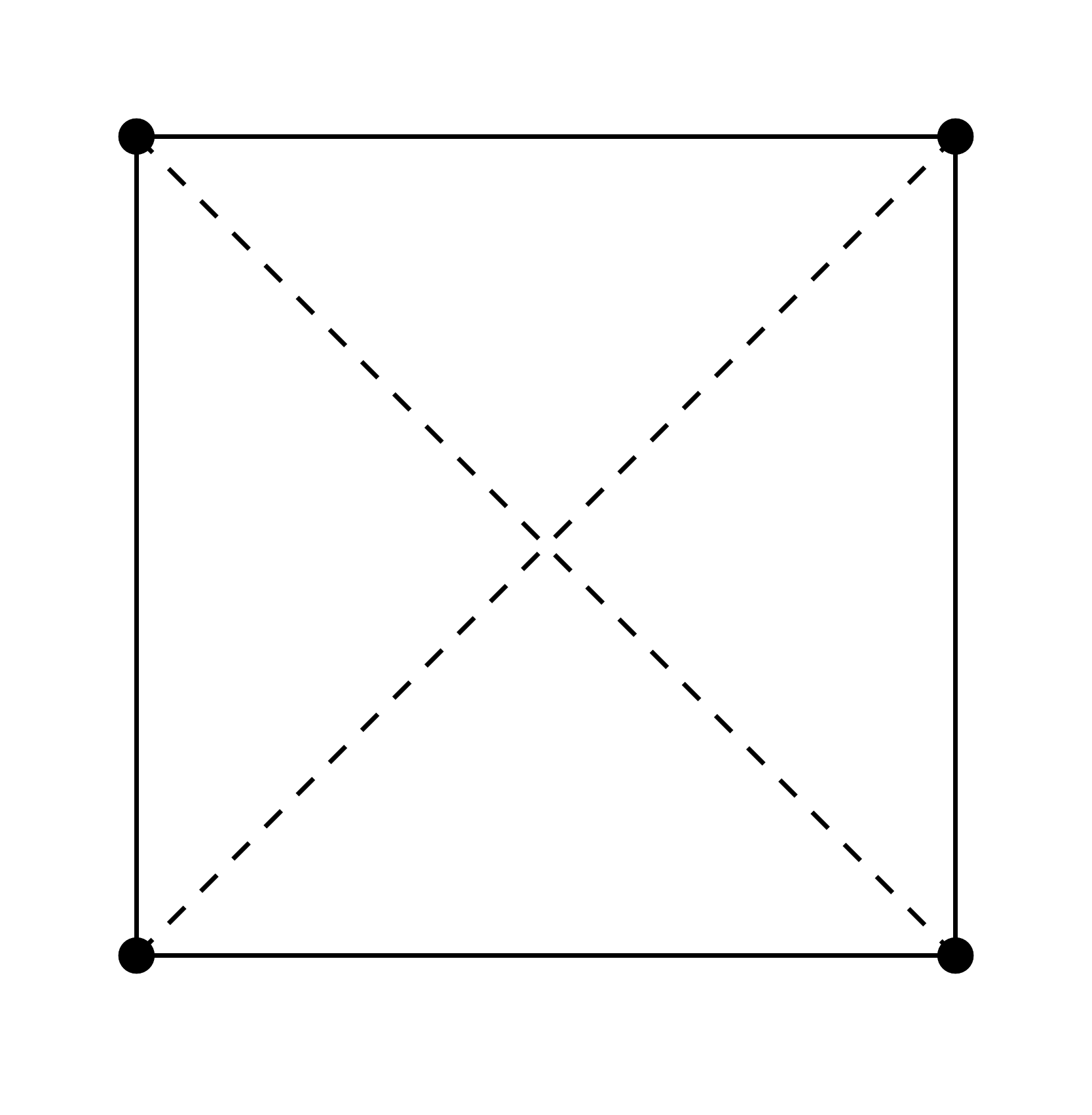}
	\vspace{-0.5em}
	\caption{$(3, 3)$-free 2-coloring of the edges of $K_4$}
	\label{figure: N_4_coloring}
	
	\vspace{1em}
	
	\begin{subfigure}{.3\textwidth}
		\centering
		\includegraphics[trim={0 470 0 0},clip,height=100px,width=100px]{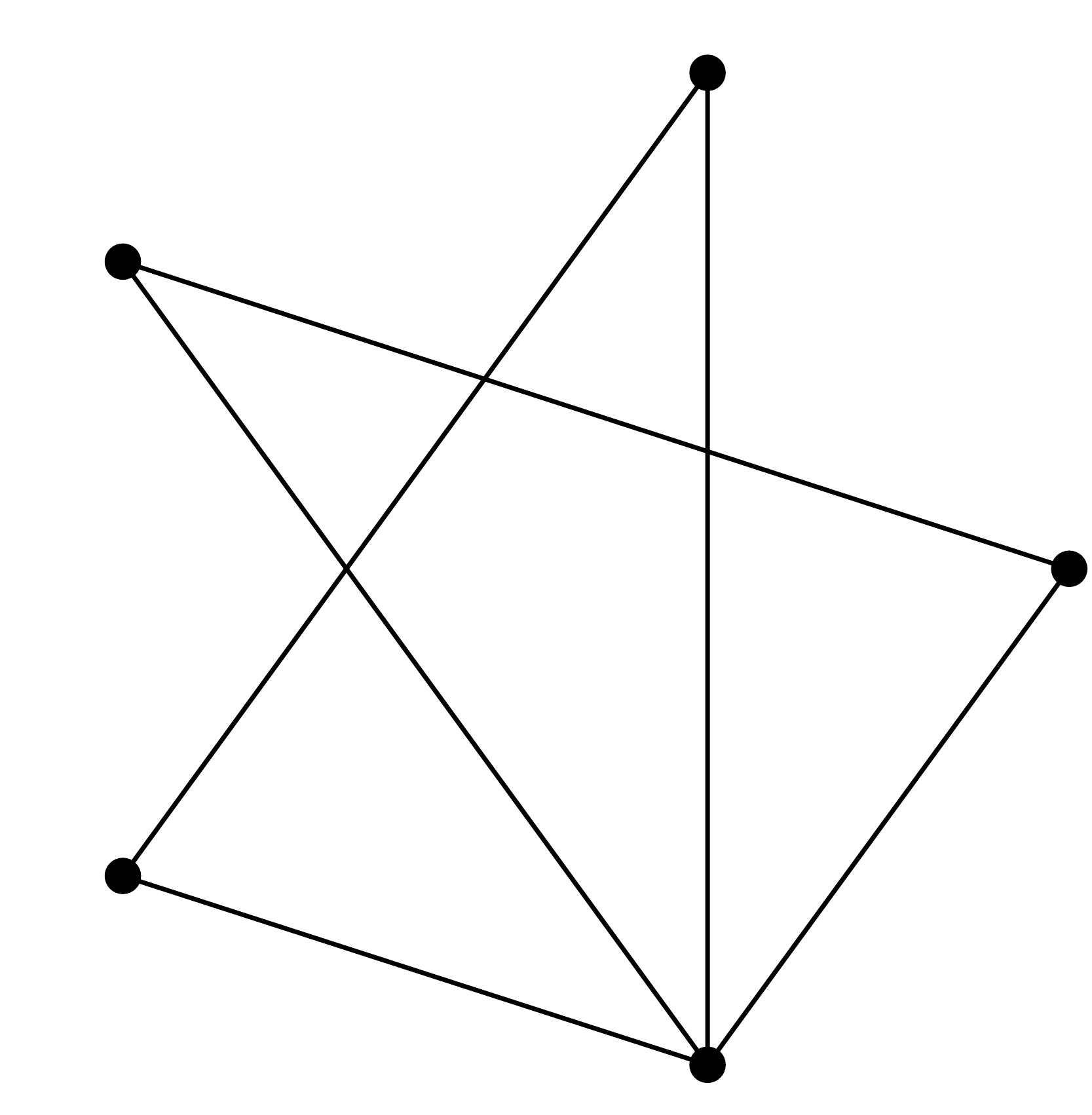}
		\caption*{$N_{5.1}$}
		\label{figure: N_5_1}
	\end{subfigure}\hfill
	\begin{subfigure}{.3\textwidth}
		\centering
		\includegraphics[trim={0 470 0 0},clip,height=100px,width=100px]{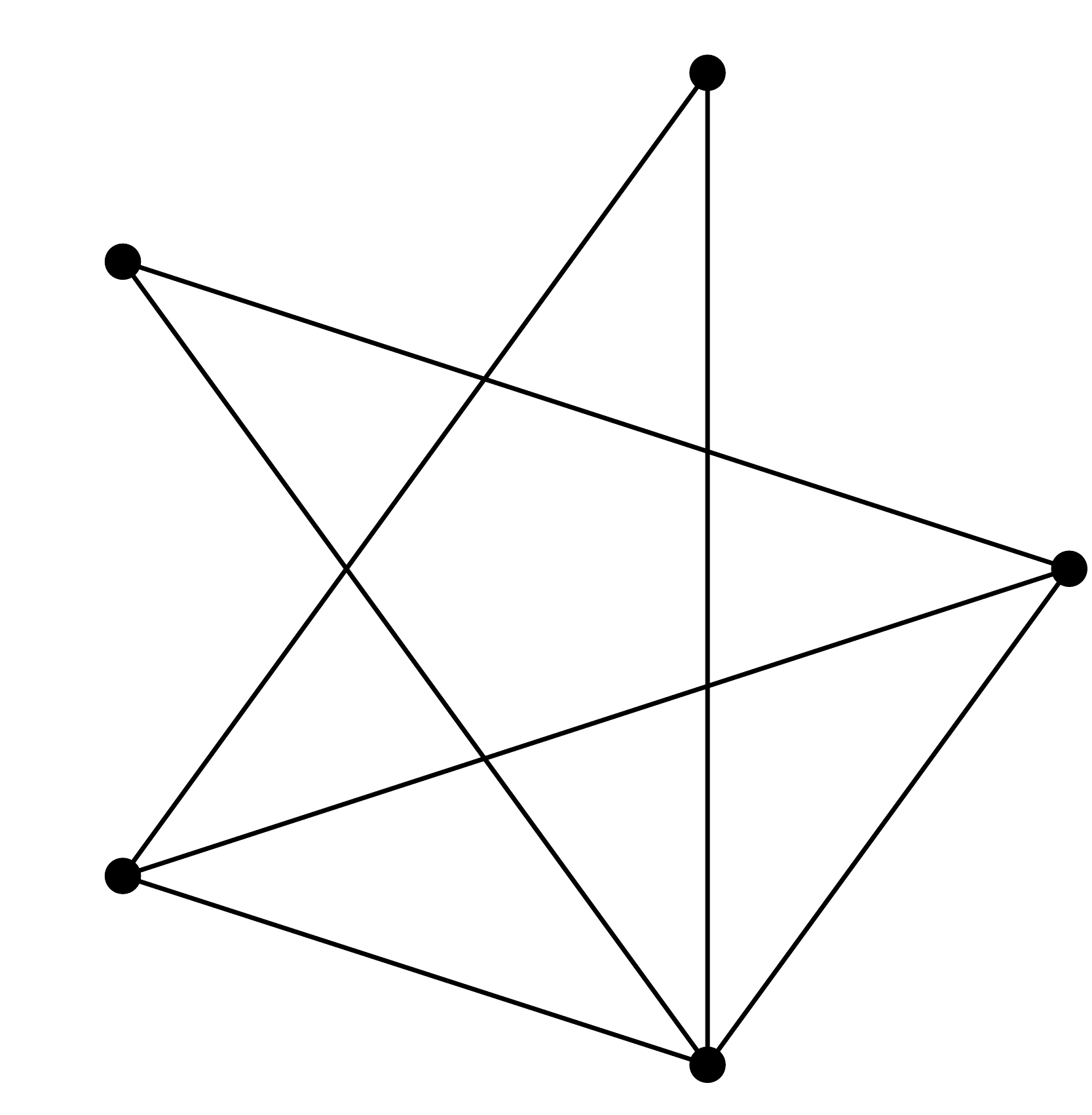}
		\caption*{$N_{5.2}$}
		\label{figure: N_5_2}
	\end{subfigure}\hfill
	\begin{subfigure}{.3\textwidth}
		\centering
		\includegraphics[trim={0 470 0 0},clip,height=100px,width=100px]{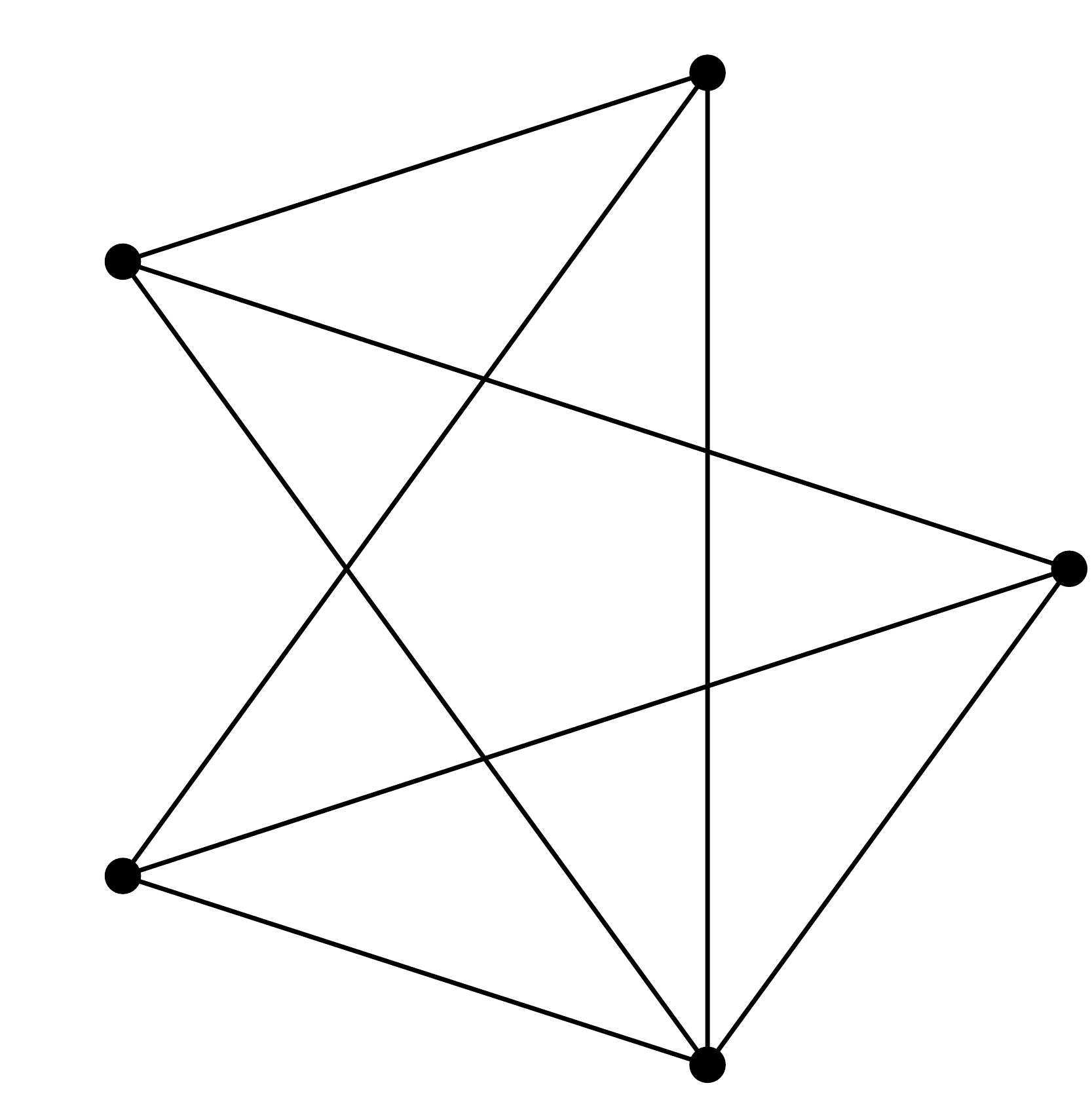}
		\caption*{$N_{5.3}$}
		\label{figure: N_5_3}
	\end{subfigure}
	\caption{The graphs $N_{5.1}$, $N_{5.2}$, $N_{5.3}$}
	\label{figure: N_5_1 N_5_2 N_5_3}

	\vspace{1em}

	\begin{subfigure}{.25\textwidth}
		\centering
		\includegraphics[trim={0 470 0 0},clip,height=80px,width=80px]{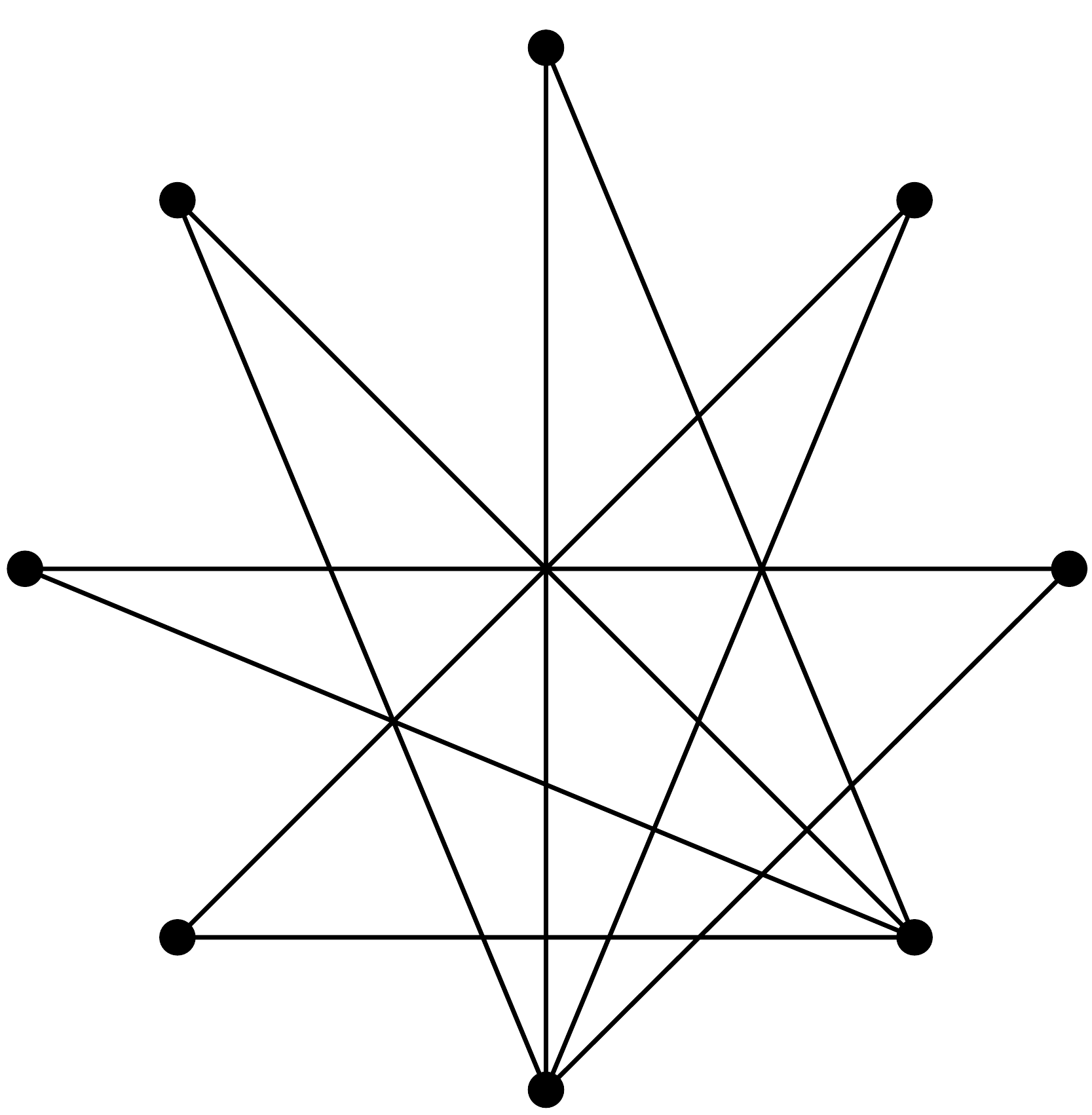}
		\caption*{$N_{8.1}$}
		\label{figure: N_8_1}
	\end{subfigure}\hfill
	\begin{subfigure}{.25\textwidth}
		\centering
		\includegraphics[trim={0 470 0 0},clip,height=80px,width=80px]{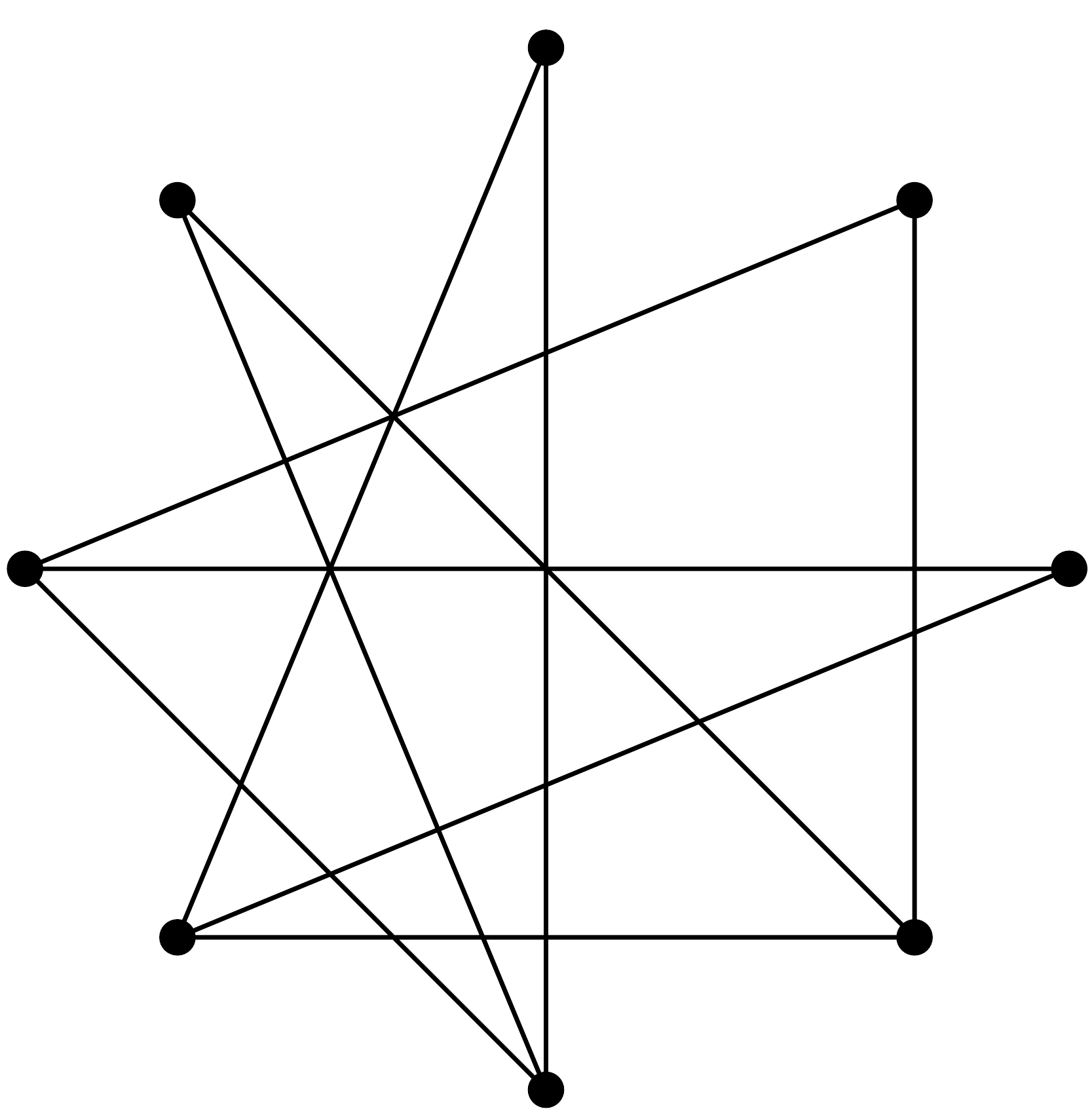}
		\caption*{$N_{8.2}$}
		\label{figure: N_8_2}
	\end{subfigure}\hfill
	\begin{subfigure}{.25\textwidth}
		\centering
		\includegraphics[trim={0 470 0 0},clip,height=80px,width=80px]{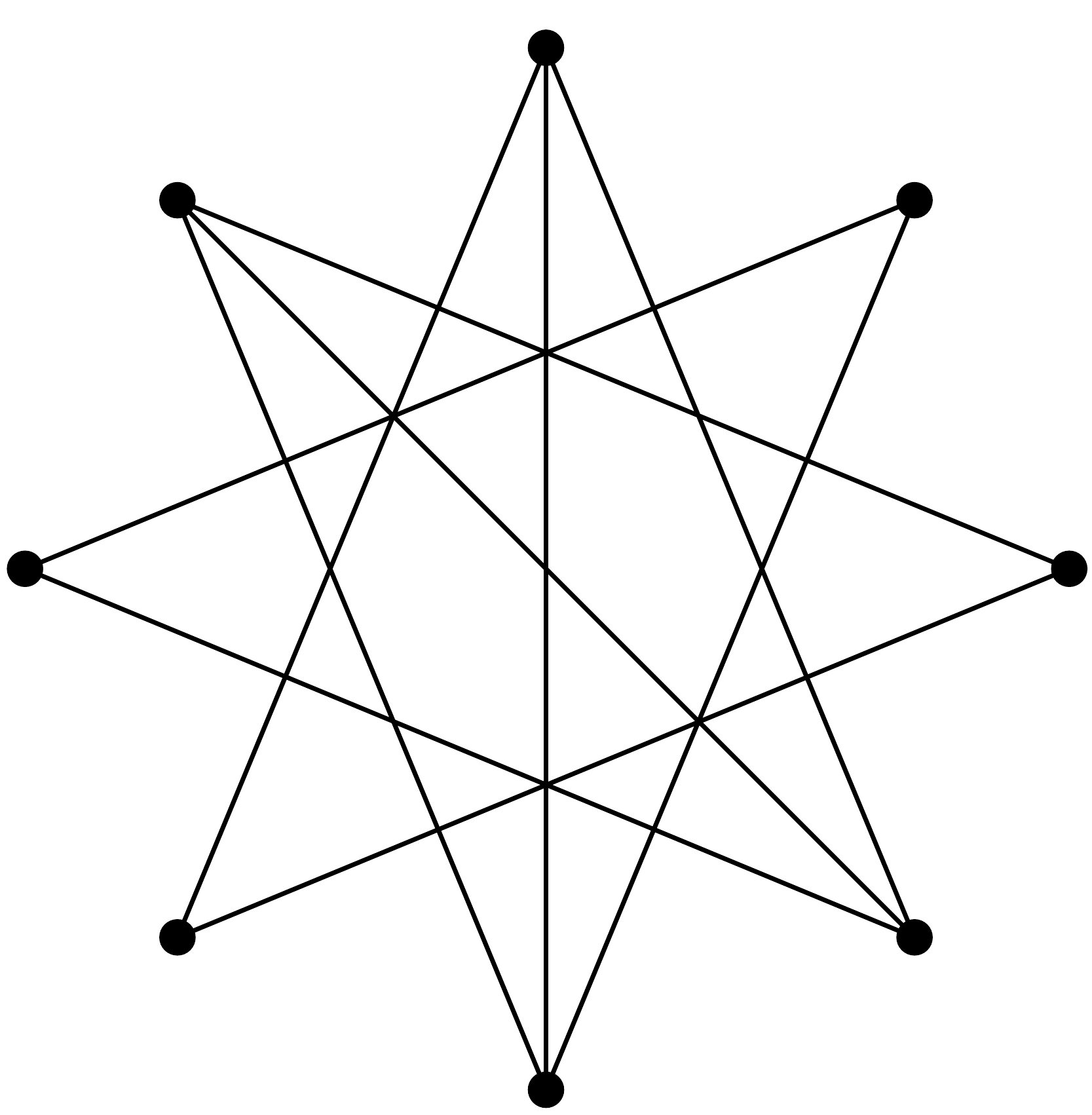}
		\caption*{$N_{8.3}$}
		\label{figure: N_8_3}
	\end{subfigure}\\
	
	\vspace{0.5em}
	
	\begin{subfigure}{.22\textwidth}
		\centering
		\includegraphics[trim={0 470 0 0},clip,height=75px,width=75px]{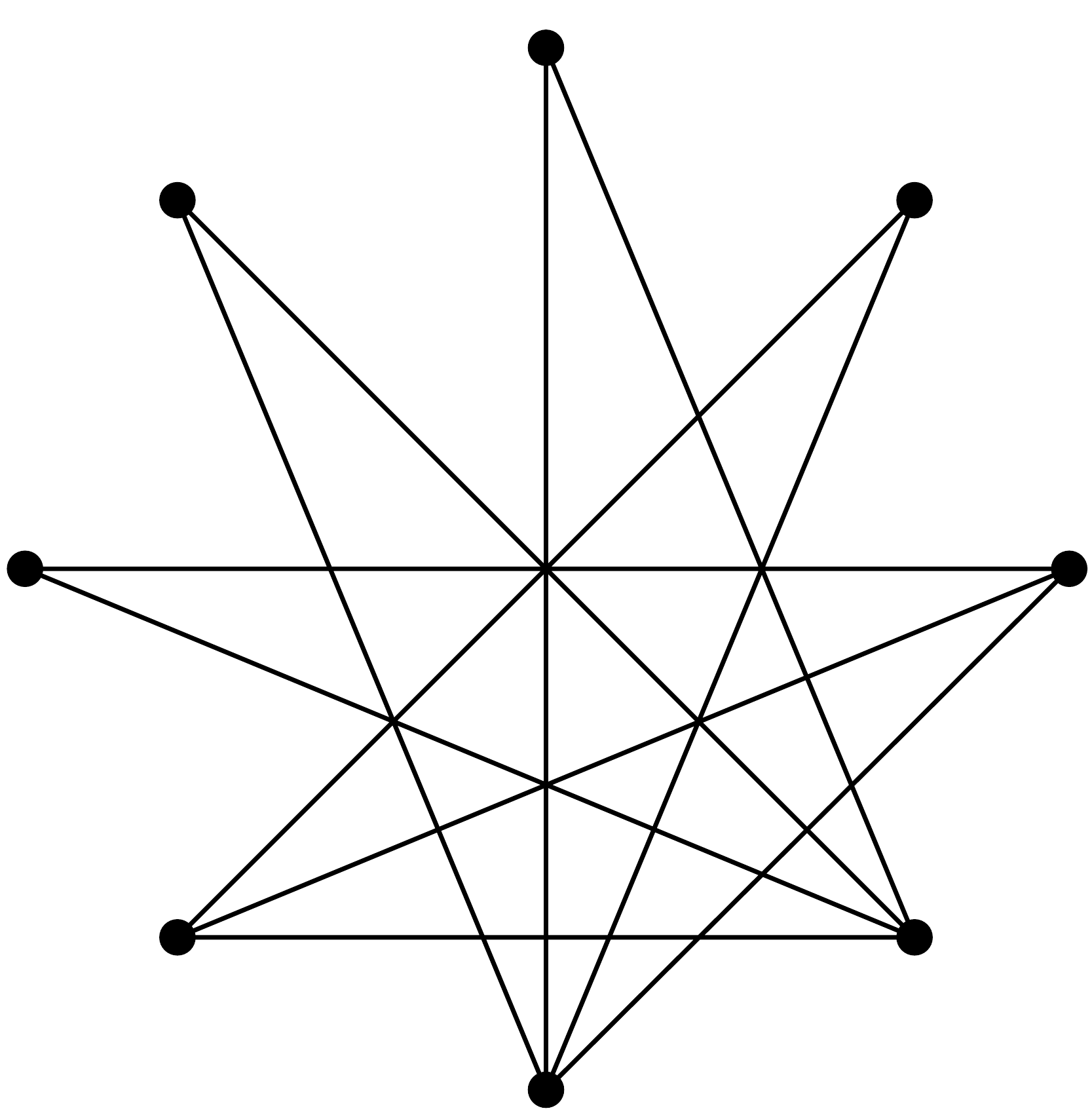}
		\caption*{$N_{8.4}$}
		\label{figure: N_8_4}
	\end{subfigure}\hfill
	\begin{subfigure}{.22\textwidth}
		\centering
		\includegraphics[trim={0 470 0 0},clip,height=75px,width=75px]{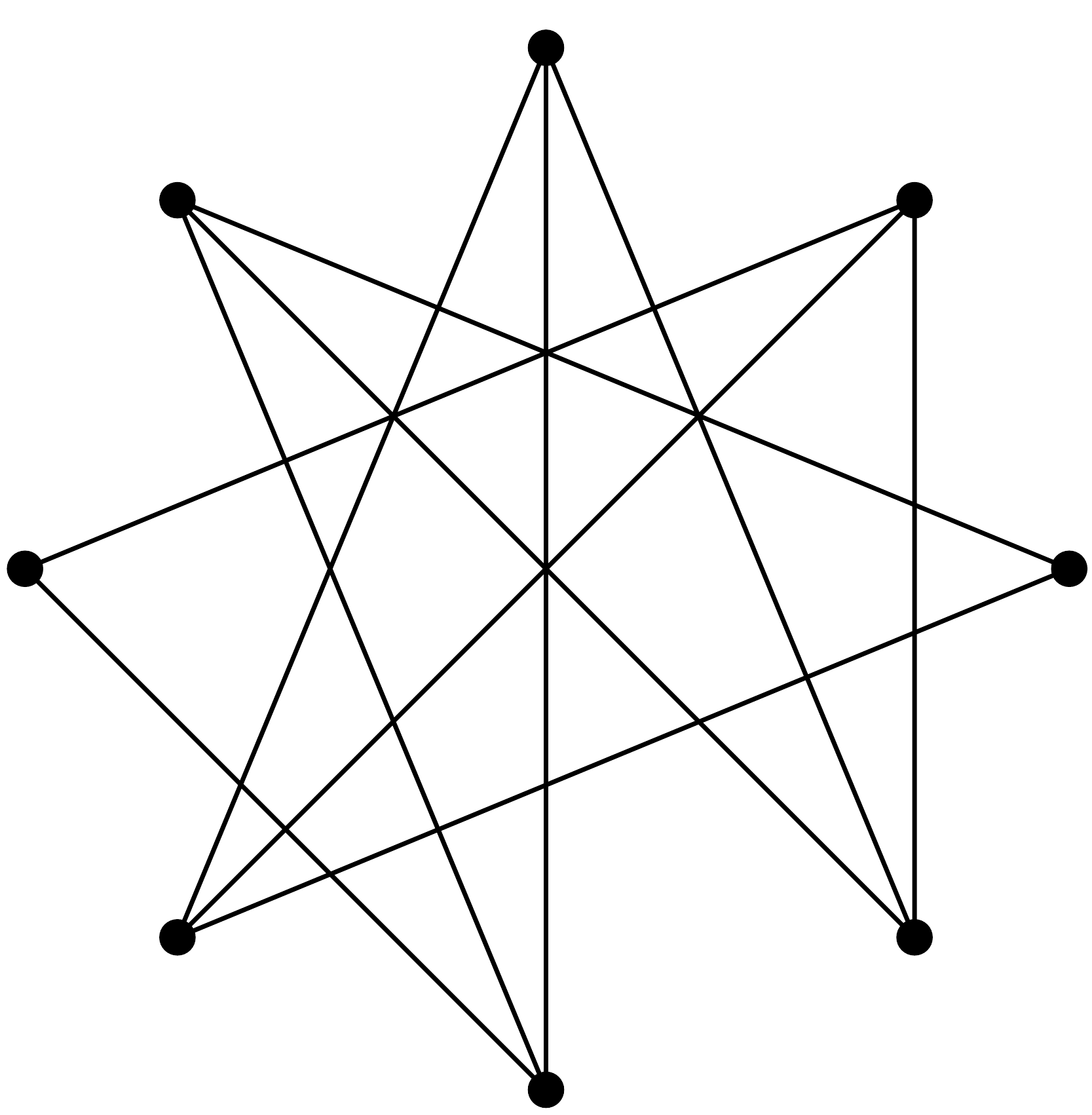}
		\caption*{$N_{8.5}$}
		\label{figure: N_8_5}
	\end{subfigure}\hfill
	\begin{subfigure}{.22\textwidth}
		\centering
		\includegraphics[trim={0 470 0 0},clip,height=75px,width=75px]{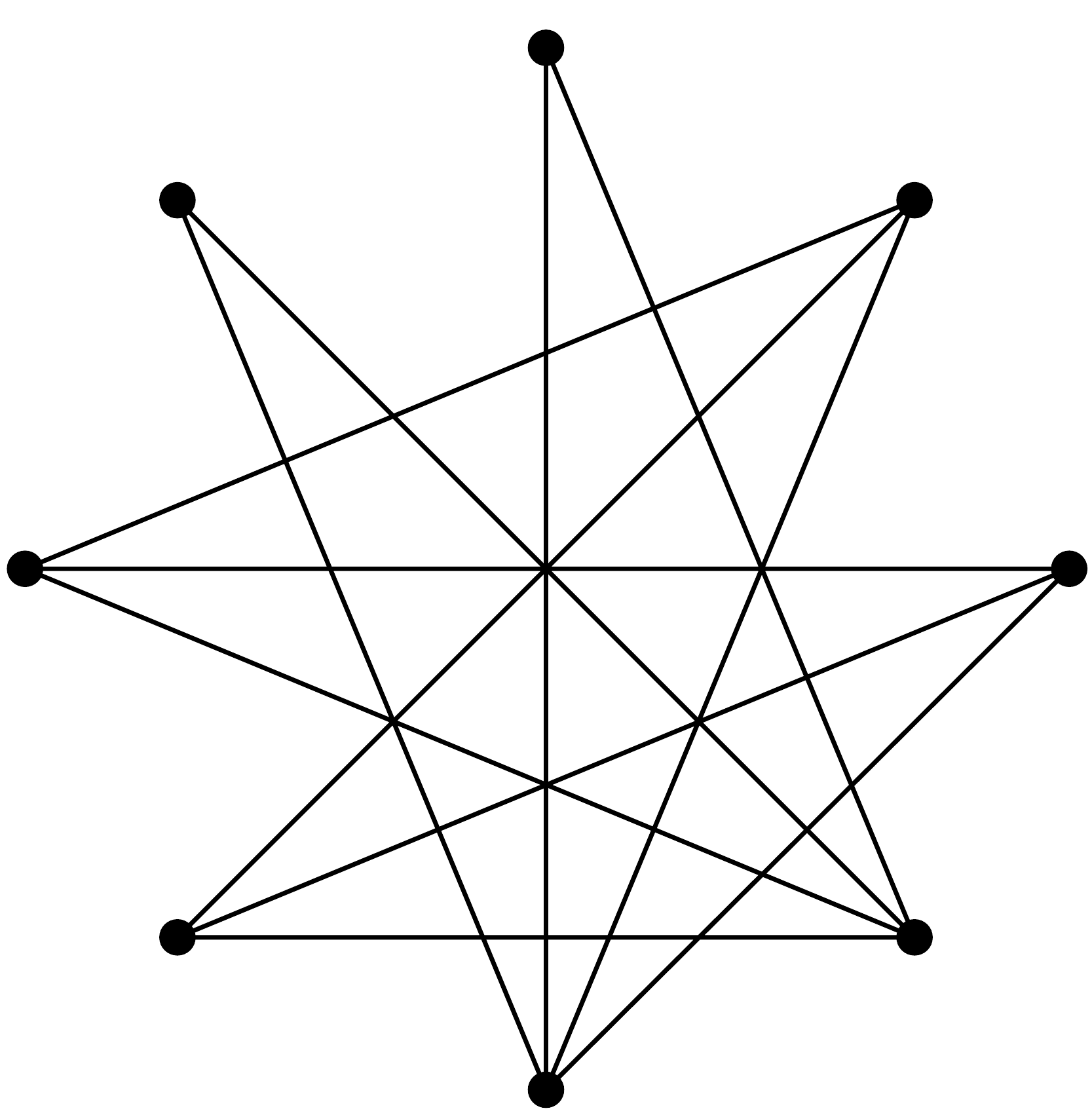}
		\caption*{$N_{8.6}$}
		\label{figure: N_8_6}
	\end{subfigure}\hfill
	\begin{subfigure}{.22\textwidth}
		\centering
		\includegraphics[trim={0 470 0 0},clip,height=75px,width=75px]{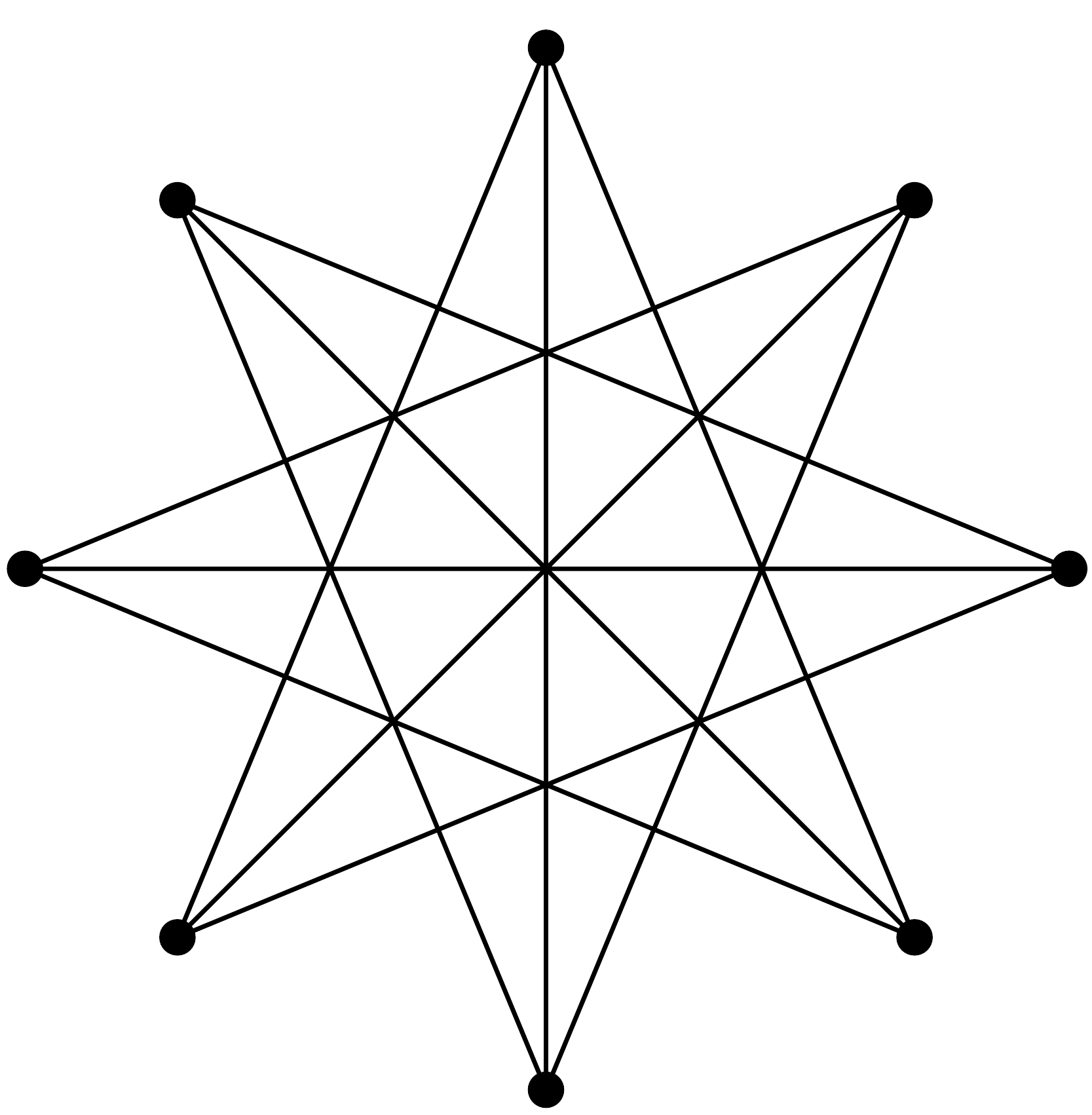}
		\caption*{$N_{8.7}$}
		\label{figure: N_8_7}
	\end{subfigure}
	\caption{The graphs $N_{8.i}, i = 1, ..., 7$}
	\label{figure: N_8_1 N_8_2 N_8_3 N_8_4 N_8_5 N_8_6 N_8_7}
\end{figure}

\vspace{1em}
Theorem \ref{theorem: delta(G) geq 5, G in mH_e(3, 3; 5)} and Theorem \ref{theorem: delta(G) geq 8, G in mH_e(3, 3; 4)} are published in \cite{Bik16}.

\begin{figure}[h]
	\captionsetup{justification=centering}
	\begin{subfigure}{.33\textwidth}
		\centering
		\includegraphics[trim={0 0 0 490},clip,height=100px,width=100px]{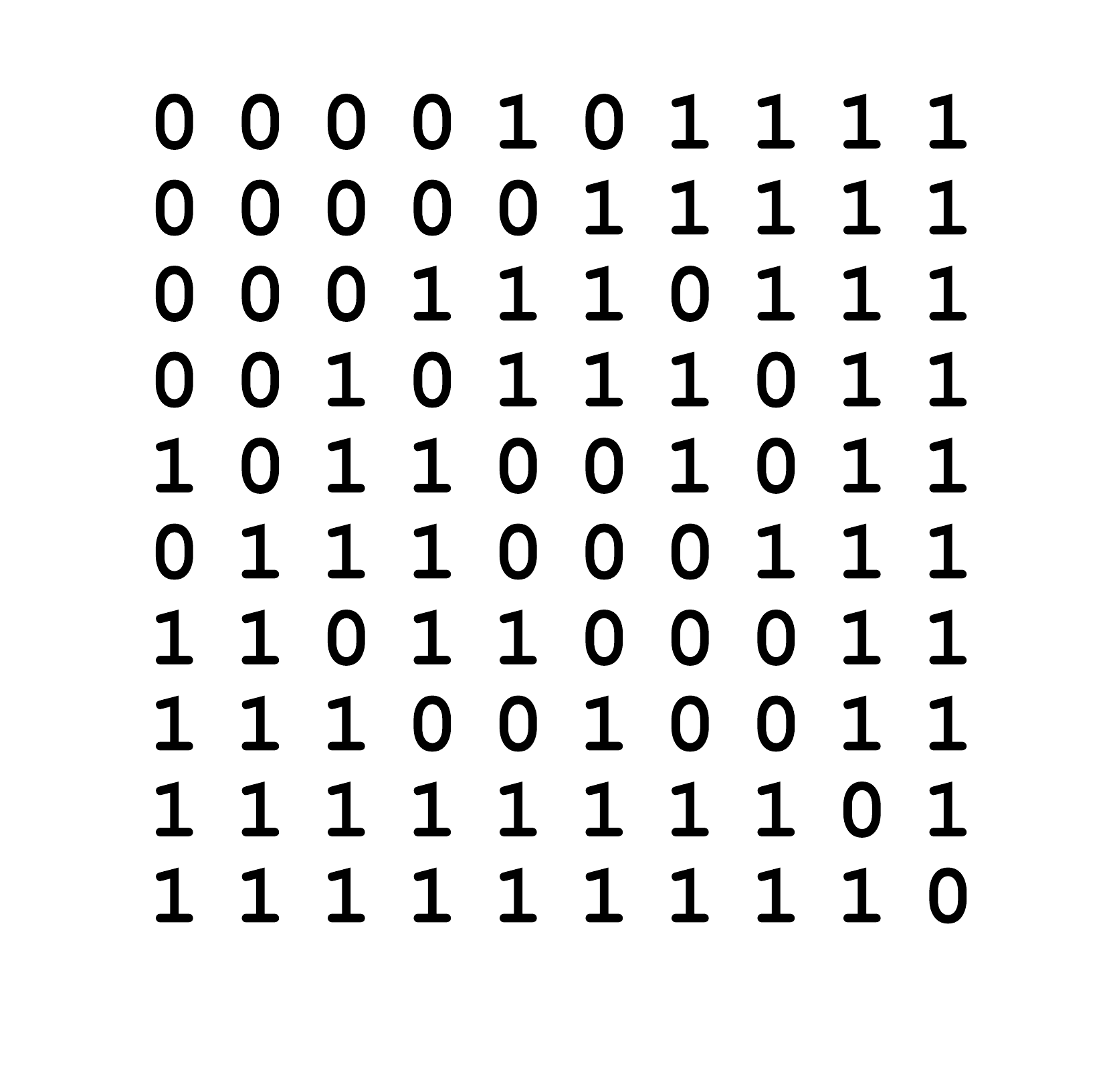}
		\vspace{-1em}
		\caption*{$G_{10.1}$}
		\label{figure: 10_1}
	\end{subfigure}\hfill
	\begin{subfigure}{.33\textwidth}
		\centering
		\includegraphics[trim={0 0 0 490},clip,height=100px,width=100px]{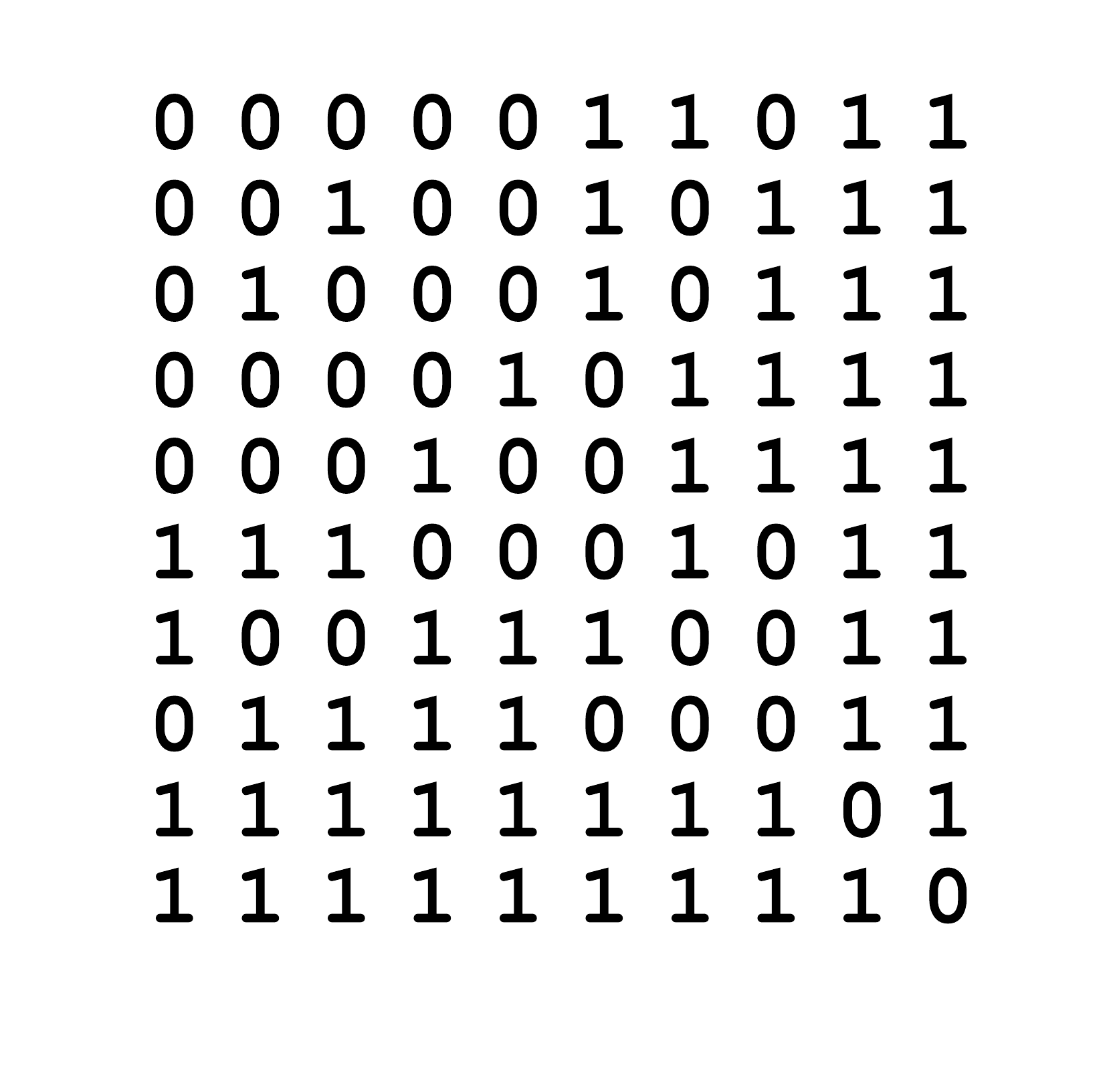}
		\vspace{-1em}
		\caption*{$G_{10.2}$}
		\label{figure: 10_2}
	\end{subfigure}\hfill
	\begin{subfigure}{.33\textwidth}
		\centering
		\includegraphics[trim={0 0 0 490},clip,height=100px,width=100px]{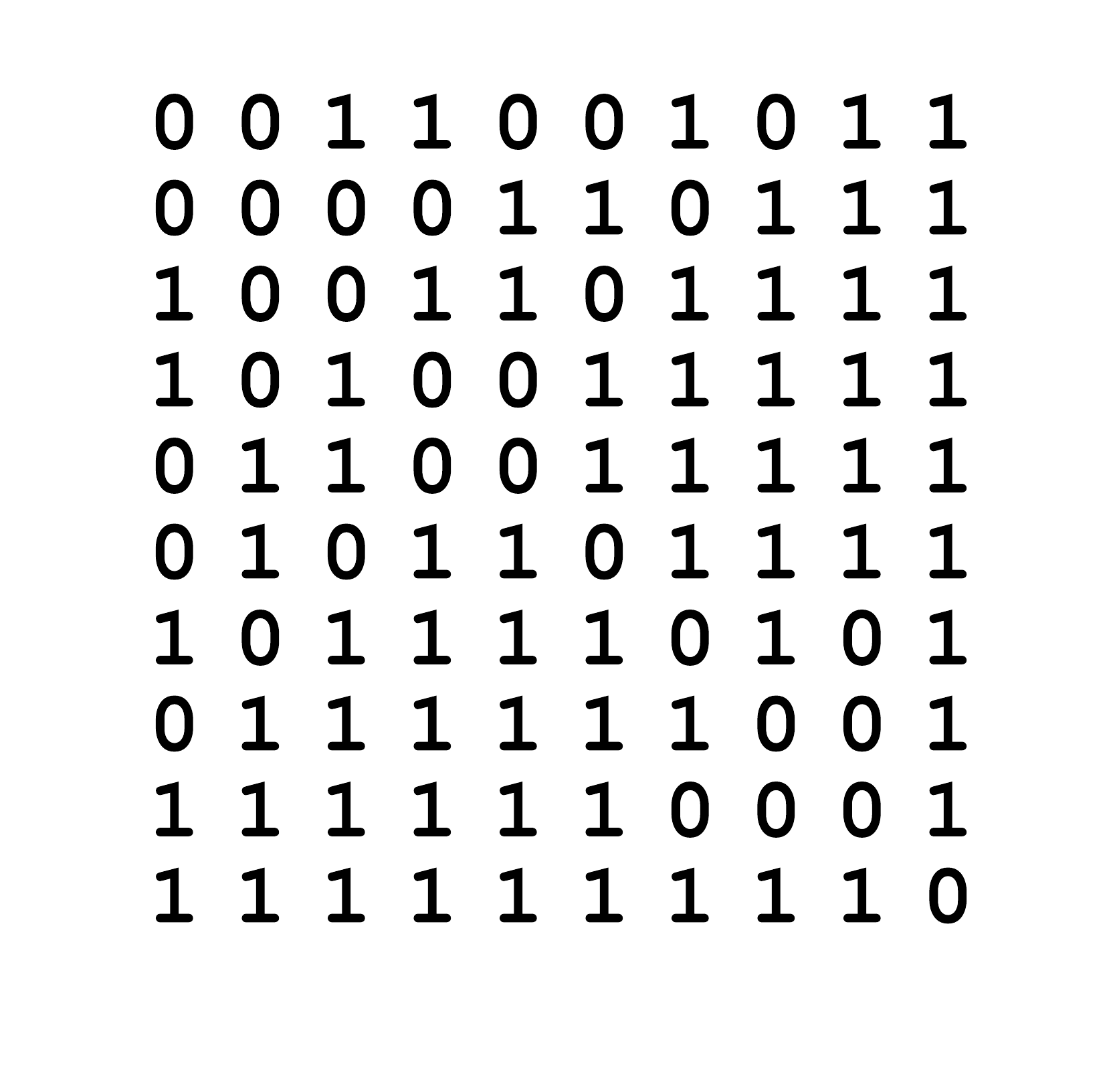}
		\vspace{-1em}
		\caption*{$G_{10.3}$}
		\label{figure: 10_3}
	\end{subfigure}
	
	\vspace{1em}
	\begin{subfigure}{.33\textwidth}
		\centering
		\includegraphics[trim={0 0 0 490},clip,height=100px,width=100px]{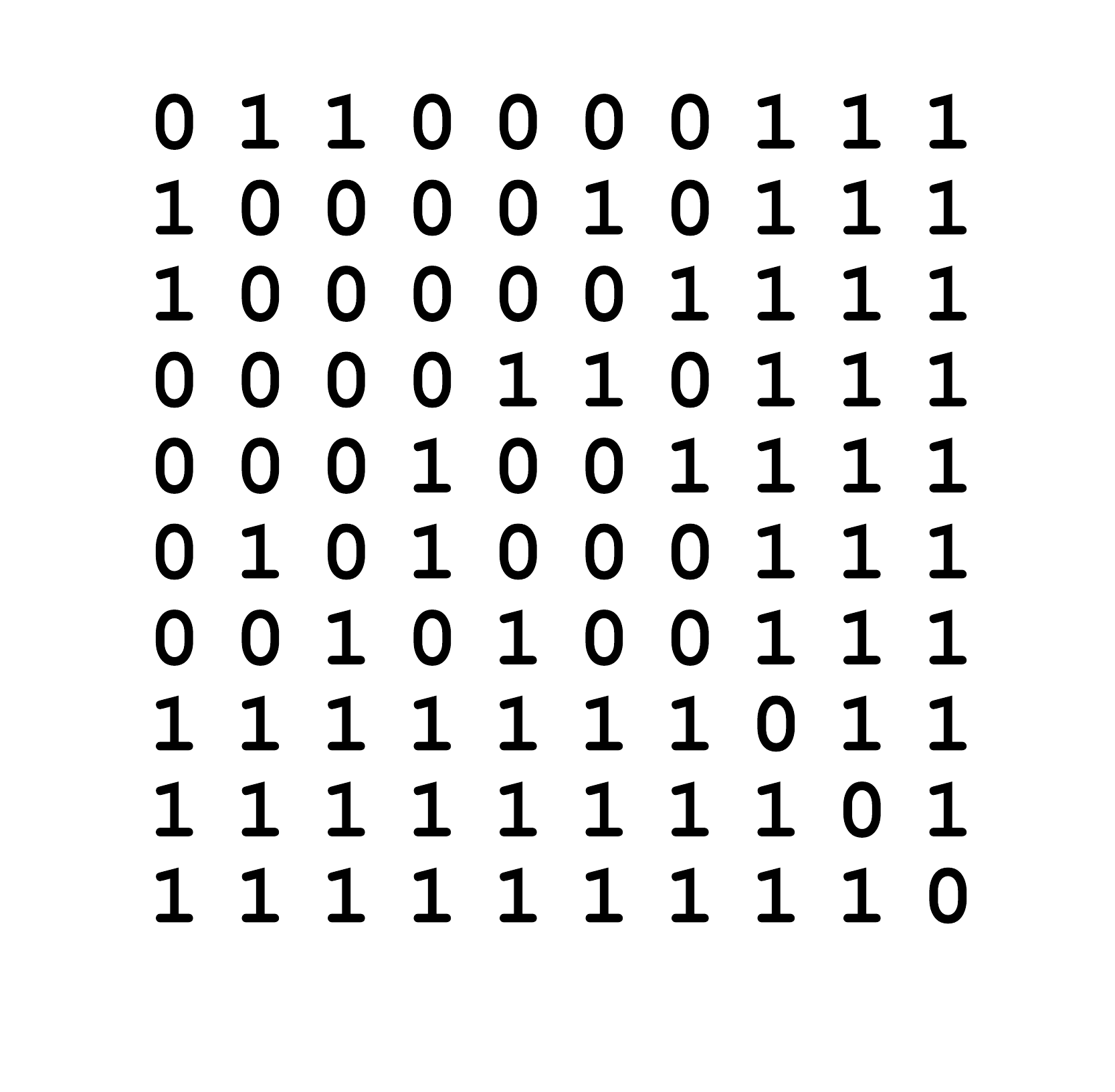}
		\vspace{-1em}
		\caption*{$G_{10.4}$}
		\label{figure: 10_4}
	\end{subfigure}\hfill
	\begin{subfigure}{.33\textwidth}
		\centering
		\includegraphics[trim={0 0 0 490},clip,height=100px,width=100px]{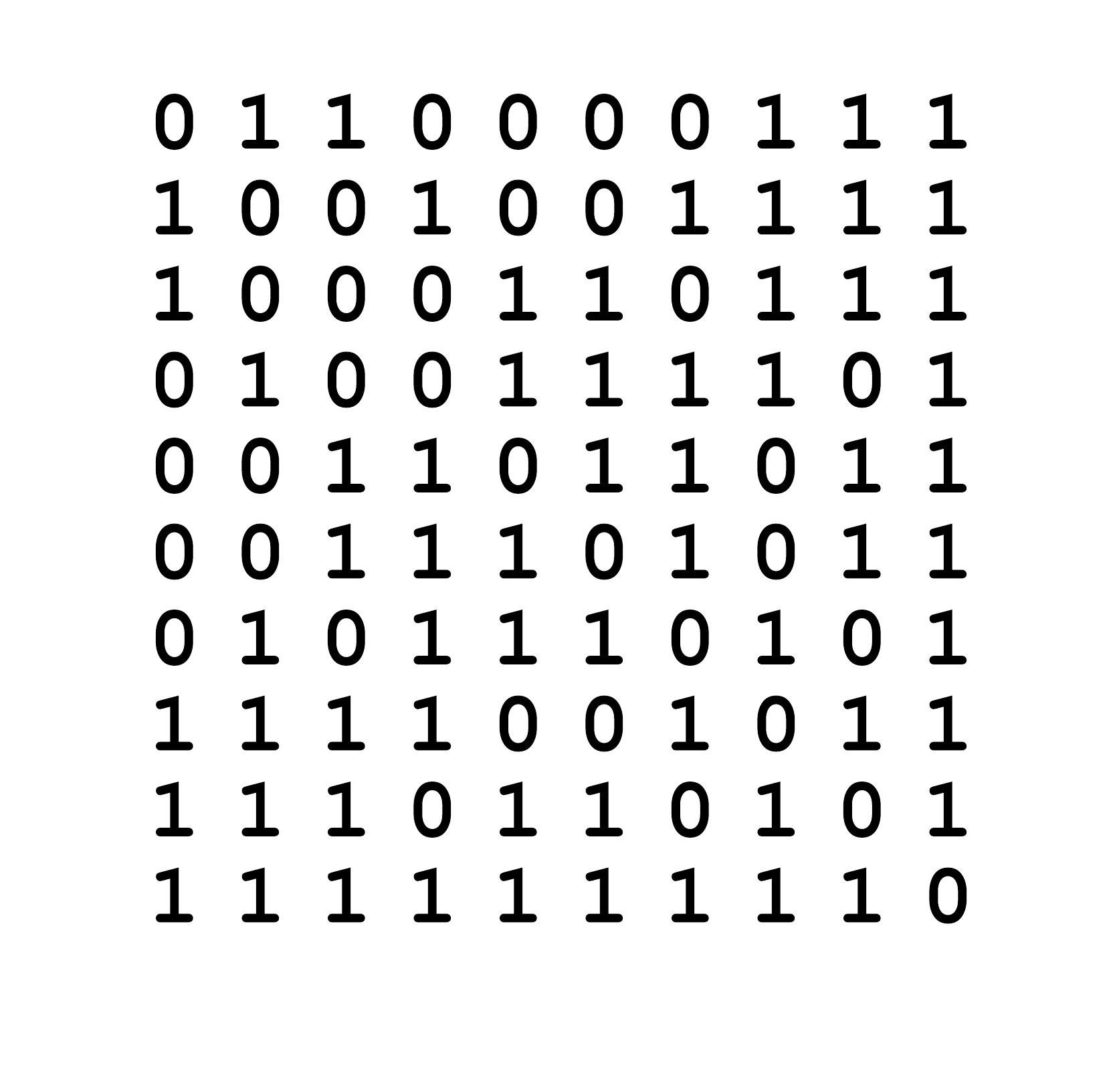}
		\vspace{-1em}
		\caption*{$G_{10.5}$}
		\label{figure: 10_5}
	\end{subfigure}\hfill
	\begin{subfigure}{.33\textwidth}
		\centering
		\includegraphics[trim={0 0 0 490},clip,height=100px,width=100px]{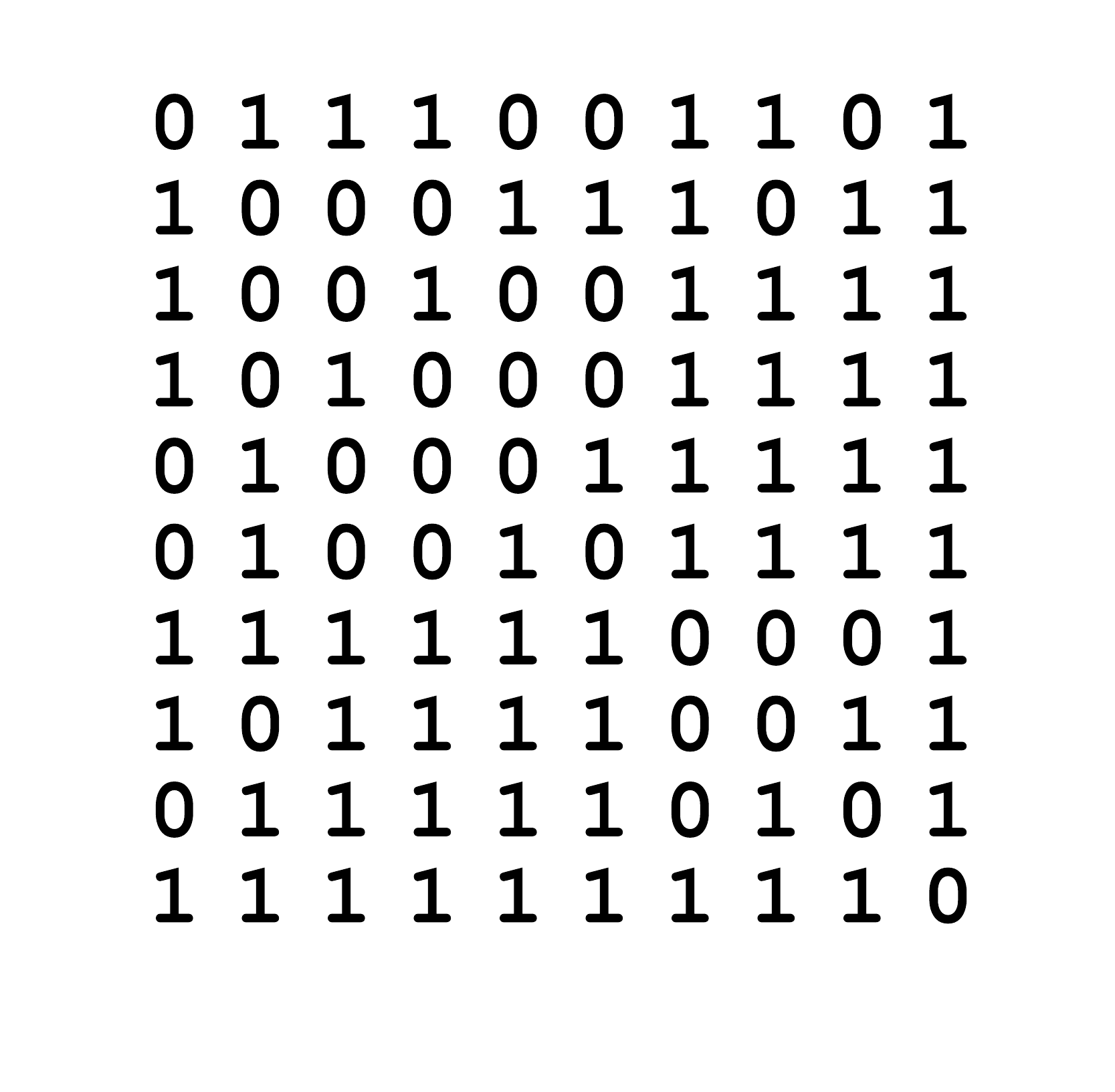}
		\vspace{-1em}
		\caption*{$G_{10.6}$}
		\label{figure: 10_6}
	\end{subfigure}
	
	\vspace{1em}
	\caption{All 6 10-vertex minimal graphs in $\mH_e(3, 3)$}
	\label{figure: 10}
\end{figure}

\begin{figure}[h]
	\captionsetup{justification=centering}
	\begin{subfigure}{.33\textwidth}
		\centering
		\includegraphics[trim={0 0 0 490},clip,height=110px,width=110px]{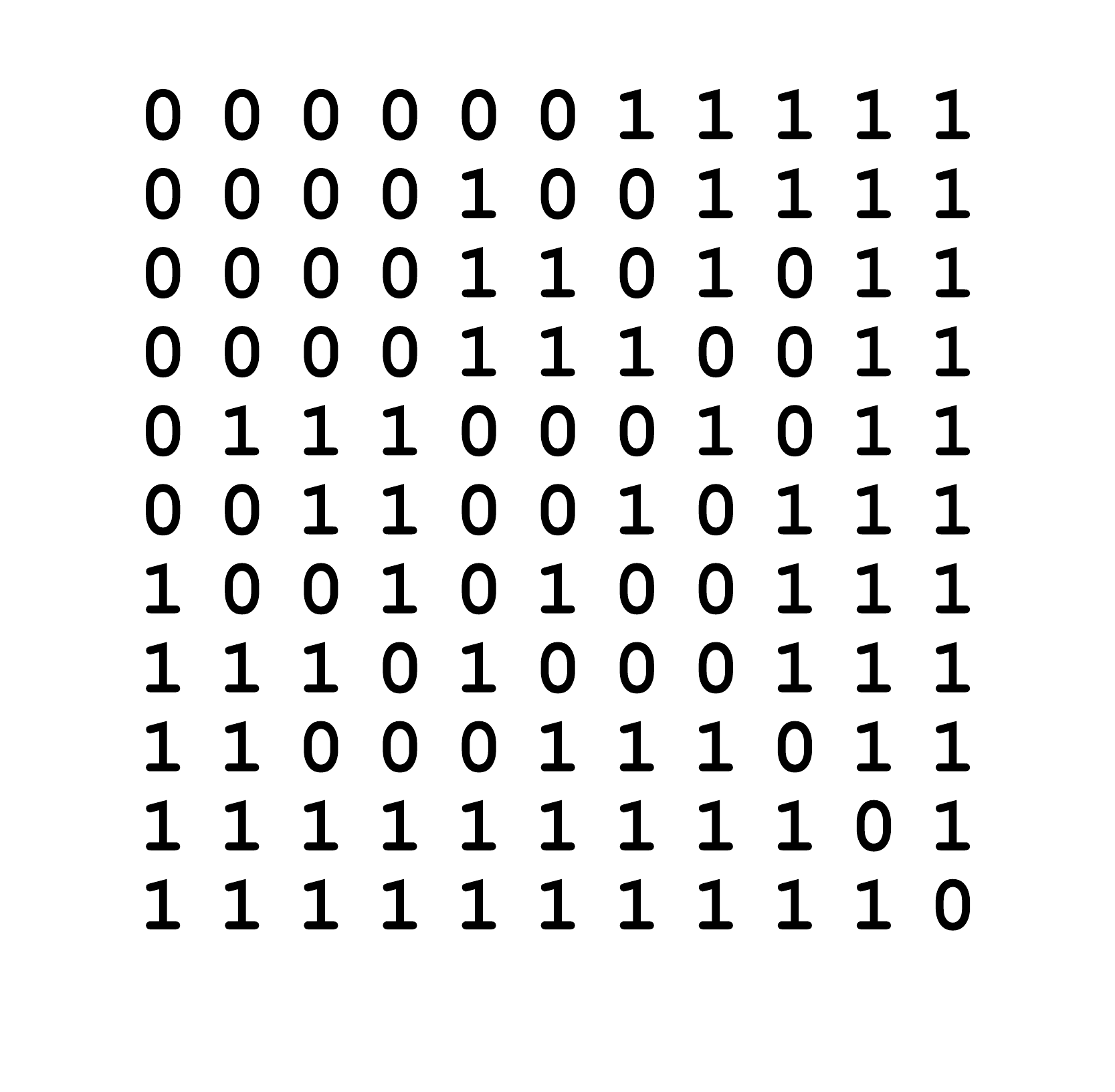}
		\vspace{-1em}
		\caption*{$G_{11.1}$}
		\label{figure: 11_1}
	\end{subfigure}\hfill
	\begin{subfigure}{.33\textwidth}
		\centering
		\includegraphics[trim={0 0 0 490},clip,height=110px,width=110px]{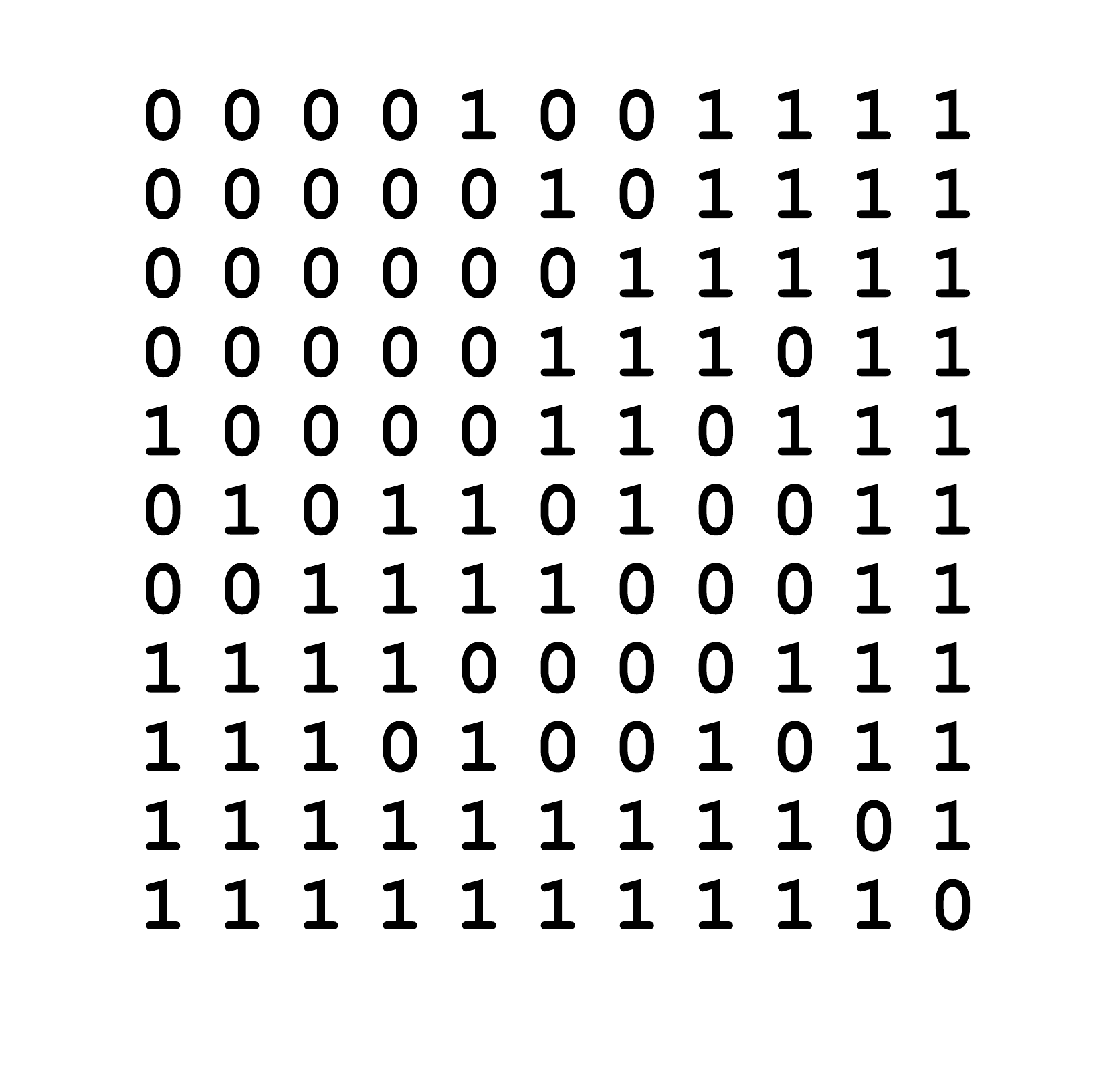}
		\vspace{-1em}
		\caption*{$G_{11.2}$}
		\label{figure: 11_2}
	\end{subfigure}\hfill
	\begin{subfigure}{.33\textwidth}
		\centering
		\includegraphics[trim={0 0 0 490},clip,height=110px,width=110px]{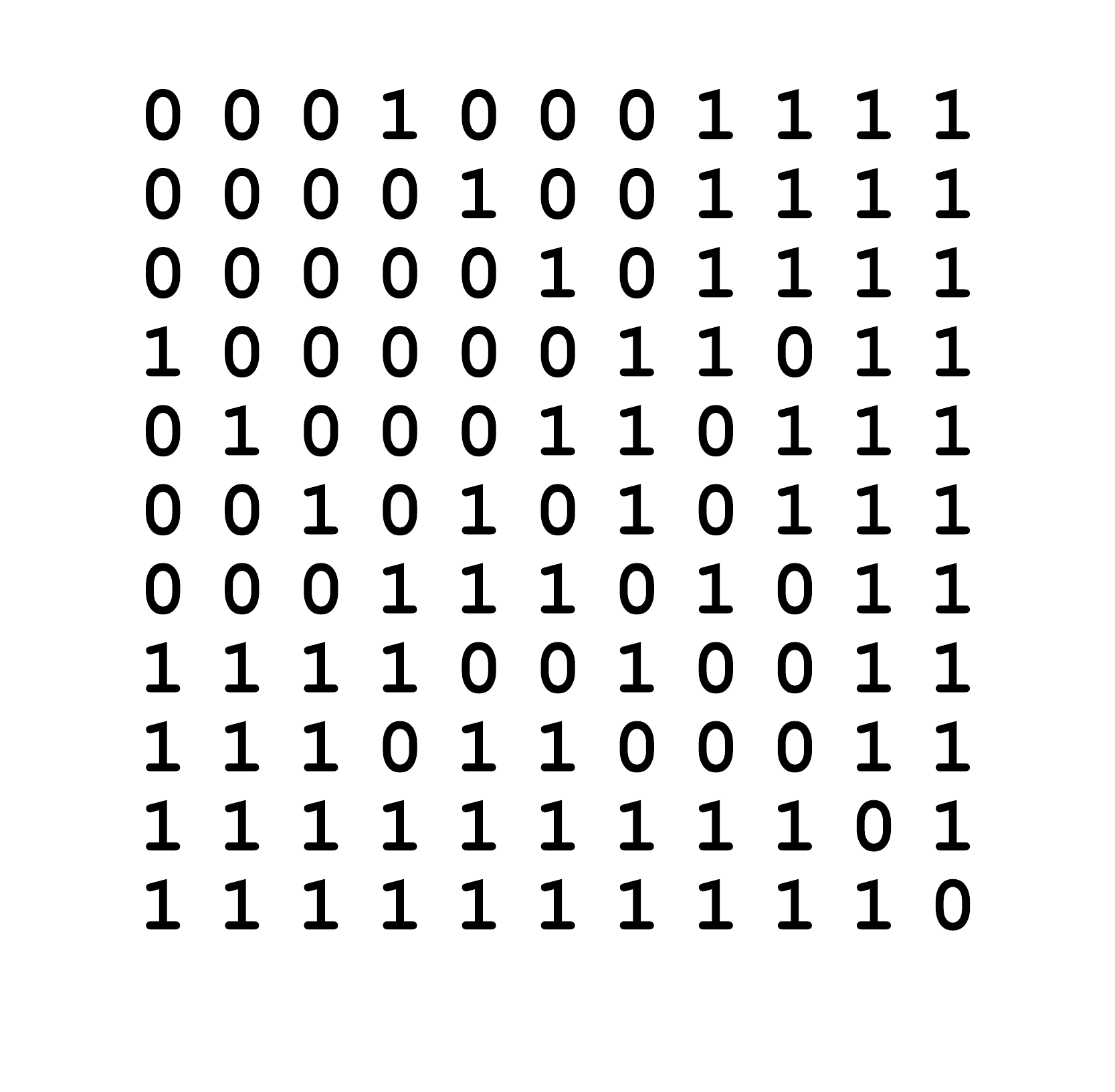}
		\vspace{-1em}
		\caption*{$G_{11.21}$}
		\label{figure: 11_21}
	\end{subfigure}
	
	\vspace{1em}
	\caption{11-vertex minimal graphs in $\mH_e(3, 3)$\\ with independence number 4}
	\label{figure: 11_a4}
\end{figure}

\begin{figure}[h]
	\captionsetup{justification=centering}
	\begin{subfigure}{.45\textwidth}
		\centering
		\includegraphics[trim={0 0 0 490},clip,height=110px,width=110px]{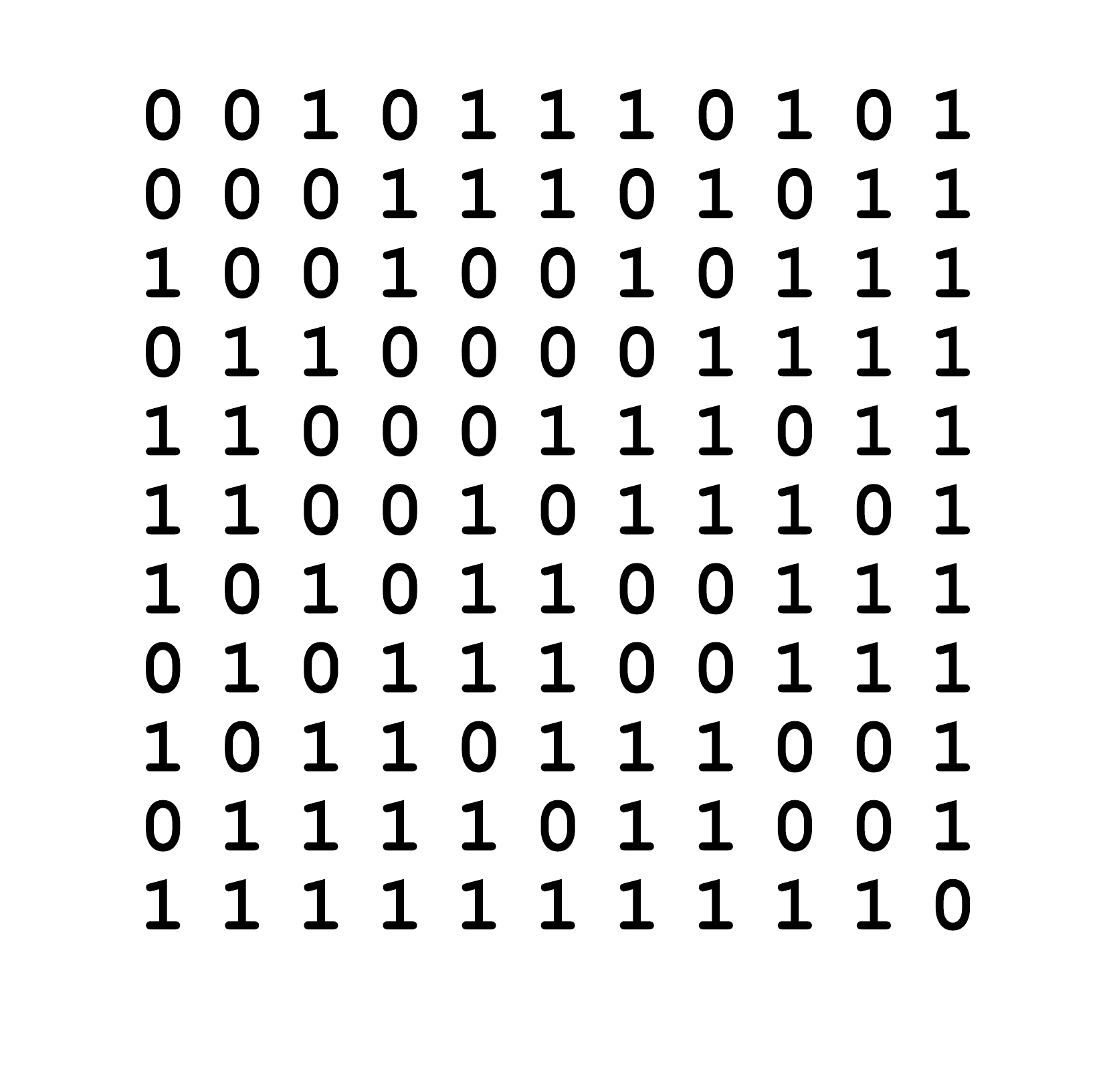}
		\vspace{-1em}
		\caption*{$G_{11.46}$}
		\label{figure: 11_46}
	\end{subfigure}\hfill
	\begin{subfigure}{.45\textwidth}
		\centering
		\includegraphics[trim={0 0 0 490},clip,height=110px,width=110px]{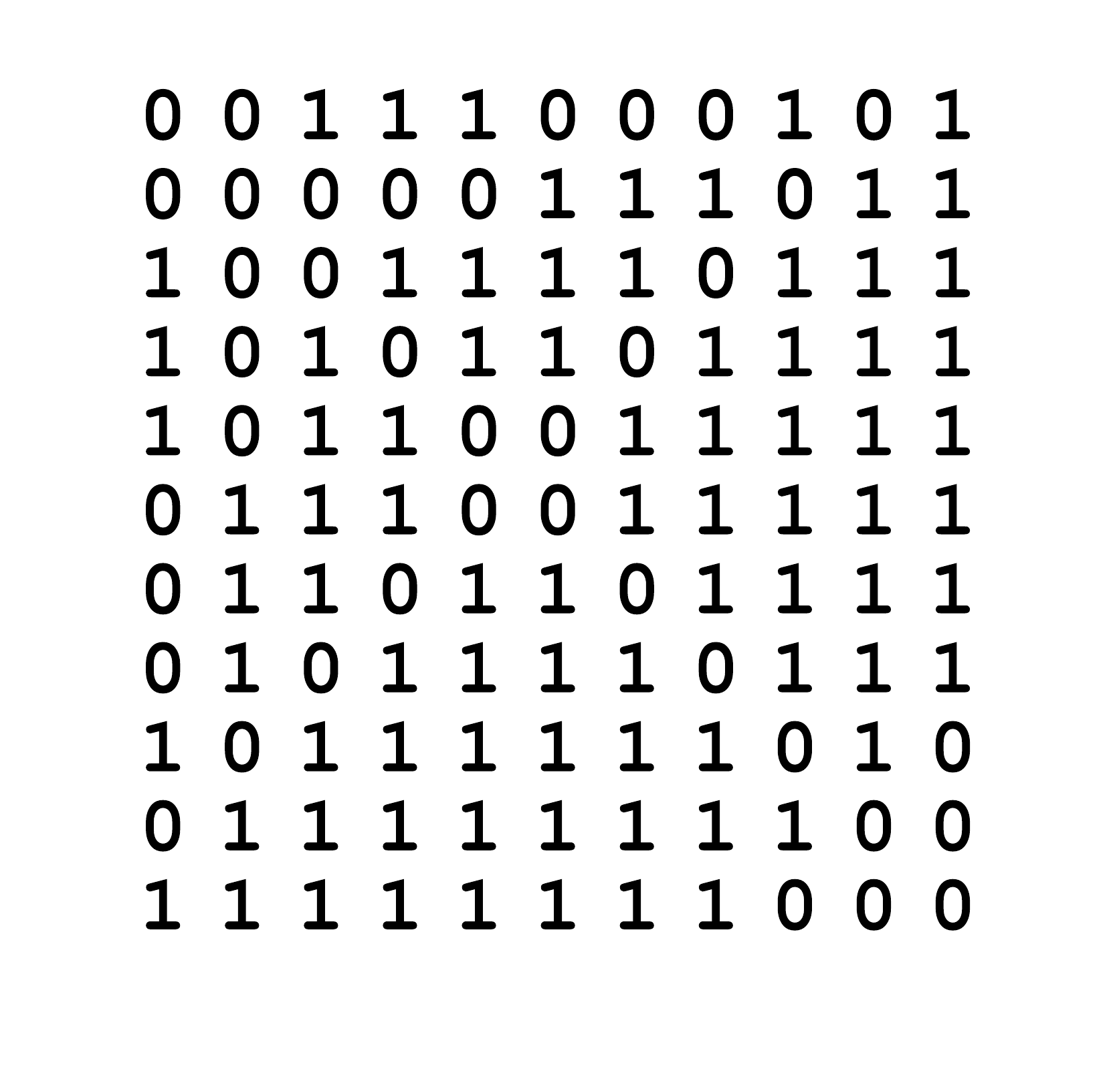}
		\vspace{-1em}
		\caption*{$G_{11.47}$}
		\label{figure: 11_47}
	\end{subfigure}
	
	\vspace{1em}
	\begin{subfigure}{.45\textwidth}
		\centering
		\includegraphics[trim={0 0 0 490},clip,height=110px,width=110px]{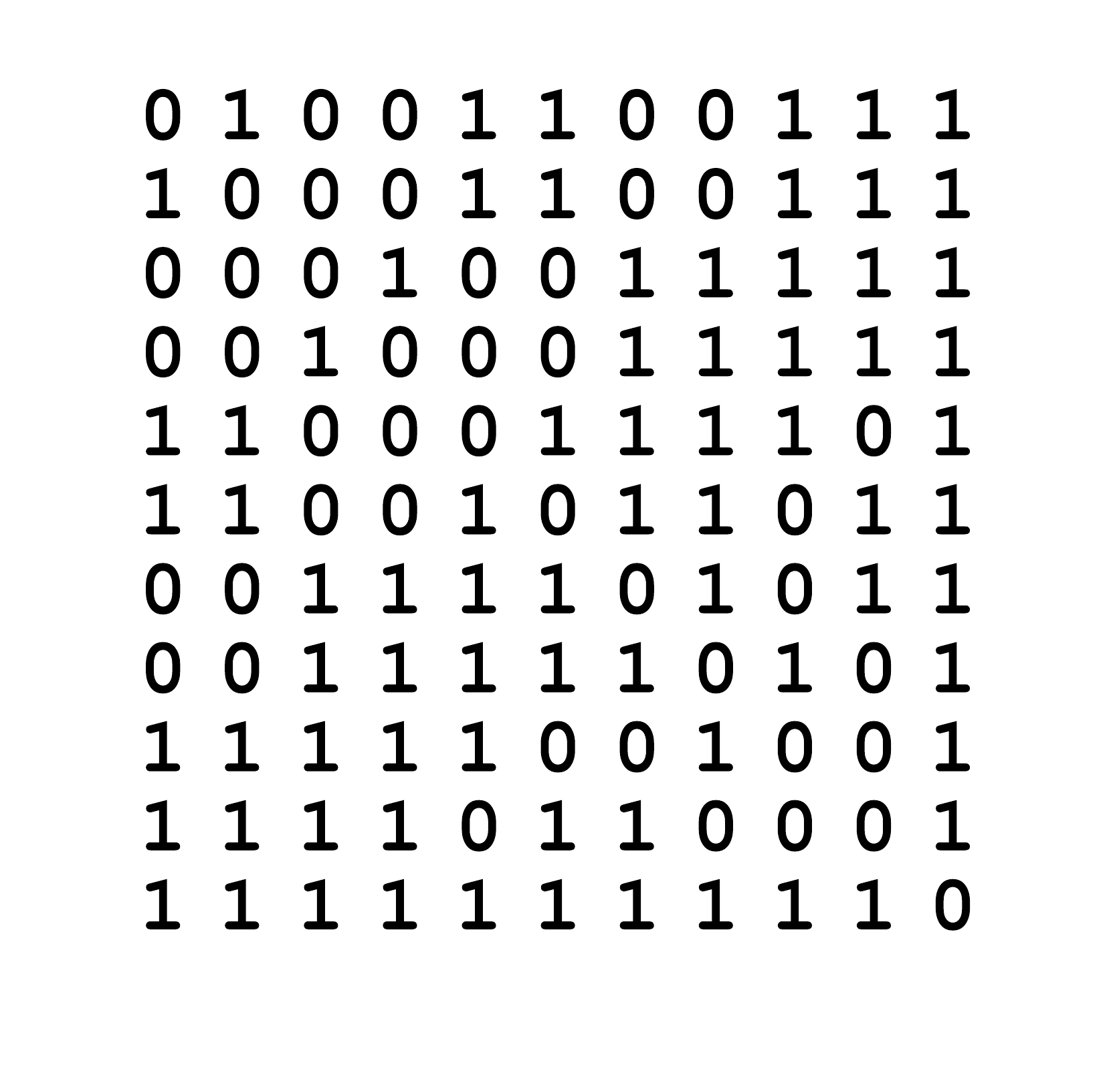}
		\vspace{-1em}
		\caption*{$G_{11.54}$}
		\label{figure: 11_54}
	\end{subfigure}\hfill
	\begin{subfigure}{.45\textwidth}
		\centering
		\includegraphics[trim={0 0 0 490},clip,height=110px,width=110px]{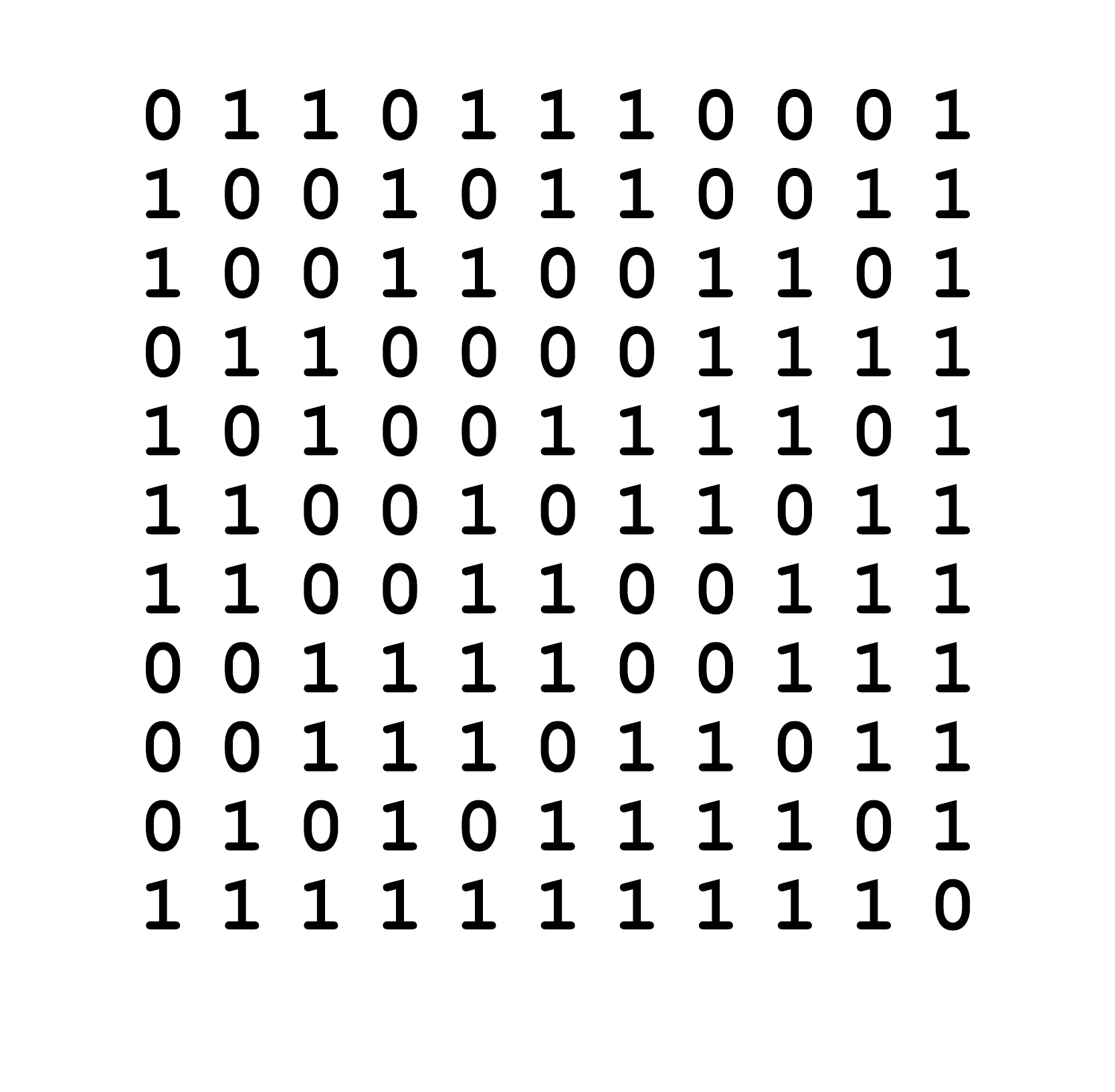}
		\vspace{-1em}
		\caption*{$G_{11.69}$}
		\label{figure: 11_69}
	\end{subfigure}
	
	\vspace{1em}
	\caption{11-vertex minimal graphs in $\mH_e(3, 3)$\\ with independence number 2}
	\label{figure: 11_a2}
\end{figure}

\begin{figure}[h]
	\captionsetup{justification=centering}
	\begin{minipage}{.45\textwidth}
		\centering
		\begin{subfigure}{\textwidth}
			\centering
			\includegraphics[trim={0 0 0 490},clip,height=120px,width=120px]{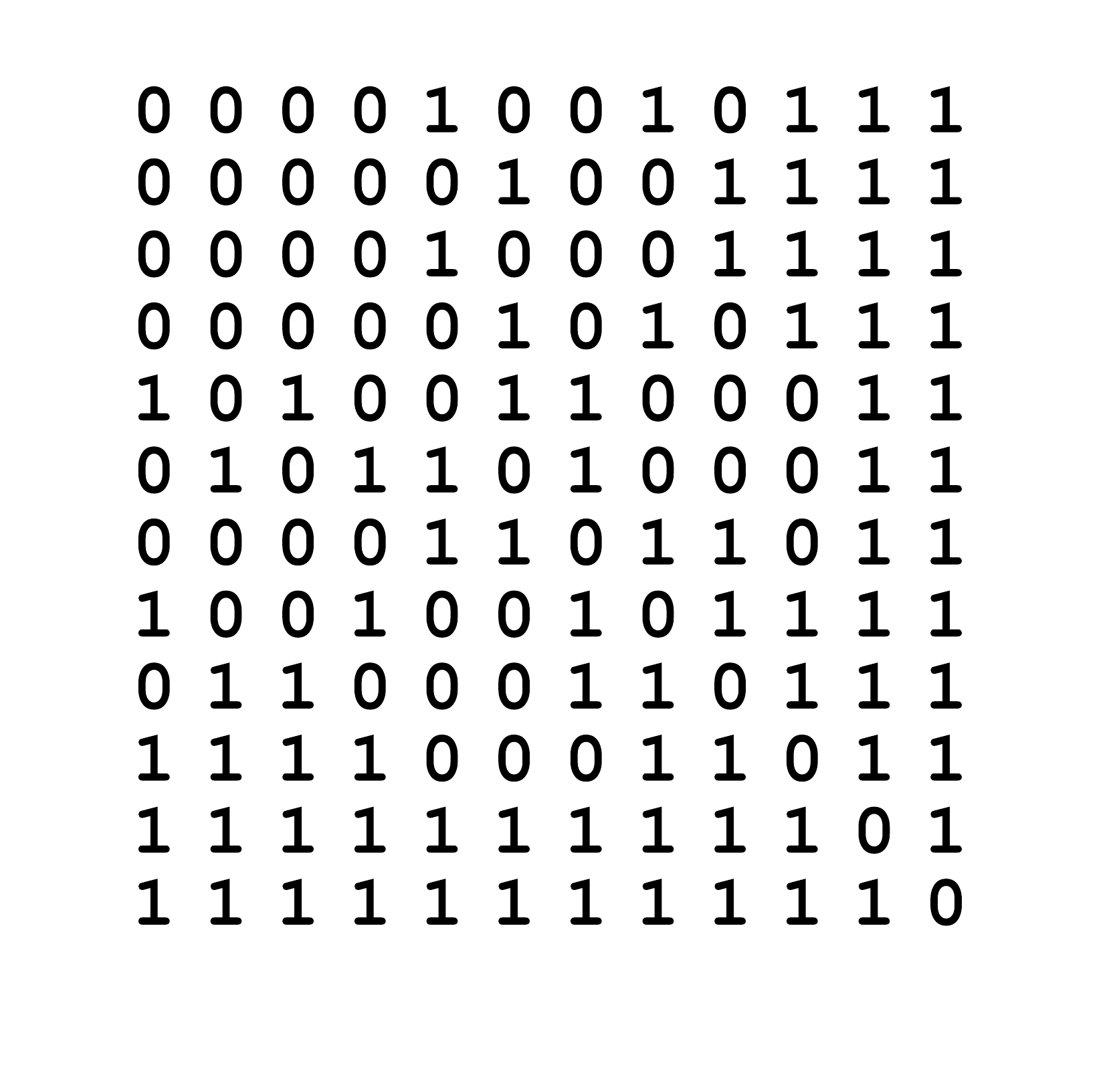}
			\vspace{-1em}
			\caption*{$G_{12.163}$}
			\label{figure: 12_163}
		\end{subfigure}
		
		\vspace{1em}
		\caption{12-vertex minimal graphs in $\mH_e(3, 3)$\\ with independence number 5}
		\label{figure: 12_a5}
	\end{minipage}\hfill
	\begin{minipage}{.45\textwidth}
		\centering
		\begin{subfigure}{\textwidth}
			\centering
			\includegraphics[trim={0 0 0 490},clip,height=120px,width=120px]{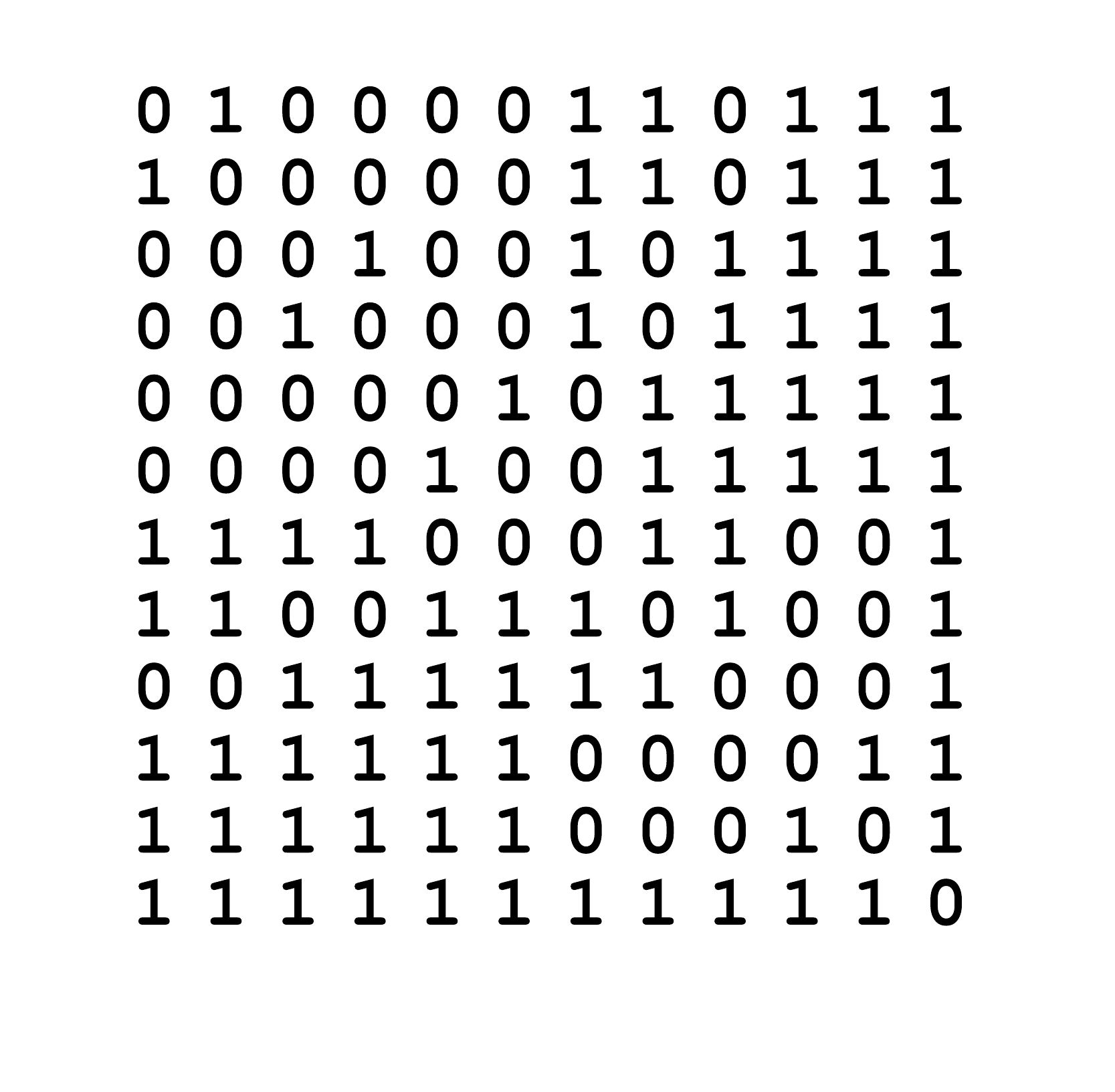}
			\vspace{-1em}
			\caption*{$G_{12.2240}$}
			\label{figure: 12_2240}
		\end{subfigure}
		
		\vspace{1em}
		\caption{12-vertex minimal graphs in $\mH_e(3, 3)$\\ with 96 automorphisms}
		\label{figure: 12_aut}
	\end{minipage}\hfill
\end{figure}

\begin{figure}[h]
	\captionsetup{justification=centering}
	\begin{subfigure}{.45\textwidth}
		\centering
		\includegraphics[trim={0 0 0 490},clip,height=130px,width=130px]{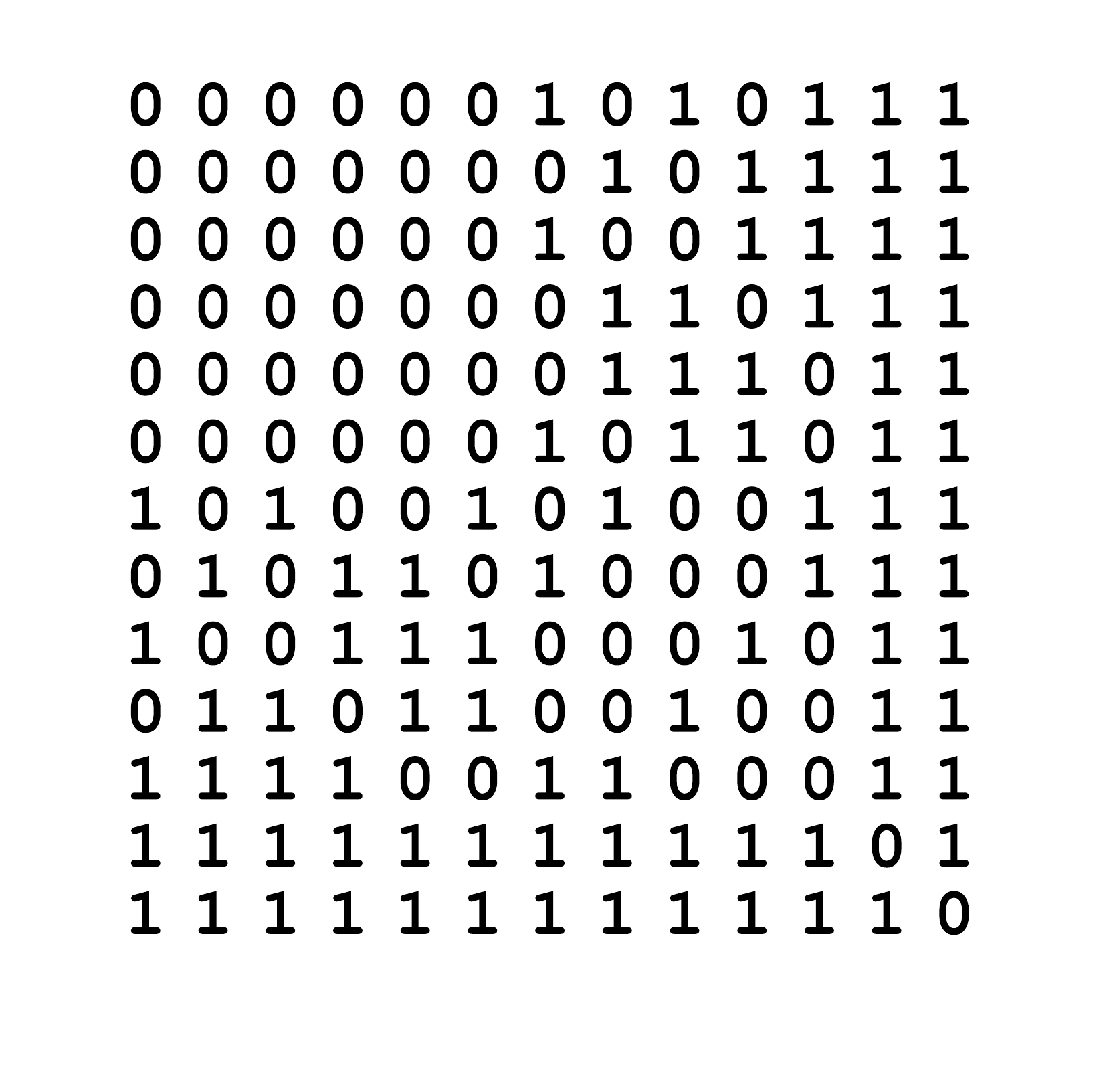}
		\vspace{-1em}
		\caption*{$G_{13.1}$}
		\label{figure: 13_1}
	\end{subfigure}\hfill
	\begin{subfigure}{.45\textwidth}
		\centering
		\includegraphics[trim={0 0 0 490},clip,height=130px,width=130px]{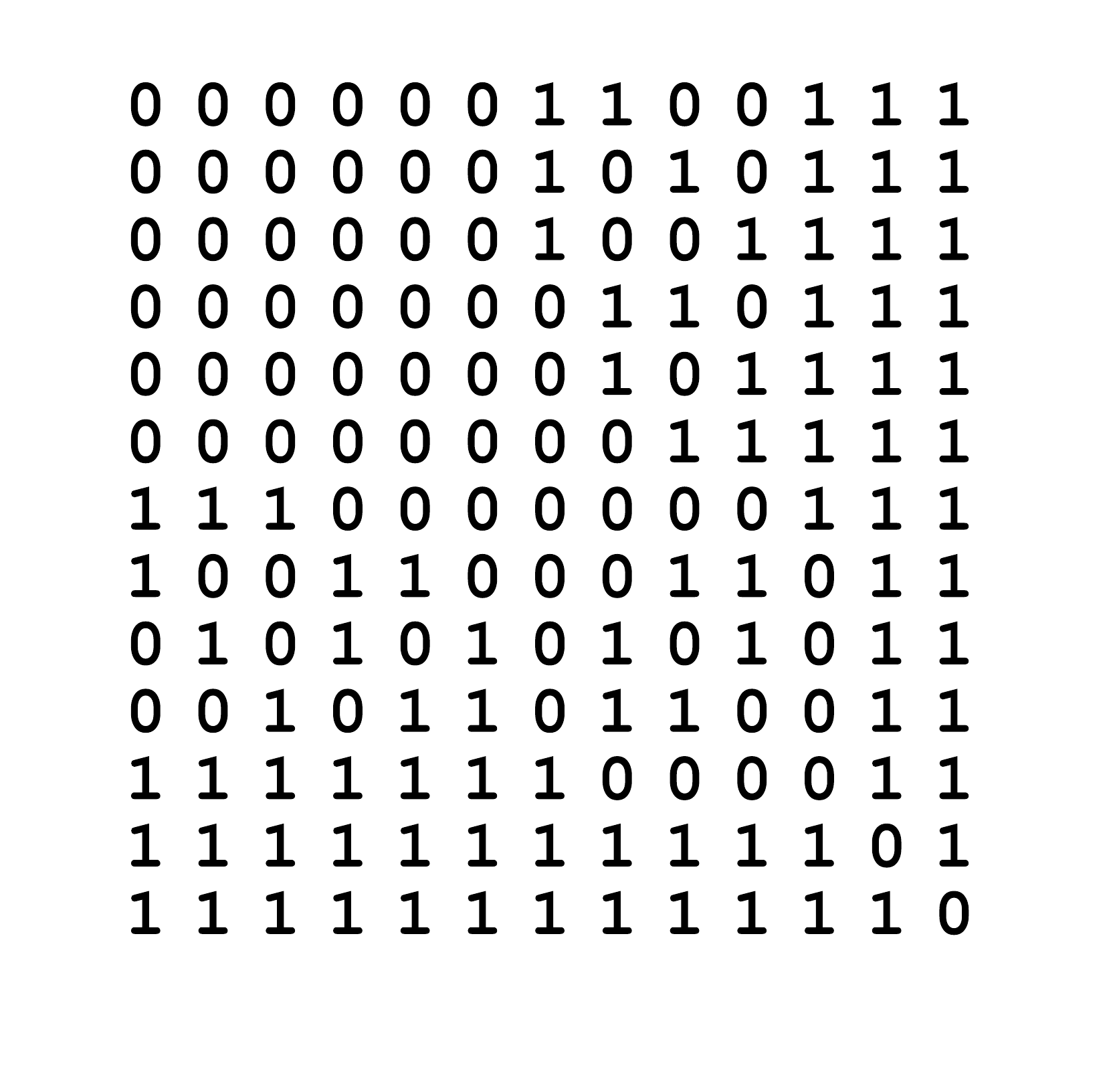}
		\vspace{-1em}
		\caption*{$G_{13.2}$}
		\label{figure: 13_2}
	\end{subfigure}
	
	\vspace{1em}
	\caption{13-vertex minimal graphs in $\mH_e(3, 3)$\\ with independence number 6}
	\label{figure: 13_a6}
\end{figure}

\begin{figure}[h]
	\captionsetup{justification=centering}
	\begin{subfigure}{.33\textwidth}
		\centering
		\includegraphics[trim={0 0 0 490},clip,height=130px,width=130px]{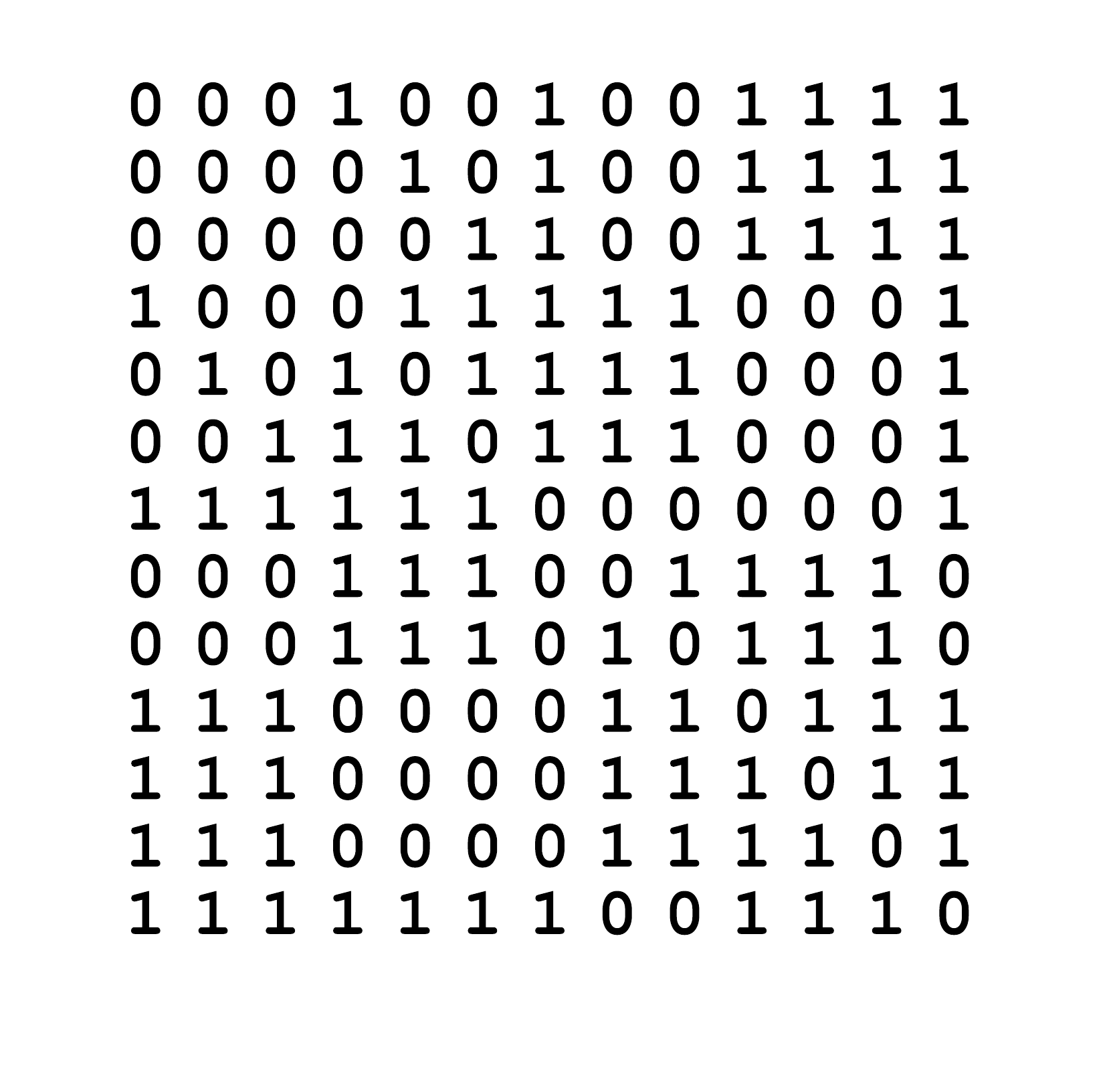}
		\vspace{-2em}
		\caption*{$G_{13.113198}$}
		\label{figure: 13_113198}
	\end{subfigure}\hfill
	\begin{subfigure}{.33\textwidth}
		\centering
		\includegraphics[trim={0 0 0 490},clip,height=130px,width=130px]{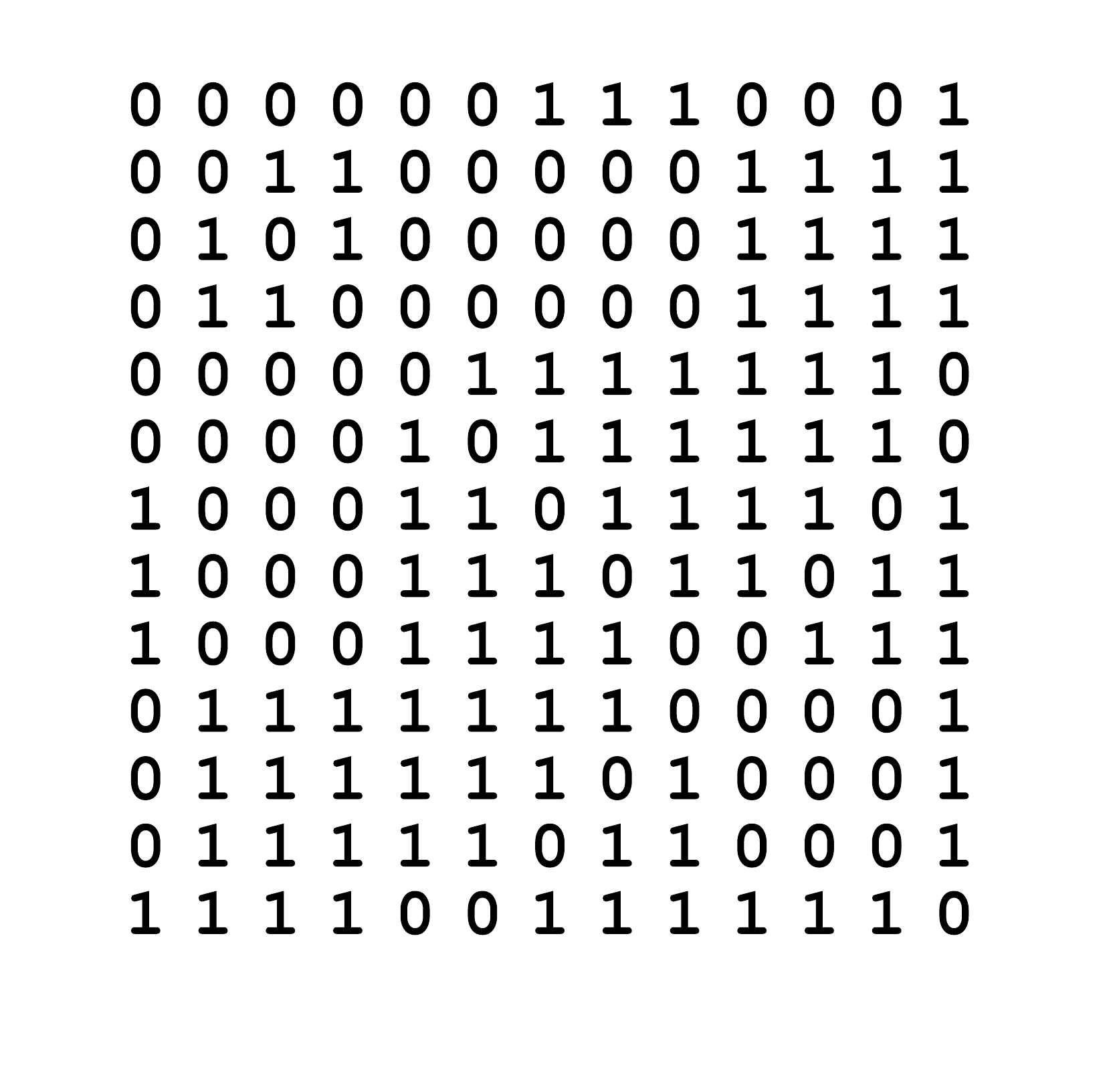}
		\vspace{-2em}
		\caption*{$G_{13.175639}$}
		\label{figure: 13_175639}
	\end{subfigure}\hfill
	\begin{subfigure}{.33\textwidth}
		\centering
		\includegraphics[trim={0 0 0 490},clip,height=130px,width=130px]{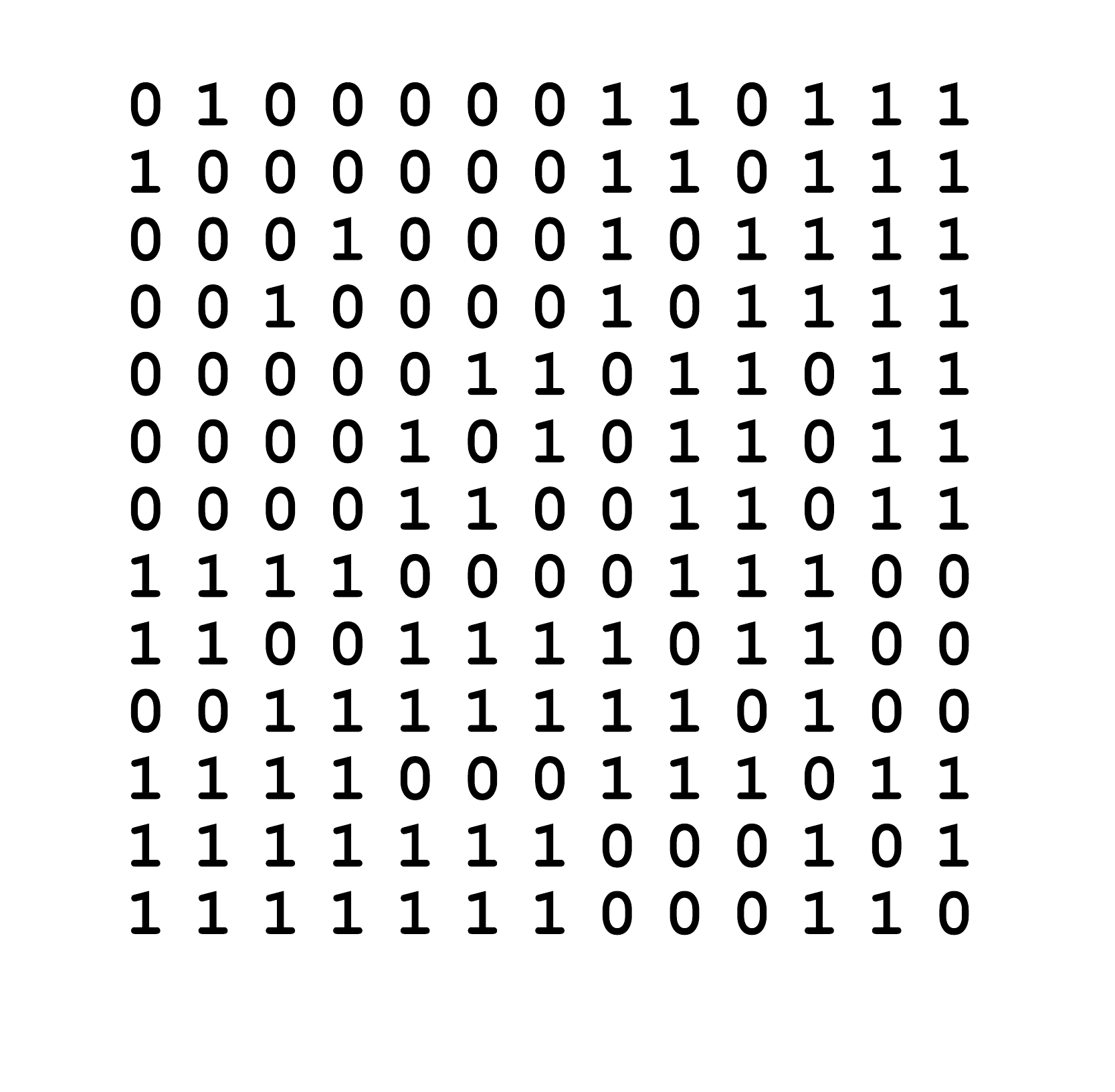}
		\vspace{-2em}
		\caption*{$G_{13.248305}$}
		\label{figure: 13_248305}
	\end{subfigure}
	
	\vspace{1em}
	\begin{subfigure}{.33\textwidth}
		\centering
		\includegraphics[trim={0 0 0 490},clip,height=130px,width=130px]{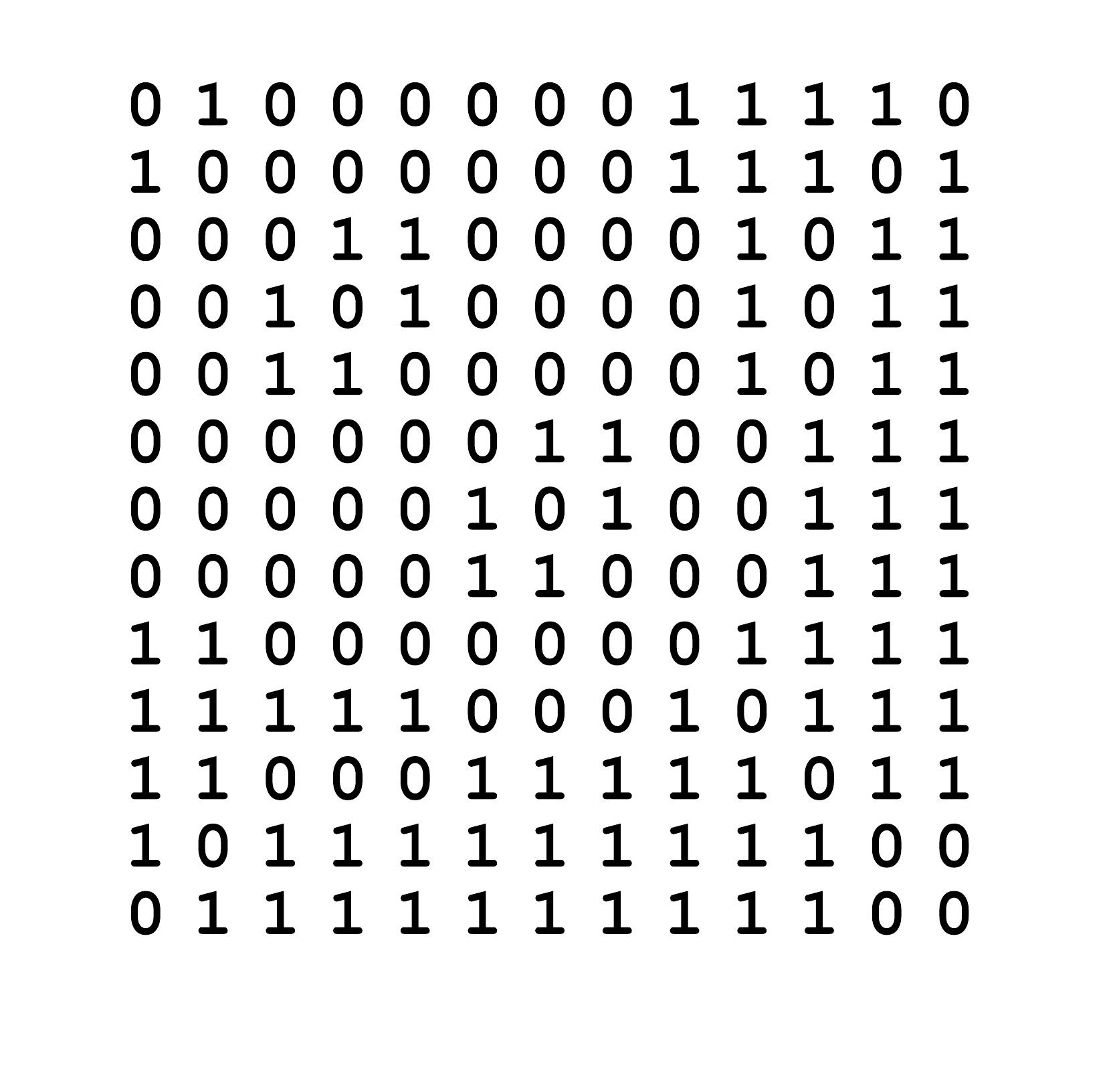}
		\vspace{-2em}
		\caption*{$G_{13.255653}$}
		\label{figure: 13_255653}
	\end{subfigure}\hfill
	\begin{subfigure}{.33\textwidth}
		\centering
		\includegraphics[trim={0 0 0 490},clip,height=130px,width=130px]{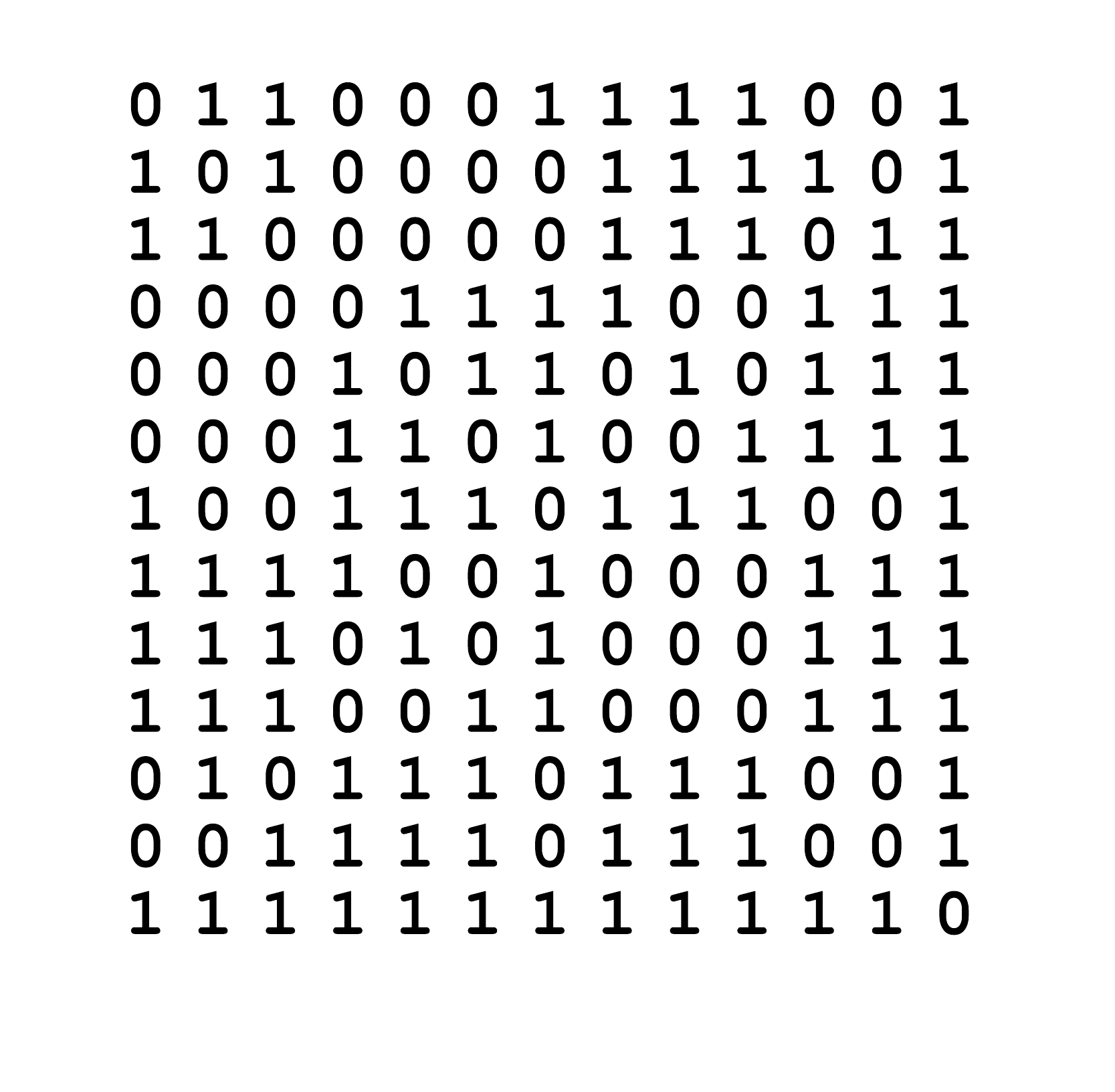}
		\vspace{-2em}
		\caption*{$G_{13.302168}$}
		\label{figure: 13_302168}
	\end{subfigure}\hfill
	\begin{subfigure}{.33\textwidth}
		\centering
		\includegraphics[trim={0 0 0 490},clip,height=130px,width=130px]{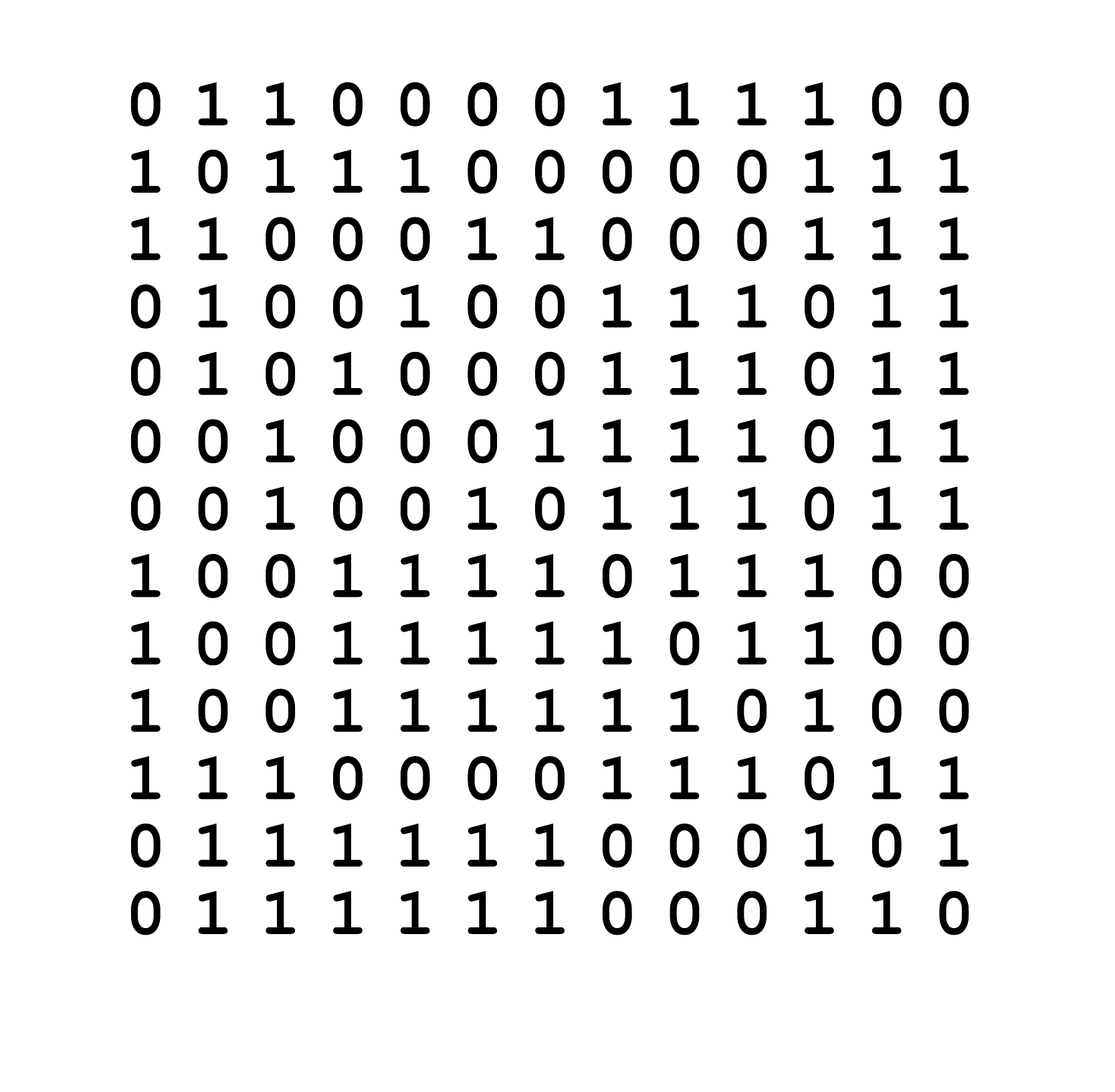}
		\vspace{-2em}
		\caption*{$G_{13.304826}$}
		\label{figure: 13_304826}
	\end{subfigure}
	
	\vspace{1em}
	\caption{13-vertex minimal graphs in $\mH_e(3, 3)$\\ with a large number of automorphisms}
	\label{figure: 13_aut}
\end{figure}

\begin{figure}[h]
	\captionsetup{justification=centering}
	\begin{subfigure}{.33\textwidth}
		\centering
		\includegraphics[trim={0 0 0 490},clip,height=100px,width=100px]{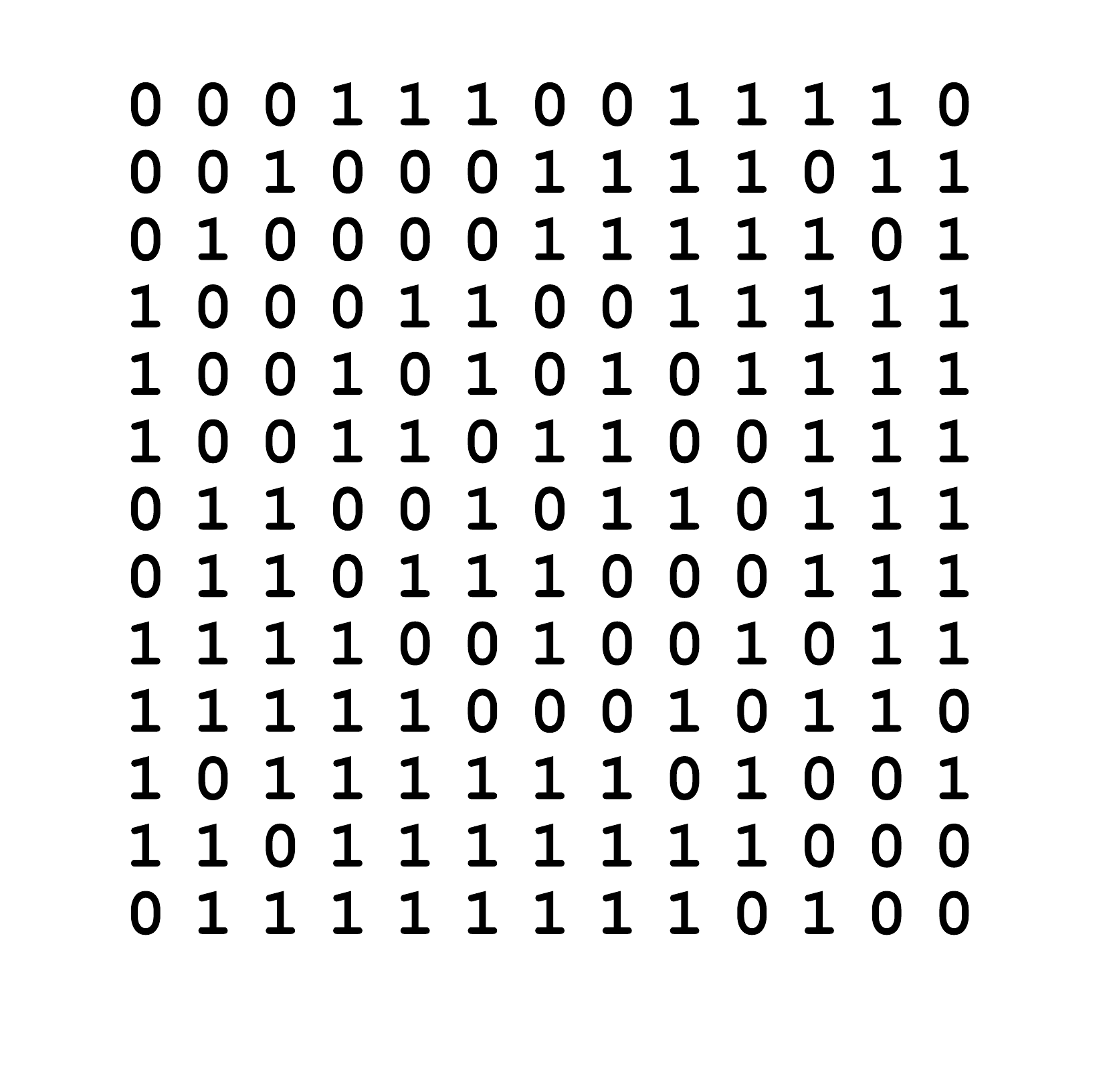}
		\vspace{-1.5em}
		\caption*{$G_{13.193684}$}
		\label{figure: 13_193684}
	\end{subfigure}\hfill
	\begin{subfigure}{.33\textwidth}
		\centering
		\includegraphics[trim={0 0 0 490},clip,height=100px,width=100px]{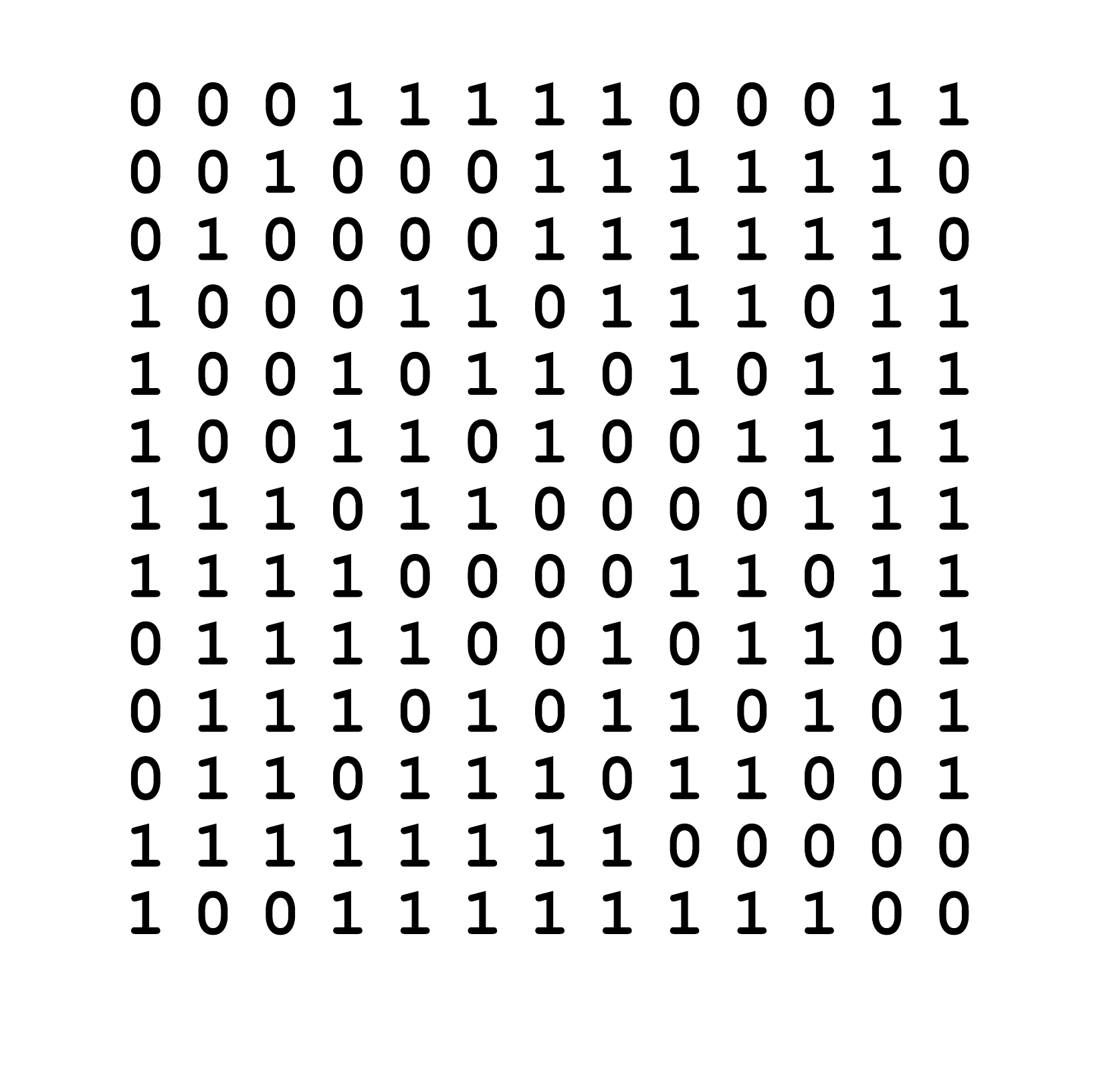}
		\vspace{-1.5em}
		\caption*{$G_{13.193760}$}
		\label{figure: 13_193760}
	\end{subfigure}\hfill
	\begin{subfigure}{.33\textwidth}
		\centering
		\includegraphics[trim={0 0 0 490},clip,height=100px,width=100px]{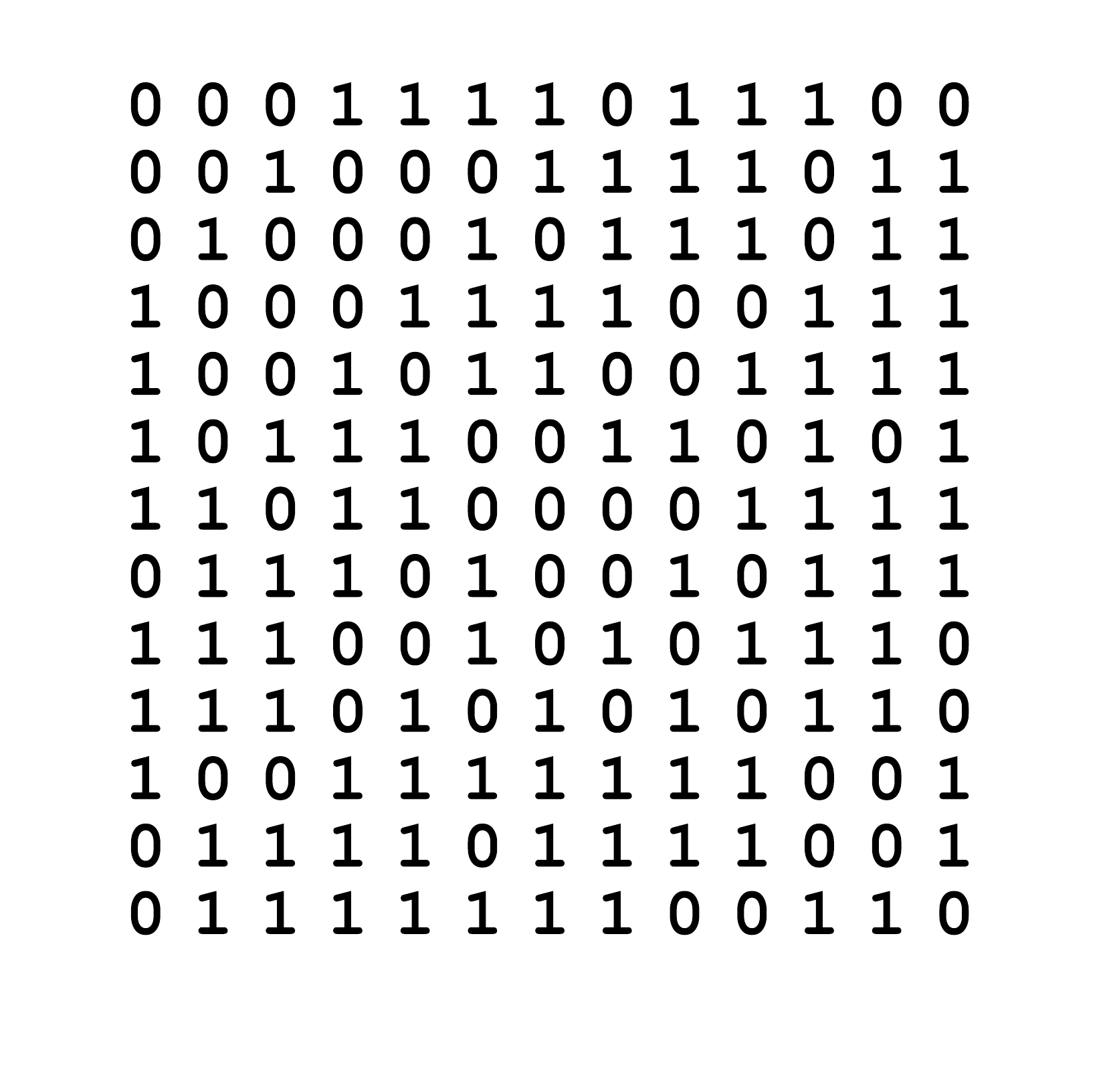}
		\vspace{-1.5em}
		\caption*{$G_{13.193988}$}
		\label{figure: 13_193988}
	\end{subfigure}
	
	\begin{subfigure}{.33\textwidth}
		\centering
		\includegraphics[trim={0 0 0 490},clip,height=100px,width=100px]{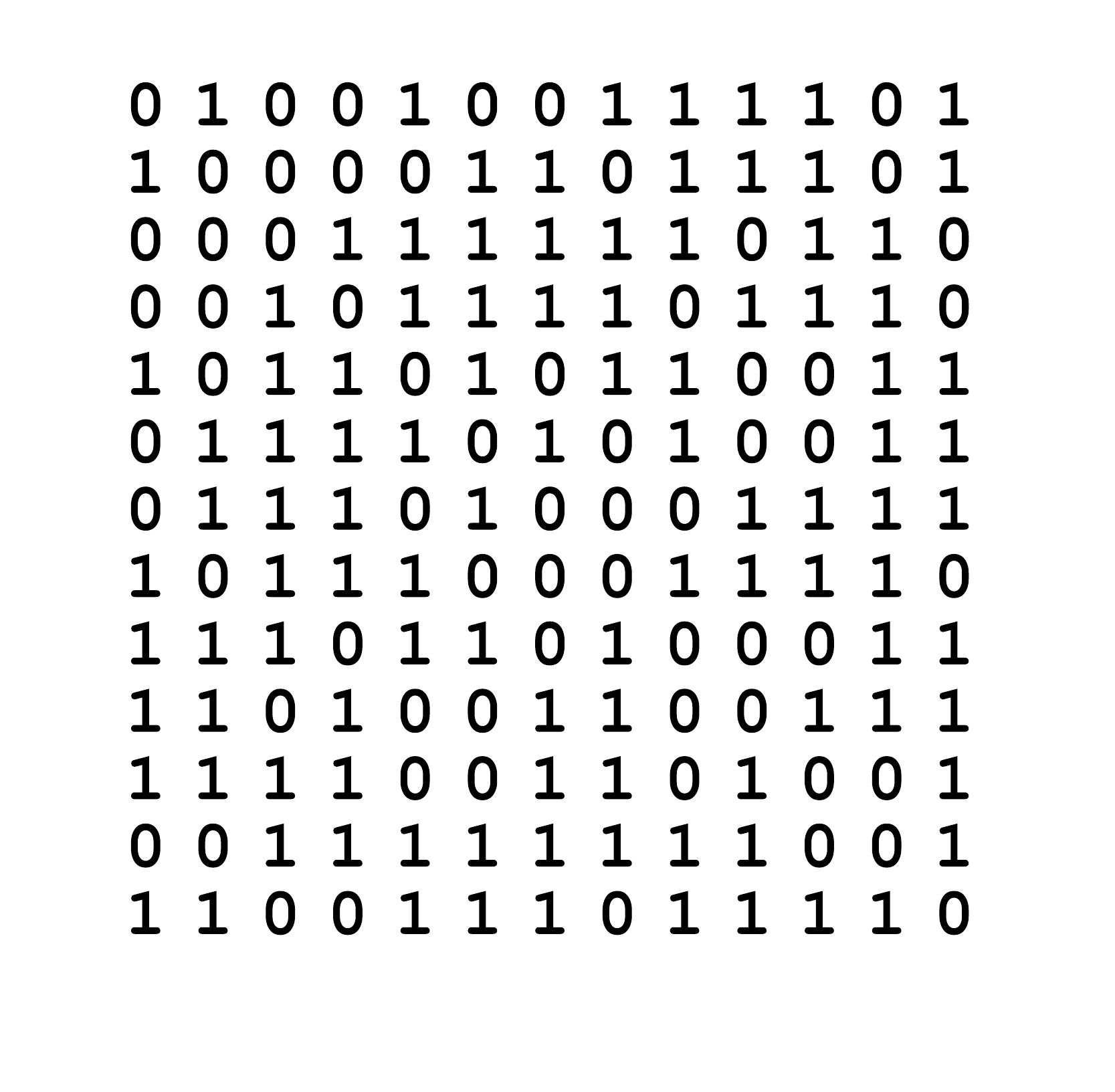}
		\vspace{-1.5em}
		\caption*{$G_{13.265221}$}
		\label{figure: 13_265221}
	\end{subfigure}\hfill
	\begin{subfigure}{.33\textwidth}
		\centering
		\includegraphics[trim={0 0 0 490},clip,height=100px,width=100px]{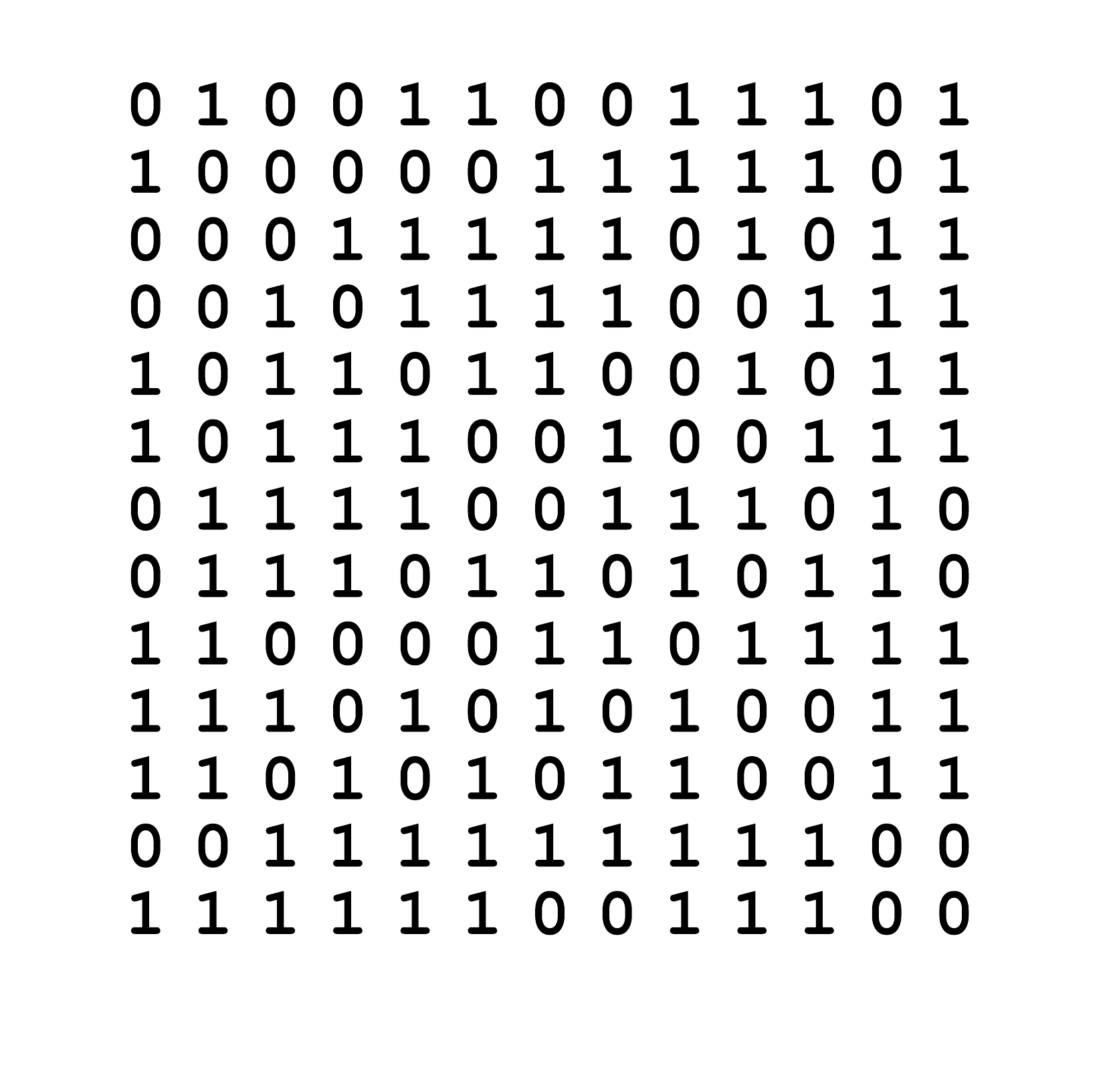}
		\vspace{-1.5em}
		\caption*{$G_{13.265299}$}
		\label{figure: 13_265299}
	\end{subfigure}\hfill
	\begin{subfigure}{.33\textwidth}
		\centering
		\includegraphics[trim={0 0 0 490},clip,height=100px,width=100px]{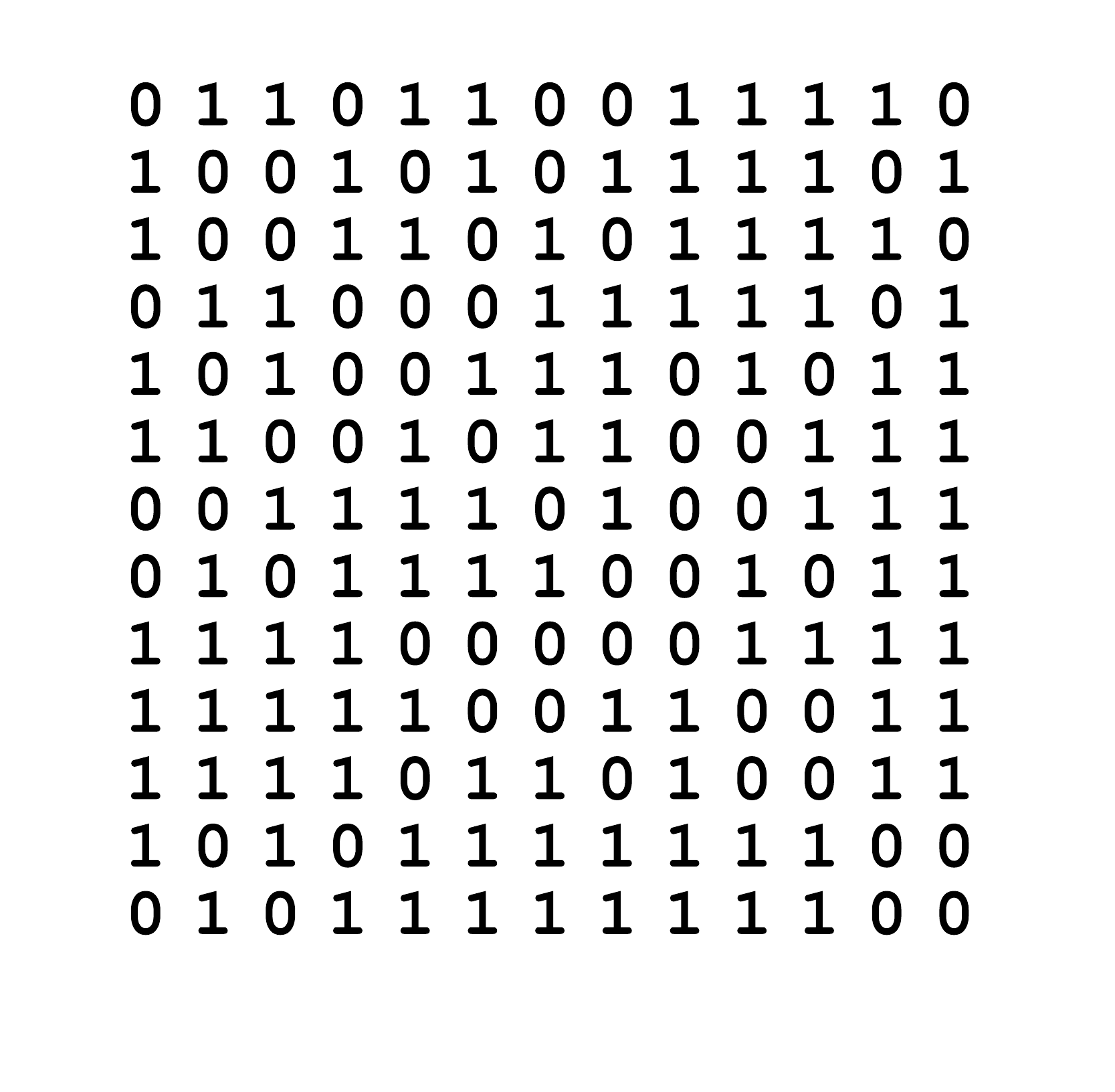}
		\vspace{-1.5em}
		\caption*{$G_{13.299797}$}
		\label{figure: 13_299797}
	\end{subfigure}
	
	\begin{subfigure}{.33\textwidth}
		\centering
		\includegraphics[trim={0 0 0 490},clip,height=100px,width=100px]{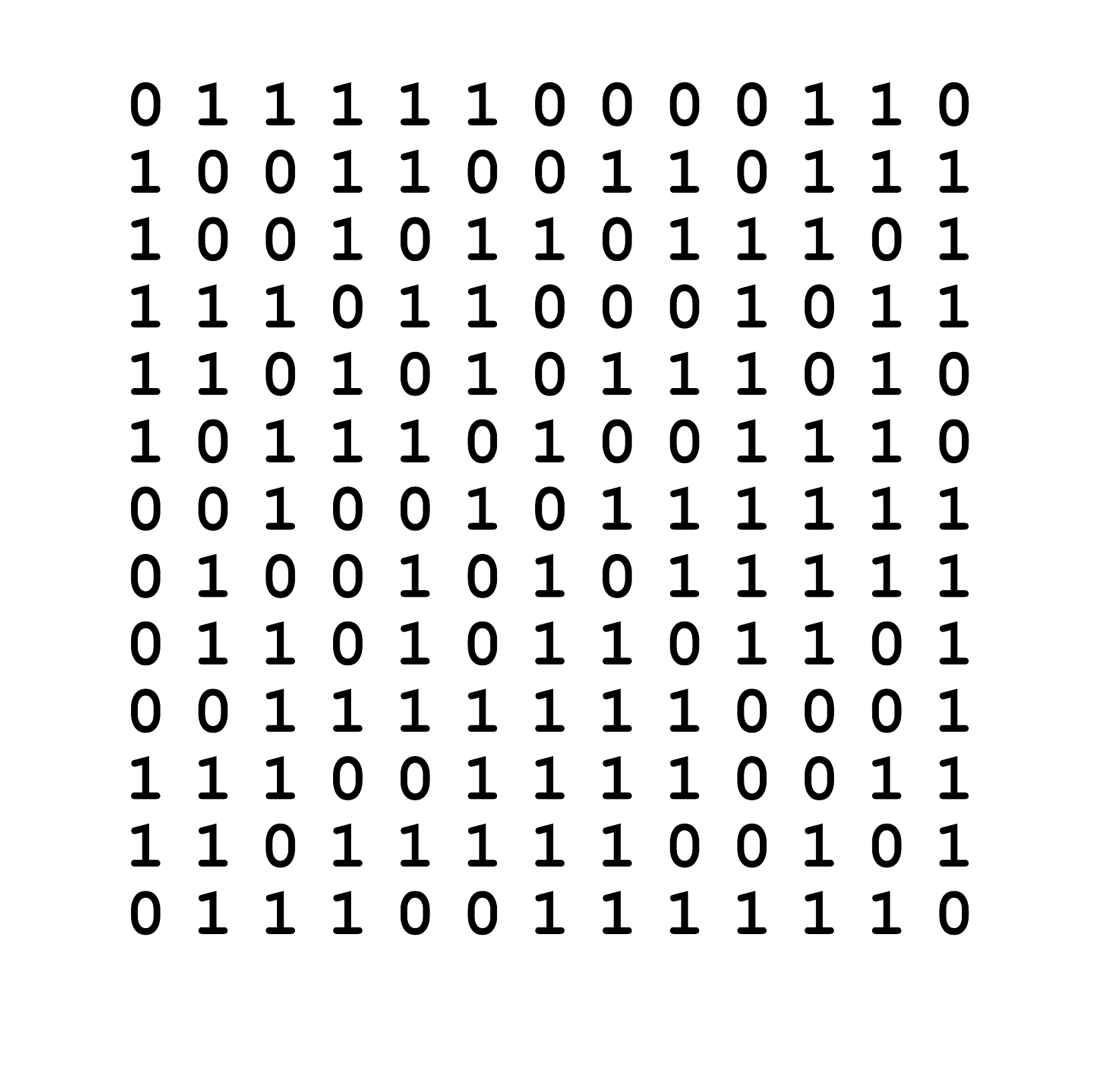}
		\vspace{-1.5em}
		\caption*{$G_{13.301368}$}
		\label{figure: 13_301368}
	\end{subfigure}\hfill
	\begin{subfigure}{.33\textwidth}
		\centering
		\includegraphics[trim={0 0 0 490},clip,height=100px,width=100px]{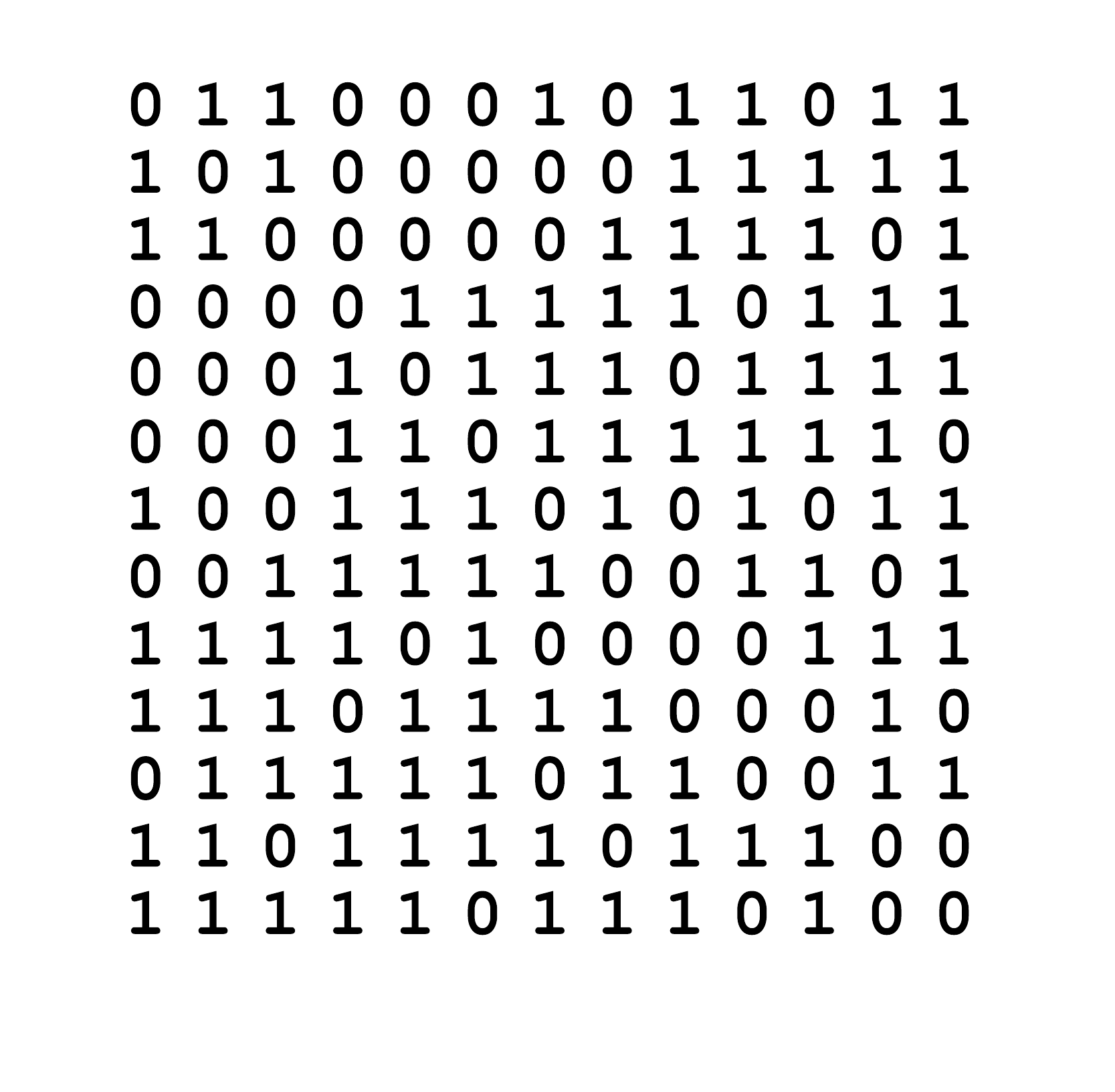}
		\vspace{-1.5em}
		\caption*{$G_{13.302151}$}
		\label{figure: 13_302151}
	\end{subfigure}\hfill
	\begin{subfigure}{.33\textwidth}
		\centering
		\includegraphics[trim={0 0 0 490},clip,height=100px,width=100px]{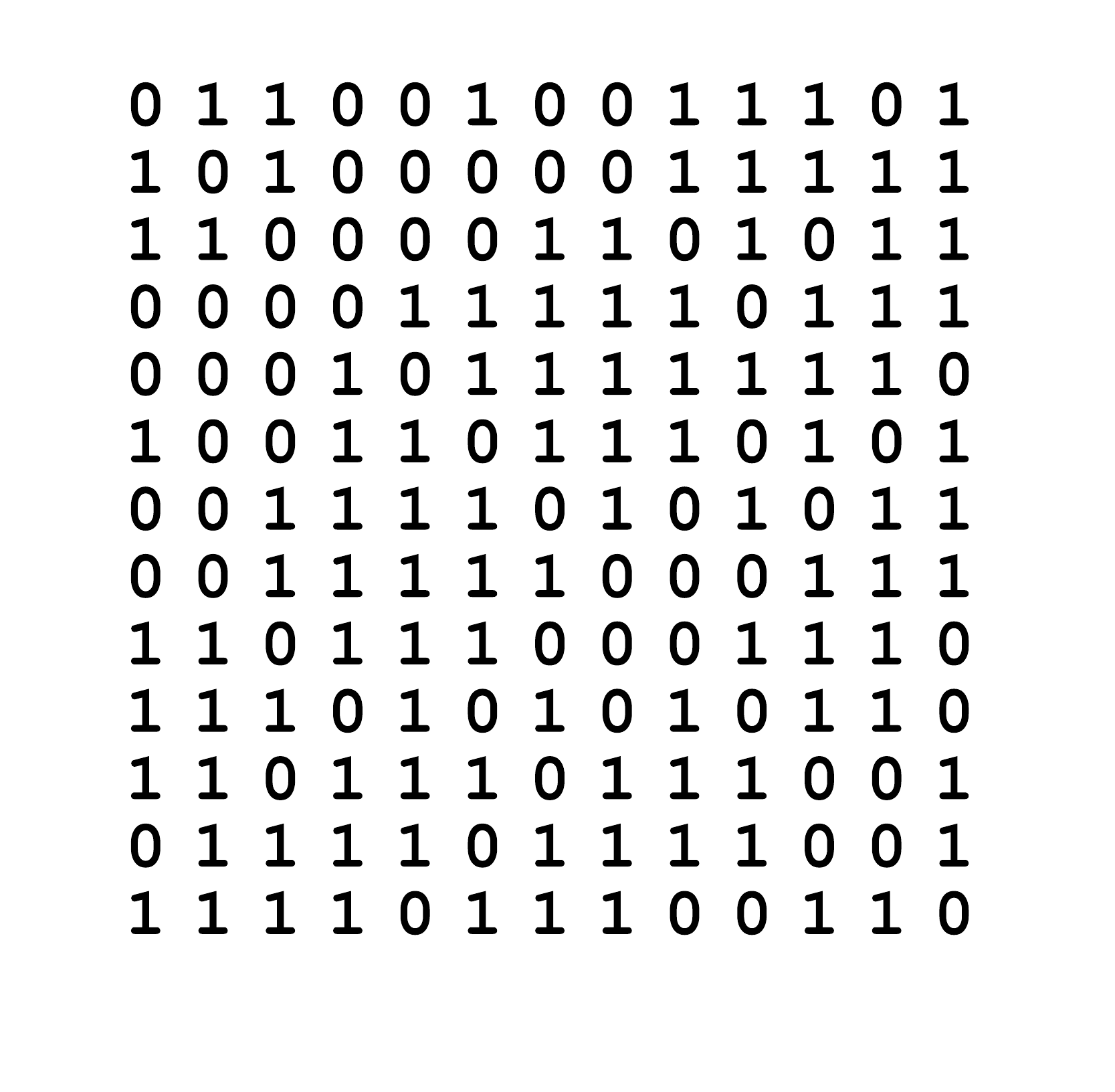}
		\vspace{-1.5em}
		\caption*{$G_{13.302764}$}
		\label{figure: 13_302764}
	\end{subfigure}
	
	\begin{subfigure}{.33\textwidth}
		\centering
		\includegraphics[trim={0 0 0 490},clip,height=100px,width=100px]{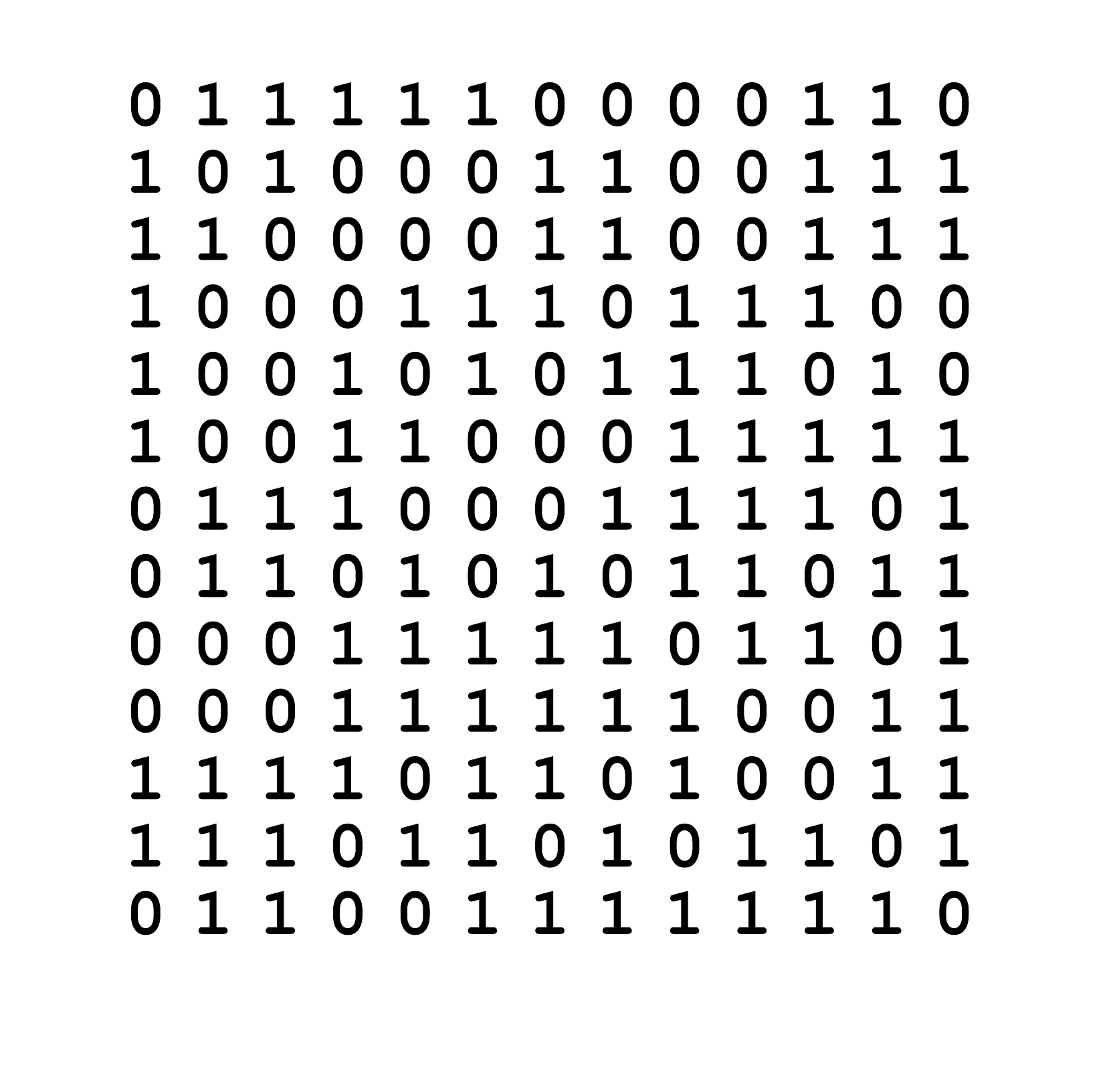}
		\vspace{-1.5em}
		\caption*{$G_{13.305857}$}
		\label{figure: 13_305857}
	\end{subfigure}\hfill
	\begin{subfigure}{.33\textwidth}
		\centering
		\includegraphics[trim={0 0 0 490},clip,height=100px,width=100px]{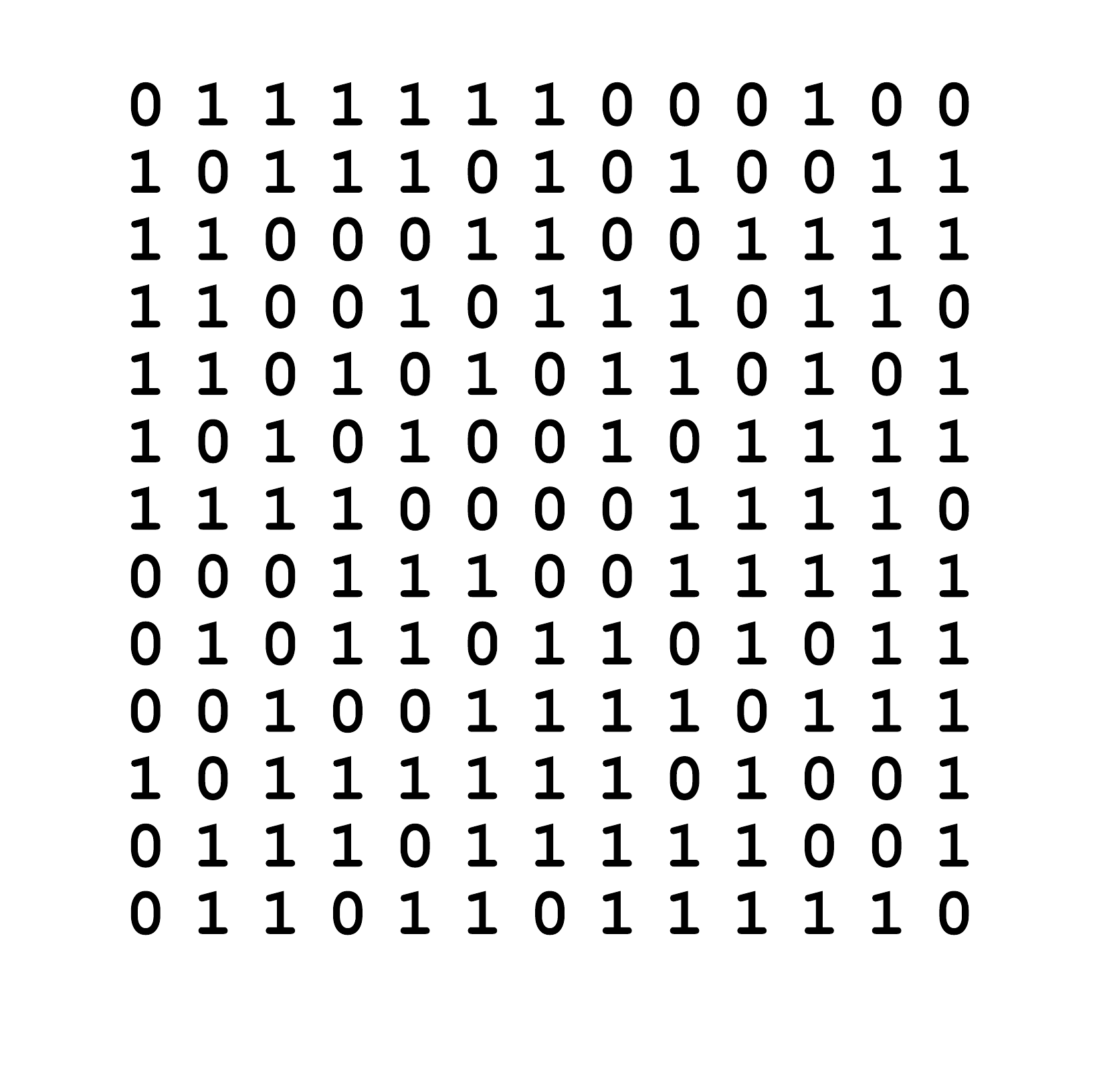}
		\vspace{-1.5em}
		\caption*{$G_{13.306448}$}
		\label{figure: 13_306448}
	\end{subfigure}\hfill
	\begin{subfigure}{.33\textwidth}
		\centering
		\includegraphics[trim={0 0 0 490},clip,height=100px,width=100px]{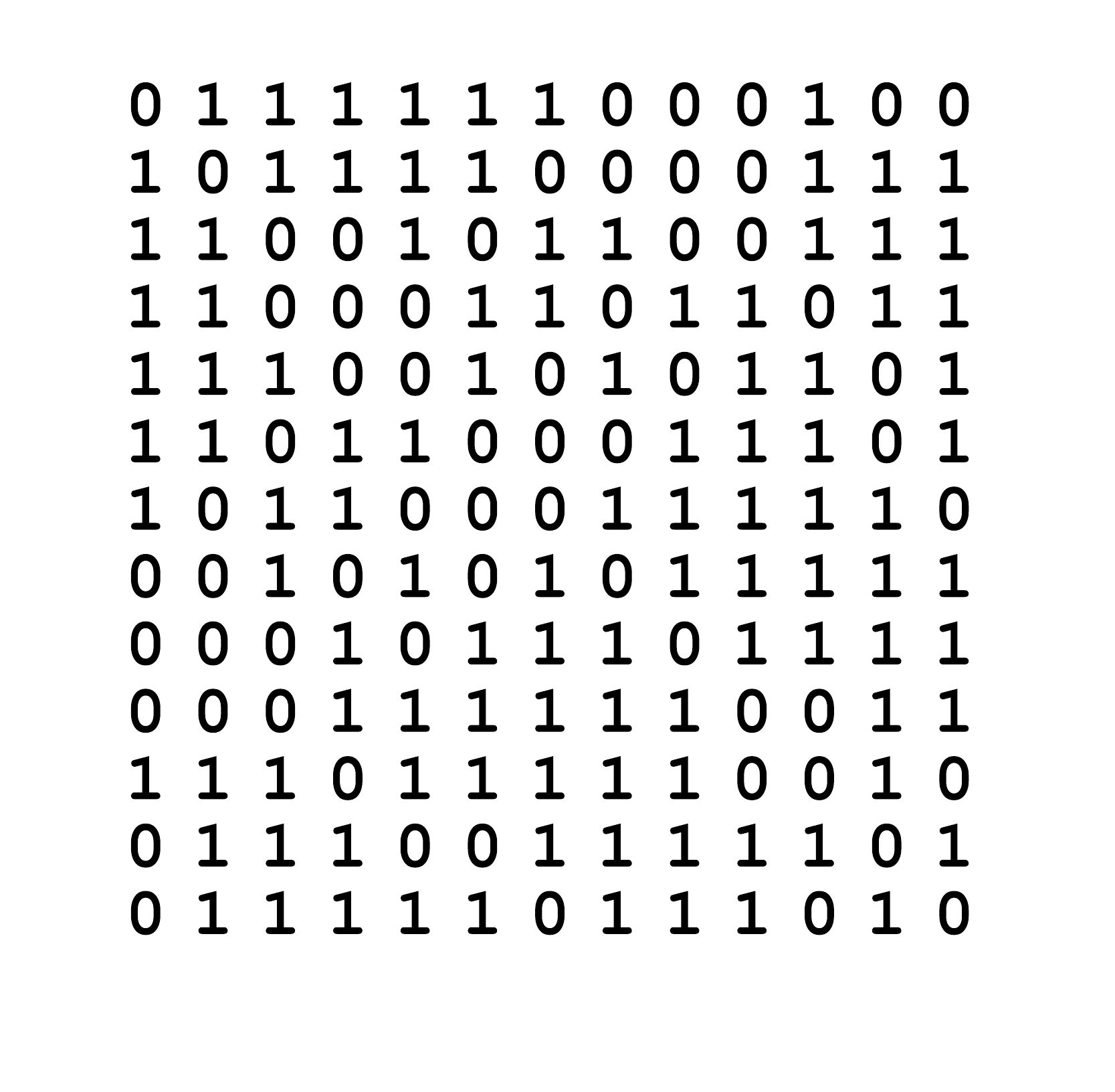}
		\vspace{-1.5em}
		\caption*{$G_{13.306460}$}
		\label{figure: 13_306460}
	\end{subfigure}
	
	\begin{subfigure}{\textwidth}
		\centering
		\includegraphics[trim={0 0 0 490},clip,height=100px,width=100px]{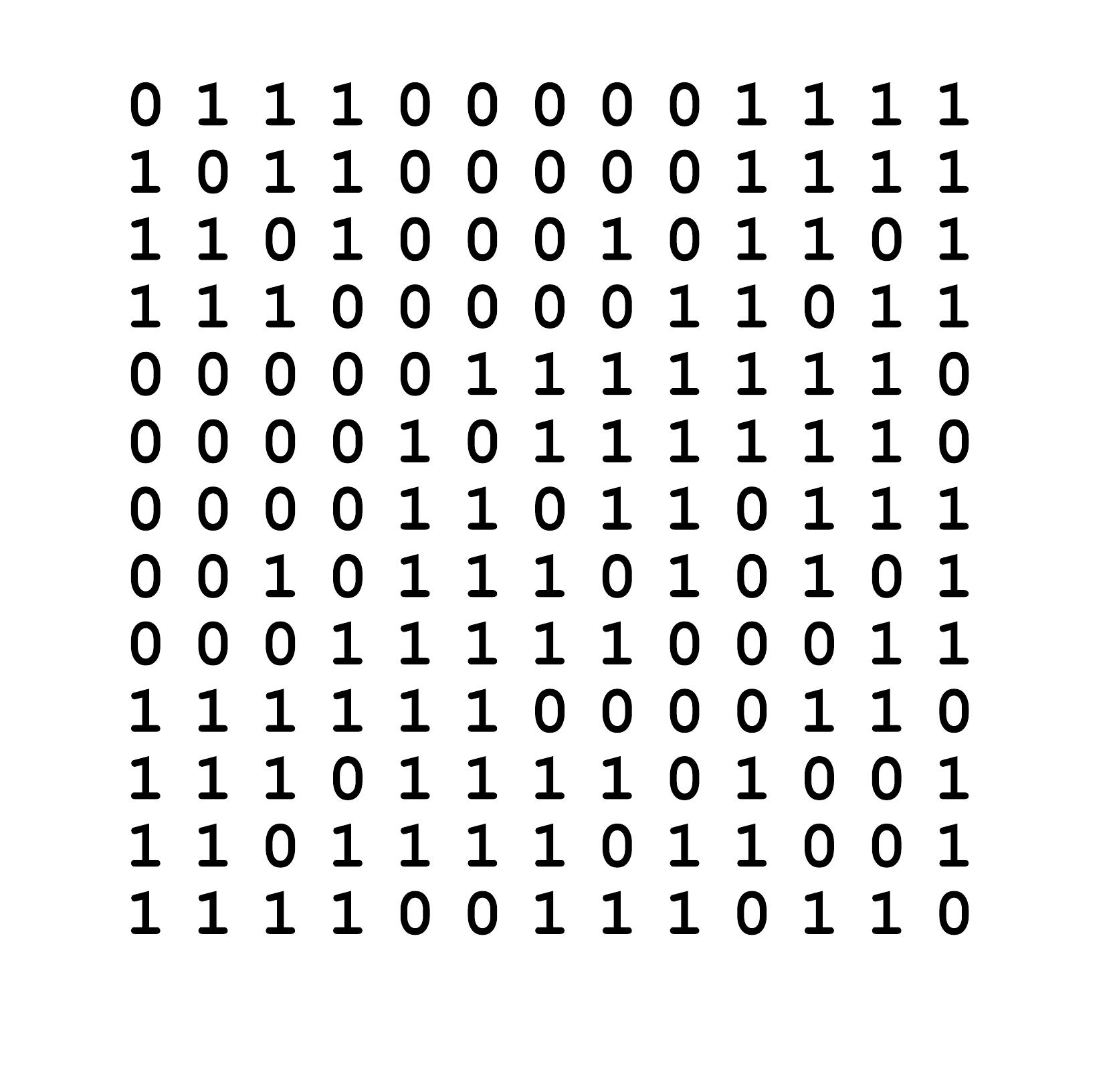}
		\vspace{-1.5em}
		\caption*{$G_{13.306470}$}
		\label{figure: 13_306470}
	\end{subfigure}
	
	\vspace{1em}
	\caption{13-vertex minimal graphs in $\mH_e(3, 3)$\\ with independence number 2}
	\label{figure: 13_a2}
\end{figure}

\chapter{New lower bound on $F_e(3, 3; 4)$}

\vspace{1em}

The graph $G$ obtained by the construction of Folkman \cite{Fol70}, for which $G \arrowse (3, 3)$ and $\omega(G) = 3$, has a very large number of vertices. Because of this, in 1975 Erd\"os \cite{Erd75} posed the problem to prove the inequality $F_e(3, 3; 4) < 10^{10}$. In 1986 Frankl and R\"odl \cite{FR86} almost solved this problem by showing that $F_e(3, 3; 4) < 7.02 \times 10^{11}$. In 1988 Spencer \cite{Spe88} proved the inequality $F_e(3, 3; 4) < 3 \times 10^9$ by using probabilistic methods. In 2008 Lu \cite{Lu08} constructed a 9697-vertex graph in $\mH_e(3, 3; 4)$, thus considerably improving the upper bound on $F_e(3, 3; 4)$. Soon after that, Lu`s result was improved by Dudek and R\"odl \cite{DR08b}, who proved $F_e(3, 3; 4) \leq 941$. The best known upper bound on this number is $F_e(3, 3; 4) \leq 786$, obtained in 2012 by Lange, Radziszowski, and Xu \cite{LRX14}. Exoo conjectured that the 127-vertex graph $G_{127}$, used by Hill and Irwing \cite{HI82} to prove the bound $R(4, 4, 4) \geq 128$, has the property $G_{127} \arrowse (3, 3)$. This conjecture was studied in \cite{RX07} and \cite{RX16}. It is still unknown whether $G_{127} \arrowse (3, 3)$.

In 1972 Lin \cite{Lin72} proved that $F_e(3, 3; 4) \geq 11$. The lower bound was improved by Nenov \cite{Nen83}, who showed in 1981 that $F_e(3, 3; 4) \geq 13$. In 1984 Nenov \cite{Nen84} proved that every 5-chromatic $K_4$-free graph has at least 11 vertices, from which it is easy to derive that $F_e(3, 3; 4) \geq 14$. From $F_e(3, 3; 5) = 15$ \cite{Nen81a}\cite{PRU99} it follows easily, that $F_e(3, 3; 4) \geq 16$. The best lower bound known on $F_e(3, 3; 4)$ was obtained in 2007 by Radziszowski and Xu \cite{RX07}, who proved with the help of a computer that $F_e(3, 3; 4) \geq 19$. According to Radziszowski and Xu \cite{RX07}, any method to improve the bound $F_e(3, 3; 4) \geq 19$ would likely be of significant interest. 

We improve the lower bound on the number $F_e(3, 3; 4)$ by proving
\begin{theorem}
	\label{theorem: F_e(3, 3; 4) geq 20}
	$F_e(3, 3; 4) \geq 20$.
\end{theorem}

\vspace{1em}

\section{Algorithm A8}

Let $G \in \mH_e(3, 3; 4; n)$, $A \subseteq \V(G)$ be an independent set of vertices of $G$, $\abs{A} = k$, and $H = G - A$. Then obviously, $G$ is a subgraph of $\overline{K}_k + H$, therefore $\overline{K}_k + H \in \mH_e(3, 3; 5; n)$, and it is easy to see that $K_1 + H \in \mH_e(3, 3; 5; n - k + 1)$. By this reasoning, in \cite{RX07} it is proved, but not explicitly formulated (see the proofs of Theorem 2 and Theorem 3), the following
\begin{proposition}
	\label{proposition: K_1 + H in mH_e(3, 3; 5; n - abs(A) + 1)}
	Let $G \in \mH_e(3, 3; 4; n)$, $A \subseteq \V(G)$ be an independent set of vertices of $G$, and $H = G - A$. Then $K_1 + H \in \mH_e(3, 3; 5; n - \abs{A} + 1)$.
\end{proposition}

With the help of Proposition \ref{proposition: K_1 + H in mH_e(3, 3; 5; n - abs(A) + 1)} in \cite{RX07} the authors prove that every graph in $\mH_e(3, 3; 4; 18)$ can be constructed by adding 4 independent vertices to a 14-vertex graph $H$ such that $K_1 + H \in \mH_e(3, 3; 5; 15)$. All 659 graphs in $\mH_e(3, 3; 5; 15)$ were obtained in \cite{PRU99}, and among them 153 graphs are of the form $K_1 + H$. With the help of a computer in \cite{RX07} it is proved that, by extending the 14-vertex graphs $H$ with 4 independent vertices, it is not possible to obtain a graph in $\mH_e(3, 3; 4; 18)$. By Proposition \ref{proposition: K_1 + H in mH_e(3, 3; 5; n - abs(A) + 1)}, $\mH_e(3, 3; 4; 18) = \emptyset$ and $F_e(3, 3; 4) \geq 19$.

This method is not suitable for proving Theorem \ref{theorem: F_e(3, 3; 4) geq 20}, because not all graphs in $\mH_e(3, 3; 5; 16)$ are known and their number is too large. Because of this, we first prove that if $\G \in \mH_e(3, 3; 4; 19)$, then $G$ can be obtained by adding 4 independent vertices to some of 1 139 033 appropriately selected 15-vertex graphs, which we obtain in advance, or by adding 5 independent vertices to some of the 14-vertex graphs known from \cite{PRU99}, mentioned above. With the help of a new computer algorithm we check that these extensions do not lead to the construction of a graph in $\mH_e(3, 3; 4; 19)$ and we derive that $\mH_e(3, 3; 4; 19) = \emptyset$ and $F_e(3, 3; 4) \geq 20$. 

For convenience, we will use the following notations:

$\mL(n; p) = \set{G : \abs{\V(G)} = n, \omega(G) < 4 \mbox{ and } K_p + G \arrowse (3, 3)}$

$\mL(n; p; k) = \set{G \in \mL(n; p) : \alpha(G) = k}$

From $R(3, 4) = 9$ it follows that
\begin{equation}
\label{equation: mL(n; p; k) = emptyset, if n geq 9 and k leq 2}
\mL(n; p; k) = \emptyset, \mbox{ if $n \geq 9$ and $k \leq 2$, }
\end{equation}

and from $R(4, 4) = 18$ it follows that
\begin{equation}
\label{equation: mL(n; p; k) = emptyset, if n geq 18 and k leq 3}
\mL(n; p; k) = \emptyset, \mbox{ if $n \geq 18$ and $k \leq 3$. }
\end{equation}

In \cite{PRU99} it is proved that $\mL(n; 1) \neq \emptyset$ if and only if $n \geq 14$, and all 153 graphs in $\mL(14; 1)$ are found. Further, we will use the following fact:
\begin{equation}
\label{equation: mL(14; 1; k) neq emptyset Leftrightarrow k in set(4, 5, 6, 7)}
\mL(14; 1; k) \neq \emptyset \Leftrightarrow k \in \set{4, 5, 6, 7}, \cite{PRU99}.
\end{equation}

From Theorem \ref{theorem: chi(G) geq R(p, q)} it follows
\begin{equation}
\label{equation: G in mL(n; p) Rightarrow chi(G) geq 6 - p}
G \in \mL(n; p) \Rightarrow \chi(G) \geq 6 - p.
\end{equation}

Obviously, $\mH_e(3, 3; 4; n) = \mL(n; 0)$. Let $G \in \mH_e(3, 3; 4)$. From the equality $F_e(3, 3; 5) = 15$ \cite{Nen81a}\cite{PRU99} and Proposition \ref{proposition: K_1 + H in mH_e(3, 3; 5; n - abs(A) + 1)} it follows that $\alpha(G) \leq \abs{\V(G)} - 14$. Therefore, either $\abs{\V(G)} \geq 20$ or $\alpha(G) \leq 5$. From $R(4, 4) = 18$ it follows that $\alpha(G) \geq 4$. Thus, we obtain
\begin{equation}
\label{equation: mH_e(3, 3; 4; 19) = mL(19; 0) = mL(19; 0; 4) cup mL(19; 0; 5)}
\mH_e(3, 3; 4; 19) = \mL(19; 0) = \mL(19; 0; 4) \cup \mL(19; 0; 5).
\end{equation}

We will need the following proposition, which follows easily from Proposition \ref{proposition: K_1 + H in mH_e(3, 3; 5; n - abs(A) + 1)}:
\begin{proposition}
	\label{proposition: H in mL(n - abs(A), p + 1)}
	Let $G \in \mL(n; p)$, $A \subseteq \V(G)$ be an independent set of vertices of $G$, and $H = G - A$. Then $H \in \mL(n - \abs{A}; p + 1)$.
\end{proposition}

We denote by $\mL_{max}(n; p; k)$ the set of all maximal $K_4$-free graphs in $\mL(n; p; k)$, i.e. the graphs $G \in \mL(n; p; k)$ for which $\omega(G + e) = 4$ for every $e \in \E(\overline{G})$. Since every graph in $\mL(19; 0)$ is contained in a maximal $K_4$-free graph in $\mL(19; 0)$, according to (\ref{equation: mH_e(3, 3; 4; 19) = mL(19; 0) = mL(19; 0; 4) cup mL(19; 0; 5)}) to prove the Theorem \ref{theorem: F_e(3, 3; 4) geq 20} it is enough to prove that $\mL_{max}(19; 0; 4) = \emptyset$ and $\mL_{max}(19; 0; 5) = \emptyset$. In the proofs of these inequalities we will use Algorithm \ref{algorithm: A8}, formulated below. 

We denote by $\mL_{+K_3}(n; p; k)$ the set of all $(+K_3)$-graphs in $\mL(n; p; k)$ (see Definition \ref{definition: (+K_p)}). Let $G \in \mL_{max}(n; p; k)$. Let $A \subseteq \V(G)$ be an independent set of vertices of $G$, $\abs{A} = k$ and $H = G - A$. According to Proposition \ref{proposition: H in mL(n - abs(A), p + 1)}, $H \in \mL(n - k, p + 1)$. Since $G$ is a maximal $K_4$-free graph, from Proposition \ref{proposition: H is a (+K_(q - 1))-graph} it follows that $H$ is $(+K_3)$-graph. From $\alpha(G) = k$ it follows that $\alpha(H) \leq k$. Therefore, $H \in \mL_{+K_3}(n - k; p + 1; k')$ for some $k' \leq k$. Thus, we proved the following
\begin{proposition}
	\label{proposition: H in bigcup(k' leq k)mL_(+K_3)(n - k; p + 1; k')}
	Let $G \in \mL_{max}(n; p; k)$. Let $A \subseteq \V(G)$ be an independent set of vertices of $G$, $\abs{A} = k$, and $H = G - A$. Then,
	
	$H \in \bigcup_{k' \leq k}\mL_{+K_3}(n - k; p + 1; k')$.
\end{proposition}

If $G \in \mL(n; p; k)$ is a Sperner graph, i.e. $N_G(u) \subseteq N_G(v)$ for some $u, v \in \V(G)$, then $G - u \in \mL(n - 1; p; k')$, for $k - 1 \leq k' \leq k$. Therefore, every Sperner graph $G \in \mL(n; p; k)$ is obtained by adding one vertex to some graph $H \in \mL(n - 1; p; k')$, $k - 1 \leq k' \leq k$. In the special case, when $G$ is a Sperner graph and $G \in \mL_{max}(n; p; k)$, from $N_G(u) \subseteq N_G(v)$ it follows that $N_G(u) = N_G(v)$. Therefore $G - u \in \mL_{max}(n - 1; p; k')$, $k - 1 \leq k' \leq k$, i.e. $G$ is obtained by duplicating a vertex in some graph $H \in \mL_{max}(n - 1; p; k')$. All non-Sperner graphs in $\mL_{max}(n; p; k)$ are obtained very efficiently with the help of Algorithm \ref{algorithm: A8}, formulated below, which is based on Proposition \ref{proposition: H in bigcup(k' leq k)mL_(+K_3)(n - k; p + 1; k')} and the following
\vspace{-0.5em}
\begin{proposition}
	\label{proposition: alpha(G) = k Leftrightarrow alpha(H - bigcup_(v in A') N_G((v)) leq k - abs(A')}
	Let $A$ be an independent set of vertices of $G$, $\abs{A} = k$, and $H = G - A$. Then, 
	
	$\alpha(G) = k \Leftrightarrow \alpha(H - \bigcup_{v \in A'} N_G(v)) \leq k - \abs{A'}, \ \forall A' \subseteq A$.
\end{proposition}

\begin{proof}
Proposition \ref{proposition: alpha(G) = k Leftrightarrow alpha(H - bigcup_(v in A') N_G((v)) leq k - abs(A')} is the special case $t = \abs{A} = k$ of Proposition \ref{proposition: alpha(G) leq t Leftrightarrow alpha(H - bigcup_(v in A') N_G((v)) leq t - abs(A')}.
\end{proof}

Now we formulate:

\vspace{-0.5em}

\begin{namedalgorithm}{A8}
	\label{algorithm: A8}
	Let $n$, $p$ and $k$ be positive integers.
		
	The input of the algorithm is the set $\mA = \bigcup_{k' \leq k}\mL_{max}(n - k; p + 1; k')$.
	
	The output of the algorithm is the set $\mB$ of all non-Sperner graphs in $\mL_{max}(n; p; k)$.
	
	\emph{1.} By removing edges from the graphs in $\mA$ obtain the set
	
	$\mA' = \bigcup_{k' \leq k}\mL_{+K_3}(n - k; p + 1; k')$.
	
	\emph{2.} For each graph $H \in \mA'$:
	
	\emph{2.1.} Find the family $\mM(H) = \set{M_1, ..., M_l}$ of all maximal $K_3$-free subsets of $\V(H)$.
	
	\emph{2.2.} Find all $k$-element subsets $N = \set{M_{i_1}, ..., M_{i_k}}$ of $\mM(H)$ which fulfill the conditions:
	
	(a) $M_{i_j} \neq N_H(v)$ for every $v \in \V(H)$ and for every $M_{i_j} \in N$.
	
	(b) $K_2 \subseteq M_{i_j} \cap M_{i_h}$ for every $M_{i_j}, M_{i_h} \in N$.
	
	(c) $\alpha(H - \bigcup_{M_{i_j} \in N'} M_{i_j}) \leq k - \abs{N'}$ for every $N' \subseteq N$.
	
	\emph{2.3.} For each of the found in step 2.2 $k$-element subsets $N = \set{M_{i_1}, ..., M_{i_k}}$ of $\mM(H)$ construct the graph $G = G(N)$ by adding new independent vertices $v_1, ..., v_k$ to $\V(H)$ such that $N_G(v_j) = M_{i_j}, j = 1, ..., k$. If $G$ is not a Sperner graph and $\omega(G + e) = 4, \forall e \in \E(\overline{G})$, then add $G$ to $\mB$.
	
	\emph{3.} Remove the isomorphic copies of graphs from $\mB$.
	
	\emph{4.} Remove from $\mB$ all graphs with chromatic number less than $6 - p$.
	
	\emph{5.} Remove from $\mB$ all graphs $G$ for which $K_p + G \not\arrowse (3, 3)$.
\end{namedalgorithm}

We will prove the correctness of Algorithm \ref{algorithm: A8} with the help of the following
\begin{lemma}
\label{lemma: algorithm A8}
After the execution of step 2.3 of Algorithm \ref{algorithm: A8}, the obtained set $\mB$ coincides with the set of all maximal $K_4$-free non-Sperner graphs $G$ with $\alpha(G) = k$ which have an independent set of vertices $A \subseteq \V(G), \abs{A} = k$ such that $G - A \in \mA'$.
\end{lemma}

\begin{proof}
Suppose that in step 2.3 of Algorithm \ref{algorithm: A8} the graph $G$ is added to $\mB$. Then $G = G(N)$ and $G - \set{v_1, ..., v_k} = H \in \mA'$, where $N$, $v_1, ..., v_k$, and $H$ are the same as in step 2.3. By $H \in \mA'$, we have $\omega(H) < 4$. Since $N_G(v_j), j = 1, ..., k$, are $K_3$-free sets, it follows that $\omega(G) < 4$. From the condition (c) in step 2.2 and Proposition \ref{proposition: alpha(G) = k Leftrightarrow alpha(H - bigcup_(v in A') N_G((v)) leq k - abs(A')} it follows that $\alpha(G) = k$. The two checks at the end of step 2.3 guarantee that $G$ is a maximal $K_4$-free non-Sperner graph. 

Let $G$ be a maximal $K_4$-free non-Sperner graph, $\alpha(G) = k$, and $A = \set{v_1, ..., v_k}$ be an independent set of vertices of $G$ such that $H = G - A \in \mA'$. We will prove that, after the execution of step 2.3 of Algorithm \ref{algorithm: A8}, $G \in \mB$. Since $G$ is a maximal $K_4$-free graph, $N_G(v_i), i = 1, ..., k$, are maximal $K_3$-free subsets of $V(H)$, and therefore $N_G(v_i) \in \mM(H), i = 1, ..., k$ (see step 2.1). Let $N = \set{N_G(v_1), ..., N_G(v_k)}$. Since $G$ is not a Sperner graph, $N$ is a $k$-element subset of $\mM(H)$, and according to Proposition \ref{proposition: N_G(u) neq N_H(v)}, $N$ fulfills the condition (a) in step 2.2. By Proposition \ref{proposition: K_(q - 2) subseteq N_G(u) cap N_G(v)}, $N$ fulfills the condition (b), and by Proposition \ref{proposition: alpha(G) = k Leftrightarrow alpha(H - bigcup_(v in A') N_G((v)) leq k - abs(A')}, $N$ also fulfills (c). Thus, we showed that $N$ fulfills all conditions in step 2.2, and since $G = G(N)$ is a maximal $K_4$-free non-Sperner graph, in step 2.3 $G$ is added to $\mB$.
\end{proof}

\begin{theorem}
	\label{theorem: algorithm A8}
	After the execution of Algorithm \ref{algorithm: A8}, the obtained set $\mB$ coincides with the set of all non-Sperner graphs in $\mL_{max}(n; p; k)$.
\end{theorem}

\begin{proof}
	Suppose that, after the execution of Algorithm \ref{algorithm: A8}, $G \in \mB$. According to Lemma \ref{lemma: algorithm A8}, $G$ is a maximal $K_4$-free non-Sperner graph and $\alpha(G) = k$. Now, from step 5 it follows that $G \in \mL_{max}(n; p; k)$. 
	
	Conversely, let $G$ be an arbitrary non-Sperner graph in $\mL_{max}(n; p; k)$. Let $A \subseteq \V(G)$ be an independent set of vertices of $G$, $\abs{A} = k$, and $H = G - A$. According to Proposition \ref{proposition: H in bigcup(k' leq k)mL_(+K_3)(n - k; p + 1; k')}, $H \in \mA'$. Now, from Lemma \ref{lemma: algorithm A8} we obtain that, after the execution of step 2.3, the graph $G$ is included in the set $\mB$. By (\ref{equation: G in mL(n; p) Rightarrow chi(G) geq 6 - p}), after the execution of step 4, $G$ remains in $\mB$. It is clear that after step 5, $G$ also remains in $\mB$.
\end{proof}

\begin{remark}
	\label{remark: delta(G) geq 8}
	Since $\mL(18; 0) = \emptyset$, from Theorem \ref{theorem: delta(G) geq 8, G in mH_e(3, 3; 4)} it follows easily that for each graph $G \in \mL(19, 0)$ we have $\delta(G) \geq 8$. Using this result we can improve Algorithm \ref{algorithm: A8} in the case $n = 19, p = 0$ in the following way:
	
	1. In step 1 we remove from the set $\mA'$ the graphs with minimum degree less than $8 - k$.
	
	2. In step 2.2 we add the following conditions for the subset $N$:
	
	(d) $\abs{M_{i_j}} \geq 8$ for every $M_{i_j} \in N$.
	
	(e) If $N' \subseteq N$, then $d_H(v) \geq 8 - k + \abs{N'}$ for every $v \not\in \bigcup_{M_{i_j} \in N'} M_{i_j}$.
	
	This way it is guaranteed that in step 2.3 only graphs $G$ for which $\delta(G) \geq 8$ are added to the set $\mB$.
\end{remark}

\vspace{1em}
Theorem \ref{theorem: algorithm A8} is published in \cite{BN16}. Algorithm \ref{algorithm: A8} is a slightly modified version of Algorithm 2.7 in \cite{BN16}.

\section{Proof of Theorem \ref{theorem: F_e(3, 3; 4) geq 20}}

According to (\ref{equation: mH_e(3, 3; 4; 19) = mL(19; 0) = mL(19; 0; 4) cup mL(19; 0; 5)}), it is enough to prove that $\mL_{max}(19; 0; 5) = \emptyset$ and $\mL_{max}(19; 0; 4) = \emptyset$.\\

1. Proof of $\mL_{max}(19; 0; 5) = \emptyset$.

With a computer check we find all 8 maximal graphs among the graphs in $\mL(14; 1)$, which are known from \cite{PRU99}. All of these graphs have independence number 4. Denote 

$\mA_3 = \mL_{max}(14; 1; 4) = \bigcup_{k' \leq 5} \mL_{max}(14; 1; k')$.

We execute Algorithm \ref{algorithm: A8}($n = 19$, $p = 0$, $k = 5$) with the set $\mA = \mA_3$ as an input. In step 1 we obtain all 85 graphs in $\mL_{+K_3}(14; 1; 4)$ and all 28 graphs in $\mL_{+K_3}(14; 1; 5)$. In step 2.3, 502 901 graphs are added to the set $\mB$, 251 244 of which remain in $\mB$ after the isomorph rejection in step 3. After step 4, 31 graphs with chromatic number 6 remain in $\mB$. In the end, after executing step 5 we obtain $\mB = \emptyset$. Since $\mL(18; 0) = \emptyset$, there are no Sperner graphs in $\mL(19; 0)$, and by Theorem \ref{theorem: algorithm A8} we obtain $\mL_{max}(19; 0; 5) = \emptyset$. \\

2. Proof of $\mL_{max}(19; 0; 4) = \emptyset$.

Using \emph{nauty} \cite{MP13} we generate all 11-vertex graphs with a computer and among them we find all 102 graphs in $\mL_{max}(11; 2; 3)$ and all 270 graphs in $\mL_{max}(11; 2; 4)$. Let us denote $\mA_1 = \mL_{max}(11; 2; 3) \cup \mL_{max}(11; 2; 4)$. By (\ref{equation: mL(n; p; k) = emptyset, if n geq 9 and k leq 2}),

$\mA_1 = \bigcup_{k' \leq 4} \mL_{max}(11; 2; k')$.

We execute Algorithm \ref{algorithm: A8} ($n = 15$, $p = 1$, $k = 4$) with the set $\mA = \mA_1$ as an input. In step 1 we obtain all 362 439 graphs in $\mL_{+K_3}(11; 2; 3)$ and all 7 949 015 graphs in $\mL_{+K_3}(11; 2; 4)$. According to Theorem \ref{theorem: algorithm A8}, after the execution of the algorithm we find all 5750 non-Sperner graphs in $\mL_{max}(15; 1; 4)$. Among the graphs in $\mL(14; 1)$, which are known from \cite{PRU99}, there are 8 maximal $K_4$-free graphs and they all have independence number 4. By adding a new vertex to each of these 8 graphs which duplicates some of their vertices, we obtain all 20 non-isomorphic Sperner graphs in $\mL_{max}(15; 1; 4)$ (see Proposition \ref{proposition: maximal Sperner graphs}). Thus, all 5770 graphs in $\mL_{max}(15; 1; 4)$ are obtained. All graphs $G$ for which $\omega(G) < 4$ and $\alpha(G) < 4$ are known and can be found in \cite{McK_r}. There are 640 such 15-vertex graphs, among which we find the only 2 graphs in $\mL_{max}(15; 1; 3)$. Let $\mA_2 = \mL_{max}(15; 1; 3) \cup \mL_{max}(15; 1; 4)$. By (\ref{equation: mL(n; p; k) = emptyset, if n geq 9 and k leq 2}),

$\mA_2 = \bigcup_{k' \leq 4} \mL_{max}(15; 1; k')$.

We execute Algorithm \ref{algorithm: A8} ($n = 19$, $p = 0$, $k = 4$) with the set $\mA = \mA_2$ as an input. In step 1 we obtain all 1 139 023 graphs in $\mL_{+K_3}(15; 1; 4)$ and all 5 graphs in $\mL_{+K_3}(15; 1; 3)$. In step 2.3, 2 551 314 graphs are added to the set $\mB$, 2 480 352 of which remain in $\mB$ after the isomorph rejection in step 3. After step 4, 2 597 graphs with chromatic number 6 remain in $\mB$. In the end, after executing step 5 we obtain $\mB = \emptyset$. Since $\mL(18, 0) = \emptyset$, there are no Sperner graphs in $\mL(19, 0)$, and by Theorem \ref{theorem: algorithm A8} we obtain $\mL_{max}(19; 0; 4) = \emptyset$. \qed

\begin{remark}
	\label{remark: mL(15; 1) and mH_v(3, 3; 4; 15)}
	
	It is easy to see that if $G \arrowsv (3, 3)$, then $K_1 + G \arrowse (3, 3)$. Thus, we derive
	\begin{equation*}
	\mH_v(3, 3; 4; n) \subseteq \mL(n; 1).
	\end{equation*}
	
	Let us note that $\mH_v(3, 3; 4; 14) = \mL(14; 1)$ \cite{PRU99}, but $\mH_v(3, 3; 4; 15) \neq \mL(15; 1)$. We found all 2 081 234 graphs $\mL(15; 1)$. In the proof of Theorem \ref{theorem: F_e(3, 3; 4) geq 20} we already found the only 2 graphs in $\mL_{max}(15; 1; 3)$ and all 5770 graphs in $\mL_{max}(15; 1; 4)$. With the help of Algorithm \ref{algorithm: A8}, similarly to the case $k = 4$, we find all graphs in $\mL_{max}(15; 1; k), \ k \geq 5$. We obtain all 826 graphs $\mL_{max}(15; 1; 5)$, all 12 graphs in $\mL_{max}(15; 1; 6)$, and $\mL_{max}(15; 1; k) = \emptyset, \ k \geq 7$. Thus, we obtain all 6 610 maximal $K_4$-free graphs in $\mL(15; 1)$. By removing edges from them, we find all 2 081 234 graphs in $\mL(15; 1)$. Some properties of these graphs are listed in Table \ref{table: mL(15; 1) statistics}. Among the graphs in $\mL(15; 1)$ there are exactly 20 graphs, which are not in $\mH_v(3, 3; 4; 15)$. Properties of these 20 graphs are given in Table \ref{table: mL(15; 1) setminus mH_v(3, 3; 4; 15) statistics}, and one of these graphs (which has 51 edges) is given in Figure \ref{figure: L_15_1}.
\end{remark}

\vspace{1em}
Theorem \ref{theorem: F_e(3, 3; 4) geq 20} is published in \cite{BN16}.

\begin{figure}
	\centering
	\includegraphics[height=300px,width=150px]{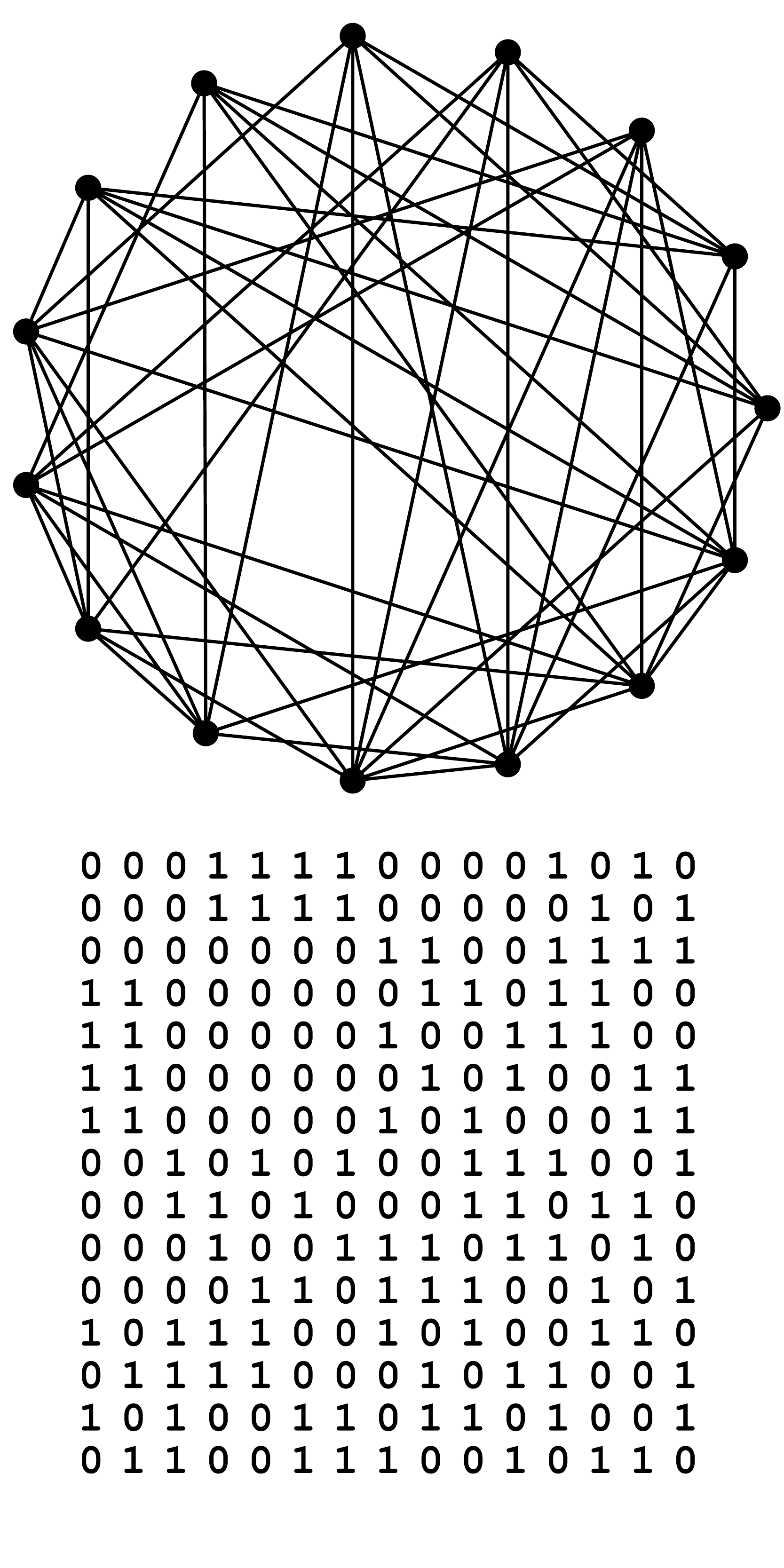}
	\caption{Example of a graph in $\mL(15; 1) \setminus \mH_v(3, 3; 4; 15)$}
	\label{figure: L_15_1}
\end{figure}

\begin{table}
	\centering
	\resizebox{0.9\textwidth}{!}{
		\begin{tabular}{ | l r | l r | l r | l r | l r | }
			\hline
			\multicolumn{2}{|c|}{\parbox{6em}{$\abs{\E(G)}$ \hfill $\#$}}&
			\multicolumn{2}{|c|}{\parbox{6em}{$\delta(G)$ \hfill $\#$}}&
			\multicolumn{2}{|c|}{\parbox{6em}{$\Delta(G)$ \hfill $\#$}}&
			\multicolumn{2}{|c|}{\parbox{6em}{$\alpha(G)$ \hfill $\#$}}&
			\multicolumn{2}{|c|}{\parbox{6em}{$\abs{Aut(G)}$ \hfill $\#$}}\\
			\hline
			42			&  1		& 0			& 153		& 7			& 65		& 3			& 5			& 1			& 2 052 543	\\
			43			&  4		& 1			& 1 629		& 8			& 675 118	& 4			& 1 300 452	& 2			& 27 729	\\
			44			&  44		& 2			& 10 039	& 9			& 1 159 910	& 5			& 747 383	& 3			& 9			\\
			45			&  334		& 3			& 34 921	& 10		& 165 612	& 6			& 32 618	& 4			& 850		\\
			46			&  2 109	& 4			& 649 579	& 11		& 80 529	& 7			& 766		& 6			& 22		\\
			47			&  9 863	& 5			& 1 038 937	& 			& 			& 8			& 10		& 8			& 55		\\
			48			&  35 812	& 6			& 339 395	& 			& 			& 			& 			& 10		& 2			\\
			49			&  101 468	& 7			& 6 581		& 			& 			& 			& 			& 12		& 11		\\
			50			&  223 881	& 			& 			& 			& 			& 			& 			& 14		& 4			\\
			51			&  378 614	& 			& 			& 			& 			& 			& 			& 16		& 4			\\
			52			&  478 582	& 			& 			& 			& 			& 			& 			& 20		& 1			\\
			53			&  436 693	& 			& 			& 			& 			& 			& 			& 24		& 4			\\
			54			&  273 824	& 			& 			& 			& 			& 			& 			& 			& 			\\
			55			&  110 592	& 			& 			& 			& 			& 			& 			& 			& 			\\
			56			&  26 099	& 			& 			& 			& 			& 			& 			& 			& 			\\
			57			&  3 150	& 			& 			& 			& 			& 			& 			& 			& 			\\
			58			&  160		& 			& 			& 			& 			& 			& 			& 			& 			\\
			59			&  4		& 			& 			& 			& 			& 			& 			& 			& 			\\
			\hline
		\end{tabular}
	}
	\caption{Some properties of the graphs in $\mL(15; 1)$}
	\label{table: mL(15; 1) statistics}
	
	\vspace{4em}
	
	\centering
	\resizebox{0.9\textwidth}{!}{
		\begin{tabular}{ | l r | l r | l r | l r | l r | }
			\hline
			\multicolumn{2}{|c|}{\parbox{6em}{$\abs{\E(G)}$ \hfill $\#$}}&
			\multicolumn{2}{|c|}{\parbox{6em}{$\delta(G)$ \hfill $\#$}}&
			\multicolumn{2}{|c|}{\parbox{6em}{$\Delta(G)$ \hfill $\#$}}&
			\multicolumn{2}{|c|}{\parbox{6em}{$\alpha(G)$ \hfill $\#$}}&
			\multicolumn{2}{|c|}{\parbox{6em}{$\abs{Aut(G)}$ \hfill $\#$}}\\
			\hline
			47			&  2		& 4			& 7			& 8			& 20		& 4			& 5			& 1			& 5			\\
			48			&  5		& 5			& 10		& 			& 			& 5			& 15		& 2			& 12		\\
			49			&  7		& 6			& 3			& 			& 			& 			& 			& 4			& 3			\\
			50			&  3		& 			& 			& 			& 			& 			& 			& 			& 			\\
			51			&  1		& 			& 			& 			& 			& 			& 			& 			& 			\\
			52			&  2		& 			& 			& 			& 			& 			& 			& 			& 			\\
			\hline
		\end{tabular}
	}
	\caption{Some properties of the graphs in $\mL(15; 1) \setminus \mH_v(3, 3; 4; 15)$}
	\label{table: mL(15; 1) setminus mH_v(3, 3; 4; 15) statistics}
\end{table}

\subsection*{Computations}

All computations were done on a personal computer. Using one processing core, the time needed to execute Algorithm \ref{algorithm: A8} in the case ($n = 19$, $p = 0$, $k = 5$) is about half a minute, and in the case ($n = 19$, $p = 0$, $k = 4$) it is about one hour. Note that in the first case among the 31 6-chromatic graphs obtained after step 4 of Algorithm \ref{algorithm: A8}, 11 have a minimum degree of 8 or more. In the second case, the total number of 6-chromatic graphs is 2597, 794 of which have a minimum degree of 8 or more. Using the improvements of Algorithm \ref{algorithm: A8} described in Remark \ref{remark: delta(G) geq 8}, the time needed for computations is reduced more than 10 times in the first case and almost 2 times in the second case.

Let us note that the result $\mL_{max}(19; 0; 5) = \emptyset$ can be obtained with the algorithm from \cite{RX07}, but a lot more computations have to be performed by the computer.

We performed various tests to check the correctness of our implementation of Algorithm \ref{algorithm: A8}. One such test was to reproduce all 153 graphs in $\mL(14; 1)$, which are known from \cite{PRU99}. First, by executing Algorithm \ref{algorithm: A8} with $n = 14, p = 1, k \geq 3$, we obtain the 8 graphs in $\mL_{max}(14; 1; 4)$, and $\mL_{max}(14; 1; k) = \emptyset, k \neq 4$. By removing edges from the 8 maximal graphs, we obtain all 153 graphs in $\mL(14; 1)$. Among the 2 081 234 graphs in $\mL(15; 1)$, there are exactly 153 graphs with one isolated vertex (see Remark \ref{remark: mL(15; 1) and mH_v(3, 3; 4; 15)} and Table \ref{table: mL(15; 1) statistics}). Thus, we obtain once again, indirectly, all graphs in $\mL(14; 1)$.

We used Algorithm \ref{algorithm: A8} to give another proof of the bound $F_e(3, 3; 4) \geq 19$. Similarly to (\ref{equation: mH_e(3, 3; 4; 19) = mL(19; 0) = mL(19; 0; 4) cup mL(19; 0; 5)}), it is easy to see that $\mH_e(3, 3; 4; 18) = \mL(18; 0) = \mL(18; 0; 4)$. By executing Algorithm \ref{algorithm: A8} ($n = 18$, $p = 0$, $k = 4$) with input $\mA = \mL_{max}(14; 1; 4) = \bigcup_{k' \leq 4} \mL_{max}(14; 1; k')$, we obtain $\mB = \emptyset$. Since there are no non-Sperner graphs in $\mL(18; 0)$, by Theorem \ref{theorem: algorithm A8} we derive $\mL_{max}(18; 0; 4) = \emptyset$, and therefore $F_e(3, 3; 4) \geq 19$. The computation time in this case is less than a second.

\addcontentsline{toc}{chapter}{Bibliography}

\nocite{*}

\bibliographystyle{plain}

\bibliography{main}

\begin{thebibliography}{100}

\bibitem{Bik13}
A.~Bikov.
\newblock {A computer research on critical Ramsey graphs (in Bulgarian)}, 2013.
\newblock Master Thesis, Sofia University "St. Kliment Ohridski".

\bibitem{Bik16}
A.~Bikov.
\newblock Small minimal $(3, 3)$-ramsey graphs.
\newblock {\em Ann. Univ. Sofia Fac. Math. Inform.}, 103:123--147, 2016.
\newblock Preprint: arxiv:1604.03716, April 2016.

\bibitem{Bik17}
A.~Bikov.
\newblock New bounds on the vertex {F}olkman number ${F}_v(2, 2, 2, 3; 4)$.
\newblock {\em Mathematics and Education. Proceedings of the 46th Spring
  Conference of the Union of Bulgarian Mathematicians}, 46:137--144, 2017.
\newblock Preprint arXiv:1611.06418, November 2016.

\bibitem{BN15b}
A.~Bikov and N.~Nenov.
\newblock Modified vertex {F}olkman numbers.
\newblock {\em Mathematics and Education. Proceedings of the 45th Spring
  Conference of the Union of Bulgarian Mathematicians}, 45:113--123, 2016.
\newblock Preprint: arxiv:1511.02125, November 2015.

\bibitem{BN16}
A.~Bikov and N.~Nenov.
\newblock The edge {F}olkman number ${F}_e(3, 3; 4)$ is greater than 19.
\newblock {\em GEOMBINATORICS}, 27(1):5--14, 2017.
\newblock Preprint: arxiv:1609.03468, September 2016.

\bibitem{BN17b}
A.~Bikov and N.~Nenov.
\newblock Lower bounding the {F}olkman numbers ${F}_v(a_1, ..., a_s; m-1)$.
\newblock {\em Ann. Univ. Sofia Fac. Math. Inform.}, 104:39--53, 2017.
\newblock Preprint: arxiv:1711.01535, November 2017.

\bibitem{BN17a}
A.~Bikov and N.~Nenov.
\newblock On the vertex {F}olkman numbers ${F}_v(a_1, ..., a_s; m - 1)$ when
  $\max\{a_1, ..., a_s\} = 6$ or $7$.
\newblock To appear in the {\em Journal of Combinatorial Mathematics and
  Combinatorial Computing}, preprint: arxiv:1512.02051, April 2017.

\bibitem{BN15a}
A.~Bikov and N.~Nenov.
\newblock The vertex {F}olkman numbers ${F}_v(a_1,...,a_s; m - 1) = m + 9$, if
  $\max\{a_1,...,a_s\} = 5$.
\newblock {\em Journal of Combinatorial Mathematics and Combinatorial
  Computing}, 103:171--198, 2017.
\newblock Preprint: arxiv:1503.08444, August 2015.

\bibitem{BCGM_HoG}
G.~Brinkmann, K.~Coolsaet, J.~Goedgebeur, and H.~M\'elot.
\newblock House of Graphs, \url{https://hog.grinvin.org/}.

\bibitem{BR90}
J.~Brown and V.~R{\"o}dl.
\newblock {A new construction of $k$-{F}olkman graphs}.
\newblock {\em Ars Comb.}, 29:265--269, 1990.

\bibitem{BEL76}
S.~Burr, P.~Erd{\"o}s, and L.~Lov{\'a}sz.
\newblock On graphs of {R}amsey type.
\newblock {\em Ars. Combinatoria}, 1(1):167--190, 1976.

\bibitem{BR80}
S.~Burr and V.~Rosta.
\newblock {On the Ramsey multiplicities of Graphs - Problems and recent
  results}.
\newblock {\em Journal of Graph Theory}, 4:347--361, 1980.

\bibitem{Chv79}
V.~Chv{\'a}tal.
\newblock The minimality of the {M}ycielski graph.
\newblock {\em Lecture Notes in Mathematics}, 406:243--246, 1979.

\bibitem{Col05}
J.~Coles.
\newblock {Algorithms for bounding Folkman numbers}, 2005.
\newblock Master Thesis, RIT,
  \url{https://ritdml.rit.edu/bitstream/handle/1850/2765/JColesThesis2004.pdf}.

\bibitem{CR06}
J.~Coles and S.~Radziszowski.
\newblock Computing the {F}olkman number ${F}_v(2,2,3;4)$.
\newblock {\em Journal of Combinatorial Mathematics and Combinatorial
  Computing}, 58:13--22, 2006.

\bibitem{DLSX13}
F.~Deng, M.~Liang, Z.~Shao, and X.~Xu.
\newblock Upper bounds for the vertex {F}olkman number ${F}_v(3, 3, 3; 4)$ and
  ${F}_v(3, 3, 3; 5)$.
\newblock {\em ARS Combinatoria}, 112:249--256, 2013.

\bibitem{DR08a}
A.~Dudek and V.~R{\"o}dl.
\newblock New upper bound on vertex {F}olkman numbers.
\newblock {\em Lecture Notes in Computer Science}, 4557:473--478, 2008.

\bibitem{DR08b}
A.~Dudek and V.~R{\"o}dl.
\newblock On the {F}olkman number ${F}(2, 3, 4)$.
\newblock {\em Experimental Mathematics}, 17:63--67, 2008.

\bibitem{Erd75}
P.~Erd{\"o}s.
\newblock Problems and results on finite and infinite graphs.
\newblock In {\em Recent Advances in Graph Theory, Proc. Second Czechoslovak
  Sympos.}, pages 183--192, April 1975.

\bibitem{EH67}
P.~Erd{\"o}s and A.~Hajnal.
\newblock {Research problem 2-5}.
\newblock {\em J. Combin. Theory}, 2:104, 1967.

\bibitem{EG18}
G.~Exoo and J.~Goedgebeur.
\newblock Bounds for the smallest k-chromatic graphs of given girth.
\newblock Preprint: arxiv:1805.06713, May 2018.

\bibitem{Fol70}
J.~{F}olkman.
\newblock Graphs with monochromatic complete subgraphs in every edge coloring.
\newblock {\em SIAM Journal on Applied Mathematics}, 18:19--24, 1970.

\bibitem{FL06}
J.~Fox and K.~Lin.
\newblock {The minimum degree of Ramsey-minimal graphs}.
\newblock {\em J. of Graph Theory}, 54(2):167--177, 2006.

\bibitem{FR86}
P.~Frankl and V.~R{\"o}dl.
\newblock {Large triangle-free subgraphs in graphs without $K_4$}.
\newblock {\em Graphs and Combinatorics}, 2:135--144, 1986.

\bibitem{GSS95}
A.~Galuccio, M.~Simonovits, and G.~Simonyi.
\newblock {On the structure of co-critical graphs}.
\newblock {\em Graph theory, combinatorics and algorithms}, Vol 1, 2
  (Kalamazoo, MI, 1992):1053--1071, 1995.

\bibitem{Goe17}
J.~Goedgebeur.
\newblock On minimal triangle-free 6-chromatic graphs.
\newblock Preprint: arxiv:1707.07581, August 2017.

\bibitem{Goo59}
A.~Goodman.
\newblock {On sets of acquaintances and strangers at any party}.
\newblock {\em Amer. Math. Monthly}, 66:778--783, 1959.

\bibitem{GB15}
R.~Graham and S.~Butler.
\newblock {\em {Rudiments of Ramsey Theory: Second Edition}}.
\newblock AMS and CBMS, 2015.

\bibitem{Gra68}
R.~L. Graham.
\newblock {On edgewise 2-colored graphs with monochromatic triangles containing
  no complete hexagon}.
\newblock {\em J. Combin. Theory}, 4:300, 1968.

\bibitem{Gra12}
R.~L. Graham.
\newblock Some graph theory problems i would like to see solved.
\newblock In {\em SIAM My Favorite Graph Theory Conjectures}, 2012.

\bibitem{GS71}
R.~L. Graham and J.~H. Spencer.
\newblock {On small graphs with forced monochromatic triangles}.
\newblock {\em Lecture Notes in Math.}, 186:137--141, 1971.
\newblock Recent Trends in Graph Theory.

\bibitem{Har69}
F.~Harary.
\newblock {\em {Graph Theory}}.
\newblock Addison - Wesley, 1969.

\bibitem{HP74}
F.~Harary and G.~Prins.
\newblock {Generalized Ramsey theory for graphs. IV: The Ramsey multiplicity of
  a graph}.
\newblock {\em Networks}, 4:163--173, 1974.

\bibitem{HI82}
R.~Hill. and R.W. Irwing.
\newblock {On group partitions associated with lower bounds for symmetric
  Ramsey numbers}.
\newblock {\em European Journal of Combinatorics}, 3:35--50, 1982.

\bibitem{Jac80}
M.~Jacobson.
\newblock {A note on Ramsey multiplicity}.
\newblock {\em Discrete Math.}, 29:201--203, 1980.

\bibitem{Jac82}
M.~Jacobson.
\newblock {A note on Ramsey multiplicity for stars}.
\newblock {\em Discrete Math.}, 42:63--66, 1982.

\bibitem{JR95}
T.~Jensen and G.~Royle.
\newblock Small graphs with chromatic number 5: a computer research.
\newblock {\em Journal of Graph Theory}, 19:107--116, 1995.

\bibitem{KWR17}
J.~Kaufmann, H.~Wickus, and S.~Radziszowski.
\newblock {On some edge Folkman numbers small and large}.
\newblock In preparation, 2017.

\bibitem{Kol08}
N.~Kolev.
\newblock A multiplicative inequality for vertex {F}olkman numbers.
\newblock {\em Discrete Mathematics}, 308:4263--4266, 2008.

\bibitem{KN06b}
N.~Kolev and N.~Nenov.
\newblock The {F}olkman number ${F}_e(3,4;8)$ is equal to 16.
\newblock {\em Comptes rendus de l'Academie bulgare des Sciences},
  59(1):25--30, 2006.

\bibitem{KN06c}
N.~Kolev and N.~Nenov.
\newblock New recurrent inequality on a class of vertex {F}olkman numbers.
\newblock In {\em Proceedings of the 35th Spring Conference of the Union of
  Bulgarian Mathematicians}, pages 164--168, April 2006.

\bibitem{KN06a}
N.~Kolev and N.~Nenov.
\newblock New upper bound for a class of vertex {F}olkman numbers.
\newblock {\em The Electronic Journal of Combinatorics}, 13, 2006.

\bibitem{LRX14}
A.~Lange, S.~Radziszowski, and X.~Xu.
\newblock Use of {MAX-CUT} for {R}amsey arrowing of triangles.
\newblock {\em Journal of Combinatorial Mathematics and Combinatorial
  Computing}, 88:61--71, 2014.

\bibitem{LR11}
J.~Lathrop and S.~Radziszowski.
\newblock Computing the {F}olkman number ${F}_v(2, 2, 2, 2, 2; 4)$.
\newblock {\em Journal of Combinatorial Mathematics and Combinatorial
  Computing}, 78:213--222, 2011.

\bibitem{LL17}
Y.~Li and Q.~Lin.
\newblock {On generalized Folkman numbers}.
\newblock {\em Taiwanese J. Math.}, 21:1--9, 2017.

\bibitem{LL15}
Q.~Lin and Y.~Li.
\newblock {A Folkman linear family}.
\newblock {\em SIAM J. Discrete Math.}, 29:1988--1998, 2015.

\bibitem{Lin72}
S.~Lin.
\newblock {On Ramsey numbers and $K_r$-coloring of graphs}.
\newblock {\em J. Combin. Theory Ser. B}, 12:82--92, 1972.

\bibitem{Lu08}
L.~Lu.
\newblock Explicit construction of small {F}olkman graphs.
\newblock {\em SIAM Journal on Discrete Mathematics}, 21:1053--1060, 2008.

\bibitem{LRU01}
T.~Luczak, A.~Ruci{\'n}ski, and S.~Urba{\'n}ski.
\newblock On minimal vertex {F}olkman graphs.
\newblock {\em Discrete Mathematics}, 236:245--262, 2001.

\bibitem{LU96}
T.~Luczak and S.~Urba{\'n}ski.
\newblock A note on restricted vertex {R}amsey numbers.
\newblock {\em Periodica Mathematica Hungarica}, 33:101--103, 1996.

\bibitem{McK_c}
B.D. McKay.
\newblock Combinatorial data, \url{http://users.cecs.anu.edu.au/~bdm/data/}.

\bibitem{McK_r}
B.D. McKay.
\newblock {R}amsey graphs,
  \url{http://users.cecs.anu.edu.au/~bdm/data/ramsey.html}.

\bibitem{MP13}
B.D. McKay and A.~Piperino.
\newblock Practical graph isomorphism, {II}.
\newblock {\em J. Symbolic Computation}, 60:94--112, 2013.
\newblock Preprint version at \href{http://arxiv.org/abs/1301.1493}{arxiv.org}.

\bibitem{Myc55}
J.~Mycielski.
\newblock Sur le coloriage des graphes.
\newblock {\em Colloquium Mathematicum}, 3:161--162, 1955.

\bibitem{Nen79}
N.~Nenov.
\newblock {Up to isomorphism there exist only one minimal $t$-graph with nine
  vertices. (in Russian)}.
\newblock {\em God. Sofij. Univ. Fak. Mat. Mekh.}, 73:169--184, 1979.

\bibitem{Nen80a}
N.~Nenov.
\newblock {On the existence of a minimal $t$-graph with a given number of
  vertices. (in Russian)}.
\newblock {\em Serdica}, 6:270--274, 1980.

\bibitem{Nen80b}
N.~Nenov.
\newblock {On the independence number of minimal $t$-graphs. (in Russian)}.
\newblock In {\em Mathematics and education in mathematics, Proc. $9^{th}$
  Spring Conf. Union Bulg. Math}, pages 74--78, Sunny Beach / Bulgaria, 1980.

\bibitem{Nen80c}
N.~Nenov.
\newblock {\em {Ramsey graphs and some constants related to them.
  (Bulgarian)}}.
\newblock PhD thesis, University of Sofia, 1980.

\bibitem{Nen81a}
N.~Nenov.
\newblock {An example of a 15-vertex Ramsey (3, 3)-graph with clique number 4.
  (in Russian)}.
\newblock {\em C. A. Acad. Bulg. Sci.}, 34:1487--1489, 1981.

\bibitem{Nen81b}
N.~Nenov.
\newblock {Certain remarks on Ramsey multiplicities. (in Russian)}.
\newblock In {\em Mathematics and education in mathematics, Proc. $10^{th}$
  Spring Conf. Union Bulg. Math}, pages 176--179, Sunny Beach / Bulgaria, 1981.

\bibitem{Nen83}
N.~Nenov.
\newblock On the {Z}ykov numbers and some its applications to {R}amsey theory.
\newblock {\em Serdica Bulgariacae Mathematicae}, 9:161--167, 1983.
\newblock (in Russian).

\bibitem{Nen84}
N.~Nenov.
\newblock The chromatic number of any 10-vertex graph without 4-cliques is at
  most 4.
\newblock {\em Comptes rendus de l'Academie bulgare des Sciences}, 37:301--304,
  1984.
\newblock (in Russian).

\bibitem{Nen85}
N.~Nenov.
\newblock Application of the corona-product of two graphs in {R}amsey theory.
\newblock {\em Ann. Univ. Sofia Fac. Math. Inform.}, 79:349--355, 1985.
\newblock (in Russian).

\bibitem{Nen98}
N.~Nenov.
\newblock On the small graphs with chromatic number 5 without 4-cliques.
\newblock {\em Discrete Mathematics}, 188:297--298, 1998.

\bibitem{Nen00}
N.~Nenov.
\newblock On a class of vertex {F}olkman graphs.
\newblock {\em Ann. Univ. Sofia Fac. Math. Inform.}, 94:15--25, 2000.

\bibitem{Nen01b}
N.~Nenov.
\newblock Computation of the vertex {F}olkman numbers ${F}(2, 2, 2, 3; 5)$ and
  ${F}(2, 3, 3; 5)$.
\newblock {\em Ann. Univ. Sofia Fac. Math. Inform.}, 95:71--82, 2001.

\bibitem{Nen01a}
N.~Nenov.
\newblock A generalization of a result of {D}irac.
\newblock {\em Ann. Univ. Sofia Fac. Math. Inform.}, 95:59--69, 2001.

\bibitem{Nen01c}
N.~Nenov.
\newblock {On the 3-coloring vertex Folkman number $F(2, 2, 4)$}.
\newblock {\em Serdica Mathematical Journal}, 27:131--136, 2001.

\bibitem{Nen02a}
N.~Nenov.
\newblock Lower bound for a number of vertices of some vertex {F}olkman graphs.
\newblock {\em Comptes rendus de l'Academie bulgare des Sciences},
  55(4):33--36, 2002.

\bibitem{Nen02b}
N.~Nenov.
\newblock On a class of vertex {F}olkman numbers.
\newblock {\em Serdica Mathematical Journal}, 28:219--232, 2002.

\bibitem{Nen09}
N.~Nenov.
\newblock On the vertex {F}olkman numbers ${F}_v(2,...,2;q)$.
\newblock {\em Serdica Mathematical Journal}, 35:251--272, 2009.
\newblock Preprint: arXiv:0903.3812 March 2009.

\bibitem{Nen10}
N.~Nenov.
\newblock {Chromatic number of graphs and edge {F}olkman numbers}.
\newblock {\em C. A. Acad. Bulg. Sci.}, 63(8):1103--1110, 2010.

\bibitem{Nen07}
N.~Nenov.
\newblock On the vertex {F}olkman numbers ${F}_v(2, ..., 2; r-1)$ and ${F}_v(2,
  ..., 2; r-2)$.
\newblock {\em Ann. Univ. Sofia Fac. Math. Inform.}, 101:5--17, Submitted in
  2007, 2013.
\newblock Preprint: arXiv:0903.3151 March 2009.

\bibitem{NK79}
N.~Nenov and N.~Khadzhiivanov(Hadziivanov).
\newblock {On edgewise 2-colored graphs containing a monochromatic triangle.
  (in Russian)}.
\newblock {\em Serdica}, 5:303--305, 1979.

\bibitem{NK85}
N.~Nenov and N.~Khadzhiivanov(Hadziivanov).
\newblock {Every Ramsey graph without 5-cliques has more than 11 vertices. (in
  Russian)}.
\newblock {\em Serdica}, 11:341--356, 1985.

\bibitem{NR76}
J.~Nesetril and V.~R{\"o}dl.
\newblock {The Ramsey property for graphs with forbidden complete subgraphs}.
\newblock {\em J. Combin. Theory, Ser. B}, 20:243--249, 1976.

\bibitem{PR01}
K.~Piwakowski and S.~Radziszowski.
\newblock {The Ramsey Multiplicity of $K_4$}.
\newblock {\em Ars. Combinatorica}, LX:131--136, 2001.

\bibitem{PRU99}
K.~Piwakowski, S.~Radziszowski, and S.~Urbanski.
\newblock Computation of the {F}olkman number ${F}_e(3, 3; 5)$.
\newblock {\em Journal of Graph Theory}, 32:41--49, 1999.

\bibitem{Rad14}
S.~Radziszowski.
\newblock Small {R}amsey numbers.
\newblock {\em The Electronic Journal of Combinatorics}, Dynamic Survey
  revision 14, January 12 2014.

\bibitem{Rad17}
S.~Radziszowski.
\newblock {Computers in Ramsey Theory. Testing, Constructions and
  Nonexistence}.
\newblock Computers in Scientific Discovery 8 Mons, Belgium, August 24, 2017.

\bibitem{RX07}
S.~Radziszowski and X.~Xiaodong.
\newblock {On the Most Wanted Folkman Graph}.
\newblock {\em Geombinatorics}, XVI(4):367--381, 2007.

\bibitem{RX16}
S.~Radziszowski and X.~Xu.
\newblock On some open questions for {R}amsey and {F}olkman numbers.
\newblock {\em Graph Theory, Favorite Conjectures and Open Problems}, 1:43--62,
  2016.

\bibitem{RXL17a}
S.~Radziszowski, X.~Xu, and M.~Liang.
\newblock {Some Folkman problems. Chromatic vertex Folkman numbers, Existence
  and non-existence, Computational challenges}.
\newblock CanaDAM, Toronto, 13 June, 2017.

\bibitem{RXL17b}
S.~Radziszowski, X.~Xu, and M.~Liang.
\newblock {Some Folkman problems. Existence and non-existence of generalized
  Folkman numbers, Computational challenges}.
\newblock GGTW, Ghent, 16 August, 2017.

\bibitem{Ram30}
P.~Ramsey.
\newblock {On a problem of formal logic}.
\newblock {\em Proc. London Math. Soc.}, 30:264--268, 1930.

\bibitem{RS77}
V.~Rosta and L.~Suranyi.
\newblock {A note on the Ramsy-multiplicity of the circuit}.
\newblock {\em Period. Math. Hung.}, 7:223--227, 1977.

\bibitem{Roy_c}
G.~Royle.
\newblock Combinatorial data,
  \url{http://staffhome.ecm.uwa.edu.au/~00013890/data.html}.

\bibitem{SLHX11}
Z.~Shao, M.~Liang, J.~He, and X.~Xu.
\newblock New lower bounds for two multicolor vertex {F}olkman numbers.
\newblock In {\em International Conference on Computer and Management (CAMAN)},
  pages 1--3, 2011.

\bibitem{SLPX12}
Z.~Shao, M.~Liang, L.~Pan, and X.~Xu.
\newblock Computation of the {F}olkman number ${F}_v(3, 5; 6)$.
\newblock {\em Journal of Combinatorial Mathematics and Combinatorial
  Computing}, 81:11--17, 2012.

\bibitem{SXL09}
Z.~Shao, X.~Xu, and H.~Luo.
\newblock Bounds for two multicolor vertex {F}olkman numbers.
\newblock {\em Application Research of Computers}, 3:834--835, 2009.
\newblock (in Chinese).

\bibitem{SXP09}
Z.~Shao, X.~Xu, and L.~Pan.
\newblock New upper bounds for vertex {F}olkman numbers ${F}_v(3, k; k + 1)$.
\newblock {\em Utilitas Mathematica}, 80:91--96, 2009.

\bibitem{SXP12}
Z.~Shao, X.~Xu, and L.~Pan.
\newblock {Computation of the vertex Folkman number $F_v(3,5;6)$}.
\newblock {\em J. of Comb. Math. and Comb. Computing}, 81:11--18, 2012.

\bibitem{Soi08}
A.~Soifer.
\newblock {\em The Mathematical Coloring Book}.
\newblock Springer, 2008.

\bibitem{Spe88}
J.~Spencer.
\newblock Three hundred million points suffice.
\newblock {\em Journal of Combinatorial Theory, Series A}, 49:210--217, 1988.
\newblock Also see erratum by M.~Hovey in 50:323.

\bibitem{Sza77}
T.~Szab{\'o}.
\newblock {On nearly regular co-critical graphs}.
\newblock {\em Discrete Math.}, 160:279--281, 1977.

\bibitem{Wes01}
D.~West.
\newblock {\em Introduction to Graph Theory}.
\newblock Prentice Hall, Inc., Upper Saddle River, 2 edition, 2001.

\bibitem{XLR17b}
X.~Xu, M.~Liang, and S.~Radziszowski.
\newblock A note on upper bounds for some generalized {F}olkman numbers.
\newblock {\em Discussiones Mathematicae Graph Theory}, in press. Preprint:
  arxiv:1708.00125, August 2017.

\bibitem{XLR17a}
X.~Xu, M.~Liang, and S.~Radziszowski.
\newblock On the nonexistence of some generalized {F}olkman numbers.
\newblock {\em Graphs Combin.}, to appear. Preprint: arxiv:1705.06268, May
  2017.

\bibitem{XLR18}
X.~Xu, M.~Liang, and S.~Radziszowski.
\newblock Chromatic vertex {F}olkman numbers.
\newblock Preprint: arxiv:1612.08136v2, May 2018.

\bibitem{XLS10}
X.~Xu, H.~Luo, and Z.~Shao.
\newblock Upper and lower bounds for ${F}_v(4, 4; 5)$.
\newblock {\em Electronic Journal of Combinatorics}, 17, 2010.

\bibitem{XS10}
X.~Xu and Z.~Shao.
\newblock On the lower bound for ${F}_v(k, k; k + 1)$ and ${F}_e(3, 4; 5)$.
\newblock {\em Utilitas Mathematica}, 81:187--192, 2010.

\end{thebibliography}

\end{document}